%% file: main.tex
\documentclass[10pt,oneside,a4paper]{amsart}
\usepackage{graphicx}
\input{preamble}

\title{Super \texorpdfstring{$K$}{K}-theory and group completion}
\author{David Aretz}
\author{Luuk Stehouwer}

\begin{document}

\maketitle

\begin{abstract}
    We develop a spectrum-level graded $K$-theory for real super Banach algebras.
    Our construction is categorical and homotopy theoretic, in the style of algebraic $K$-theory: the graded $K$-theory spectrum is obtained by a (co)fiber sequence from the $(\infty,1)$-categorical group completion of topological groupoids of finitely generated projective graded modules, rather than from spaces of Fredholm operators or Kasparov cycles.
    We define a connective spectrum $k^{\ABS}_A$ refining the Atiyah--Bott--Shapiro construction as a cofiber, together with its periodification $K^{\mathrm{gr}}_A$, and show that both are lax symmetric monoidal and functorial in bimodules, not merely in homomorphisms.
    The failure of graded $A$-modules to present all cocycles in $K^{\mathrm{gr}}_0(A)$ is shown to be purely a $\pi_0$-phenomenon on $k^{\ABS}_A$.
    We obtain a natural equivalence $K^{\mathrm{gr}}_{A\grotimes \Cl_{p,q}} \simeq \Sigma^{p-q}K^{\mathrm{gr}}_A$, which links topological Bott periodicity with the Morita equivalence between $\Cl_8$ and $\R$.
    Restricting to invertible finite-dimensional semisimple super algebras yields a symmetric monoidal functor $\Pic(\Bim(\sBan_{\R})^{\mathrm{fd}}) \to \Pic(\ModInf(KO))$ which splits off the bottom three Postnikov layers of $\Pic(\ModInf(KO))$, giving a direct link between super division algebras and invertible $KO$-modules.
    We also give spectral refinements of Karoubi's and van Daele's graded $K$-groups, with explicit comparison equivalences, therefore connecting to $KK$-theory.
    We also provide an extensive general treatment for $K$-theory of ungraded topological rings that might be of independent interest. In particular, we characterize connective topological $K$-theory of ungraded Banach algebras by a universal property.
\end{abstract}

\tableofcontents

\newpage
\section{Introduction}
\label{sec:intro}
Every existing construction of the $K$-theory spectrum $K^{\gr}_A$ of a graded $C^*$-algebra $A$ is built from Fredholm operators or Kasparov cycles on infinite rank Hilbert modules.
The purpose of this article is to provide a purely categorical homotopy theoretic framework that functorially builds $K^{\gr}_A$ for any graded Banach algebra $A$ in the style of algebraic $K$-theory without any functional analysis.

\subsubsection*{Topological \texorpdfstring{$K$}{K}-theory without analysis}
We construct the $K$-theory of ungraded Banach algebras $A$ through the lens of algebraic $K$-theory, obtaining a blueprint for the graded theory.
That a group completion definition of topological $K$-theory exists at all is not widely known, see for example \cite{MOsegalmachine, MOhenriquesKspace, MOkrausektheory}.
However, with some care (Warning~\ref{ex:wrong}) we can define the \emph{connective $K$-theory spectrum $k_A$ of $A$} as the (higher) group completion $k(\ModEn_A)$ of a suitable monoid of finitely generated projective modules under direct sum.
More precisely, $k_A$ is the group completion of $N^{hc}(\ModEn_A^{\cong},\oplus)$, the $\E_\infty$-monoid obtained from taking the homotopy coherent nerve of the $\sSet$-enriched groupoid $\ModEn_A^{\cong}$ coming from the Banach space norm on module maps (Definition~\ref{def:KthBanAlg}).
These are modules in the algebraic sense; we never need Hilbert bimodules or $C^*$-correspondences in this article.

Since our setup is far more general, it immediately applies to topological rings (Section~\ref{sec:sSetKth}) and simplicially enriched module categories (Section~\ref{sec:enrichKth}), and we expect adaptations to bornological or condensed~\cite{aoki2024semi} settings to be straightforward.
However, Banach algebras are close enough to topology that the key features of topological $K$-theory become categorical statements about module categories: the rigid behavior of their idempotents yields a Serre--Swan theorem (Proposition~\ref{prop:serreswan}), and by the open mapping theorem surjective homomorphisms behave like fibrations (Example~\ref{ex:serrefunctor}), which is the source of excision (Corollary~\ref{cor:excision}).
All our results apply to Banach algebras over both $\mathbb{F} = \R$ and $\mathbb{F} = \C$, though our main focus will be $\mathbb{F} = \R$ for definiteness.

\subsubsection*{The Atiyah--Bott--Shapiro construction}

The paper~\cite{ABS} of Atiyah, Bott and Shapiro is the source of the influential idea that higher $K$-theory classes can be represented by bundles of Clifford modules.
More precisely, for a compact Hausdorff space $X$ there is a map we call the \emph{ABS construction}
\begin{equation}
    \label{eq:ABSconstruction}
    t\colon \mathrm{mod}_{\Cl_{-n}}(X)/\mathrm{mod}_{\Cl_{-(n+1)}}(X) \to KO^n(X) \,,
\end{equation}
where $\mathrm{mod}_{\Cl_{-n}}(X)$ is the monoid of bundles of finitely generated projective modules over the Clifford algebra (Definition~\ref{def:clifford}) under direct sum $\oplus$. 
The subtlety that $t$ is \emph{not} an isomorphism for general $X$ (Example~\ref{ex:circle})
has been repeatedly observed~\cite{MR2490588}~\cite[second remark on p.~8]{andreK-theory}~\cite[Section~2.1]{thiang}~\cite[Appendix~A]{bourne2026analyticindextheoryspectral}---yet has never been properly explained.
We explain this failure by showing this is purely a $\pi_0$-phenomenon, and use it as a motivation to develop $K$-theory of graded algebras in terms of their module categories.

In more detail, analogously to $k_A$, let $k^{C_2}_A =k(\ModEn^{C_2}_A)$ be the group completion of finitely generated projective graded $A$-modules (Proposition~\ref{prop:kC_2}).
Define the connective graded $K$-theory spectrum $k^{\ABS}_A$ for every graded Banach algebra $A$ as the \emph{cofiber} of the map $k^{C_2}_{A \grotimes \Cl_{-1}} \to k^{C_2}_A$ given by forgetting the $\Cl_{-1}$-action (Definition~\ref{def:ABSKth}).
Here $\grotimes$ is the graded tensor product~\eqref{gradedtensoralgebra} and we have $k^{\ABS}_A \cong k_A$ if $A$ is ungraded (Example~\ref{example:ungraded}).
We define a map of connective spectra
\begin{equation}
    \label{eq:ABSpectra}
    t\colon k_{A \grotimes \Cl_{-n}}^{ABS}\to \Omega^n k^{ABS}_A\,.
\end{equation}
which recovers \eqref{eq:ABSconstruction} on $\pi_0$ for $A = C(X)$.
We prove that $\pi_i(t)$ is an isomorphism for all $i>0$, providing in addition a single obstruction for $\pi_0(t)$, and hence $t$, to be an isomorphism~(Corollary~\ref{cor:quotientcondition}).
This is a form of Bott periodicity that connects it (non-tautologically) to algebraic Bott periodicity.

\subsubsection*{The product structure}

Topological $K$-theory of a space is a ring, and more generally we prove that there is a product structure 
\begin{equation}
\label{eq:ungrproduct}
k_A \otimes k_B \to k_{A \otimes B}     
\end{equation} 
for ungraded Banach algebras $A$ and $B$ (Theorem~\ref{thm:omnibusktheory}, known for $C^*$-algebras~\cite{MR3788857}). 
Here we used the tensor product in the $(\infty,1)$-category $\Sp$ of spectra and the completed projective tensor product of Banach algebras~\cite{Ryan2002}, see Section~\ref{sec:Banach}.
The precise choice of tensor product on the algebra side is not too important here---other tensor products, such as $\otimes_{\mathrm{min}},\otimes_{\mathrm{max}}$ from the $C^*$-world, work just as well, since we don't claim the lax monoidal structure \eqref{eq:ungrproduct} is strong.

Relating graded tensor products of modules to a multiplication on spectra is more delicate: $k^{\ABS}$ is not obviously lax monoidal.
We resolve this by exhibiting the algebra homomorphism $\R \to \Cl_{-1}$ as a Smith ideal (Definition \ref{def:smith}) in $\sBanAlgInf^{\op}$ in Section~\ref{sec:smith}.
For this, it is important to realize the Smith ideal as the image of $\R^0 \to \R^1$ under the Clifford functor $\Cl$ of Proposition~\ref{prop:Clfunctor}.
This functor maps the $(\infty,1)$-category of finite-dimensional vector spaces and injections to the $(\infty,1)$-category $\BanAlgInf$ of Banach algebras and homomorphisms obtained by localizing the $1$-category $\Alg(\Ban)$ at homotopy equivalences (Proposition~\ref{prop:localizationofBanAlg}).
One of our main contributions is to use the Smith ideal machinery to construct a spectrum-level product structure $k_A^{\ABS} \otimes k_B^{\ABS} \to  k_{A \grotimes B}^{\ABS}$~(Theorem~\ref{thm:abslaxmonoidal}).
The core idea of our approach is present but not fully worked out in \cite{TheMscThesis}, and is closely connected to Weiss' orthogonal calculus~\cite{WeissOrthogonalCalculus} as observed by Rezk~\cite{MOrezkclifford}.

\subsubsection*{Periodic graded \texorpdfstring{$K$}{K}-theory}

Taking $B = \R$ and using $k^{ABS}_\R = ko$ (Examples~\ref{example:ungraded} and \ref{ex:ku}), we obtain a $ko$-module action on $k_A^{ABS}$.
We can thus define the \emph{periodic $K$-theory spectrum $K^{\mathrm{gr}}_A := k_A^{ABS}[\beta^{-1}]$} by inverting the action by the Bott class $\beta \in \pi_8 ko$ (Definition~\ref{def:periodick}).
Even though $k_A^{ABS}$ does not satisfy excision, we show that $K^{\mathrm{gr}}$ does~(Theorem~\ref{thm:periodicKtheorylax}) and $k^{\ABS}\to K^{\mathrm{gr}}$ is a $\pi_{\geq 1}$-isomorphism.

\subsubsection*{Morita functoriality}

One advantage of our setup is that (graded) $K$-theory is automatically functorial for bimodules that are finitely generated and projective on one side.
In particular, graded $K$-theory is Morita invariant and gives wrong way maps for certain algebra homomorphisms.

In more detail, we start with the very general construction $\Mor(\sBan)$ of the Morita $(2,2)$-category (Appendix~\ref{app:algebras}) in which
\begin{itemize}
    \item objects are graded Banach algebras
    \item $1$-morphisms are graded Banach bimodules
    \item $2$-morphisms are grading-preserving bimodule maps
\end{itemize}
We use this to construct the Morita $(\infty,1)$-category $\sBimBanInf$ (Definition~\ref{def:BimBaninfty}) of graded Banach algebras by
\begin{enumerate}
    \item restricting to $1$-morphisms that admit a left adjoint, or equivalently are finitely generated and projective over the target;
    \item restricting to invertible $2$-morphisms, resulting in the $(2,1)$-category $\BimBanOne$;
    \item localizing $\BimBanOne$ with respect to bimodule homotopies.
\end{enumerate}
We construct graded $K$-theory as a functor out of $\sBimBanInf$, so that bimodules gives maps of $K$-theory spectra while both bimodule isomorphisms as well as bimodule homotopies give homotopies of spectrum maps.
Even though bicategories of $C^*$-algebras, von Neumann algebras and bornological algebras have been constructed~\cite{MR1327199,MR3056650, MR1972774,MR1867561,ARETZ2026111618}, the specific higher categorical structure on Banach algebras we need does not appear in the literature.

\subsubsection*{The role of invertible superalgebras in \texorpdfstring{$K$}{K}-theory}

The observation that the Brauer--Wall group $\Z/8$ of central real graded division algebras~\cite{MR167498} and the periodicity of $KO$ are related is at least 60 years old.
A full comparison between the whole Picard $2$-groupoid of Morita-invertible finite-dimensional graded algebras $\Pic(\Mor(\sVectOne))$ and the Picard spectrum $\Pic(\ModInf(KO))$ of invertible $KO$-module spectra had been made abstractly~\cite{beardsley2023brauerwallgroupstruncatedpicard,FreedTwistedKTheoryOrientifolds} but never by a functor.
We produce a functor $\Pic(\Mor(\sVectOne)) \to \Pic(\ModInf(KO))$ simply by restricting $K^{\mathrm{gr}}$ itself, on which it becomes strong symmetric monoidal~(Corollary~\ref{cor:strongmonoidalfd}).
This identifies $\Pic(\Mor(\sVectOne))$ with a $2$-truncation of $\Pic(\ModInf(KO))$, providing a direct link between graded division algebras and invertible $KO$-modules~(Theorem~\ref{thm:picardgroupoid}).
This explains the connection with twisted $K$-theory through bundles of invertible superalgebras in homotopical language.

\subsection{Main results}
Our main results are highlighted as Theorems~A, B and C below.
In brief, Theorem~A characterizes connective topological $K$-theory of ungraded Banach algebras as the initial homotopy invariant functor under algebraic $K$-theory and establishes its basic properties: lax symmetric monoidality, excision, and preservation of filtered colimits.
Theorem~B extends this to the graded Morita setting: the functors $k^{\ABS}$ and $K^{\mathrm{gr}}$ are lax symmetric monoidal on $\sBimBanInf$, homotopy invariant, excisive, preserve filtered colimits and they satisfy spectral Bott periodicity.
Theorem~C identifies the Picard $\infty$-groupoid of finite-dimensional semisimple graded algebras with a $2$-truncation of $\Pic(\ModInf(KO))$.

\begin{theoremA*}\hypertarget{thmA}{}
    The connective topological $K$-theory functor
    \begin{equation*}
        k\colon \Alg(\Ban) \to \Sp_{\geq 0}
    \end{equation*}
    is initial among homotopy invariant functors equipped with a natural transformation $\mathcal{K}^{\mathrm{alg}}\to k$ from algebraic $K$-theory.
    The functor $k$ is lax symmetric monoidal.
    It satisfies excision and preserves filtered colimits along contractive homomorphisms.
    Topological $K$-theory can be computed as
    \begin{align*}
        k_A \simeq \colim_{\Delta^{\op}} \mathcal{K}^{\mathrm{alg}}(C(\Delta^n;A)) \,.
    \end{align*}
\end{theoremA*}
These results also extend to nonunital Banach algebras (Section~\ref{sec:nonunital}).

\begin{theoremB*}\hypertarget{thmB}{}
    $k^{\ABS}$ and $K^{\mathrm{gr}}$ refine to lax symmetric monoidal functors
    \begin{align*}
        k^{\ABS}\colon \sBimBanInf \to \Sp_{\geq 0} \,, \qquad K^{\mathrm{gr}}\colon \sBimBanInf \to \ModInf(KO) \,,
    \end{align*}
    such that the natural transformations $k^{C_2}\to k^{\ABS}\to K^{\mathrm{gr}}$ are monoidal.
    The restriction of $k^{\ABS}$ to ungraded Banach algebras coincides with the connective $K$-theory functor $k$.
    The functors are homotopy invariant and preserve filtered colimits along contractive homomorphisms, and $K^{\mathrm{gr}}$ is excisive.
    There is a natural equivalence of $KO$-modules 
    \begin{equation}
        \label{eq:Bottper}
            K^{\mathrm{gr}}_{A\grotimes \Cl_{p,q}} \simeq \Sigma^{p-q} K^{\mathrm{gr}}_A \,.        
    \end{equation}
\end{theoremB*}
Equation~\eqref{eq:Bottper} connects \emph{algebraic} Bott periodicity (the statement that $\Cl_8$ is Morita equivalent to $\R$) to \emph{topological} Bott periodicity ($\Sigma^8KO \simeq KO$).
This is proven in a connective version in Theorem~\ref{thm:gradedbottperiodicity} and is essentially a spectral reformulation of strong Bott periodicity.
For this approach, the Morita functoriality in our setup is crucial.

We prove a splitting result of Picard spectra as a corollary of our setup:
\begin{theoremC*}\hypertarget{thmC}{}
    The restriction of $K^{\mathrm{gr}}$ to the Picard $\infty$-groupoid of $\grotimes$-invertible finite-dimensional semisimple graded algebras is a symmetric monoidal functor
            \begin{equation*}
                \Pic(\Bim(\sBan_{\R}))^{\mathrm{fd}}) \longrightarrow \Pic(\ModInf(KO)) \,.
            \end{equation*}
        The map is an isomorphism on $\pi_0,\pi_1,\pi_2$ and $\Pic(\Bim(\sBan_{\R})^{\mathrm{fd}})$ is $2$-coconnective.
        Hence, $K^{\mathrm{gr}}$ witnesses a splitting of Picard $\infty$-groupoids:
        \begin{equation*}
            \Pic(\ModInf(KO)) \simeq \Pic(\Bim(\sBan_{\R})^{\mathrm{fd}}) \times \tau_{\geq 3}\Pic(\ModInf(KO))\,.
        \end{equation*}
\end{theoremC*}

We now summarize other results that we have not discussed so far.

Theorem \ref{thm:fibergroupcompletion} expresses fibers of group completions as cofibers, which could be of independent interest.
We use it to connect with Karoubi's definitions of relative $K$-theory of a Banach functor, and also in our proof of Bott periodicity.

We prove a K\"{u}nneth theorem for a large class of graded Banach algebras~(Corollary~\ref{cor:periodicKunneth}), a general graded Wood fiber sequence~(Theorem~\ref{thm:wood}), and we recover the iterated loop spaces of $ko$ from the representation theory of Clifford algebras in Section~\ref{sec:Bottclock}.

We construct spectral refinements of both van Daele~\cite{MR947500,MR961241} and Karoubi~\cite{Karoubi68} $K$-theory (Sections \ref{sec:vD} and \ref{sec:karoubikt-h}).
We reconcile these spectra and several other sign conventions systematically with our main approach in Section \ref{sec:atlas}, and summarize them in Table~\ref{table:mastercomparisonintro}---the first place these are all lined up.

\begin{table}[ht]
\mastercomparisonbody
\caption{The comparison of conventions for graded $K$-theory proved in Section~\ref{sec:atlas}. The rows record whether one takes a fiber or a cofiber, and whether the defining map is induced covariantly or contravariantly; the columns record whether one adjoins $\Cl_{+1}$ or $\Cl_{-1}$. All eight entries are equivalent to a shift of $K^{\mathrm{gr}}_A$ or of $K^{\mathrm{gr}}_{A^{\op}}$.}
\label{table:mastercomparisonintro}
\end{table}

\subsubsection*{Comparison with the literature}
Kasparov's $KK$-theory~\cite{MR582160} incorporates graded $C^*$-algebras from the outset, but it is functional analytic in nature and group-valued.
The modern approaches to $KK$-theory through stable $(\infty,1)$-categories~\cite{bunkeKKEth,bunke2023survey} by default only treat ungraded algebras.
In these frameworks gradings are implemented as a $C_2$-action, which does not encode the Koszul sign rule; e.g.\ Clifford algebras are not multiplicative in $KK^{C_2}$.
Bott periodicity is an analytic input rather than a theorem. The connection to projective modules or spaces of projections is proven a posteriori and often highly involved.
Bunke on the other hand proposes graded $KK$-theory as a category of comodules over $\mathcal{S} = C_0(\R)$ in~\cite{bunkenoncommhom2}, which rests on Trout's reformulation~\cite{MR1775323} of $KK$-theory.
Haag~\cite{MR1694805} develops a $\Z/2$-graded analogue of Cuntz's picture~\cite{MR733641, MR899916}, identifying the graded $\Ext$-functor with $C_2$-equivariant $KK$ up to a degree shift.

Spectrum-level refinements of periodic operator $K$-theory have been obtained by analytic means.
Bunke, Joachim and Stolz~\cite{MR2048716} use the Fredholm and the unbounded pictures of Kasparov theory to construct classifying spaces for the $K$-theory of a $\Z/2$-graded $\sigma$-unital $C^*$-algebra; the unbounded picture yields an orthogonal spectrum, functorial for \emph{essential} $*$-homomorphisms.
In~\cite{MR1808222}, Joachim constructs a symmetric ring spectrum representing $KO$ out of spaces of homomorphisms of $\Z/2$-graded $C^*$-algebras and their Clifford algebras, so that the $K$-theory spectra of graded $C^*$-algebras become module spectra over it.
Joachim and Stolz~\cite{MR2545610} lift the additive enrichment of the Kasparov category to an enrichment over symmetric spectra, working in Cuntz's picture of $KK$.
Dell'Ambrogio, Emerson, Kandelaki and Meyer~\cite{dell2011functorial} construct a lax symmetric monoidal functor from graded ($G$-equivariant) $C^*$-algebras to symmetric module spectra over a commutative symmetric ring spectrum; their model is close to that of~\cite{MR2048716} but is functorial for \emph{all} $*$-homomorphisms, and it rests on Trout's description~\cite{MR1775323} of $\hat{K}_0$ of a graded $C^*$-algebra by graded $*$-homomorphisms out of $\mathcal{S}$.
Closest in spirit to the present paper is~\cite{MR2122155}, where Joachim builds $G$-equivariant $K$-theory for a compact Lie group $G$ as an orthogonal $G$-spectrum indexed on inner product spaces, again out of $\Z/2$-graded $C^*$-algebras.
In such orthogonal models Bott periodicity is largely built into the indexing rather than established as a theorem;
in our setting it is instead proved directly (\ThmB).
Independently from a different vantage point~\cite{anupam} recovers our fiber and cofiber sequences.
Focusing on the module categories instead, our approach to $K^{\mathrm{gr}}$ is functorial in suitable bimodules, not just in homomorphisms---a feature which is only recognized a posteriori in all modern treatments of $KK$-theory.

Beyond the $C^*$-setting, bivariant $K$-theories have been developed for locally convex and bornological algebras by Cuntz, Meyer and Rosenberg~\cite{MR2340673};
Grensing~\cite{grensing2013noncommutative} reconstructs $KK$ and $E$-theory by stable homotopy theoretic methods with such generalizations in view.
Further homotopical models for ungraded operator $K$-theory are given by solid or condensed spectra~\cite{aoki2024semi} and by sheaves of spectra in the differential setting~\cite{zbMATH06558447}.
In \cite{KamelCoherentABS} a different method is used to realize the graded tensor product of Clifford modules coherently as the multiplication $\Omega^k ko\times \Omega^l ko\to \Omega^{k+l}ko$.

\subsubsection*{Relation to physics}
 Symmetry-protected topological
(SPT) phases of free fermions are classified by $K$-theory groups~\cite{kitaev2009periodic}.
In particular, Freed--Moore used twisted equivariant $K$-theory~\cite{MR3119923,gomi_freed-moore_2021}, and later accounts used the generalization to graded $K$-theory of real $C^*$-algebras~\cite{thiang, MR4121611}. Explicit Hamiltonian quantum systems representing a given topological phase are most easily described in Karoubi's (or van Daele's~\cite{MR3665214}) picture as opposed to a Fredholm picture.
More specifically, gapped free-fermion systems protected by a
$\Z/2$-graded symmetry algebra $A$ (with the grading recording time-reversing
symmetries) form the group $K_2^{\mathrm{gr}}(A)$~\cite{stehouwer2025free}. 

Several results in condensed matter such as the unification of superficially different tenfold-way schemes implicitly rely on  Morita invariance of graded $K$-theory.
For us this is built in, since $K^{\mathrm{gr}}$ is a functor from the Morita $(\infty,1)$-category
$\sBimBanInf$.
Bimodule functoriality also supplies the wrong-way maps used for the necessary equivariant transfer arguments.

Freed and Hopkins proposed an approach to classify interacting SPT phases using invertible unitary TQFTs~\cite{MR4268163}.
In that setup, the comparison between free and interacting SPT phases proceeds through their
free-to-interacting map, the Anderson dual of an
Atiyah--Bott--Shapiro orientation twisted by one of the ten symmetry
types of the tenfold way.
The domain of this map is exactly the group our $\Sigma^2 K_A^{\mathrm{gr}}$
spectrifies, and the ABS orientation is a necessary input.
In separate work, the first author will construct a spin orientation of $KO$ within
the present framework, without recourse to the usual Fredholm operator type
analysis.
This construction is directly relevant to spectrum-level formulations
of the free-to-interacting map and sheds light on formulating it for general symmetry algebras, a problem on which partial progress was made in~\cite{debray2026unravelingbottspiral}.

\subsection{Notational conventions}
We will freely use the language of $(\infty,1)$-categories as developed in Lurie's works~\cite{HTT,HA}.
All concepts are by default homotopical, e.g.\ $\colim$ always refers to the (homotopy) colimit in the $(\infty,1)$-categorical sense.

Throughout this document, we utilize the following standard notations and conventions.
We also refer to Appendices~\ref{sec:enriched} and~\ref{app:categories} for more details.

\vspace{0.3cm}
\noindent\textbf{Typefaces for categories.}
We distinguish three kinds of category by typeface:
\begin{itemize}
    \item $\mathrm{C}$, upright roman: a $1$-category, $(2,2)$-category or $(2,1)$-category. For example $\Alg(\Ban)$, $\ModOne_A$, $\Mor(\Ban)$.
    \item $\mathsf{C}$, sans-serif: an enriched category. For example $\BanAlgEn$, $\ModEn_A$, $\VectEn$.
    \item $\mathrm{C}_\infty$, upright roman with a subscript $\infty$: an $(\infty,1)$-category obtained from a named enriched category. For example $\BanAlgInf$, $\BimBanInf$, $\InnInf$.
    \item $\calC$, calligraphic: a generic, unnamed $(\infty,1)$-category.
\end{itemize}
Two classes of exception are worth recording.
$(\infty,1)$-categories that have no $1$-categorical analogue we ever wish to consider carry no subscript: $\Sp$, $\Spc$, $\CMon(\calC)$, and $\Cat_\infty$, whose subscript is part of its standard name.
Likewise $\ModInf(R)$, the $(\infty,1)$-category of modules over an algebra $R$ in some $(\infty,1)$-category (e.g.\ a ring spectrum) is distinguished from the $1$-category $\ModOne_A$ of Definition~\ref{def:fgpmod} in that it is written in parentheses rather than as a subscript.

\vspace{0.3cm}
\noindent\textbf{Enrichment.}
Unless stated otherwise, a sans-serif category $\mathsf{C}$ is enriched in $\Ban$.
Every $\Ban$-enriched category is regarded as $\Top$-enriched via the norm topology on its mapping spaces, and as $\sSet$-enriched via $\Sing$ (Conventions \ref{conv:topenr} and \ref{conv:banenr}); both functors are suppressed from the notation.
Thus $\ModEn_A$ denotes the Banach category of Definition~\ref{def:fgpmod}, and the same symbol is used for its underlying $\Top$- and $\sSet$-enriched categories whenever a construction requires one of these.
The passage to an $(\infty,1)$-category is written $N^{\mathrm{hc}}(\mathsf{C})$.

\vspace{0.3cm}
\noindent\textbf{Monoids}
We will use the terms `commutative monoid' and `$\E_\infty$-monoid' interchangeably.
We will often implicitly identify grouplike commutative monoids in $\Spc$ with connective spectra.

\vspace{0.5cm}
\noindent\textbf{Categories and $\infty$-Categories} \vspace{0.2cm} \\
\noindent
\begin{tabularx}{\textwidth}{@{} p{3.7cm} X @{}}
    $\Sp, \Sp_{\geq 0}$ & The $(\infty,1)$-category of spectra and the full subcategory of connective spectra equipped with the symmetric monoidal smash product $\otimes$. \\
    $\Spc$ & The $(\infty,1)$-category of spaces. \\
    $\CMon(\mathcal{C}), \CGrp(\mathcal{C})$ & The $(\infty,1)$-categories of commutative monoids and grouplike commutative monoids in a symmetric monoidal $(\infty,1)$-category $\mathcal{C}$. \\
    $\Cat_\infty^\Sigma$ & The $(\infty,1)$-category of small $(\infty,1)$-categories that admit finite coproducts. \\
    $\Cat_1(\mathcal{V})$ & The $(2,1)$-category of $\mathcal{V}$-enriched categories, where $\mathcal{V}$ is a symmetric monoidal $1$-category.\\
    $\Top$ & The category of compactly generated weak Hausdorff spaces. \\
    $\sSet$ & The category of simplicial sets. \\
    $\Ban, \Ban_\leq$ & The categories of Banach spaces with bounded linear maps and short maps (contractions), respectively. \\
    $\Alg(\Ban), \sBanAlgOne$ & The $1$-categories of Banach algebras and graded Banach algebras.\\
    $\BanAlgInf, \sBanAlgInf$ & The $(\infty,1)$-categories of Banach algebras and graded Banach algebras and algebra homomorphisms. \\
    $\BimBanInf, \sBimBanInf$ & The Morita $(\infty,1)$-categories of Banach algebras (resp. graded Banach algebras) and bimodules.
\end{tabularx}

\vspace{0.5cm}
\noindent\textbf{Algebras and Modules} \vspace{0.2cm} \\
\noindent
\begin{tabularx}{\textwidth}{@{} p{3.5cm} X @{}}
    $\ModOne_A$ & The 1-category of finitely generated projective $A$-modules (Definition~\ref{def:fgpmod}). \\
    $\ModEn_A$ & The same category regarded as a Banach category, and, via the conventions above, as a $\Top$- or $\sSet$-enriched category. \\
    $\ModOne_A^{C_2}, \ModEn_A^{C_2}$ & The corresponding categories of finitely generated projective graded modules over $A$. \\
    $\ModInf(R)$ & The $(\infty,1)$-category of modules over a ring spectrum $R$. \\
    $\Cl_{\pm n}$ & The real Clifford algebra on $\R^n$ with the quadratic form $\pm\langle -, - \rangle$. \\
    $\Cxl_n$ & The complex Clifford algebra on $\C^n$. \\
    $|A|$ & The underlying ungraded algebra of a graded algebra $A$. \\
    $A^{\op}$ & The opposite algebra, equipped with the Koszul sign rule for graded algebras. \\
    $A^{\flop}$ & The ``flopposite'' algebra, defined as the ungraded opposite of a graded algebra.
\end{tabularx}

\vspace{0.5cm}
\noindent\textbf{Functors and Operations} \vspace{0.2cm} \\
\noindent
\begin{tabularx}{\textwidth}{@{} p{3.5cm} X @{}}
    $N^{\mathrm{hc}}$ & The homotopy coherent nerve functor. \\
    $(-)^{\gp}$ & The group completion functor. \\
    $\Pic(\calC)$ & The Picard $\infty$-groupoid of $\otimes$-invertible objects in a symmetric monoidal $(\infty,1)$-category $\calC$, regarded as a connective spectrum (Definition~\ref{def:picard}). \\
    $\otimes_\pi$ & The (completed) projective tensor product of Banach spaces. \\
    $\grotimes$ & The graded tensor product of graded vector spaces or graded algebras, incorporating the Koszul sign rule $\beta(v \otimes w) = (-1)^{|v||w|} w \otimes v$. \\
    $\Pi$ & The parity shift functor for graded modules, defined by tensoring with the odd line $\Pi \C$.
\end{tabularx}

\vspace{0.5cm}
\noindent\textbf{$K$-theory} \vspace{0.2cm} \\
\noindent
\begin{tabularx}{\textwidth}{@{} p{3.5cm} X @{}}
    $\mathcal{K}^{\mathrm{alg}}$ & Algebraic $K$-theory~\eqref{eq:algebraicktheory}. \\
    $k_A$ & The connective $K$-theory spectrum of a Banach algebra $A$ (Definition~\ref{def:KthBanAlg}); $K_A:=k_A[\beta^{-1}]$ is its periodic version. \\
    $k^{C_2}_A$ & The $K$-theory $k(\ModEn_A^{C_2})$ of finitely generated projective \emph{graded} $A$-modules; $K_A^{C_2}:=k_A^{C_2}[\beta^{-1}]$. \\
    $k^{\ABS}_A$ & Connective graded (Atiyah--Bott--Shapiro) $K$-theory of a graded Banach algebra $A$ (Definition~\ref{def:ABSKth}). \\
    $K^{\mathrm{gr}}_A$ & Periodic graded $K$-theory, $k^{\ABS}_A[\beta^{-1}]$ (Definition~\ref{def:periodick}). \\
    $k^{\Kar}_A, K^{\Kar}_A$ & Connective and periodic Karoubi $K$-theory (Definition~\ref{def:Karoubispectral}). \\
    $ko, ku, KO, KU$ & The connective and periodic real and complex topological $K$-theory spectra. \\
\end{tabularx}

\subsubsection*{Artificial intelligence disclosure statement}
The mathematical content of this paper---the definitions, constructions, theorems and their proofs---is our own.
In the later stages of the project we used large language models, principally Claude (Anthropic), Gemini (Google) and ChatGPT (OpenAI) for literature search, drafting expository material,\footnote{The sections \ref{sec:enriched} and \ref{app:algebras} of the Appendix are revised Claude-written text based on an earlier draft.} reorganizing existing text, proof exploration, finding typos, enforcing notational consistency, and as an additional independent verification of correctness.
We have verified every mathematical statement in this paper ourselves and take full responsibility for its correctness.

\subsubsection*{Acknowledgements}

Both authors express their thanks to Ulrich Bunke for helpful input, to Shachar Carmeli for openly sharing his closely related ideas on Smith ideals, to Anupam Datta for explaining several important subtleties in defining operator space topologies, and to Peter Teichner for his interest and feedback.
David also thanks Christian Blohmann for helpful discussions.
Luuk would like to thank Cameron Krulewski, Lukas M\"uller and Natalia Pacheco-Tallaj for helpful discussions, and Thomas Nikolaus for asking a question that motivated this line of research.

Luuk is supported by ERC Consolidator Grant SYMSPEC.
We thank the Max Planck Institute for Mathematics in Bonn for its hospitality; most of this work was carried out there.

\section{Ungraded \texorpdfstring{$K$}{K}-theory by group completion}
\label{sec:ungradedkthy}

The goal of this section is to construct the connective topological $K$-theory of ungraded Banach algebras using modern homotopical group completion technology.
Several of the topics discussed repackage known results, and serve as a blueprint for the
graded theory of Section~\ref{sec:graded kthy}.

We begin in Section~\ref{sec:algKth} by recalling direct sum algebraic
$K$-theory $K(R) = (\ModOne_R^{\cong})^{\gp}$ and its lax
symmetric monoidal and Morita functoriality following~\cite{Gepner2015}.
Section~\ref{sec:groupcompletion} develops the group completion
toolkit following~\cite{nikolaus_groupcompletion} on which the rest of the paper relies.
\emph{Telescopic} $\E_\infty$-monoids play an important role, for which the group completion
is computed by a mapping telescope.
In the telescopic case, we identify the fiber of the group completion of a quasi-surjective map
$f\colon M \to N$ with an $\E_\infty$-refinement of Karoubi's
$K$-theory of $f$ (Theorem~\ref{thm:fibergroupcompletion}). The latter
result is what will later allow us to recognize the $K$-theories of
Karoubi and van Daele as the homotopy groups of our spectra.

We extend $K$-theory to simplicial rings and 
simplicially enriched categories in
Section~\ref{sec:kthysimplicial} by geometric realization.
We set up our conventions for
categories of Banach algebras in Section~\ref{sec:Banach}: the
symmetric monoidal $(2,1)$-category $\BimBanOne$ of Banach algebras,
one-sided finitely generated projective bimodules and bimodule
isomorphisms, as well as the $(\infty,1)$-category $\BanAlgInf$,
which we characterize as the localization of Banach algebras at the
homotopy equivalences. 
In Section~\ref{sec:BanachKth} we give two
definitions of the $K$-theory spectrum $k_A$ of a Banach algebra
$A$: as the group completion of the simplicially enriched groupoid
of finitely generated projective modules, and as the $K$-theory of
the simplicial ring $C(\Delta^\bullet; A)$. We prove they agree
via a spectrum-level Serre--Swan theorem
(Proposition~\ref{prop:serreswan}). A key simplification over the
algebraic case is that $\ModOne_A^{\cong}$ is telescopic, so that its
group completion is $K_0(A) \times B\GL_\infty(A)$. We characterize
$k$ as the initial homotopy invariant functor equipped with a
transformation from algebraic $K$-theory
(Section~\ref{sec:kan}), and show it satisfies excision,
preserves filtered colimits along contractive homomorphisms, and
extends to nonunital algebras
(Sections~\ref{sec:excision}--\ref{sec:nonunital}). 
Finally, Section~\ref{sec:moritafunctoriality} upgrades $k$ to a lax symmetric monoidal
functor $\BimBanInf \to \Sp_{\geq 0}$ on the Morita
$(\infty,1)$-category $\BimBanInf= \BimBanOne[W^{-1}]$, obtained by
localizing at the bimodule homotopy equivalences.
We view this Morita category as an important contribution that is mid-way between purely algebraic statements and bivariant $K$-theory; we observe that $\Hom_{\BimBanInf}(C(X),\R)$ group completes to \emph{$K$-homology} of $X$ in Theorem~\ref{th:segal}.

\subsection{Recollections on algebraic \texorpdfstring{$K$}{K}-theory}
\label{sec:algKth}
Starting with a discrete ring $R$, one considers the category $\ModOne_R$ of finitely generated projective left modules.
This category admits direct sums, which supplies a symmetric monoidal structure $\oplus$.
The maximal subgroupoid $\ModOne_R^{\cong}$ carries the induced structure of a commutative monoid.
Recall that an $\E_\infty$-monoid $M$ is \emph{grouplike} if the commutative monoid $\pi_0M$ is a group.
The algebraic $K$-theory of $R$ is defined as the group completion (cf.\ Section~\ref{sec:groupcompletion})
\begin{equation*}
    \mathcal{K}^{\mathrm{alg}}(R):=(\ModOne_R^{\cong})^{\gp} \in \Sp_{\geq 0} \,.
\end{equation*}
Here we used the equivalence of categories $\CGrp(\Spc)\simeq \Sp_{\geq 0}$ between grouplike $\E_\infty$-monoids and connective spectra (\cite[Theorem 5.2.6.10]{HA}), which is classically known as May's recognition principle.
The $K$-theory functor is given by 
\begin{equation}
    \label{eq:algebraicktheory}
    \begin{tikzcd}
        \mathcal{K}^{\mathrm{alg}}: \Cat_\infty^{\Sigma} \ar[r,hookrightarrow] & \CMon(\Cat_\infty) \ar[r,"(-)^\simeq"] & \CMon(\Spc) \ar[r,"(-)^{\gp}"] & \Sp 
    \end{tikzcd} \,,
\end{equation}
where $\Spc$ is the $(\infty,1)$-category of spaces ($\infty$-groupoids), $\Cat^\Sigma_\infty$ is the $(\infty,1)$-category of small $(\infty,1)$-categories with coproducts and functors which preserve them, and $\CMon$ denotes the $(\infty,1)$-category of commutative monoids.
The following is proven in~\cite{Gepner2015}.
\begin{theorem}
    \label{th:algebraicktheory}
    All categories in the composition \eqref{eq:algebraicktheory} carry universal symmetric monoidal structures and all functors in \eqref{eq:algebraicktheory} admit lax symmetric monoidal refinements.
    The inclusion $\Cat_\infty^{\Sigma} \hookrightarrow \CMon(\Cat_\infty)$ and the group completion $(-)^{\gp}$ are strong symmetric monoidal functors.
    Also the functor $\ModOne\colon \Ring\to \Cat_\infty^\Sigma$ admits a lax symmetric monoidal refinement.
\end{theorem}

This has significant implications:
Firstly, if $R$ is a commutative ring, then $\mathcal{K}^{\mathrm{alg}}(R)$ is canonically an $\E_\infty$-ring in $\Sp$.
Secondly, if $R$ is commutative and $S$ is an $R$-algebra, then $\mathcal{K}^{\mathrm{alg}}(S)$ is canonically a $\mathcal{K}^{\mathrm{alg}}(R)$-module.
Thirdly, from the factorization~\eqref{eq:algebraicktheory}, $\mathcal{K}^{\mathrm{alg}}$ is functorial in Morita bimodules, which implement coproduct-preserving functors ${}_{S}N_R\otimes_R-: \ModOne_R\to \ModOne_S$.

We will apply a similar approach to Banach algebras while taking the Banach topology into account.

\subsection{Group completion}
\label{sec:groupcompletion}
The inclusion of abelian groups into abelian monoids has a left adjoint called the Grothendieck group completion.
Explicitly, the group completion of $M$ can be modeled by a quotient of $M \times M$, thought of as formal differences.
More specifically, $m_1 - m_2 = m_1' - m_2'$ if and only if there is a $m \in M$ such that $m_1 +m_2' + m = m_1' + m_2 + m$.

The same constructions can be made in the homotopical setting where we replace $\CMon(\Set)$ by $\CMon(\Spc)$ and $\CGrp(\Set)$ by $\CGrp(\Spc) \simeq \Sp_{\geq 0}$.
The quotient (cofiber) of the diagonal map $M \to M \times M$ in the category $\CMon(\Spc)$ still gives a model for the group completion, see \cite[Proposition 8.7]{lehner2024groupcompletionactioninftycategory}.

\subsubsection{Telescopic group completion}
Let $(M,\oplus)$ be a commutative monoid in $\Spc$.
For $m\in M$ and $X$ a space with an $M$-action we can form
\begin{equation*}
    X[m^{-1}]:= \colim \left(X\xrightarrow{m \oplus (-) } X\xrightarrow{m \oplus  (-) }X\xrightarrow{m \oplus (-) } \dots \right) \,.
\end{equation*}
The space $X[m^{-1}]$ still carries an $M$-action and the colimit in spaces coincides with the colimit in the category $\LMod(M)$ of $M$-spaces~\cite[Corollary 4.2.3.5]{HA}. 
Choose a well-ordering on $\pi_0 M$.
We inductively define $X[\{m_{1}\dots m_{n}\}^{-1}]=(X[\{m_{1}\dots m_{n-1}\}^{-1}])[m_{n}^{-1}]$ for $m_1<\dots <m_n$. 
Given a subset $S \subseteq \pi_0 M$, write $X[S^{-1}] = \colim_{A\subseteq \pi_0 M \text{ finite}} X[A^{-1}]$. 
This is a filtered colimit of spaces with an $M$-action.
Observe that if $I$ is a set of generators for $\pi_0M$, then the induced map $X[I^{-1}] \to X[\pi_0(M)^{-1}]$ is an equivalence.

$X\mapsto X[\pi_0(M)^{-1}]$ defines a functor $\LMod(M)\to \LMod(M)$.
Applying this functor to $M\to M^\gp $ we get a natural $M$-equivariant map $M[\pi_0(M)^{-1}]\to M^\gp[\pi_0(M^{\gp})^{-1}]\simeq M^\gp$, because for $X=M^\gp$ the transition maps in the filtered colimit are all equivalences.
The following terminology comes from the fact that filtered colimits are modeled by mapping telescopes.
\begin{definition}
    $M$ is called \emph{telescopic} if $M[\pi_0(M)^{-1}] \to M^{\gp}$ is an equivalence.
\end{definition}

A group is called \emph{hypoabelian} if it does not contain any nontrivial perfect subgroup.
Equivalently, $G$ is hypoabelian if the derived series $G^{(\alpha+1)}=[G^{(\alpha)},G^{(\alpha)}]$ terminates in $G^{(\alpha)}=\{e\}$ for some (potentially transfinite) $\alpha$.

For the following proposition, we observe that the different orders in which to add $n$ copies of some $m \in M$ yield a map
\begin{equation}
\label{eq:braid}
B\Sigma_n \to (M^n)_{h\Sigma_n} \xrightarrow{\oplus} M
\end{equation}
coming from the higher commutativity of $M$.

\begin{proposition}[{\cite[Proposition~6]{nikolaus_groupcompletion}}]
    \label{prop:cyclicinvariance}
    The following are equivalent:
    \begin{enumerate}[label=\normalfont{(\alph*)}]
        \item $M$ is telescopic.
        \item The fundamental group of every component of $M[\pi_0(M)^{-1}]$ is abelian.
        \item The fundamental group of every component of $M[\pi_0(M)^{-1}]$ is hypoabelian.
        \item \label{item:permutationcondition}
        For some $k\geq 2$ and all $m_i,i\in I$, the permutation $(1 \dots k)$ is in the kernel of the symmetric braiding $\Sigma_k\to \pi_1(M,k \cdot m_i)\to \pi_1(M[\pi_0(M)^{-1}],k \cdot m_i)$.
    \end{enumerate}
\end{proposition}
Importantly, hypoabelian groups have a permanence property that abelian groups do not have.
\begin{lemma}
    \label{lem:hypoabelianSerre}
    The class of hypoabelian groups is a Serre class.
    That is, they are closed under subgroups, quotients and extensions of groups.
    The full subcategory $\Spc^{\mathrm{hyp}}$ of $\Spc$ on those spaces whose every component has hypoabelian fundamental group is closed under pullback.
\end{lemma}
\begin{remark}
    In fact, hypoabelian groups are the smallest subclass of groups containing abelian groups and having the above permanence properties.
\end{remark}
\begin{example}
    Let $M$ be the category of finite sets and bijections equipped with the disjoint union operation. Then $M[\pi_0(M)^{-1}] = \Z \times B\Sigma_\infty$ where $\Sigma_\infty = \colim_n \Sigma_n$.
    For every basepoint we have that $\pi_1(M[\pi_0(M)^{-1}]) = \Sigma_\infty$ is a nonabelian group, and so $M$ is not telescopic.
    In fact, $M^{\gp}$ is the sphere spectrum by the Barratt--Priddy theorem~\cite{MR314940}.
\end{example}

\begin{warning}
    Given an $\E_\infty$-monoid $M$ and a subset $S \subseteq \pi_0M$, our notation $M[S^{-1}]$ is motivated by the fact that $M[S^{-1}]$ is the universal $M$-module on which $S$ acts invertibly.
    By the previous example, $M[\pi_0(M)^{-1}]$ need not be an $\E_\infty$-monoid at all unless $M$ is telescopic.
\end{warning}

If $f \colon M \to N$ is an $\E_\infty$-map, then the induced map $M[\pi_0(M)^{-1}] \to N[\pi_0(N)^{-1}]$ factors as $M[\pi_0M^{-1}] \to N[f(\pi_0M)^{-1}] \to N[\pi_0N^{-1}]$.
However, if $f\colon M\to N$ is $\pi_0$-surjective, then $\{f(m)\}_{m \in \pi_0(M)}$ is a set of generators for $\pi_0 N$, and hence the second map is an equivalence.
We can weaken the assumption to $f$ being quasi-surjective (compare~\cite[II.2.6.]{karoubi_k-theory_1978}):
\begin{definition}
    A map $f\colon M \to N$ of $\E_\infty$-spaces is called \emph{quasi-surjective} (also called \emph{cofinal}) if for every $n \in \pi_0 N$ there exist $m \in \pi_0 M$ and $n' \in \pi_0 N$ such that $n+n' = \pi_0(f)(m)$.
\end{definition}

\begin{remark}
\label{rem:qsur}
    If $f_1 \colon M_1 \to M_2$ and $f_2 \colon M_2 \to M_3$ are maps of $\E_\infty$-spaces, then $f_2 f_1$ quasi-surjective implies that $f_2$ is quasi-surjective.
\end{remark}

\begin{lemma}
    \label{lem:quasi-surjlocalization}
    Let $f \colon M \to N$ be a quasi-surjective homomorphism of $\E_\infty$-monoids.
    Then
    \begin{equation*}
        N[\pi_0M^{-1}]\xrightarrow{\simeq} N[\pi_0N^{-1}] \,.
    \end{equation*}
\end{lemma}
\begin{proof}
    Since $f$ is quasi-surjective, the fact that every $m\in M$ acts invertibly on $N[\pi_0M^{-1}]$ implies that every $n \in \pi_0N$ acts invertibly: write $n+n'=m$.
    Then an inverse to the action by $n$ is given by the action by $n'$ followed by the inverse of acting by $m$.
\end{proof}

\begin{lemma}
    \label{lem:imagetelescopic}
    Let $f\colon M\to N$ be $\pi_0$-surjective and let $M$ be telescopic.
    Then $N$ is telescopic.
\end{lemma}
\begin{proof}
    We want to check condition~\ref{item:permutationcondition} in Proposition~\ref{prop:cyclicinvariance}.
    There is a commutative diagram of groups
    \begin{equation*}
        \begin{tikzcd}
            \Sigma_k \ar[r] \ar[dr]& \pi_1(M,k \cdot m) \ar[r]\ar[d]& \pi_1(M[\pi_0(M)^{-1}],k \cdot m) \ar[d] \\
            & \pi_1(N,k \cdot f(m)) \ar[r]& \pi_1(N[\pi_0(N)^{-1}],k \cdot f(m))
        \end{tikzcd} \,.
    \end{equation*}
    The triangle commutes by the naturality of \eqref{eq:braid}. 
    The right vertical map is the induced map on filtered colimits, and so the square commutes by naturality of the map $M \to M[\pi_0(M)^{-1}]$.
    Since $(1,\dots,k)$ is in the kernel of the upper horizontal composition, it also maps to $0$ in $\pi_1(N[\pi_0(N)^{-1}],f(m))$.
    Since $f\colon M\to N$ is $\pi_0$-surjective, the lower right corner for suitable $m\in M$ can hit $\pi_1(N[\pi_0(N)^{-1}], k \cdot n)$ for every $n \in N$.
    This now implies condition~\ref{item:permutationcondition}.
\end{proof}

\subsubsection{Fibers of group completions}
\label{sec:fibersgroupcompletion}
This section discusses the fiber of maps between group completions.
The results could be of independent interest.
We will use them in the proof of Bott periodicity in Section~\ref{sec:proofs} and to define the spectral version of Karoubi $K$-theory in Section~\ref{sec:karoubikt-h}.
We start with a few lemmas which will be useful for defining graded $K$-theory using group completion.

\begin{lemma}
    \label{lem:fibercofibersequences}
    Let $A\xrightarrow{f} B\xrightarrow{g} C$ be maps of grouplike $\E_\infty$-spaces or, equivalently, of connective spectra.
    \begin{itemize}
        \item If $A\to B\to C$ is a cofiber sequence, then it is a fiber sequence.
        \item If $A\to B\to C$ is a fiber sequence and $g\colon \pi_0B\to \pi_0C$ is surjective, then it is a cofiber sequence.
    \end{itemize}
\end{lemma}
\begin{proof}
    Connective spectra form a coreflective subcategory of spectra:
    \begin{equation*}
    \begin{tikzcd}
         \Sp_{\geq 0} \ar[r,shift left=.5ex,hook]& \Sp \ar[l, shift left=.5ex,"\tau_{\geq 0}"]
    \end{tikzcd} \,.
    \end{equation*}
    
    Suppose that $A\to B\to C$ is a cofiber sequence in $\Sp_{\geq 0}$.
    Since the inclusion is a left adjoint, it is also a cofiber sequence in $\Sp$, and hence a fiber sequence.
    Applying the right adjoint $\tau_{\geq 0}$ the sequence is unchanged and stays a fiber sequence.

    Suppose the sequence $A\to B\to C$ is a fiber sequence and that $\pi_0(g)$ is surjective.
    There is an associated long exact sequence of abelian groups
    \begin{equation*}
        \dots \to \pi_1 B \to \pi_1 C \to \pi_0 A \to \pi_0 B \xrightarrow{g}\pi_0 C \to 0 \,,
    \end{equation*}
    by surjectivity.
    Now consider $A\to B\to C$ as a fiber sequence in $\Sp$ and let $\cofib(f)$ be the cofiber of $f\colon A\to B$ in $\Sp$.
    We get a map $\cofib(f)\to C$.
    This gives a comparison map between the long exact sequences which end in 
    \begin{equation*}
        \begin{tikzcd}
        \dots \ar[r] & \pi_1 \cofib(f) \ar[r]\ar[d]& \pi_0 A \ar[r]\ar[d,equals]& \pi_0B\times \pi_0 C \ar[r]\ar[d,equals]& \pi_0 \cofib(f) \ar[r] \ar[d] & 0 \\
        \dots \ar[r] & \pi_1 C \ar[r]& \pi_0 A \ar[r]& \pi_0B\ar[r]& \pi_0 C \ar[r]& 0
        \end{tikzcd} \,.
    \end{equation*}
    The $5$-lemma now proves that $\cofib(f)\to C$ is an equivalence.
\end{proof}
\begin{corollary}
    \label{cor:pullbackpushoutsquare}
    Consider the square of grouplike $\E_\infty$-spaces.
    \begin{equation}
        \label{eq:pullbackpushoutsquare}
        \begin{tikzcd}
            A\ar[r]\ar[d]& B \ar[d,"p"]\\
            C\ar[r,"q"]&  D
        \end{tikzcd} \,.
    \end{equation}
    \begin{itemize}
        \item If it is a pushout square, then it is also a pullback square.
        \item If $p+q\colon \pi_0(B\times C)\to \pi_0D$ is surjective and this square is a pullback square, then the square is a pushout square.
    \end{itemize}    
\end{corollary}
\begin{proof}
    Follows from the fact that $A\to B\times C\xrightarrow{p-q} D$ is a (co)fiber sequence if and only if \eqref{eq:pullbackpushoutsquare} is a pullback (pushout). 
\end{proof}

\begin{lemma}
    \label{lem:groupcompletionpullback}
    Let 
    \begin{equation*}\begin{tikzcd}
	M & M_1 \\
	M_2 & N
	\arrow[from=1-1, to=1-2]
	\arrow[from=1-1, to=2-1]
	\arrow["\lrcorner"{anchor=center, pos=0.125}, draw=none, from=1-1, to=2-2]
	\arrow[from=1-2, to=2-2]
	\arrow[from=2-1, to=2-2]
\end{tikzcd}    \end{equation*}
    be a pullback square in $\CMon(\Spc)$ where all maps are quasi-surjective.
    Assume further that $M$ is telescopic.
    Then, the group completed square is again a pullback.
    \begin{equation*}\begin{tikzcd}
	M^\gp & M_1^\gp \\
	M_2^\gp & N^\gp
	\arrow[from=1-1, to=1-2]
	\arrow[from=1-1, to=2-1]
	\arrow["\lrcorner"{anchor=center, pos=0.125}, draw=none, from=1-1, to=2-2]
	\arrow[from=1-2, to=2-2]
	\arrow[from=2-1, to=2-2]
\end{tikzcd}    \end{equation*}
\end{lemma}
\begin{proof}
    We may assume without loss of generality that the maps are $\pi_0$-surjective by restricting the codomain to the essential image.
    This does not change the pullbacks, or group completions by cofinality of the essential image.
    By Lemma~\ref{lem:imagetelescopic}, all corners are telescopic and so
    by Lemma~\ref{lem:quasi-surjlocalization}, it is sufficient to localize all corners at $\pi_0 M$.
    Localization is a filtered colimit and preserves pullbacks.
    Further using that the corners are telescopic, the maps $N[\pi_0M^{-1}]\to N[\pi_0N^{-1}]\to N^{\gp}$ are equivalences of spaces (even of $M$-modules).
    The group completed square is a pullback since the underlying square of spaces is.
\end{proof}

\begin{warning}
     It is necessary to assume that all maps in Lemma \ref{lem:groupcompletionpullback} are quasi-surjective.
     For example, the pullback of commutative monoids
    \begin{equation*}
    \begin{tikzcd}
        0 \ar[d] \ar[r] & \Z_{\leq 0} \ar[d]
        \\
        \Z_{\geq 0} \ar[r] & \Z
        \arrow["\lrcorner"{anchor=center, pos=0.125}, draw=none, from=1-1, to=2-2]
    \end{tikzcd}
    \end{equation*}
    does not group complete to a pullback square.
     It also follows that quasi-surjective maps are not stable under pullbacks; both maps into $\Z$ are quasi-surjective, but neither of the maps from $0$ is.
\end{warning}
\begin{remark}
    Instead of requiring that $M$ is telescopic, we can require that $M_1,M_2,N$ are telescopic. 
    It is possible to deduce that $M$ is telescopic:
    We just have to show that $M[\pi_0M^{-1}]$ is hypoabelian.
    But it is a pullback of hypoabelian spaces by the fact that the spaces $N[\pi_0M^{-1}]$ are equivalent to their group completions. 
\end{remark}

The following definition is motivated by~\cite[II.2.13]{karoubi_k-theory_1978}.

\begin{definition}
    \label{def:abstractKar}
    Let $f\colon M\to N$ be quasi-surjective.
    The \emph{$\E_\infty$-monoid of Karoubi triples for $f$} is the pullback
    \begin{equation*}
    \begin{tikzcd}
        \Gamma(f) \ar[r,"p_1"] \ar[d,"p_2"] & M\ar[d,"f"]
        \\
        M \ar[r,"f"] & N
        \arrow["\lrcorner"{anchor=center, pos=0.125}, draw=none, from=1-1, to=2-2]
    \end{tikzcd}
    \end{equation*}
    in the $(\infty,1)$-category of $\E_\infty$-monoids.
    The universal property yields a diagonal map $\Delta\colon M\to \Gamma(f)$. 
    We define the \emph{$\E_\infty$-Karoubi $K$-theory space for $f$} to be the cofiber
    \begin{equation*}
        k(f)=\cofib(\Delta\colon M\to \Gamma(f))
    \end{equation*}
    in the category of $\E_\infty$-spaces.
    We define the \emph{Karoubi $K$-theory group of $f$} to be $K^{\Kar}(f)=\pi_0k(f)$.
\end{definition}
 \begin{remark}
        If the map $f$ of $\E_\infty$-spaces arises as the core of a functor between symmetric monoidal $(\infty,1)$-categories, then there is a version of $\Gamma(f)$ which is a symmetric monoidal $(\infty,1)$-category so that its core recovers $\Gamma(f)$.
        Bass uses the weak pullback in the $(2,1)$-category $\Cat_1^\otimes$ to define a similar category $co(f)$ in~\cite[Ch. VII, \textsection5]{bassktheory}.
    \end{remark}
The underlying space of $\Gamma(f)$ can be computed as the pullback in the $(\infty,1)$-category of spaces.
    Explicitly, a point in this space is a triple $(m_1,m_2,\phi)$, where $m_1,m_2 \in M$ and $\phi$ is a path from $f(m_1)$ to $f(m_2)$ in $N$.
    In particular, $\pi_0\Gamma(f)$ is a quotient of such points by the equivalence relation of homotopy.
    Explicitly, a homotopy is a path $(m_1(t),m_2(t),\phi_t)$ where $\phi_t$ is a path homotopy with endpoints $f(m_1(t))$ and $f(m_2(t))$.
    The $\E_\infty$-structure on $\Gamma(f)$ is given pointwise by $(m_1,m_2,\phi)\oplus(m_1',m_2',\phi')=(m_1\oplus m_1',m_2\oplus m_2',\phi\oplus\phi')$.
    This induces a monoid structure on $\pi_0 \Gamma(f)$.
    \begin{definition}
        The commutative monoid $\pi_0\Gamma(f)$ is called the monoid of \emph{Karoubi triples}.
        Representatives of path components are denoted $(m_1,m_2,\phi)\in \pi_0\Gamma(f)$.
        Elements in the essential image of $\Delta$ are called \emph{elementary} Karoubi triples.
    They are (homotopic to) triples of the form $(m,m,\id_{f(m)})$.
    \end{definition}

The next definition provides an explicit model for the quotient in the category of commutative monoids (in $\Set$).
\begin{definition}
\label{def:monoidquotient}
    Let $f\colon A\to B$ be a map of commutative monoids.
    We declare $b,b'\in B$ to be equivalent if there exists $a,a'\in A$ with $b+f(a)=b'+f(a')$.
    The \emph{quotient monoid} is $B/f(A):=B/\sim$.
\end{definition}
 $B/f(A)$ is indeed a commutative monoid.
\begin{lemma}
\label{lem:triples}
    The Karoubi $K$-theory group $K^{\Kar}(f)$ is the quotient monoid of Karoubi triples by elementary Karoubi triples. 
    It is a group.
\end{lemma}
\begin{proof}
    $\pi_0\colon \CMon(\Spc)\to \CMon$ is a left adjoint and preserves cofiber sequences.
    Applying $\pi_0$ to the cofiber sequence $M\to \Gamma(f)\to k(f)$,
    we can express $K^{\Kar}(f)=\pi_0 k(f)$ as a cofiber in $\CMon$.
    The cofiber of a map of monoids is explicitly described in Definition~\ref{def:monoidquotient}.

    It follows by~\cite[Lemma C.1]{gomi_freed-moore_2021} that $K^{\Kar}(f)$ is a group, because the map
    \begin{equation*}
    I\colon \pi_0(\Gamma(f)) \to  \pi_0(\Gamma(f)) \quad I(m_1, m_2, \phi) = (m_2, m_1, \phi^{-1})
    \end{equation*}
    satisfies the necessary properties:
    \begin{enumerate}
        \item it is a monoid map
        \item $I(x) + x$ is an elementary triple for all $x \in \pi_0(\Gamma(f))$.\footnote{Indeed, note that if $\gamma$ is a path from $n_1$ to $n_2$, then $\gamma^{-1} \oplus \gamma$ is trivial in $\pi_1(N,n_1 \oplus n_2)$ by an Eckmann--Hilton argument.}
        \item $I$ is the identity on elementary triples. \qedhere
    \end{enumerate}
\end{proof}
Since every $\E_\infty$-space for which $\pi_0$ is a group is grouplike, we obtain
\begin{corollary}
    The $\E_\infty$-space $k(f)$ is grouplike.
\end{corollary}
\begin{theorem}
\label{thm:fibergroupcompletion}
    Let $f\colon M \to N$ be a quasi-surjective map of $\E_\infty$-spaces and let $f^{\gp}\colon M^\gp \to N^\gp$ be the induced map on group completions.
    Assume further that $M$ is telescopic.
    Then there is a natural equivalence of grouplike $\E_\infty$-spaces $k(f)\simeq \fib f^{\gp}$.
\end{theorem}

\begin{proof}
    As in the proof of Lemma \ref{lem:groupcompletionpullback}, we may reduce to $f$ being $\pi_0$-surjective, which implies that $N$ is telescopic by Lemma~\ref{lem:imagetelescopic}.
    This does not change the pullbacks, the fiber nor the property that $k(f)\simeq \fib(f^\gp)$. 
    The following diagram commutes:
    \begin{equation}
        \label{eq:Karoubifiberdiagram}
        \begin{tikzcd}[row sep={40,between origins}, column sep={40,between origins}]
          & M \ar[rr,"\Delta"]\ar{dd} & & M\times M \ar{dd} \ar[rr]& & M^{\gp} \ar[dd,"f^{\gp}"]\\
        M \ar[crossing over,"\Delta", near start]{rr} \ar{dd}\ar[ru,equals]\ar[dd] & & \Gamma(f) \ar[rr,crossing over]\ar[ru]\ar[dd]& & k(f) \ar[ru]\ar[dd,crossing over] &\\
          & N  \ar[rr,"\Delta"] \ar{rr} & &  N\times N \ar{rr} && N^{\gp}  \\
        N \ar[rr,equals]\ar[ru,equals] && N\ar[ru,"\Delta"] \ar[rr]\ar[from=uu,crossing over] && 0\ar[ru] &
    \end{tikzcd} \,.
    \end{equation}
    The right face is the cofiber of the middle and leftmost face. 
    Both the left and the middle face are pullbacks in $\CMon(\Spc)$.
    Passing to group completions and using that $k(f)\simeq k(f)^\gp$, we obtain the following diagram:
    \begin{equation}
    \label{eq:groupcompleteddiagram}
    \begin{tikzcd}[row sep={40,between origins}, column sep={40,between origins}]
          & M^\gp \ar[rr,"\Delta"]\ar{dd} & & M^\gp\times M^\gp \ar{dd} \ar[rr]& & M^{\gp} \ar[dd,"f^{\gp}"]\\
        M^\gp \ar[crossing over,"\Delta", near start]{rr} \ar{dd}\ar[ru,equals]\ar[dd] & & \Gamma(f)^\gp \ar[rr,crossing over]\ar[ru]\ar[dd]& & k(f) \ar[ru]\ar[dd,crossing over] &\\
          & N^\gp  \ar[rr,"\Delta"] \ar{rr} & &  N^\gp\times N^\gp \ar{rr} && N^{\gp}  \\
        N^\gp \ar[rr,equals]\ar[ru,equals] && N^\gp\ar[ru,"\Delta"] \ar[rr]\ar[from=uu,crossing over] && 0\ar[ru] &
    \end{tikzcd} \,.    
    \end{equation}

    Since group completion preserves colimits, the right face is still the cofiber of the middle and left face.
    The left face is still a pullback after group completion.
    \textbf{Claim:} the middle face is still a pullback after group completion.
    
    Let us see how the claim implies the theorem.
    We assumed that $f^\gp$ is surjective on $\pi_0$.
    By Corollary~\ref{cor:pullbackpushoutsquare}, the left and middle face of diagram~\eqref{eq:groupcompleteddiagram} are pushouts.
    Since group completion is a left adjoint, the right face is still the cofiber of the left and middle faces in diagram~\eqref{eq:groupcompleteddiagram}.
    The cofiber of pushouts is a pushout square, i.e.\ $k(f)\to M^\gp \to N^\gp$ is a cofiber sequence.
    Again, since $f^\gp$ is surjective on $\pi_0$, it is also a fiber sequence, i.e.\ $k(f)\simeq \fib(f^\gp)$.

    To finish, we prove the claim.
    We first show that $\Gamma(f)$ is telescopic.
    In the defining pullback diagram
    \begin{equation*}
        \begin{tikzcd}
            \Gamma(f) \ar[r,"p_1"]\ar[d,"p_2"]& M \ar[d,"f"]\\
            M  \ar[r,"f"] & N
            \arrow["\lrcorner"{anchor=center, pos=0.125}, draw=none, from=1-1, to=2-2]
        \end{tikzcd} 
    \end{equation*}
    $p_1,p_2$ are $\pi_0$-surjective since $f$ is.
    Since the functor $(-)[S^{-1}]$ is given by a filtered colimit, it commutes with finite limits.
    We thus obtain a pullback diagram of spaces:
    \begin{equation}
        \label{eq:sqGamma}
        \begin{tikzcd}
            \Gamma(f)[\pi_0\Gamma(f)^{-1}] \ar[r,"p_1"]\ar[d,"p_2"]& M[\pi_0(M)^{-1}] \ar[d,"f"]\\
            M[\pi_0(M)^{-1}]  \ar[r,"f"] & N[\pi_0(N)^{-1}]
            \arrow["\lrcorner"{anchor=center, pos=0.125}, draw=none, from=1-1, to=2-2]
        \end{tikzcd} \,.
    \end{equation}
    By Proposition~\ref{prop:cyclicinvariance}, it is enough to check that $\pi_1(\Gamma(f)[\pi_0\Gamma(f)^{-1}],x)$ is hypoabelian for every $x\in \pi_0\Gamma(f)$.
    But a pullback of spaces with hypoabelian fundamental groups has hypoabelian fundamental groups, because hypoabelian groups form a Serre class.
    Since all spaces are telescopic and the maps are $\pi_0$-surjective,
    we obtain $\Gamma(f)^{\gp}\simeq M^{\gp}\times_{N^{\gp}} M^{\gp}$.
    Therefore, the middle face of diagram~\eqref{eq:groupcompleteddiagram} is also a pullback in $\Spc$ and hence a fortiori in $\CGrp(\Spc)$.
    This concludes the proof.
\end{proof}

\begin{remark}
    A particular version of the long exact sequence associated to the fibration $k(f) \to M^\gp \to N^\gp$ appears  in~\cite[Theorem II.3.22]{karoubi_k-theory_1978}.
\end{remark}

\begin{warning}
    Theorem~\ref{thm:fibergroupcompletion} is false in case $f$ is not quasi-surjective.
    For example, take $M = 0$.
    We obtain that $\Gamma(f) = \Omega N$ and so $k(f) = \Omega N$.
    On the other hand, we have that the fiber of $M^{\gp} \to N^{\gp}$ is given by $\Omega N^{\gp}$.
    These are different in general: 
    Let $N=\bigsqcup_{n\in \N} BO(n)$ be the monoid of real vector spaces under $\oplus$ we will consider in Example~\ref{ex:ku}.
    Then $\pi_0\Omega N= \pi_0 O(0) = 0$ which does not equal $\pi_0\Omega N^\gp = \pi_0 O=\Z/2$.
\end{warning}

\subsection{Topological \texorpdfstring{$K$}{K}-theory by simplicial techniques}
\label{sec:kthysimplicial}
Our main goal in the coming few sections is to define the connective $K$-theory spectrum of an ungraded Banach algebra $A$ in the same spirit as Section~\ref{sec:algKth}.
Already in the most famous case $A = \C$ we arrive at the problem that algebraic $K$-theory of $\C$ and topological $K$-theory of $\C$ do not agree.
We therefore need to remember the topology of $A$ when defining the group completion of $\ModOne_A^{\cong}$.
We will provide several equivalent approaches to do so.

In this section, we will focus on simplicial techniques.
More specifically, we define the $K$-theory spectrum of a simplicial ring $R$ and more generally a simplicial object in categories in Subsection~\ref{sec:sSetKth}.
We define the $K$-theory of a category enriched in topological spaces (or simplicial sets) in \ref{sec:enrichKth}.
We prove a relation between these two approaches in Theorem~\ref{th:comparisontopktheory}.

\subsubsection{\texorpdfstring{$K$}{K}-theory of simplicial rings}
\label{sec:sSetKth}
Any topological ring $A$ (such as a Banach algebra) determines a simplicial ring $A_\bullet$ given by the functor
\begin{align*}
    A_\bullet \colon \Delta^{\op} & \longrightarrow \Ring \,,\\
    [n] & \longmapsto C(\Delta^n;A) \,.
\end{align*}
Here, the set of continuous functions $C(\Delta^n;A)$ is equipped with the pointwise multiplication and addition.
To each of these rings we can separately apply the algebraic $K$-theory functor $\mathcal{K}^{\mathrm{alg}}$, obtaining a simplicial diagram in spectra.
 
\begin{definition}
Let $R$ be a simplicial object in the $1$-category of rings.
    The (connective) \emph{$K$-theory of $R$} is the geometric realization
\begin{equation*}
    k_R:= |\mathcal{K}^{\mathrm{alg}}(R_\bullet)|  = \colim_{[n]\in \Delta^{\op}} \mathcal{K}^{\mathrm{alg}}(R_n) 
\end{equation*}
in connective spectra.
\end{definition}

More generally we define 

\begin{definition}
\label{def:topktheory1}
The \emph{$K$-theory of a simplicial object in the $(\infty,1)$-category of $(\infty,1)$-categories with coproducts} is the image under the composition
\begin{equation*}
(\Cat_\infty^\Sigma)^{\Delta^{\op}} \xrightarrow{\mathcal{K}^{\mathrm{alg}}_*} \Sp^{\Delta^{\op}} \xrightarrow{|-|} \Sp,
\end{equation*}
where $\mathcal{C}^{\Delta^{\op}} := \Fun(\Delta^{\op},\mathcal{C})$ denotes the functor $(\infty,1)$-category.
\end{definition}

\begin{remark}
    Definition~\ref{def:topktheory1} is a variant of semi-topological $K$-theory~\cite{MR1910042, MR3477639}.
\end{remark}

We equip the functor categories $(\Cat_\infty^\Sigma)^{\Delta^{\op}}$ and $\Sp^{\Delta^{\op}}$ with the pointwise symmetric monoidal structure.

\begin{proposition}
\label{prop:Kthlax1}
The $K$-theory functor of Definition~\ref{def:topktheory1} admits a lax symmetric monoidal refinement.    
\end{proposition}

    Note that the postcomposition with $\mathcal{K}^{\mathrm{alg}}$ is lax monoidal because $\mathcal{K}^{\mathrm{alg}}$ is lax monoidal by Theorem~\ref{th:algebraicktheory}.
    Hence Proposition~\ref{prop:Kthlax1} follows from:

\begin{lemma}    
    The colimit functor $|-|:\Sp^{\Delta^{\op}}\to \Sp$ admits a canonical symmetric monoidal refinement.
\end{lemma}
\begin{proof}
    See \ref{lem:geomrealizationmonoidal}.
\end{proof}

\subsubsection{\texorpdfstring{$K$}{K}-theory of simplicially enriched categories}
\label{sec:enrichKth}

If $R$ is a Banach algebra, the category of finitely generated projective $R$-modules carries an enrichment in simplicial sets.
Namely, as we will discuss in more detail in Section~\ref{sec:Banach}, $R$-modules form a Banach category, and the Banach space topology on module homomorphisms has a corresponding simplicial set.
This motivates us to define the (connective) $K$-theory $k(\mathsf{C})$ of any simplicially enriched category $\mathsf{C}$ with finite coproducts in this section.

Using the homotopy coherent nerve $N^{\mathrm{hc}}$~\cite[\S 1.1.5]{HTT}, we can regard $\mathsf{C}$ as an $(\infty,1)$-category with finite coproducts.\footnote{We will assume all simplicially enriched categories are fibrant, i.e.\ all hom-simplicial sets are Kan. This will ensure that the homotopy coherent nerve construction lands in quasicategories. Since the simplicial sets we consider are always induced by a topological space, this will not be a restriction.}
The most naive approach of defining the topological $K$-theory of $\mathsf{C}$ as $\mathcal{K}^{\mathrm{alg}}(N^{\mathrm{hc}}\mathsf{C})$ is not ``correct'':

\begin{warning}
\label{ex:wrong}
Consider for $R = \C$ the category $\ModOne_R = \VectOne_\C$ of finite-dimensional complex vector spaces.
    The Euclidean topology on $\C$ induces the standard topology on linear maps, which makes $\VectOne_\C$ into a $\Top$-enriched category, which we write $\VectEn_\C$.
Since the mapping spaces $M_{m\times n}(\C)$ are all contractible, $N^{\mathrm{hc}}\VectEn_\C \simeq *$ is the terminal $(\infty,1)$-category.
    We therefore see that $\mathcal{K}^{\mathrm{alg}}(N^{\mathrm{hc}}\VectEn_\C)=0\in \Sp$, which is quite different from the desired topological $K$-theory spectrum $ku$.
    
    Instead, we have to take the homotopy coherent nerve $N^{\mathrm{hc}}$ of the symmetric monoidal topologically enriched groupoid $(\VectEn_\C^{\cong}, \oplus)$ of vector spaces and linear isomorphisms.
    It turns out group completing the resulting $\E_\infty$-monoid does yield $ku$. 
    We will revisit this computation in more detail in Example~\ref{ex:ku}.
\end{warning}
To prevent the problem above, we need to differentiate between \emph{isomorphisms} and \emph{equivalences} in $\mathsf{C}$.
\begin{definition}
    Let $\mathsf{C}$ be a simplicially enriched category.
    The \emph{iso-core} of $\mathsf{C}$ is the simplicially enriched category $\mathsf{C}^{\cong}$ defined as follows:
    It has the same objects as $\mathsf{C}$.
    The simplicial set $\mathsf{C}^{\cong}(c_1,c_2)$ is the sub-simplicial set whose vertices are the isomorphisms in $\mathsf{C}(c_1,c_2)$.
\end{definition}
The enriched category $\mathsf{C}^{\cong}$ is a weak enriched groupoid, i.e.\ the underlying category is a groupoid, but it is generally not an enriched groupoid in the sense of Definition~\ref{def:enrichedgroupoid}, see Remark~\ref{rmk:weakgroupoid}.
This is enough to guarantee that the $(\infty,1)$-category $N^{\mathrm{hc}}(\mathsf{C}^{\cong})$ is an $\infty$-groupoid.
\begin{definition}
    \label{def:topktheory2}
    Let $\mathsf{C}$ be a simplicially enriched category with enriched coproducts (see Appendix \ref{sec:enriched}).
    The (connective) \emph{topological $K$-theory} of $\mathsf{C}$ is the group completion of the $\E_\infty$-monoid $(N^{\mathrm{hc}}(\mathsf{C}^{\cong}),\oplus)$:
    \begin{equation}
    \label{eq:topktheory2}
    k\colon \Cat^\Sigma_1(\sSet) \xrightarrow{(-)^{\oplus}} \CMon(\Cat_1(\sSet)) \xrightarrow{(-)^{\cong}} \CMon(\Grpd_1(\sSet)) \xrightarrow{N^{\mathrm{hc}}} \CMon(\Spc) \xrightarrow{(-)^{\gp}} \Sp \,.
    \end{equation}
\end{definition}
Here $\Grpd_1(\sSet)$ is the symmetric monoidal $(2,1)$-category of weak enriched groupoids and $(-)^\oplus$ assigns the cocartesian symmetric monoidal structure to any category with enriched coproducts, see Appendix \ref{sec:enriched} for details.
\begin{remark}
\label{rem:coresdontcommute}
     The underlying abstract reason for the disparity discussed in Warning~\ref{ex:wrong} is that the diagram
    \begin{equation*}
    \begin{tikzcd}
        \Cat(\sSet) \ar[r,"N^{\mathrm{hc}}"] \ar[d,"(-)^{\cong}"] & \Cat_\infty \ar[d,"(-)^{\simeq}"]
        \\
        \Grpd(\sSet) \ar[r,"N^{\mathrm{hc}}"] & \Spc
    \end{tikzcd} \,,
    \end{equation*}
    where the horizontal arrows are given by the homotopy-coherent nerve, does not commute.
\end{remark}
\begin{convention}
\label{conv:topenr}
The functor $\Sing \colon \Top \to \sSet$ from the category $\Top$ of compactly generated weak Hausdorff spaces is right adjoint to geometric realization, and so preserves finite products.
We will always consider topologically enriched categories as simplicially enriched categories using the change of enrichment functor $\Cat(\Top) \to \Cat(\sSet)$.
This functor preserves enriched coproducts and hence Definition \ref{def:topktheory2} compiles for topologically enriched categories.
\end{convention}

\paragraph{Comparison with simplicial objects in categories}

If $\mathsf{C}$ is a simplicially enriched category, it induces a simplicial object in categories $j(\mathsf{C})$ by mapping $[n]$ to the category $j(\mathsf{C})[n]$ with objects $\ob \mathsf{C}$ and morphisms the $n$-simplices of $\Hom_{\mathsf{C}}(x,y)$.
This Yoneda functor $j \colon \Cat_1(\sSet) \to \Cat_1^{\Delta^{\op}}$ is fully faithful and essentially surjects onto those simplicial diagrams that have constant objects. 
Note that simplicially enriched categories with enriched coproducts give simplicial diagrams of categories with coproducts and so we get $j \colon \Cat_1^\Sigma(\sSet) \to (\Cat_1^\Sigma)^{\Delta^{\op}}$.

Given the machinery we developed in Section~\ref{sec:sSetKth}, we can therefore alternatively define the (connective) $K$-theory of $\mathsf{C}$ as the $K$-theory of $j(\mathsf{C})$:
\begin{equation}
    \label{eq:topktheory1'}
    \begin{tikzcd}
        k:\Cat_1^{\Sigma}(\sSet) \ar[r,"j"]& (\Cat^{\Sigma}_1)^{\Delta^{\op}} \ar[r] &  (\Cat^{\Sigma}_\infty)^{\Delta^{\op}} \ar[r,"\mathcal{K}^{\mathrm{alg}}_*"] & \Sp^{\Delta^{\op}} \ar[r, "|-|"]& \Sp  
    \end{tikzcd} \,.
\end{equation}
We will now show that this composition agrees with Definition~\ref{def:topktheory2}.

\begin{theorem}
    \label{th:comparisontopktheory}
    The compositions of functors in \eqref{eq:topktheory2} and \eqref{eq:topktheory1'} are naturally equivalent.
\end{theorem}
\begin{proof}
    This follows from the following commutative diagram:
    \begin{equation*}
        \begin{tikzcd}
            \Cat_1^{\Sigma}(\sSet)\ar[d,"j"] \ar[r] & \CMon(\Cat_1(\sSet)) \ar[r,"(-)^{\cong}"] \ar[d,"j"]& \CMon(\Grpd_1(\sSet)) \ar[d,"j"]\ar[rd,"N^{\mathrm{hc}}"] & &  \\
             (\Cat_1^\Sigma)^{\Delta^{\op}} \ar[d] \ar[r] & \CMon\left(\Cat_1^{\Delta^{\op}} \right)\ar[d] \ar[r,"(-)^{\cong}"]& \CMon\left(\Grpd_1^{\Delta^{\op}}\right) \ar[d] & \CMon(\Spc)\ar[r,"(-)^{\gp}"]& \Sp \\
            \left( \Cat^{\Sigma}_\infty \right)^{\Delta^{\op}} \ar[r]& \CMon\left(\Cat_\infty^{\Delta^{\op}} \right) \ar[r,"(-)^\simeq"] \ar[d,"\simeq"] & \ar[d,"\simeq"] \CMon\left(\Spc^{\Delta^{\op}}\right)  \ar[ru,"|-|"] & \Sp^{\Delta^{\op}} \ar[ru,"|-|", swap] & 
            \\
            &\CMon(\Cat_\infty)^{\Delta^{\op}} \ar[r,"(-)^{\simeq}"]& \CMon(\Spc)^{\Delta^{\op}} \ar[ru,"(-)^{\gp}", swap] &&
        \end{tikzcd} \,.
    \end{equation*}
    Note that we can take $\CMon$ without any problems since all functors involved preserve products.
    The right triangle commutes by Lemma~\ref{lem:homotopycoherentnerve}.
    Commutativity of all other squares is immediate.
\end{proof}
\begin{remark}
    We do not know how to construct the monoidal structure on $\Cat_1^\Sigma(\sSet)$.
    There is an operadic replacement for this and in this sense the $K$-theory of additive simplicially enriched categories is lax symmetric monoidal.
\end{remark}

\subsection{Banach algebras and Banach categories}
\label{sec:Banach}

In this section we review operator-algebraic facts we will need in the main text.

\paragraph{Banach algebras}

Let $\mathbb{F}$ be either the real or complex numbers. 
Recall that a \emph{Banach algebra} over $\mathbb{F}$ is a Banach space $(A,\Vert-\rVert)$ over $\mathbb{F}$ together with an algebra structure such that $\| a b \| \leq \| a\|  \|b\|$ for all $a,b \in A$. 
 Unless stated otherwise, Banach algebras are assumed to be unital. 

More abstractly, let $\Ban$ be the category of Banach spaces in which the hom sets $\Hom_{\Ban}(X,Y) = B(X,Y)$ are bounded linear maps.
Let $\Ban_\leq$ be the wide subcategory on the \emph{contractions} (also called \emph{short maps}), which are those morphisms $T \colon X \to Y$ such that $\|Tx\| \leq \|x\|$.
Both $\Ban$ and $\Ban_{\leq}$ are symmetric monoidal closed using the (completed) \emph{projective tensor product}~\cite{Ryan2002}, which is universal with respect to bounded bilinear maps.
There is a tautological symmetric monoidal functor $\Ban_{\leq} \to \Ban$.
The categories $\Ban_{\leq}$ and $\Ban$ are additive and $\Ban_{\leq}$ is bicomplete~\cite{MR533819} but $\Ban$ does not have infinite products. Both categories are closed with internal hom given by the Banach space of bounded linear maps with the operator norm.

An algebra object in $\Ban_\leq$ is exactly a Banach algebra.
Every algebra object in $\Ban$ is isomorphic to a Banach algebra by norm rescaling.
In the category $\Alg(\Ban_{\leq})$ on the other hand equivalent norms do \emph{not} define isomorphic objects.
\begin{definition}
    Let $(\Alg(\Ban),\otimes_\pi)$ be the symmetric monoidal $1$-category whose objects are Banach algebras and whose morphisms are all bounded Banach algebra homomorphisms.
    The multiplication on $A \otimes_\pi B$ is given by 
    \begin{equation*}
    (a_1 \otimes b_1) (a_2 \otimes b_2) = a_1 a_2 \otimes b_1 b_2.
    \end{equation*}
\end{definition}

\begin{convention}
\label{conv:tensor}
    From now on, an unadorned $\otimes$ between Banach spaces, Banach algebras, etc.\ denotes the completed projective tensor product $\otimes_\pi$; we write the subscript only for emphasis or contrast.
\end{convention}

\begin{example}
    Every finite-dimensional algebra $A$ over $\R$ is uniquely Banach up to continuous isomorphism.
    Indeed, pick any norm on $A$ which will be complete and the linear map $A \otimes A \to A$ given by multiplication is continuous since $A$ is finite-dimensional.
    If necessary, rescale the norm so that $\|ab\| \leq \|a\|\|b\|$.
\end{example}

\paragraph{Banach modules}
The coequalizer of $f \colon X \to Y$ and $g \colon X \to Y$ is $Y/\overline{\im(f-g)}$~\cite[1.18]{MR533819}. The projective tensor product $\otimes$ preserves coequalizers in each variable.

Let $\Mor(\Ban)$ be the Morita $(2,2)$-category of Banach algebras (Lemma~\ref{lem:morexists}). 
A $1$-morphism in $\Mor(\Ban)$ from $B$ to $A$, denoted $B\rightsquigarrow A$, is a Banach $(A,B)$-bimodule, which is a Banach space $M$ equipped with bounded maps $A \otimes_\pi M \to M$ and $M \otimes_\pi B \to M$ satisfying the usual (unital) bimodule axioms.
We warn the reader that the two indexings run in opposite directions: $\Mor(\Ban)(B,A)$ is the category of $(A,B)$-bimodules.

\begin{remark}
We could also consider $\Mor(\Ban_{\leq})$ in which Banach $(A,B)$-bimodules come with contractive actions.
The functor $\Mor(\Ban_{\leq}) \to \Mor(\Ban)$ is not just essentially surjective on objects, but also on $1$-morphisms;
for left modules see~\cite[Discussion below Def.~2.6.1]{MR1816726}.
\end{remark}

By the universal property of the projective tensor product, an action map $A \otimes_\pi M \to M$ is the same as a continuous bilinear map $A \times M \to M$.
The composition of an $(A,B)$-bimodule $M$ and a $(B,C)$-bimodule $N$ is given by the relative (projective) tensor product
\begin{equation}
    \label{eq:reltensor}
M \otimes_B N = \frac{M \otimes_\pi N}{\overline{(mb \otimes n - m \otimes bn)}}
\end{equation}
as a special case of \eqref{eq:generalreltensor}.
This vector space is equipped with the usual quotient norm of the Banach space $M \otimes_\pi N$ by a closed subspace~\cite[Theorem 1.5.3]{kadisonringrose}.
The symmetric monoidal structure on $\Mor(\Ban)$ is again given by the projective tensor product of algebras.
\begin{convention}
\label{conv:leftright}
    Unlike in most of the operator algebra literature, we use the convention that an $A$-module has an action on the left unless specified otherwise.
    An endomorphism of the free module $A^n$ is given by multiplication by a matrix on the right, and since such multiplications compose in the opposite order we obtain
    \begin{equation}
        \label{eq:endop}
        \End_A(A^n) \cong M_n(A)^{\op},
    \end{equation}
    in particular $\End_A(A) \cong A^{\op}$.
    This distinction will become especially important in  Section~\ref{sec:graded alg}.
\end{convention}

\paragraph{Finitely generated projective modules}

Because of the special role of finitely generated projective modules in $K$-theory, we will be especially interested in finitely generated projective Banach $A$-modules.
\begin{definition}
    \label{def:fgpmod}
    A \emph{finitely generated projective $A$-module} $M$ is a left $A$-module which is (algebraically) isomorphic to a direct summand of the free module $A^n$, for some $n\in \N$.
    We write $\ModOne_A$ for the category of finitely generated projective left $A$-modules.
\end{definition}
Every $A$-linear map between free $A$-modules is a matrix with entries in $A$ and hence continuous, so algebraic direct summands of $A^n$ are closed by the splitting of bounded idempotents in $\Ban$ recalled in Appendix~\ref{sec:enriched}.
Consequently any finitely generated projective module carries a unique Banach topology, and module maps between finitely generated projective Banach modules are automatically bounded.

It follows by the universal property of \eqref{eq:reltensor} that for an
$(A,B)$-bimodule $M$, $M \otimes_B (-) \colon \ModOne_B \to \ModOne_A$ is left adjoint to $\Hom_A(M,-)$, the Banach module of bounded $A$-linear maps.

\begin{proposition}\label{prop:adj}
An $(A,B)$-bimodule $M$ admits a right adjoint in $\Mor(\Ban)$ if and only if $M$
is finitely generated projective as a left $A$-module, in which case the right
adjoint is $\Hom_A(M,A)$.
\end{proposition}
See~\cite{MR733829} for the algebraic version of this statement.
\begin{proof}
We first observe as a consequence of a theorem of Houzel~\cite[Corollaire de la Proposition 6]{houzel1973} that a Banach $A$-module $M$ is finitely generated projective if and only if $\id_M$ is \emph{nuclear}\footnote{see~\cite[Section 6]{kelly2025localising} for an introduction to nuclearity in a similar categorical setup.
}, i.e.\ lies in the image of $M \otimes_A \Hom_A(M,A) \to \End_A(M)$.
More specifically, he proved the analogous result for a multiplicatively convex unital bornological algebra $A$ and a bornological $A$-module,
see~\cite[Proposition~4.3]{MR2608194}. A unital Banach algebra equipped with its von Neumann bornology is a complete multiplicatively convex unital bornological algebra, Banach modules are bornological modules, and on Banach spaces the completed bornological tensor product and equibounded hom agree with the projective tensor product~\cite[Sec.~1.3.6]{zbMATH05176991} and the Banach space of bounded maps.
Houzel's theorem therefore applies verbatim to Banach modules.

To prove the proposition: if $M$ is a direct summand of $A^{n}$, a finite dual basis
$\{(\phi_i,m_i)\}_{i\le n}\subseteq\Hom_A(M,A)\times M$ furnishes a bounded coevaluation with evaluation as the counit, realizing $\Hom_A(M,A)$ as a right adjoint.
Conversely, given a right adjoint $N$ with unit $\eta\colon\id_B\Rightarrow N\otimes_A M$ and counit $\epsilon\colon M\otimes_B N\Rightarrow\id_A$, choose a representative $\eta(1_B)=\sum_{i\in\N}\psi_i\otimes m_i$ in the projective tensor product with $\sum_i\|\psi_i\|\,\|m_i\|<\infty$, and put $\phi_i:=\epsilon(-\otimes\psi_i)\in\Hom_A(M,A)$.
The triangle identities give $\id_M=\sum_{i}\phi_i(\cdot)\,m_i$ with $\sum_i \|\phi_i\|\|m_i\| \leq \|\epsilon\| \sum_i \|\psi_i\|\|m_i\| < \infty$, so $\id_M$ is nuclear and Houzel's theorem applies.
\end{proof}

The statement for left adjoints in $\Mor(\Ban)$ is analogous.

\begin{definition}
\label{def:bimbamrm}
    Let $\BimBanOne$ be the symmetric monoidal $(2,1)$-subcategory of $\Mor(\Ban)$ defined in \ref{def:bimgeneral}.
    By Proposition~\ref{prop:adj} this $2$-category has
    \begin{itemize}
        \item the same objects as $\Mor(\Ban)$;
        \item those $1$-morphisms $M \colon A \rightsquigarrow B$ which are finitely generated and projective as $B$-modules;
        \item all invertible $2$-morphisms between those.
    \end{itemize} 
    Note that $\Hom_{\BimBanOne}(\mathbb{F},A)$ coincides with $\ModOne_A^{\cong}$, the groupoid of finitely generated projective left $A$-modules.
\end{definition}

There are symmetric monoidal functors $\Alg(\Ban) \to \Mor(\Ban)$ (respectively $\Alg(\Ban)^{\op} \to \Mor(\Ban)$) sending an algebra homomorphism $f \colon A \to B$ to the $(B,A)$-bimodule $B_f:A\rightsquigarrow B$ with action twisted on the right by $f$ (respectively $f \mapsto {}_f B$ twisted on the left).
Since $B_f$ is always a finitely generated projective $B$-module, but $_f A$ not necessarily, this induces functors $\Alg(\Ban) \to \BimBanOne$ and $\Alg(\Ban)^{\op}_{\mathrm{fgp}} \to \BimBanOne$.
This is why only some homomorphisms will give wrong-way maps in $K$-theory, compare~\cite[Lemma 2.3.5]{mertsch2020geometric}.
The induced functor $f_* := B_f \otimes_A (-) \colon \ModOne_A \to \ModOne_B$ is called \emph{extension of scalars} and $f^* := {}_f B \otimes_B (-)\colon \ModOne_B \to \ModOne_A$ \emph{restriction of scalars}.

\paragraph{Morita functoriality of taking modules}

It follows from Proposition~\ref{prop:daniel} applied
to the $(2,2)$-category $B \subseteq \Mor(\Ban)$ of Banach algebras, one-sided finitely generated projective modules and bimodule maps that 
    there is a lax symmetric monoidal functor $\Hom_{\BimBanOne}(\mathbb{F},-) = \ModOne\colon \BimBanOne \to \Cat^\Sigma_1$ of symmetric monoidal $(2,1)$-categories.
The functor concretely maps
\begin{itemize}
    \item the Banach algebra $A$ to the category $\ModOne_A$;
    \item the $(A,B)$-bimodule $M$ to the functor $M \otimes_B (-) \colon \ModOne_B \to \ModOne_A$;
    \item the (invertible) $(A,B)$-bimodule map $\phi \colon M_1 \to M_2$ to the natural transformation $M_1 \otimes_B (-) \Rightarrow M_2 \otimes_B (-)$ given on $N \in \ModOne_B$ by $\phi \otimes \id_N$;
    \item the monoidal data $\ModOne_A \times \ModOne_B \to \ModOne_{A \otimes_\pi B}$ takes $\otimes_\pi$ of finitely generated projective modules.
\end{itemize}
Next we will more generally consider the Banach space of maps between $A$-modules.

\paragraph{Banach categories}

A \emph{Banach category} is a category enriched in $\Ban_{\leq}$.
Explicitly, a Banach category over $\mathbb{F}$ is a category $\mathsf{C}$ equipped with a Banach space structure on $\mathsf{C}(x,y)$ over $\mathbb{F}$ for all $x,y \in \mathsf{C}$ such that composition is bilinear and contractive, i.e.\ $\|f \circ g\| \leq \|f\| \|g\|$.
The $1$-morphisms $F \colon \mathsf{C} \to \mathsf{D}$ are called \emph{Banach functors} and are defined to be $\Ban$-enriched functors, i.e.\ the functor data includes \emph{bounded} maps $\mathsf{C}(x,y) \to \mathsf{D}(Fx,Fy)$.
For a category enriched in $\Ban$ composition is only continuous and
Banach categories span a full $(2,2)$-subcategory $
\BanCat \subseteq \Cat_1(\Ban)$.
The category $\BanCat$ carries the usual symmetric monoidal structure $\boxtimes$ of enriched categories given by $(\mathsf{C}\boxtimes \mathsf{D})((c_1,d_1),(c_2,d_2))=\mathsf{C}(c_1,c_2)\otimes_\pi \mathsf{D}(d_1,d_2)$, see Appendix \ref{sec:enriched} for more on enriched categories.

Let $\BanCat^{\mathrm{add}}$ be the full subcategory of $\BanCat$ on the \emph{additive Banach categories}, i.e.\ those categories which admit finite coproducts.
Finite coproducts are automatically enriched direct sums by Lemma~\ref{lem:openmapping} and by Lemma~\ref{lem:enrichedbiproduct} every Banach functor preserves them.

The self enrichment of $\Mor(\Ban_{\leq})$ given in general in Equation~\eqref{eq:selfenrich} is a closed subspace of the Banach space of bounded linear maps, so $\underline{\Hom}_{A,B}(M,N)$ is the Banach space of bounded bimodule maps with the operator norm.
In particular $\ModOne_A = \Hom_{\BimBanOne}(\mathbb{F}, A)$ is a Banach category.
Recall that for the special case of endomorphisms of the free module $A^n$, the operator norm agrees with the standard Banach algebra structure on $M_n(A) = \End_A(A^n)^{\op}$, cf.~\eqref{eq:endop}. 
Therefore the operator norm on $\underline{\Hom_A}(M,N) := \underline{\Hom}_{A,\mathbb{F}}(M,N)$ for $M,N \in \ModOne$ is equivalent to the norm on the matrix algebra obtained by choosing embeddings of $M$ and $N$ into free modules, also see~\cite[Remark I.6.22]{karoubi_k-theory_1978}. Outside of this paragraph we suppress the underline and simply write $\Hom_A(M,N)$ for this Banach space.
To emphasize the extra structure, we denote the $\Ban_{\leq}$-enriched category by $\ModEn_A$.
Even in the enriched sense, $\ModEn_A$ is the idempotent completion of its full subcategory $\Free_A$ of free modules (Lemma~\ref{lem:idemfree}).

\paragraph{\texorpdfstring{$C^*$}{C*}-algebras}

A \emph{complex $C^*$-algebra} is a complex Banach algebra equipped with a complex antilinear map $* \colon A \to A$ such that $(ab)^* = b^* a^*$, $a^{**} = a$ and $\| a^* a \| = \| a\|^2$ for all $a,b \in A$.

A \emph{Real $C^*$-algebra} is a complex $C^*$-algebra $A$ equipped with a complex antilinear $*$-algebra homomorphism $\overline{(-)} \colon A \to A$ such that $\overline{\overline{a}} = a$.
A \emph{real $C^*$-algebra} is a real Banach algebra $B$ for which there exists a Real $C^*$-algebra $A$ such that
\begin{equation*}
B = \{ a \in A : \overline{a} = a\} \,,
\end{equation*}
with the induced Banach algebra structure.

Both complex and real $C^*$-algebras form a category $\CStarAlg$ in which morphisms $f \colon A \to B$ are \emph{$*$-homomorphisms}; algebra homomorphisms such that $f(a^*) = f(a)^*$.
Any $*$-homomorphism is automatically contractive and so there is a functor $\CStarAlg \to \Alg(\Ban_\leq)$.

The projective tensor product of $C^*$-algebras need not be a $C^*$-algebra and there are two common choices of $C^*$-tensor product: $\otimes_{\mathrm{min}}$ and $\otimes_{\mathrm{max}}$.
The forgetful functor $(\CStarAlg, \otimes_{\mathrm{max}}) \to (\Alg(\Ban_\leq), \otimes_\pi)$ preserves filtered colimits and is lax symmetric monoidal, i.e.\ if $A,B$ are $C^*$-algebras, there are continuous maps $A\otimes_\pi B \to A\otimes_{\mathrm{max}} B \to A\otimes_{\mathrm{min}}B$.

\begin{remark}
    Let $A$ be a $C^*$-algebra and $M$ a finitely generated and projective $A$-module.
    Then $M$ admits an $A$-valued inner product which is unique up to unitary isomorphism~\cite[Remark A.4.3]{higsonroe}.
    The Banach space structure induced by this inner product is equivalent to the Banach space structure on $M$ induced by an embedding $M \hookrightarrow A^n$ into a free module. Moreover, for maps $T\colon M_1 \to M_2$ being adjointable is equivalent to being $A$-linear. The norm underlying the usual $C^*$-algebra structure on $\End_A M$ is the operator norm and so reproduces the desired inner hom in Banach spaces.
\end{remark}

If $X$ is a compact Hausdorff space, then the complex algebra $C(X;\C)$ of continuous functions on $X$ with the supremum norm is a Real $C^*$-algebra with $(-)^*$ and $\overline{(-)}$ both given by pointwise complex conjugation.

\paragraph{The \texorpdfstring{$(\infty,1)$}{(infinity,1)}-category of Banach algebras and homomorphisms}
This section establishes the $(\infty,1)$-category of Banach algebras and homomorphisms.
We provide a concrete model as a $\Top$-enriched category and a characterization of this as a localization of the $1$-category of Banach algebras and homomorphisms at the homotopy equivalences in Proposition~\ref{prop:localizationofBanAlg}.

\begin{definition}
Let $A$ and $B$ be Banach algebras.
    A \emph{homotopy} between Banach algebra homomorphisms $f_0,f_1\colon A \to B$ is a homomorphism $H\colon A\to C(\Delta^1; B)$ such that $\ev_i\circ H=f_i$ for both $i$.
    A Banach algebra homomorphism $f\colon A\to B$ is called a \emph{homotopy equivalence}, if there is $g\colon B\to A$ and homotopies $\id_B\simeq fg$, $\id_A\simeq gf$.
\end{definition}

Gelfand duality identifies commutative $C^*$-algebras with compact Hausdorff spaces.
It is usually stated for $*$-homomorphisms, whereas the morphisms of $\Alg(\Ban)$ are bounded algebra homomorphisms, required neither to be contractive nor to preserve the $*$-operation.
The following lemma says that no generality is gained or lost: the notions of morphism agree, and each is the same as a continuous map in the opposite direction.

\begin{lemma}
\label{lem:gelfandhom}
    Let $X$ and $Y$ be compact Hausdorff spaces and $\mathbb{F} \in \{\R, \C\}$.
    Every algebra homomorphism $\varphi \colon C(X;\mathbb{F}) \to C(Y;\mathbb{F})$ is of the form $\varphi(f) = f \circ \tau$ for a unique continuous map $\tau \colon Y \to X$.
    In particular $\varphi$ is automatically bounded, contractive and $*$-preserving, and
    \begin{equation*}
        \Hom_{\Alg(\Ban)}(C(X;\mathbb{F}),C(Y;\mathbb{F})) \;\cong\; \Top(Y,X) \,.
    \end{equation*}
\end{lemma}
\begin{proof}
    First let $\mathbb{F} = \C$.
    Recall that the \emph{Gelfand spectrum} $\Spec A$ of a commutative unital complex $C^*$-algebra $A$ is the set of algebra homomorphisms $A \to \C$, which are automatically bounded of norm one~\cite[Theorem 4.43]{allan2011banach}, equipped with the topology of pointwise convergence.
    \emph{Gelfand duality} states that $\Spec$ and $C(-;\C)$ are mutually inverse equivalences between $\CHaus^{\op}$ and the category of commutative complex $C^*$-algebras and $*$-homomorphisms; in particular the canonical map $X \to \Spec C(X;\C)$ is a homeomorphism, i.e.\ every character of $C(X;\C)$ is an evaluation at a point.
    For $y \in Y$ the composite $\ev_y \circ \varphi$ is a character of $C(X;\C)$, so $\ev_y\circ \varphi = \ev_{\tau(y)}$ for a unique $\tau(y) \in X$.
    Thus $\varphi(f) = f\circ \tau$ pointwise, whence $\|\varphi(f)\| = \sup_{y}|f(\tau(y))| \leq \|f\|$ and $\varphi(\bar f) = \bar f \circ \tau = \overline{\varphi(f)}$.
    The map $\tau$ is continuous because $f \circ \tau$ is continuous for every $f \in C(X;\C)$ and the topology of the compact Hausdorff space $X$ is the initial topology for $C(X;\C)$.
    For $\mathbb{F} = \R$, view $C(X;\R)$ as the fixed points of the Real $C^*$-algebra $C(X;\C)$: a character of $C(X;\R)$ extends $\C$-linearly along $C(X;\R) \otimes_\R \C \cong C(X;\C)$ to a character of $C(X;\C)$, hence is again an evaluation, and the same argument applies.
\end{proof}

\begin{remark}
    Let $A = C(X)$ and $B = C(Y)$ for compact Hausdorff spaces $X$ and $Y$ and let $f_0, f_1 \colon Y \to X$ be continuous maps.
    Then homotopies between $f_0$ and $f_1$ in the classical sense correspond to homotopies between the induced maps $C(X) \to C(Y)$: by Lemma~\ref{lem:gelfandhom}, a homomorphism $C(X) \to C(\Delta^1; C(Y)) \cong C(Y \times \Delta^1)$ is the same as a continuous map $Y \times \Delta^1 \to X$.
\end{remark}
In $\Top$, the internal hom $\Top(X,Y)$ is equipped with the $k$-ification of the compact-open topology, see \cite{strickland2009category} for a more detailed exposition.
We topologize the set of bounded algebra homomorphisms $\Hom(A,B)$ via the adjunction
\begin{equation}
    \label{eq:adjunctionhomomorphisms}
    \Top(X,\Hom(A,B)) \cong \Hom(A,C(X;B)) \,,
\end{equation}
where $X$ is a compact Hausdorff space.
Explicitly, this topology is the $k$-ification of the compact-open topology or the $k$-ification of the point-norm topology.

\begin{definition}
\label{def:BanAlgInf}
    We denote by $\BanAlgEn$ the topologically enriched category of Banach algebras equipped with the $k$-ification of the compact-open topology on $\Hom(A,B)$.
    The $(\infty,1)$-category $\BanAlgInf$ of Banach algebras and Banach algebra homomorphisms is defined as the homotopy coherent nerve $N^{\mathrm{hc}}\BanAlgEn$.
\end{definition}

The set of morphisms $A\to B$ in the homotopy category $\Ho(\BanAlgInf)$ is the set of homotopy classes of bounded algebra homomorphisms.

\begin{warning}
The operator norm topology is strictly finer than the topology on the hom-spaces of $\BanAlgEn$; let $A = C([0,1];\C )$ and $B = \C$.
    Then the sequence $\ev_{1/n}$ does not converge in the supremum norm since for every $n \neq m$
    \begin{equation*}
    \|\ev_{1/m} - \ev_{1/n}\|  = \sup_{\|f\| = 1} \|f(1/n) - f(1/m)\| \geq 1
    \end{equation*}
    by taking a bump function around $1/m$ that vanishes at $1/n$.
    In the topology of $\BanAlgEn$, on the other hand, $\ev_{1/n}$ does converge to $\ev_0$:
    applying the adjunction~\eqref{eq:adjunctionhomomorphisms} to the identity homomorphism $C([0,1];\C) \to C([0,1];\C)$ exhibits $t \mapsto \ev_t$ as a continuous map $[0,1] \to \Hom(C([0,1];\C),\C)$.
\end{warning}

\begin{remark} \label{rmk:banalgsimplicialdescription}
    By Lemma \ref{lem:homotopycoherentnerve}, we have
    \begin{equation*}
        \BanAlgInf \simeq N^{\mathrm{hc}} \BanAlgEn \simeq \colim_{\Delta^{\op}} j(\BanAlgEn)_n \,.
    \end{equation*}
    The category $j(\BanAlgEn)_n$ admits the following description:
    Its objects are Banach algebras and the morphisms are $\Delta^n$-families of Banach algebra homomorphisms using the $k$-ification of the compact-open topology, or equivalently, Banach algebra homomorphisms $A\to C(\Delta^n; B)$.
\end{remark}

\begin{definition}
    Let $(\calC,\otimes)$ be symmetric monoidal and $W$ a class of morphisms in $\calC$. 
    We say that $W$ is \emph{$\otimes$-stable} if $\id_c\otimes f\in W$  for each $c\in C$ and $f\in W$.
\end{definition}

The class $W$ of homotopy equivalences is the smallest class that satisfies the $2$-out-of-$3$-property, is $\otimes$-stable and contains the map $A\to C(\Delta^1;A)$ for all $A$.
Any symmetric monoidal functor that maps $A\to C(\Delta^1;A)$ to an equivalence will hence automatically factor through the localization $\Alg(\Ban)[W^{-1}]$.

\begin{proposition}
\label{prop:localizationofBanAlg}
    The map $\Alg(\Ban) \to \BanAlgInf$ exhibits the target as the Dwyer--Kan localization at homotopy equivalences, i.e.\ for every symmetric monoidal $(\infty,1)$-category $\calC$ we have an equivalence
    \begin{equation*}
        \Fun(\BanAlgInf,\calC) \xrightarrow{\simeq} \Fun^W(\Alg(\Ban),\calC) \,,
    \end{equation*}
    which identifies functors out of $\BanAlgInf$ with functors out of $\Alg(\Ban)$ that map $W$ to equivalences in $\calC$.
    The localization is symmetric monoidal, i.e.\ the localization map is symmetric monoidal and we also obtain an equivalence
    \begin{equation*}
        \Fun_{\mathrm{lax}}(\BanAlgInf,\calC) \xrightarrow{\simeq} \Fun_{\mathrm{lax}}^W(\Alg(\Ban),\calC) \,.
    \end{equation*}
\end{proposition}
\begin{proof}
    The proof is the same as in~\cite[Prop. 3.5]{bunkeKKEth}.
    It only depends on the adjunction~\eqref{eq:adjunctionhomomorphisms}. 
    Monoidality follows since $W$ is $\otimes$-stable.
\end{proof}
In particular, the class $W$ of homotopy equivalences of Banach algebras is \emph{saturated}, i.e.\ the homotopy equivalences are precisely the algebra homomorphisms which are mapped to equivalences in the localization $\BanAlgInf[W^{-1}]$.

\begin{remark}
\label{rem:generatinghomotopy}
    In the setting of $C^*$-algebras there is a slight simplification.
    We have $C(X;A) \cong C(X) \otimes_{\mathrm{min}} A$ for any compact Hausdorff space $X$~\cite[Theorem 6.4.17]{MR1074574}, and in fact for any tensor product since every commutative $C^*$-algebra is nuclear~\cite[Thm.~II.9.4.4]{MR1656031}.
    Therefore for $C^*$-algebras, $W$ is even generated by the single element $\R \to C(\Delta^1)$ when requiring closure under tensor products.
    
    For Banach algebras $A$, it is still true that there is an isomorphism $C(X)\otimes_\epsilon A\cong C(X;A)$ of Banach algebras for the injective tensor product~\cite[Proposition 1.5.6]{MR2458901}, but generally the projective tensor product $C(\Delta^1)\otimes_\pi A$ is smaller than the injective one.
    This discrepancy causes no further issues.
\end{remark}
\begin{remark}
    Note that the map $C(\Delta^1)\otimes_\pi A\to C(\Delta^1;A)$ is a homotopy equivalence since both algebras are homotopy equivalent to $A$.
    Moreover, $W$ is also the smallest class containing the two types of maps $A\to C(\Delta^1)\otimes_\pi A$ and $C(\Delta^1)\otimes_\pi A \to C(\Delta^1;A)$ which is $\otimes$-stable and closed under $2$-out-of-$3$. 
\end{remark}

\subsection{The \texorpdfstring{$K$}{K}-theory spectrum of a Banach algebra}
\label{sec:BanachKth}

In this section we give several definitions of the $K$-theory spectrum of a Banach algebra, and show their equivalence.
More specifically, in Subsection~\ref{sec:SerreSwan} we show that for a Banach algebra $A$ the $K$-theory of the simplicial ring $C(\Delta^n; A)$ (cf.\ Section~\ref{sec:sSetKth}) and the $K$-theory of the simplicially enriched category $\ModEn_A$ (cf.\ Section~\ref{sec:Banach}) are equivalent.
This follows from rigidity of idempotents or, geometrically, from a Serre--Swan theorem.
In Subsection~\ref{sec:kan} we will provide yet another characterization of $K$-theory of Banach algebras as a universal homotopy invariant approximation of algebraic $K$-theory.
The latter approach will play a main role in the rest of the paper.
From now on $\ModEn_A$ will be regarded as a simplicially enriched category via the norm topology on its mapping spaces, in more detail:

\begin{convention}
\label{conv:banenr}
    Using the lax symmetric monoidal functor $(\Ban_\leq,\otimes_\pi) \to (\Top,\times)$, we can regard every Banach category as a $\Top$ or $\sSet$-enriched category (by Convention \ref{conv:topenr}).
    This functor preserves finite products, hence the $\Top$-enriched category underlying a Banach category has enriched coproducts.
    We will implicitly consider additive Banach categories as symmetric monoidal simplicially enriched categories in this way.
\end{convention}

\begin{definition}
\label{def:KthBanAlg}
    The \emph{$K$-theory spectrum of a Banach algebra $A$} is 
    \begin{enumerate}
        \item the $K$-theory of the simplicially enriched category $\ModEn_A$ in the sense of Definition~\ref{def:topktheory2}.
        \begin{equation}
        \label{eq:def1}
        k(\ModEn_A)=(N^{\mathrm{hc}}(\ModEn_{A}^{\cong}) ,\oplus)^{\gp}
         \end{equation}
        \item the $K$-theory of the simplicial ring $C(\Delta^\bullet;A)$:
        \begin{equation}
        \label{eq:def2}
            \colim_{[n]\in \Delta^{\op}} \mathcal{K}^{\mathrm{alg}}(C(\Delta^n;A))
        \end{equation}
    \end{enumerate}
\end{definition}
\begin{theorem}[Theorem A]
    \label{thm:omnibusktheory}
    The definitions~\eqref{eq:def1} and~\eqref{eq:def2} are naturally equivalent.
    Moreover, the connective topological $K$-theory functor 
    \begin{equation*}
        k\colon \Alg(\Ban) \longrightarrow \Sp_{\geq 0}
    \end{equation*}
    is the initial homotopy invariant functor with a natural transformation $\mathcal{K}^{\mathrm{alg}}\to k$ from algebraic $K$-theory.
    $k$ is lax symmetric monoidal.
    It satisfies excision and preserves filtered colimits along contractive homomorphisms.
\end{theorem}
The equivalence of the definitions is proven in Proposition~\ref{prop:compareK}.
The universal characterization and lax monoidality are established in Section~\ref{sec:kan}.
Excision is stated and proven in Section~\ref{sec:excision}.
The preservation of filtered colimits along contractive homomorphisms is proven in Proposition~\ref{prop:filteredcolimits}.

We first need to establish some properties of the $\E_\infty$-monoid $N^{\mathrm{hc}}(\ModEn_A^{\cong})$.
The key simplification over algebraic $K$-theory is that it is telescopic, and so we recover the main result of \cite{MR1404916}:
\begin{lemma}
    \label{lem:modulestelescopic}
        Let $A$ be a Banach algebra.
        The $\E_\infty$-monoid $(N^{\mathrm{hc}}(\ModEn_A^{\cong}),\oplus)$ is telescopic and the group completion is given by
        \begin{equation*}
             K_0(A) \times B\GL_{\infty}(A) \,.
        \end{equation*}
    \end{lemma}
    \begin{proof}
    Free modules form a cofinal system in projective modules in the sense that the inclusion from the $\E_\infty$-submonoid of free modules of finite rank is quasi-surjective.
        We obtain by Lemma~\ref{lem:quasi-surjlocalization} that the telescope is given by 
        \begin{equation*}
        \colim(M \xrightarrow{\oplus A} M \xrightarrow{\oplus A}\dots),
        \end{equation*}
        where $M := N^{\mathrm{hc}}(\ModEn_A^{\cong})$.
        Since every finitely generated projective module is a summand of a free module, we thus get an equivalence of underlying spaces
        \begin{equation*}
        M[\pi_0(M)^{-1}] \simeq K_0(A) \times B\GL_{\infty}(A),
        \end{equation*}
        where $\GL_{\infty}(A) = \colim_n \GL_n(A)$ is given the subspace topology for the norm topology on $M_n(A)$.
        \footnote{This is \emph{not} a splitting of $\E_\infty$-monoids, only of underlying spaces.}
        By Lemma~\ref{lem:K1}, $\pi_0 \GL_{\infty}(A) \cong \pi_0 U_{\infty}(A) \cong K_1(A)$ is an abelian group.
        Since this result is independent of the basepoint, the result follows by Proposition~\ref{prop:cyclicinvariance}.
    \end{proof}
    The following result is standard, see \cite[Proposition 7.1.2]{MR1222415} for complex $C^*$-algebras.
    \begin{lemma}
\label{lem:K1}
    Let $A$ be a Banach algebra.
    Then $ \pi_0 \GL_\infty(A)$ is an abelian group.
\end{lemma} 
\begin{proof}
    For $c,d \in \GL_n(A)$ there are paths between the block diagonal matrices in $\GL_{2n}(A)$
    \begin{equation*}
         \begin{pmatrix}
            cd&  \\
            & 1
        \end{pmatrix}
        \simeq
        \begin{pmatrix}
            c & \\
            & d
        \end{pmatrix} \simeq
        \begin{pmatrix}
            d & \\
            & c
        \end{pmatrix}
        \simeq 
        \begin{pmatrix}
            dc & \\
            & 1
        \end{pmatrix}
        \,.
    \end{equation*}
    The first path is given by
    \begin{equation*}
        t\mapsto \begin{pmatrix}
            c & \\
            & 1
        \end{pmatrix}U(t)\begin{pmatrix}
            1&\\
            & d
        \end{pmatrix}U(t)^{-1}\,,
        \qquad U(t)=\begin{pmatrix}
            \cos t & -\sin t\\
            \sin t & \cos t
        \end{pmatrix}\,.
    \end{equation*}
    The second path is given by
    \begin{equation*}
        t\mapsto U(t)\begin{pmatrix}
            c& \\
            & d
        \end{pmatrix}
        U(t)^{-1} \,.
    \end{equation*}
    Finally, the third path is the same as the first one, with $c,d$ interchanged.
    This yields $cd=dc$ in $\pi_0 \GL_\infty(A)$.
\end{proof}
\begin{example}
\label{ex:ku}
Consider the Banach algebra $\C$ and $\VectEn_{\C}$ the Banach category of finite-dimensional complex vector spaces.
    The topologically enriched groupoid $\VectEn_\C^{\cong}$ has objects $\C^n$ and spaces of automorphisms are given by $\GL_n(\C)$.
    It carries a symmetric monoidal structure $\oplus$ with $\C^n\oplus \C^m=\C^{n+m}$ and $\GL_n(\C)\times \GL_m(\C) \to \GL_{n+m}(\C)$ given by block direct sum of matrices.
    Applying $N^{\mathrm{hc}}$, we obtain the $\E_\infty$-monoid 
    \begin{equation*}
        \bigsqcup_{n\in \N} B\GL_n(\C) \simeq \bigsqcup_{n\in \N} BU_n\,.
    \end{equation*}
    Using the telescopic property of Banach algebras \ref{lem:modulestelescopic}, its group completion is 
    \begin{equation*}
        ku=\colim_{\N}\left( \bigsqcup_{n\in \N} BU_n \xrightarrow{-\oplus \C} \bigsqcup_{n\in \N}BU_n \xrightarrow{-\oplus \C} \dots \right) =\Z\times BU \,.
    \end{equation*}
    This example works verbatim when replacing $\C$ by $\R$, $U_n$ by $O_n$ and $ku$ by $ko$.
\end{example}

\begin{remark}
    Since $\ModEn_A$ is the idempotent completion of $\Free_A$ (Lemma~\ref{lem:idemfree}), part 1 of Definition~\ref{def:KthBanAlg} can be reformulated as a group completion of the space of projections in $M_\infty(A)$ under $\oplus$.
    On $\pi_0$, this translation in particular recovers the classic formulation of $K_0(A)$ as the group completion of homotopy classes of projections.
\end{remark}

\subsubsection{Serre--Swan and comparing definitions}
\label{sec:SerreSwan}
Throughout this section, $A$ is a Banach algebra and $X$ a compact Hausdorff space.
\begin{definition}
    We define $\ModEn_A(X)$ as the Banach category of bundles of $A$-modules that are locally isomorphic to a trivial bundle $U\times P\to U$, where the $A$-module $P$ is finitely generated projective.
\end{definition}
\begin{lemma}
\label{lem:idempotents}
    Let $e\colon X\to A$ be a continuous family of idempotents.
    Then for every $x_0\in X$ there is a neighborhood $V$ with $e(x)=z(x)e(x_0)z(x)^{-1}$ for all $x\in V$.
    Here, $x\mapsto z(x)$ is a continuous path of invertible elements in $A$ with $z(x_0)=1$.
    Moreover, if $X = [0,1]$ then $V$ can be chosen to be all of $X$.
\end{lemma}
\begin{proof}
    Take a closed neighborhood $V$ of $x_0$ such that $\sup_{x \in V}\|e(x) - e(x_0)\|_A$ is sufficiently small.
    Let $f \colon V \to A$ be the idempotent constantly equal to $e(x_0)$.
    It follows by~\cite[Proposition 4.3.2]{MR1656031} that there is an invertible element $z \in C(V;A)$ such that $z^{-1} e z = f$.
    For the final statement, see~\cite[Proposition 4.3.3]{MR1656031}.
\end{proof}
The following is a version of the Serre--Swan theorem.
\begin{proposition}
    \label{prop:serreswan}
    The global sections functor
    \begin{equation*}
        \Gamma\colon \ModEn_A(X) \to \ModEn_{C(X;A)} 
    \end{equation*}
    induces an isomorphism between the Banach categories of finitely generated projective modules over $C(X;A)$ and of bundles of finitely generated projective $A$-modules over $X$.
\end{proposition}
\begin{proof}
    Consider the functor $F \colon \Free_{C(X;A)} \to \ModEn_A(X)$ that takes $C(X;A)^n$ to the trivial bundle $A^n \times X \to X$ of rank $n$.
    Using the universal property of idempotent completion (Lemma~\ref{lem:idemcompletion}), $F$ will uniquely extend to a functor from $\Idem(\Free_{C(X;A)}) \cong \ModEn_{C(X;A)}$ (Lemma~\ref{lem:idemfree}) if we can show that the image of $F$ is idempotent complete.
    This follows by Lemma~\ref{lem:idempotents} applied to $M_n(A)$; if $e \in M_n(C(X;A)) \cong C(X; M_n(A))$ is an idempotent on the free module, then the subbundle of $A^n \times X$ it picks out is locally trivial.

    There is a functor $\Gamma$ in the other direction which maps a bundle to its module of global sections.
    Since the global sections of a trivial rank $n$ bundle are $C(X;A^n)$, it follows by the uniqueness of extensions to the idempotent completion (Lemma~\ref{lem:idemcompletion}) that $\Gamma$ is a right inverse of $F$.

    Hence we are done once we can show $F$ is essentially surjective.
    Let $\{U_i\}$ be a covering of local trivializations for the object $E \to X$ of $\ModEn_A(X)$ giving $E|_{U_i} \cong P_i \times U_i$ for some finitely generated projective modules $P_i$.
    For every $U_i$ pick an embedding $P_i \subseteq A^{n_i}$ giving a corresponding idempotent $e_i \in M_{n_i}(A)$.
    Because $X$ is compact, we can assume the covering is finite, and hence we can assume $n_i = n$ is constant in $i$.
    
    In every local trivialization $U_i$ we obtain a family $f_i \colon U_i \to M_n(A)$ of idempotents such that $f_i(x) e_i = f_i(x) = e_i f_i(x)$ for all $x \in U_i$.
    We have to show that the subbundle picked out by $f_i$ is locally trivial.
    This follows by Lemma~\ref{lem:idempotents} applied to $M_n(A)$; after choosing a potentially finer finite covering, the subbundle over $U_i$ corresponding to $f_i$ becomes equivalent to a constant idempotent.
\end{proof}

\begin{corollary}
\label{cor:contractibleSerreSwan}
    If $X$ is a contractible space, then the functor $\ModEn_A \to \ModEn_{C(X;A)}$ of Banach categories induced by $X \to \pt$ is essentially surjective.
\end{corollary}
\begin{proof}
     It follows by Proposition~\ref{prop:serreswan} that any finitely generated projective $C(X;A)$-module is isomorphic to a bundle of finitely generated projective $A$-modules over $A$.
    Such a bundle is necessarily trivial since $\Delta^n$ is contractible.
\end{proof}

\begin{proposition}
\label{prop:compareK}
    Let $A$ be a Banach algebra. 
    Consider $\ModEn_A$ as a simplicially enriched category coming from the Banach category structure on $\ModEn_A$.
    Then $k(\ModEn_A)$ agrees with the $K$-theory of the simplicial ring $C(\Delta^n;A)$.
    That is, there is a natural equivalence between
    \begin{equation}
        k(\ModEn_A)=(N^{\mathrm{hc}}(\ModEn_{A}^{\cong}) ,\oplus)^{\gp} \,,
    \end{equation}
    and 
    \begin{equation}
        \label{eq:colimitexpressionktheory}
        k_A = \colim_{\Delta^{\op}} \mathcal{K}^{\mathrm{alg}}(C(\Delta^\bullet;A)) \,.
    \end{equation}
\end{proposition}
\begin{remark}
    The formula in the preceding proposition is known, e.g.\ Formula~\eqref{eq:colimitexpressionktheory} appears as~\cite[Thm.~5.16]{aoki2024semi}.
    We could however not locate this exact comparison result in the literature.
\end{remark}
\begin{proof}
It suffices to show that the simplicial categories $j(\ModEn_A)$ and $[n] \mapsto \ModOne_{C(\Delta^n;A)}$ are equivalent.
   Consider the functor
    \begin{align*}
        j(\ModEn_A)_n &\longrightarrow \ModOne_{C(\Delta^n;A)} \\
        M &\longmapsto C(\Delta^n;M),
    \end{align*} 
    where $C(\Delta^n;M) \cong C(\Delta^n;A) \otimes_A M$ is the extension of scalars of $M$ to $C(\Delta^n;A)$ along the algebra homomorphism $A \to C(\Delta^n;A)$. 
    This means that it is given on morphisms by mapping 
    \begin{equation*}
    \Hom_{j(\ModEn_A)_n }(M_1,M_2) = \Top(\Delta^n, \Hom_A(M_1, M_2)) \cong \Hom_A(M_1, C(\Delta^n,M_2)) \ni f
    \end{equation*}
    to its extension of scalars in 
    \begin{equation*}
    \Hom_{C(\Delta^n;A)}(C(\Delta^n;A) \otimes_A M_1, C(\Delta^n;M_2)).
    \end{equation*}
    This assignment is a bijection by the adjunction given by extension and restriction of scalars and so the functor is fully faithful.
    Our functor is clearly a map of simplicial objects.
    It follows from Corollary~\ref{cor:contractibleSerreSwan} that the functor is essentially surjective.
 \end{proof}

\subsubsection{Excision and Mayer--Vietoris for topological \texorpdfstring{$K$}{K}-theory}
\label{sec:excision}

In this subsection, we show that connective topological $K$-theory of ungraded Banach algebras satisfies excision (Definition~\ref{def:excisive}).
This can be used to extend it to nonunital algebras.
To establish this, we will use the Banach category perspective.

\begin{lemma}
    \label{lem:serrefunctor}
    Let $p\colon \mathsf{C}\to \mathsf{D}$ be a Banach functor.
    The following are equivalent:
    \begin{enumerate}
        \item $\Aut_{\mathsf{C}}(c)\to \Aut_{\mathsf{D}}(p(c))$ is a Serre fibration $\forall c \in \mathsf{C}$.
        \item $\End_{\mathsf{C}}(c)\to \End_{\mathsf{D}}(p(c))$ is surjective $\forall c \in \mathsf{C}$.
    \end{enumerate}
\end{lemma}
\begin{proof}
    This is~\cite[Prop.~1.2.7]{Karoubi68}.
\end{proof}
\begin{definition}
    If $p\colon \mathsf{C}\to \mathsf{D}$ satisfies the equivalent conditions of Lemma~\ref{lem:serrefunctor}, we say $p$ is a \emph{Serre functor}.
\end{definition}
\begin{example}
    \label{ex:serrefunctor}
    Let $p\colon A\twoheadrightarrow B$ be a surjection of Banach algebras.
    Then $\ModEn_A\to \ModEn_B$ is a Serre functor.
    To see this, first note that $M_n(A)\to M_n(B)$ is surjective for all $n \geq 0$. Observe now that an endomorphism of a finitely generated projective module $\im(e\colon A^n\to A^n)$ can be identified with a matrix $X$ satisfying $X=eX=Xe$. 
    So let $p(e)X'=X'=X'p(e)$ be an endomorphism of $\im(p(e))$. Pick a lift $p(\Tilde{X})=X'$ and set $X=e\Tilde{X}e$. We then have $p(X)=X'$ and $X$ is an endomorphism of $\im(e)$.
\end{example}
\begin{remark}
    \label{rmk:mappingcylinder}
    Every map $f\colon A\to B$ of Banach algebras factors as a homotopy equivalence and a surjection:
    \begin{equation*}
        \begin{tikzcd}
            A\ar[rr,"f"]\ar[rd,"\simeq"]& & B\\
            & Z_f\ar[ru,"\ev_1"]& 
        \end{tikzcd} \,,
    \end{equation*}
    where $Z_f=\{(a,\phi)| a\in A,\phi\in C([0,1];B),f(a)=\phi(0)\}$ is the mapping cylinder, see e.g. \cite[Exercise 6.M]{MR1222415}.
\end{remark}

A concrete model of the pullback in the $(2,1)$-category $\Cat_1(\Ban)$ is given by the iso-comma category.
We will apply Lemma~\ref{lem:htpypullbackoftopcats} for which the following definition will be useful, see Appendix~\ref{sec:enriched} for details.

\begin{definition}
    A \emph{Serre pullback} of Banach categories is a pullback square 
    \begin{equation*}
        \begin{tikzcd}
            \mathsf{C}^{12} \ar[d]\ar[r]& \mathsf{C}^1 \ar[d,"p"]\\
            \mathsf{C}^2 \ar[r]& \mathsf{C}^0
    \arrow["\lrcorner"{anchor=center, pos=0.125}, draw=none, from=1-1, to=2-2]
        \end{tikzcd}
    \end{equation*}
    in $\Cat_1(\Ban)$ where $p$ is a Serre functor.
    In particular, the underlying square in $\Cat_1$ is a pullback square.
\end{definition}
We want to think of Serre pullbacks as homotopy pullback squares of Banach categories.
The following is a version of Milnor patching~\cite{MR349811}.

\begin{lemma}
\label{lem:milnorpatch}
    Let 
    \begin{equation*}
        \begin{tikzcd}
            A\ar[d]\ar[r]& A_1 \ar[d,"p",twoheadrightarrow]\\
            A_2 \ar[r]& B
    \arrow["\lrcorner"{anchor=center, pos=0.125}, draw=none, from=1-1, to=2-2]
        \end{tikzcd}
    \end{equation*}
    be a pullback diagram of Banach algebras with $p$ a surjection.\footnote{The pullback in $\Alg(\Ban_{\leq})$ and $\Alg(\Ban)$ agree.}
    Then
    \begin{equation*}
        \begin{tikzcd}
            \ModEn_A \ar[d]\ar[r]& \ModEn_{A_1} \ar[d,"p_*"]\\
            \ModEn_{A_2} \ar[r]& \ModEn_{B}
    \arrow["\lrcorner"{anchor=center, pos=0.125}, draw=none, from=1-1, to=2-2]
        \end{tikzcd}
    \end{equation*}
    is a Serre pullback of Banach categories.
\end{lemma}
\begin{proof}
    By Example \ref{ex:serrefunctor}, $p$ induces a Serre functor.
    By~\cite[Thm~I.2.7]{weibelkbook}, every ``descent datum'' of finitely generated projective $A_i$-modules $P_i$ together with an isomorphism $g:B\otimes_{A_1} P_1\cong B\otimes_{A_2}P_2$ ``patches'' to a finitely generated projective $A$-module $P$, i.e.\ the map $\ModOne_A\to \ModOne_{A_1}\times_{\ModOne_B}\ModOne_{A_2}$ is essentially surjective.
    We are left to check fully faithfulness.
    Let $P,Q\in \ModOne_A$.
    We have to show that the square of Banach spaces
    \begin{equation*}
        \begin{tikzcd}
            \Hom_A(P,Q) \ar[r]\ar[d]& \Hom_{A_1}(A_1\otimes_A P,A_1\otimes_{A}Q) \ar[d]\\
            \Hom_{A_2}(A_2\otimes_A P,A_2\otimes_{A}Q) \ar[r] & \Hom_B(B\otimes_A P,B\otimes_A Q)
        \end{tikzcd}
    \end{equation*}
    is a pullback square.
    By the open mapping theorem, this can be checked on underlying sets. 
    We can use the isomorphic square
    \begin{equation*}
        \begin{tikzcd}
            \Hom_A(P,Q) \ar[r]\ar[d]& \Hom_{A}( P,A_1\otimes_{A}Q) \ar[d]\\
            \Hom_{A}(P,A_2\otimes_{A}Q) \ar[r] & \Hom_A(P,B\otimes_A Q)
        \end{tikzcd}
        \,.
    \end{equation*}
    But by projectivity, $\Hom_A(P,-)$ and $-\otimes_A Q$ are exact.
    So this is a pullback square since
    \begin{equation*}
        \begin{tikzcd}
            A \ar[r]\ar[d]& A_1 \ar[d]\\
            A_2 \ar[r] & B
    \arrow["\lrcorner"{anchor=center, pos=0.125}, draw=none, from=1-1, to=2-2]
        \end{tikzcd}
    \end{equation*}
    is a pullback.
\end{proof}

\begin{lemma}
    \label{lem:homotopypullbackofbanachcategories}
    Let 
    \begin{equation*}
        \begin{tikzcd}
            \mathsf{C}^{12} \ar[d]\ar[r]& \mathsf{C}^1 \ar[d,"p"]\\
            \mathsf{C}^2 \ar[r]& \mathsf{C}^0
    \arrow["\lrcorner"{anchor=center, pos=0.125}, draw=none, from=1-1, to=2-2]
        \end{tikzcd}
    \end{equation*}
    be a Serre pullback square of Banach categories.
    Then, the square 
        \begin{equation*}
                \begin{tikzcd}
                    N^{\mathrm{hc}}(\mathsf{C}^{12,\cong}) \ar[d]\ar[r]& N^{\mathrm{hc}}(\mathsf{C}^{1,\cong})\ar[d,"p"]\\
                    N^{\mathrm{hc}}(\mathsf{C}^{2,\cong}) \ar[r]& N^{\mathrm{hc}}(\mathsf{C}^{0,\cong})
    \arrow["\lrcorner"{anchor=center, pos=0.125}, draw=none, from=1-1, to=2-2]
                \end{tikzcd}
            \end{equation*}
    is a pullback square in $\CMon(\Spc)$.
\end{lemma}
\begin{proof}
    By assumption that $p$ is a Serre functor, the square 
    \begin{equation*}
        \begin{tikzcd}
            \mathsf{C}^{12,\cong} \ar[d]\ar[r]& \mathsf{C}^{1,\cong} \ar[d,"p"]\\
            \mathsf{C}^{2,\cong} \ar[r]& \mathsf{C}^{0,\cong}
    \arrow["\lrcorner"{anchor=center, pos=0.125}, draw=none, from=1-1, to=2-2]
        \end{tikzcd}
    \end{equation*}
    is a homotopy pullback square of $\Top$-enriched categories.
    We can conclude by Lemma~\ref{lem:htpypullbackoftopcats}.
\end{proof}

\begin{corollary}
    \label{cor:kthypullback}
    Let 
    \begin{equation*}
        \begin{tikzcd}
            \mathsf{C}^{12} \ar[d]\ar[r]& \mathsf{C}^1 \ar[d,"p"]\\
            \mathsf{C}^2 \ar[r]& \mathsf{C}^0
    \arrow["\lrcorner"{anchor=center, pos=0.125}, draw=none, from=1-1, to=2-2]
        \end{tikzcd}
    \end{equation*}
    be a pullback square of Banach categories and let $p$ be a Serre functor. 
    Assume that all functors are quasi-surjective and that $N^{\mathrm{hc}}(\mathsf{C}^{12,\cong})$ is telescopic.
    Then the square
    \begin{equation*}
        \begin{tikzcd}
            k(\mathsf{C}^{12}) \ar[d]\ar[r]& k(\mathsf{C}^1) \ar[d,"p"]\\
            k(\mathsf{C}^2) \ar[r]& k(\mathsf{C}^0)
    \arrow["\lrcorner"{anchor=center, pos=0.125}, draw=none, from=1-1, to=2-2]
        \end{tikzcd}
    \end{equation*}
    is a pullback square of connective spectra.
\end{corollary}
\begin{proof}
    Combine Lemma~\ref{lem:homotopypullbackofbanachcategories} with Lemma~\ref{lem:groupcompletionpullback}.
\end{proof}

\begin{definition}
    \label{def:excisive}
    A pullback square in $\Alg(\Ban)$ is a \emph{Milnor square} if one of the legs is surjective.
    A functor $F\colon \Alg(\Ban) \to \calC$ is \emph{excisive} if it maps Milnor squares to pullback squares.
\end{definition}

\begin{corollary}
\label{cor:excision}
    Topological $K$-theory is excisive, i.e.\ it maps Milnor squares of Banach algebras to pullback squares in $\Sp_{\geq 0}$.
\end{corollary}
The same result is obtained for the space of stable projections in a $C^*$-algebra in~\cite[Prop.~10.4]{bunkeKKEth}.
\begin{proof}
    Let
    \begin{equation*}
        \begin{tikzcd}
            A\ar[d]\ar[r]& A_1 \ar[d,"p",twoheadrightarrow]\\
            A_2 \ar[r]& B
    \arrow["\lrcorner"{anchor=center, pos=0.125}, draw=none, from=1-1, to=2-2]
        \end{tikzcd}
    \end{equation*}
    be a Milnor square of Banach algebras.
    By Lemma~\ref{lem:modulestelescopic}, $N^{\mathrm{hc}}(\ModEn_A^{\cong})$ is telescopic.
    The Banach functors induced covariantly from algebra homomorphisms are all quasi-surjective, since for $f\colon A\to B$ the image of $f_*=B\otimes_A-:\ModEn_A \to \ModEn_B$ includes all finitely generated free modules $B^n$.
    Furthermore, $p_*$ is a Serre functor by Example~\ref{ex:serrefunctor}.
    Corollary~\ref{cor:kthypullback} implies that
    \begin{equation*}
        \begin{tikzcd}
            k_A\ar[d]\ar[r]& k_{A_1} \ar[d,"p"]\\
            k_{A_2} \ar[r]& k_B
    \arrow["\lrcorner"{anchor=center, pos=0.125}, draw=none, from=1-1, to=2-2]
        \end{tikzcd}
    \end{equation*}
    is a pullback square.
    This concludes the proof.
\end{proof}
\begin{remark}
    The $K$-theory of a nonunital Banach algebra $I$ is usually defined in terms of a unitization of $A$. More specifically, suppose $I \hookrightarrow A$ sits inside a unital Banach algebra $A$ as an ideal.
    Define $k_I := \fib(k_A \to k_{A/I})$.
    By Corollary~\ref{cor:excision}, this definition is independent of the choice of unitization $A$.
    Indeed, a commutative diagram 
    \begin{equation*}
    \begin{tikzcd}
        I \ar[r, hookrightarrow] \ar[dr, hookrightarrow] & A \ar[d] \\
        & B
    \end{tikzcd}
    \end{equation*}
    induces a pullback of Banach algebras
    \begin{equation*}
    \begin{tikzcd}
        A \ar[d] \ar[r] & A/I \ar[d] \\
        B \ar[r] & B/I
    \arrow["\lrcorner"{anchor=center, pos=0.125}, draw=none, from=1-1, to=2-2]
    \end{tikzcd} \,,
    \end{equation*}
    with surjective horizontal maps.
    Hence the induced pullback in spectra establishes an equivalence between fibers $\fib(k_A \to k_{A/I})$ and $\fib(k_B \to k_{B/I})$.
    Note that the category of unital Banach algebras under $I$ has an initial object which is the unitization.
    We discuss $K$-theory for nonunital algebras in more detail in Section~\ref{sec:nonunital}.
\end{remark}
We denote by $k_A(X)$ the $K$-theory spectrum of $\ModEn_A(X)$ and $k(X) := k_{\mathbb{F}}(X)$.
\begin{example}
    Let $X=U\cup V$ be a closed cover of a compact Hausdorff space $X$.
    By the pasting lemma, we obtain a pullback of Banach algebras
    \begin{equation*}
        \begin{tikzcd}
            C(X) \ar[r]\ar[d]& C(U) \ar[d]\\
            C(V) \ar[r]& C(U\cap V)
    \arrow["\lrcorner"{anchor=center, pos=0.125}, draw=none, from=1-1, to=2-2]
        \end{tikzcd} \,,
    \end{equation*}
    and $C(X)\to C(U)$ is surjective by Tietze's extension theorem.
    Hence, there is a pullback square
    \begin{equation*}
        \begin{tikzcd}
            k(X) \ar[r]\ar[d]& k(U) \ar[d]\\
            k(V) \ar[r]& k(U\cap V)
    \arrow["\lrcorner"{anchor=center, pos=0.125}, draw=none, from=1-1, to=2-2]
        \end{tikzcd} \,.
    \end{equation*}
    Equivalently, this is a fiber sequence of connective spectra $k(X)\to k(U)\oplus k(V) \to k(U\cap V)$.
    It induces a long exact Mayer--Vietoris sequence
    \begin{equation*}
        \dots \to k^i(X) \to k^i(U)\oplus k^i(V) \to k^{i}(U\cap V) \to k^{i+1}(X) \to \dots \,,
    \end{equation*}
    where we wrote $k^i(X)=\pi_i(k(X))$.
    The analogous statement also holds for any Banach algebra $A$ and $C(X;A)$.
\end{example}
\begin{remark}
There are other formulations of Mayer--Vietoris sequences for Banach algebras, compare~\cite[21.5]{MR1656031} and~\cite[Theorem 4.1]{MR757510}.
\end{remark}
\begin{lemma}
    \label{lem:excisivefunctor}
    Let $F \colon \CHaus \to \Sp_{\geq 0}$ be an excisive homotopy invariant functor such that $F(\emptyset) = 0$.
    If $X$ is a compact Hausdorff space and homotopy finitely dominated (a homotopy retract of a finite CW complex), there is an equivalence
    \begin{equation*}
    \tau_{\geq 0} \underline{\Hom}_{\Sp}(\Sigma^\infty_+ X, F) \simeq F(X)\,,
    \end{equation*}
    which is natural in $X$ and $F$.
\end{lemma}
\begin{proof}
    The proof is a connective spectrum version of \cite[Lemma 3.1]{bunke2023survey}, keeping in mind that $\Sp_{\geq0} \to \Sp$ doesn't preserve limits, but has the coreflection $\tau_{\geq 0}$.
\end{proof}
An inductive application of Mayer--Vietoris therefore yields:
\begin{corollary}
    \label{cor:homotopicalSwan}
    If $X$ is a homotopy finitely dominated compact Hausdorff space, there is an equivalence
    \begin{equation*}
    \tau_{\geq 0} \underline{\Hom}_{\Sp}(\Sigma^\infty_+ X, k_A) \simeq k_A(X)\,,
    \end{equation*}
    which is natural in $X$ and $A$.    
\end{corollary}
\begin{proof}
Follows since $k_A$ is excisive and homotopy invariant, compare ~\cite[Proposition 3.2]{bunke2023survey}.
\end{proof}
\begin{example}
    Let $X$ be a compact Hausdorff space that is homotopy finitely dominated, for example a finite CW complex, and consider $A=\R$, $\ModEn_A=\VectEn_\R$.
    Then the functor $X\mapsto N^{\mathrm{hc}}\VectEn(X)^{\cong}$ satisfies descent and is homotopy invariant.
    Its classifying space is $N^{\mathrm{hc}}\VectEn(\pt)^{\cong}=N^{\mathrm{hc}}(\VectEn^{\cong})\simeq \bigsqcup_{n} BO(n)$.
    Moreover, we have $\VectEn(X)^{\cong,\gp}\simeq [\Sigma_+^\infty X,ko]$.
    In particular, $ko^0(X)\cong K_0(\VectOne(X))$ is the classical topological $KO$-theory group.
\end{example}
\begin{definition}
An $A$-representation $\pi\colon A\to \End(M)$ is \emph{topologically irreducible} if the only closed subrepresentations are $0$ and $M$.
The \emph{radical} of $A$ is the two-sided ideal
\begin{equation*}
    \mathrm{rad}(A)=\bigcap_{\pi \text{ irrep}} \ker(\pi) \,.
\end{equation*}
$A$ is \emph{semisimple} if $\mathrm{rad}(A)=0$.
\end{definition}
All $C^*$-algebras are semisimple and $A/\mathrm{rad}(A)$ is semisimple.
For more details on radicals see~\cite[Ch.~III]{MR423029} or ~\cite[Ch.~4.3]{MR1270014}.
\begin{proposition}
    \label{prop:semisimplekthy}
    The quotient map $A\to A/\mathrm{rad}(A)$ induces an equivalence $k_A\simeq k_{A/\mathrm{rad}(A)}$.
\end{proposition}
\begin{proof}
It follows from the proof of \cite[Prop.~2.3.11.]{Karoubi68}, which we now spell out.
    The Banach functor $\ModEn_A \to \ModEn_{A/{\mathrm{rad}(A)}}$ is a Serre functor.
    It hence induces a surjective fibration
    \begin{equation}
    \label{eq:radicalquotient}
         \GL_n(A) \to \GL_n(A/\mathrm{rad}(A)) \,,
    \end{equation}
    whose kernel is the subspace $\{1+x:x\in M_n(\mathrm{rad}(A))\}$. 
    Note that any such matrix is indeed invertible and that this subspace is contractible.
    Therefore, \eqref{eq:radicalquotient} is a homotopy equivalence.
    We can conclude that $K_i(A) \cong K_{i}(A/\mathrm{rad}(A))$ for $i\geq 1$.
    The $K_0$ case follows by idempotent lifting (or by Bott periodicity).
\end{proof}

\subsubsection{Filtered colimits}
\label{sec:filteredcolimits}
The category of Banach algebras and contractive homomorphisms $\Alg(\Ban_{\leq})$ has all limits and colimits.
The filtered colimit can be computed in the underlying Banach space~\cite{MR4370428} and so is easy to describe: Let $I\ni i\mapsto A_i$ be a filtered system of Banach algebras with contractive transition maps $\phi_{ij}:A_i\to A_j$.
The filtered colimit of underlying sets is the $\R$-algebra $A^{\mathrm{alg}}_\infty=\bigcup_i A_i/\sim$, where the equivalence relation is $a_i\sim \phi_{ij}(a_i)$.
It has a submultiplicative seminorm given by $\lVert [a_i]\rVert = \inf_{j\geq i} \lVert \phi_{ij}(a_i)\rVert_{A_j}$\,.
Let $N=\{a\in A^{\mathrm{alg}}_\infty:\lVert a\rVert =0 \}$.
This defines a closed two-sided ideal.
The filtered colimit is the completion of the quotient with respect to the norm:
\begin{equation}
    \label{eq:filteredcolimdef}
    \colim_I A_i = \overline{A^{\mathrm{alg}}_\infty/N}^{\lVert-\rVert} \,.
\end{equation}

\begin{warning}
    Consider the dual numbers $\C[\epsilon]/(\epsilon^2)$ with norm $|a+b\epsilon|=|a|+|b|$.
    Fix $1>t>0$ and consider the contractive homomorphism $g_t:\epsilon \mapsto t\epsilon$.
    The filtered system $\N \to \Alg(\Ban_{\leq})$ where each transition map is $g_t$ has colimit $\C$.
    In $\Alg(\Ban)$, $g_t$ has an inverse and the actual colimit is $\C[\epsilon]/(\epsilon^2)$.
\end{warning}
\begin{convention}
    By the term \emph{filtered colimit of Banach algebras} we refer to the filtered colimit in $\Alg(\Ban_{\leq})$ as in~\eqref{eq:filteredcolimdef} for which we will write 
    $  \colim^{\leq}_I A_i $.
\end{convention}

A very general group level version of the following Proposition for local Banach algebras and injective transition maps appears in~\cite[Thm~2.62]{MR2340673}.
For complex $C^*$-algebras, a proof of a spectral version along the same lines may be found in~\cite[Prop.~10.3]{bunkeKKEth}.
\begin{proposition}
    \label{prop:filteredcolimits}
    Let $i\in I\mapsto A_i$ be a filtered system of Banach algebras with contractive transition maps $\phi_{ij}\colon A_i\to A_j$.
    Then there is a natural equivalence
    \begin{equation}
        \label{eq:filteredcolim}
        \colim_{i\in I} k_{A_i} \xrightarrow{\simeq} k_{\colim^\leq_{i\in I} A_i} \,.
    \end{equation}
\end{proposition}
\begin{proof}
    Denote the filtered colimit by $A_\infty=\colim^\leq_I A_i$ with maps $\phi_i\colon A_i\to A_\infty$.
    Note that $\bigcup \phi_i(A_i)$ defines a dense subset and the norm of an element $\phi_i(a_i), a_i\in A_i$ is given by $\lVert \phi_i(a_i)\rVert_{A_\infty}=\inf_{j\geq i} \lVert \phi_{ij}(a_i)\rVert_{A_j}$.
    The claim follows once we have proven that
    \begin{equation*}
        \colim_I \GL_n(A_i) \to \GL_n(A_\infty)
    \end{equation*}
    is a weak homotopy equivalence of topological spaces for all $k\in \N\cup \{\infty\}$, since then $\colim_I \Omega k_{A_i}=\colim_I \GL_{\infty}(A_i)$ is equivalent to $\Omega k_{A_\infty}=\GL_\infty(A_\infty)$.
    From this one can conclude using Bott periodicity that a fortiori the map~\eqref{eq:filteredcolim} is an equivalence.\footnote{There is a longer version of the proof that avoids using Bott periodicity and gives a separate argument on $\pi_0$ using lifting of idempotents.}

    Let $X$ be a compact Hausdorff space. We have to show that the map between homotopy classes of maps
    \begin{equation*}
        \colim_I [X,\GL_n(A_i)] \to [X,\GL_n(A_\infty)]
    \end{equation*}
    is a bijection.
    We can rewrite this as 
    \begin{equation*}
        \colim_I\pi_0 \GL_n(C(X;A_i)) \to \pi_0 \GL_n(C(X;A_\infty)) \,.
    \end{equation*}
    We may therefore assume that $X$ is a point.
    To show surjectivity, we let $u\in M_k(A_\infty)$ be invertible.
    We prove that for $\epsilon>0$ there is $j\in I$ and $v\in \GL_n(A_j)$ such that $\lVert \phi_j(v)-u\rVert_{A_\infty} < \epsilon$.
    Since invertible elements are locally path-connected, choosing $\epsilon$ small yields a path from $\phi_j(v)$ to $u$.
    
    Since the images $\phi_i(M_k(A_i))$ are dense in $M_k(A_\infty)$, for any $\epsilon>0$ there exists $i\in I,v\in M_k(A_i)$ with $\lVert \phi_i(v)-u\rVert_{A_\infty} < \epsilon$.
    We claim that there exists $j\in I$ such that $\phi_{ij}(v)$ is invertible.
    Since the invertible elements are open in a Banach algebra, we may assume that $\phi_i(v)$ is invertible in $M_k(A_\infty)$ by choosing $\epsilon$ small enough.
    Again, there is a $j\in I, j\geq i$ and $v'\in M_k(A_j)$ with $\phi_j(v')$ close to $\phi_i(v)^{-1}$.
    Without loss of generality, we assume $i=j$, by trading $v$ for $\phi_{ij}(v)$.
    Also we can assume that $j$ is large enough so that
    \begin{align*}
        \lVert vv' -1\rVert_{M_k(A_j)} &\leq \lVert \phi_j(v)\phi_j(v')-1\rVert_{M_k(A_\infty)}+\epsilon \\
        & \leq \lVert \phi_j(v)(\phi_j(v')-\phi_j(v)^{-1})\rVert_{M_k(A_\infty)} \leq 2\epsilon \,.
    \end{align*}
    Since $vv'$ is close to $1$, it is invertible in the Banach algebra $M_k(A_j)$.
    This proves that $v$ has a right inverse. Similarly, one proves that it has a left inverse, i.e.\ $v\in \GL_n(A_j)$.
    This proves surjectivity.
    
    To prove injectivity, let $u_t\colon [0,1]\to \GL_n(A_\infty)$ be a path from $u_0=\phi_i(v)$ to $u_1=\phi_{i'}(v')$.
    We may immediately assume $i=i'$.
    By the above, applied to the filtered system $C([0,1];A_i)$, by possibly enlarging $j$ there is a path $v_t\in C([0,1];M_k(A_j))$ with $\sup_t \lVert \phi_j(v_t)-u_t \rVert_{M_k(A_\infty)} \leq \epsilon$.
    By choosing $j$ large enough we can guarantee $\lVert v-v_0 \rVert_{M_k(A_j)} \leq \lVert \phi_j(v-v_0)\rVert_{M_k(A_\infty)}+\epsilon \leq 2\epsilon$ and similarly $\lVert v'-v_1\rVert_{M_k(A_j)}\leq 2\epsilon$.
    Since again $\GL_n(A_j)$ is locally path-connected, there is a path in $\GL_n(A_j)$ from $v$ to $v_0$ to $v_1$ to $v'$.
    This proves injectivity.
\end{proof}

\subsubsection{Topological \texorpdfstring{$K$}{K}-theory via Kan extension}
\label{sec:kan}
Here we prove that connective topological $K$-theory may be characterized as the initial homotopy invariant functor with a natural transformation $\mathcal{K}^{\mathrm{alg}}\to k$.

\begin{equation*}\begin{tikzcd}
	\Alg(\Ban) && {\Sp_{\geq 0}} \\
	\\
	\BanAlgInf
	\arrow["{\mathcal{K}^{\mathrm{alg}}}", from=1-1, to=1-3]
	\arrow["i"', from=1-1, to=3-1]
	\arrow[""{name=0, anchor=center, inner sep=0}, "k"', from=3-1, to=1-3]
	\arrow[Rightarrow, from=1-1, to=0,shorten >=4, shorten <=2]
\end{tikzcd} \,.\end{equation*}
For every cocomplete $(\infty,1)$-category $\calD$ there is an adjunction
\begin{equation*}
    \begin{tikzcd}
        \Fun(\BanAlgInf, \calD) \ar[r,hookrightarrow,shift right=1ex,"i^*"'] & \Fun(\Alg(\Ban), \calD) \ar[l,"i_!"',shift right=1ex]
    \end{tikzcd} \,, 
\end{equation*}
where the precomposition functor $i^*$ is an equivalence onto the subcategory of functors mapping $W$ to equivalences in $\calD$.
The left Kan extension $i_!F$ can be interpreted as the initial homotopy invariant functor with a map $F\to i^*i_!F$.
If $\calD$ is (presentably) symmetric monoidal, then $i_!$ is symmetric monoidal with respect to Day convolution~\cite[Lemma~3.4]{MR4679205} and hence refines to a map between categories of lax symmetric monoidal functors.

\begin{theorem}
    \label{thm:kanextension}
    The left Kan extension $i_!\mathcal{K}^{\mathrm{alg}}$ can naturally be identified with connective topological $K$-theory $k$.
    That is, $k$ is the initial homotopy invariant functor on Banach algebras under algebraic $K$-theory $\mathcal{K}^{\mathrm{alg}}$.
\end{theorem}

For the proof we use simplicial resolutions:

\begin{definition}
    \label{def:resolution}
    Let $(\calC,W)$ be a relative category, i.e.\ a category with a subcategory $W$ of ``weak equivalences''.
    A \emph{simplicial resolution of $Y\in \calC$} is a diagram $Y_\bullet\colon \Delta^{\op} \to \calC_{Y/}$ such that 
    \begin{enumerate}
        \item The structure maps $Y\to Y_n$ are weak equivalences for all $n$.
        \item \label{cond:2}The functor $\colim_{\Delta^{\op}}\Hom_{\calC}(-,Y_\bullet)$ maps weak equivalences in $\calC$ to equivalences in $\Spc$.
    \end{enumerate}
\end{definition}

We will use the Derived Mapping Space Lemma~\ref{lem:derivedmappingspace} and its consequence Lemma~\ref{lem:leftKanextension} to compute mapping spaces in localizations via simplicial resolutions.

\begin{lemma}
    \label{lem:simplicialresBanAlg}
    For every Banach algebra $A$, the diagram 
    \begin{align*}
        \Delta^{\op} &\longrightarrow \Alg(\Ban)_{A/} \\
        [n] & \longmapsto C(\Delta^n;A)
    \end{align*}
    is a simplicial resolution.
\end{lemma}
\begin{proof}
    The maps $A\to C(\Delta^n;A)$ are homotopy equivalences.
    We are left to prove that 
    \begin{equation*}
        \colim_{\Delta^{\op}}\Hom_{\Alg(\Ban)}(-,C(\Delta^\bullet;A))
    \end{equation*}
preserves homotopy equivalences in the first variable.
    The algebra homomorphisms from $B$ to $A$ are topologized so that
    \begin{equation*}
         \Hom_{\Alg(\Ban)}(B,C(\Delta^n;A))\cong \Top(\Delta^n,\Hom_{\BanAlgEn}(B,A))\,.
    \end{equation*}
    Therefore we obtain an equivalence of spaces $\colim_{\Delta^{\op}} \Hom_{\Alg(\Ban)}(B,C(\Delta^n;A)) \simeq \Hom_{\BanAlgEn}(B,A)$ .
    Clearly, $\BanAlgEn$ maps homotopies of algebra homomorphisms to homotopies.
\end{proof}

\begin{proof}[Proof of Theorem \ref{thm:kanextension}]
    We know that $k$ is homotopy invariant and there is a natural transformation $\mathcal{K}^{\mathrm{alg}}\to i^*k$, which is adjoint to a natural transformation $i_!\mathcal{K}^{\mathrm{alg}}\to k$.
    We claim that this transformation is an equivalence of functors.
    Let $A\to C(\Delta^\bullet;A)$ be the simplicial resolution of Lemma~\ref{lem:simplicialresBanAlg}.
    By Lemma~\ref{lem:leftKanextension}, there is a natural equivalence $i^*i_!\mathcal{K}^{\mathrm{alg}}(A)\simeq \colim_{\Delta^{\op}}\mathcal{K}^{\mathrm{alg}}(C(\Delta^n;A))$.
    We can conclude by Proposition~\ref{prop:compareK}.
\end{proof}
\subsubsection{Extending to nonunital algebras}
\label{sec:nonunital}
Here, we explain how the theory we have developed thus far for unital Banach algebras extends to nonunital Banach algebras without any problems.

The forgetful functor from unital to nonunital Banach algebras has a left adjoint $A\mapsto A_+ = A \oplus_{\ell^1} \R$ called unitization.
\begin{equation*}
\begin{tikzcd}
	\Alg(\Ban) && {\Alg(\Ban)^{\mathrm{nu}}}
	\arrow[hook,"i"', from=1-1, to=1-3,shift right=1.3]
	\arrow["{(-)_+}"', shift right=1.3, from=1-3, to=1-1]
\end{tikzcd} \,.
\end{equation*}
The nonunital Banach algebra $0$ is both initial and terminal, i.e.\ it has a zero object.
A category with a zero object is called \emph{pointed}.
\begin{definition}
    Let $\calC$ be pointed.
    A functor $F\colon \Alg(\Ban)^{\mathrm{nu}}\to \calC$ is called \emph{reduced} if $F(0)=0$.
    It is \emph{exact} if it maps short exact sequences of Banach algebras to fiber sequences of spectra.
\end{definition}
If $F\colon \Alg(\Ban)^{\mathrm{nu}}\to \calC$ is any functor, then $\tilde{F}(A)=\fib(F(A)\to F(0))$ defines a reduced functor.
This fits into a chain of adjunctions of functor categories:
\begin{equation*}
    \begin{tikzcd}[column sep = 40]
        \Fun^{\mathrm{red}}(\Alg(\Ban)^{\mathrm{nu}},\calC) \ar[r,hookrightarrow,shift left=1ex]& \Fun(\Alg(\Ban)^{\mathrm{nu}},\calC) \ar[l,shift left=1ex,"\widetilde{(-)}"] \ar[r,shift left=1ex,"i^*"]&\Fun(\Alg(\Ban),\calC) \ar[l,shift left=1ex,"i_*=(-)_+^*"]  
    \end{tikzcd} \,. 
\end{equation*}
\begin{lemma}
\label{lem:reducedextension}
    If $F\colon \Alg(\Ban)\to \calC$ preserves filtered colimits, then so does the reduced extension $\tilde{F}\colon \Alg(\Ban)^{\mathrm{nu}} \to \calC$, where
    \begin{equation*}
        \tilde{F}(A)=\fib\left(F(A_+) \to F(\R) \right) \,.
    \end{equation*}
    If $F$ is homotopy invariant, then so is $\tilde{F}$.
    If $F$ satisfies excision, then $\tilde{F}$ is exact.
    Moreover, the adjunction restricts to an equivalence
    \begin{equation*}
        \Fun^{\mathrm{red,exact}}(\Alg(\Ban)^{\mathrm{nu}},\calC) \simeq \Fun^{\mathrm{excisive}}(\Alg(\Ban),\calC) \,.
    \end{equation*}
    This equivalence further restricts to the subcategories of homotopy invariant and filtered colimit preserving functors.
\end{lemma}
\begin{proof}
    For the first statement, note that $(-)_+$ preserves all colimits as a left adjoint and $\fib$ commutes with filtered colimits.
    Homotopy invariance of $\tilde{F}$ is immediate.
    Let $F$ satisfy excision and let $0\to A\to B\to C\to 0$ be an exact sequence of nonunital Banach algebras.
    Then there is an induced Milnor square, i.e.\ a pullback square of unital Banach algebras, in which one leg is a surjection:
    \begin{equation*}
        \begin{tikzcd}
            A_+ \ar[r]\ar[d]& \R \ar[d] \\
            B_+ \ar[r,twoheadrightarrow] & C_+
    \arrow["\lrcorner"{anchor=center, pos=0.125}, draw=none, from=1-1, to=2-2]
        \end{tikzcd} \,.
    \end{equation*}
    By excision, this induces a pullback square
    \begin{equation*}
        \begin{tikzcd}
            F(A_+) \ar[r]\ar[d]& F(\R) \ar[d] \\
            F(B_+) \ar[r] & F(C_+)
    \arrow["\lrcorner"{anchor=center, pos=0.125}, draw=none, from=1-1, to=2-2]
        \end{tikzcd} \,.
    \end{equation*}
    We obtain a fiber sequence $\tilde{F}(A)\to\tilde{F}(B) \to \tilde{F}(C)$ by the following diagram, in which the left face is a pullback since it is obtained as a fiber of two pullback squares:
    \begin{equation*}
        \begin{tikzcd}[row sep={40,between origins}, column sep={40,between origins}]
          & \tilde{F}(B) \ar[rr]\ar{dd} & & F(B_+) \ar{dd} \ar[rr]& & F(\R) \ar[dd]\\
        \tilde{F}(A) \ar[crossing over]{rr} \ar{dd}\ar[ru]\ar[dd] & & F(A_+) \ar[rr,crossing over]\ar[ru]\ar[dd]& & F(\R) \ar[ru]\ar[dd,crossing over] &\\
          & \tilde{F}(C)  \ar[rr] \ar{rr} & &  F(C_+) \ar{rr} && F(\R) \\
        0 \ar[rr,equals]\ar[ru,equals] && F(\R)\ar[ru] \ar[rr]\ar[from=uu,crossing over] && F(\R) \ar[ru] &
    \end{tikzcd} \,.
    \end{equation*}
    
    We are left to show that the unit and counit of the adjunction are equivalences.
    Let $G\colon \Alg(\Ban)^{\mathrm{nu}}\to \calC$ be exact and reduced.
    Then the unit $G(A)\to \widetilde{i^*G}(A)$ is the map
    \begin{equation*}
        G(A) \to \fib\left( G(A_+) \to G(\R) \right) \,,
    \end{equation*}
    which is an equivalence by exactness of $A\to A_+\to \R$.
    Let $F\colon \Alg(\Ban)\to \calC$ be excisive. 
    The counit is the map $i^*\tilde{F}(A) \to F(A)$ given by
    \begin{equation*}
        \fib(F(A_+) \to F(\R)) \to F(A) \,,
    \end{equation*}
    which is an equivalence since 
    \begin{equation*}
        \begin{tikzcd}
            A_+ \ar[r]\ar[d]& \R \ar[d] \\
            A \ar[r,twoheadrightarrow] & 0
    \arrow["\lrcorner"{anchor=center, pos=0.125}, draw=none, from=1-1, to=2-2]
        \end{tikzcd} 
    \end{equation*}
    is a Milnor square exhibiting $A_+$ as the product of Banach algebras $A\times \R$.
\end{proof}
\begin{corollary}
    The connective topological $K$-theory functor on unital Banach algebras extends to a reduced exact functor $k\colon \Alg(\Ban)^{\mathrm{nu}}\to \Sp_{\geq 0}$ which is homotopy invariant and preserves filtered colimits (along contractive homomorphisms).
    The extension is given by
    \begin{equation*}
        k_A=\fib( k_{A_+} \to k_{\R} ) \,.
    \end{equation*}
\end{corollary}
\begin{example}
    Consider the real $C^*$-algebras $M_n(\R)$ with nonunital transition map $M_n(\R) \hookrightarrow M_{n+1}(\R)$.
    Then $\colim_n^{\leq} M_n(\R)=\mathbb{K}(H)$ is the $C^*$-algebra of compact operators on a separable real Hilbert space $H$.
    The contractive transition maps induce equivalences $k_{M_n(\R)} \simeq k_{M_{n+1}(\R)}$.
    We obtain $k_{\mathbb{K}(H)}\simeq ko$.
\end{example}
\subsection{Morita functoriality}
\label{sec:moritafunctoriality}
The goal of this section is to enhance the functoriality of connective $K$-theory to a lax symmetric monoidal functor 
\begin{equation*}
    \BimBanInf \to \Sp_{\geq 0}
\end{equation*} 
from a Morita $(\infty,1)$-category of Banach algebras, Banach bimodules and bimodule maps.
In particular, an $(A,B)$-bimodule $M$ induces a functor $M\otimes_B-:\ModOne_B \to \ModOne_A$ and thereby a map of spectra $k_B\to k_A$.
Our starting point is the symmetric monoidal $(2,1)$-category $\BimBanOne$ defined in Definition~\ref{def:bimbamrm}:
Firstly, since $\Sp_{\geq 0}$ is an $(\infty,1)$-category, we need to restrict to invertible bimodule maps.
Secondly, we have to restrict to $(A,B)$-bimodules $M$ that are finitely generated as $A$-modules, since only those bimodules induce functors $M \otimes_B (-) \colon \ModOne_B \to \ModOne_A$ on categories of finitely generated projective modules.

Homotopies between algebra homomorphisms $A \to B$ induce homotopies between the maps $k_A \to k_B$.
This is not yet encoded in a functor from $\BimBanOne$.
Therefore we will need to add ``continuous paths of bimodules'' as additional $2$-morphisms in $\BimBanInf$.
Also adding in higher simplices, we arrive at the $(\infty,1)$-category $\BimBanInf$ (Definition~\ref{def:BimBaninfty}).
$\BimBanInf$ is close to a bivariant $K$-theory: We prove that the group completion $\Hom_{\BimBanInf}(C(X),\mathbb{F})$ recovers the $K$-homology of the space $X$ in Theorem~\ref{th:segal}.

We will provide a universal definition of $\BimBanInf$ as a Dwyer--Kan localization, but we will also be able to identify its mapping spaces concretely as sketched above.

\begin{definition}
    A \emph{bimodule homotopy} is a morphism $M\colon A \rightsquigarrow C(\Delta^1; B)$ in $\Bim(\Ban)$, i.e.\ a $C(\Delta^1;B)$-$A$-bimodule which is finitely generated projective as a left module.
    We say that two bimodules $N_0,N_1\colon A\rightsquigarrow B$ are \emph{bimodule-homotopic} if there exists a bimodule homotopy $M$ and bimodule isomorphisms $N_0\cong \ev_0\circ M$ and $N_1\cong \ev_1 \circ M$.
    $M\colon A\rightsquigarrow B$ is a \emph{homotopy equivalence} if there exists $N\colon B\rightsquigarrow A$ and homotopies $N\circ M\simeq \id_A$, $M\circ N \simeq \id_B$.
\end{definition}
Note that the restrictions $\ev_0\circ M$ and $\ev_1\circ M$ are isomorphic as left $B$-modules by Lemma~\ref{lem:idempotents}.
Being homotopic defines an equivalence relation on isomorphism classes of bimodules which is stable under pre- and postcomposition. 
The class $W$ of bimodule homotopy equivalences satisfies the $2$-out-of-$3$-property and it is $\otimes$-stable, i.e.\ $f \otimes \id_A \in W$ if $f \in W$.
It is the smallest class with these properties containing $A\rightsquigarrow C(\Delta^1;A)$ for all $A$.

\begin{remark}
Similarly to Remark~\ref{rem:generatinghomotopy}, the class $W$ is in fact the smallest containing the bimodule $\R \rightsquigarrow C(\Delta^1)$ if we restrict to $C^*$-algebras.
\end{remark}

\begin{definition}
\label{def:BimBaninfty}
    The \emph{Morita $(\infty,1)$-category of Banach algebras and bimodules} is defined as the localization of $\BimBanOne$ at the homotopy equivalences $\BimBanInf:=\BimBanOne[W^{-1}]$.
\end{definition}

The homotopy category $\Ho(\BimBanInf)$ has objects Banach algebras and morphisms $A\rightsquigarrow B$ are bimodule-homotopy classes of $(B,A)$-bimodules that are finitely generated projective on the left.

\begin{theorem}
    \label{thm:localizationbimban}
    The localization $\BimBanOne\to \BimBanInf$ is symmetric monoidal.
    The mapping spaces in the target category are computed as 
    \begin{align}
        \label{eq:mappingspacesbimban1}
            \Hom_{\BimBanInf}(A,B) &\simeq \colim_{[n]\in\Delta^{\op}} \Hom_{\BimBanOne}(A, C(\Delta^n;B)) \\
         \label{eq:mappingspacesbimban2}
            &\simeq \bigsqcup_{[P]} \Hom_{\BanAlgInf}(A,\End_B(P)^{\op}) \sslash \mathsf{Aut}_B(P)
            \,,
    \end{align}
    where $[P]$ runs over isomorphism classes of finitely generated projective left $B$-modules and $\sslash$ denotes the homotopy quotient.
    There is a symmetric monoidal functor $\BanAlgInf \to \BimBanInf$.
\end{theorem}
\begin{proof}
    By~\cite[Prop. 4.1.7.4]{HA}, the localization is symmetric monoidal since $W$ is $\otimes$-stable.
    The derived mapping space Lemma~\ref{lem:derivedmappingspace} together with the simplicial resolution from Lemma~\ref{lem:simplicalresolution} imply  Equation~\eqref{eq:mappingspacesbimban1}.
    The symmetric monoidal functor $\Alg(\Ban)\to \BimBanOne$ maps homotopy equivalences to bimodule homotopy equivalences and hence induces a functor on localizations.

    For $P$ a finitely generated projective $B$-module, we can consider the topological action groupoid $\mathcal{A}_P=\Hom_{\BanAlgInf}(A,\End_B(P)^{\op}) \rtimes \mathsf{Aut}_B(P)$, with $\mathsf{Aut}_B(P)$ acting by conjugation.
    There is a map of simplicial groupoids
    \begin{align*}
        \bigsqcup_{[P]} \Sing(\mathcal{A}_P)_n &\longrightarrow \Hom_{\BimBanOne}(A,C(\Delta^n;B)) \,,
    \end{align*}
    which on objects maps $\rho:A\to C(\Delta^n;\End_B(P)^{\op})$ to the bimodule $(C(\Delta^n;P),\rho)$ with constant $C(\Delta^n;B)$-action and left $A$-action given by $\rho$.
    On morphisms, the morphism $(f,\rho):\rho \to f\rho f^{-1}$ for $f:\Delta^n\to  \mathsf{Aut}_B(P)$ is mapped to the bimodule isomorphism $f:(C(\Delta^n;P),\rho)\to (C(\Delta^n;P),f\rho f^{-1})$.
    This is easily seen to be natural in $[n]$ and to be essentially surjective and fully faithful, i.e.\ an equivalence.
    Therefore, the geometric realizations of both sides agree and we obtain \eqref{eq:mappingspacesbimban2}.
\end{proof}

We can consider the lax symmetric monoidal functor 
\begin{equation*}
\ModOne=\Hom_{\BimBanOne}(\mathbbm{1},-)\colon \BimBanOne \to \Cat_1^\Sigma \,.
\end{equation*}
and the composite lax symmetric monoidal algebraic $K$-theory functor 
\begin{equation*}
    \BimBanOne \xrightarrow{\ModOne} \Cat_1^\Sigma \xrightarrow{\mathcal{K}^{\mathrm{alg}}} \Sp_{\geq 0}\,,
\end{equation*}
which we also denote by $A\mapsto \mathcal{K}^{\mathrm{alg}}_A$.
\begin{definition}
\label{def:ungradedBimBanKth}
    We define the \emph{topological $K$-theory functor} $k\colon \BimBanInf \to \Sp_{\geq 0}$ as the initial homotopy invariant functor under $\mathcal{K}^{\mathrm{alg}}$, i.e.\ as the left Kan extension of $\mathcal{K}^{\mathrm{alg}}$ along $\gamma\colon \BimBanOne \to \BimBanInf$.
\end{definition}
Since the localization $\BimBanOne\to \BimBanInf$ is symmetric monoidal, $k$ admits a canonical refinement to a lax symmetric monoidal functor.
The inclusion $\Alg(\Ban) \to \BimBanOne$ induces a map $\BanAlgInf\to \BimBanInf$ upon localizing both at the homotopy equivalences.
\begin{theorem}
\label{th:Kthextension}
    The restriction of $k\colon \BimBanInf\to \Sp_{\geq 0}$ to $\BanAlgInf$ coincides with the topological $K$-theory of Banach algebras in the sense of \ref{thm:kanextension}.
\end{theorem}
\begin{proof}
    As in the proof of Theorem~\ref{thm:kanextension}, it is sufficient to prove that $A\rightsquigarrow C(\Delta^n;A)$ is a simplicial resolution in the category $\BimBanOne$.
    This is the content of Lemma~\ref{lem:simplicalresolution}.
\end{proof}

\begin{warning}
    It is not true that $k$ is excisive on the $(2,1)$-category $\BimBanOne$.
    For example, let $M_1,M_2$ be nonisomorphic irreducible modules of a finite-dimensional algebra $A$.
    Then
    \begin{equation*}
    \begin{tikzcd}[column sep=45]
        0 \ar[r] \ar[d] & \VectEn \ar[d,"V \mapsto V \otimes M_1"]
        \\
        \VectEn \ar[r,"V \mapsto V \otimes M_2"] & \ModEn_A
    \arrow["\lrcorner"{anchor=center, pos=0.125}, draw=none, from=1-1, to=2-2]
    \end{tikzcd}
    \end{equation*}
    is a pullback of Banach categories.
    Indeed, the iso-comma category has no nonzero objects, as those would give isomorphisms $M_1^{k_1} \cong M_2^{k_2}$ for some positive integers $k_1,k_2$.
    Observe that for the functors to $\ModOne_A$ to be Serre functors, we need that $\End(\R^k)\twoheadrightarrow\End_A(M_i^k)$ surjects, which happens if and only if $\End_A(M_i)\cong \R$ is an irreducible representation of real type.
    By taking $A = \R^{\times s}$ we see that applying connective $K$-theory to the above diagram yields the square
    \begin{equation*}
    \begin{tikzcd}
        0 \ar[r] \ar[d] & ko \ar[d]
        \\
        ko \ar[r] & k_A \cong ko^{\oplus s}
    \end{tikzcd} \,,
    \end{equation*}
    which is obviously not a pullback.
    We explain this failure by noting that algebra homomorphisms have the additional property of inducing quasi-surjective functors.
\end{warning}

We will now consider the relative category $(\calC,W)$ to be $\BimBanOne$ equipped with the class of homotopy equivalences.
\begin{lemma}
    \label{lem:simplicalresolution}
    For every Banach algebra $A$, the diagram 
    \begin{align*}
        \Delta^{\op} &\longrightarrow (\BimBanOne)_{A/} \\
        [n] &\longmapsto C(\Delta^n;A)
    \end{align*}
    is a simplicial resolution.
\end{lemma}
\begin{proof}
    Clearly, $A\to C(\Delta^n;A)$ is a homotopy equivalence.
    We are left to check condition~\ref{cond:2} of Definition~\ref{def:resolution}.
    It suffices to prove:
    The colimit of mapping spaces
    \begin{equation*}
        \colim_{\Delta^{\op}} \Hom_{\BimBanOne}(-, C(\Delta^n;A))
    \end{equation*}
    maps bimodule homotopies to homotopies in $\Spc$. It will then automatically map homotopy equivalences to equivalences.

    Let $\Phi\colon B\rightsquigarrow C(\Delta^1;B')$ be a homotopy between $\Phi_0,\Phi_1\colon B\rightsquigarrow B'$.
    We regard $\Phi$ as an element of $\Hom_{\BimBanOne}(B, C(\Delta^1;B'))$.
    Define a map  
    \begin{equation}
        \begin{aligned}
        H_\Phi:\Hom_{\BimBanOne}(B',A) & \longrightarrow \Hom_{\BimBanOne}(B, C(\Delta^1;A)) \\
        M & \longmapsto \left( C(\Delta^1;M)\right) \otimes_{ C(\Delta^1;B')} \Phi
    \end{aligned} \,.
    \end{equation}
    
    It is easy to see that this indeed defines a finitely generated projective left module and is natural in $A$.
    Furthermore, for the maps $\ev_i\colon C(\Delta^1;A)\to A$, one can see that the composite
    \begin{equation*}
        \begin{tikzcd}
            \Hom_{\BimBanOne}(B',A) \ar[r,"H_{\Phi}"]& \Hom_{\BimBanOne}(B, C(\Delta^1;A)) \ar[r,"(\ev_i)_*"]& \Hom_{\BimBanOne}(B,A)
        \end{tikzcd}
    \end{equation*}
    is naturally equivalent to $(\Phi_i)^*$, by means of the isomorphism
    \begin{equation*}
        A_{\ev_i} \otimes_{C(\Delta^1;A)}  C(\Delta^1;M) \otimes_{ C(\Delta^1;B')} \Phi \cong M_{\ev_i} \otimes_{C(\Delta^1;B')}\Phi \cong M\otimes_{B'} \Phi_i \,.
    \end{equation*}
    Upon applying $H_{\Phi}$ to the entire resolution we obtain:
    \begin{align*}
        \colim_{\Delta^{\op}} \Hom_{\BimBanOne}(B', C(\Delta^n;A)) & \longrightarrow  \colim_{\Delta^{\op}} \Hom_{\BimBanOne}(B, C(\Delta^1;C(\Delta^n;A))) \,.
    \end{align*}
    Using $C(\Delta^1;C(\Delta^n;A)\cong C(\Delta^n;C(\Delta^1;A)))$
    that amounts to a map
    \begin{align*}
        \mathcal{H}_{\Phi}\colon \Hom_{\BimBanInf}(B', A) \longrightarrow \Hom_{\BimBanInf}(B,C(\Delta^1;A)) \,,
    \end{align*}
    which we may regard as a homotopy $\Phi_0^*\simeq \Phi_1^*$:
    \begin{equation*}
        \Hom_{\BimBanInf}(B',A)\times \Delta^1 \xrightarrow{\mathcal{H}_\Phi\times \id} \Hom_{\BimBanInf}(B, C(\Delta^1;A))\times \Delta^1 \xrightarrow{\ev} \Hom_{\BimBanInf}(B,A)\,.
    \end{equation*}
\end{proof}

\paragraph{Geometric construction of \texorpdfstring{$\BimBanInf$}{BimBan}}
For every topological space $X$ there is a symmetric monoidal $(2,2)$-category $\BimBanOne(X)$ whose objects are Banach algebras $A$, interpreted as trivial algebra bundles over $X$.
The category of morphisms $B\rightsquigarrow A$ is the category of $(A,B)$-bimodule bundles over $X$ which are locally trivial as $A$-module bundles, i.e.\ locally isomorphic to $U\times P$, where $U\subseteq X$ and $P$ is a finitely generated projective left $A$-module.
The fiberwise tensor product equips this with a symmetric monoidal structure.

This assignment is contravariantly functorial in $X$ by pullback, i.e.\ defines a diagram
\begin{align*}
    \CHaus^{\op} \longrightarrow& \CMon(\Cat_{(2,1)}) \subseteq \CMon(\Cat_\infty)\\
    X \longmapsto & \BimBanOne(X)
    \,.
\end{align*}
Note that the objects are not allowed to vary in $X$.
Restricting to geometric simplices $\Delta^{\op} \subseteq \CHaus^{\op}$ we obtain a simplicial diagram of symmetric monoidal $(\infty,1)$-categories.
The colimit of this diagram can be identified with $\BimBanInf$.
\begin{equation*}
    \colim_{\Delta^{\op}} \BimBanOne(\Delta^n) \simeq \BimBanInf \,.
\end{equation*}
\begin{proof}[Sketch]
    Sifted colimits of symmetric monoidal $(\infty,1)$-categories are computed underlying~\cite[Cor.~3.2.3.2]{HA}.
The set of objects is constant and the diagram is sifted. 
Using~\cite[Prop. 1.3.4.10]{HA}, we can conclude that the colimit has the same set of objects and mapping spaces in the colimit are the colimits of the mapping spaces, as in Equation~\eqref{eq:mappingspacesbimban1}.
However, Lurie only proves a version for simplicially enriched categories.
\end{proof}

\begin{remark}
    In the paper~\cite{kristel2022insidious}, interesting subtleties surrounding bicategories of $X$-families of algebras and bimodules were pointed out.
    There are two reasons why our setup does not suffer from similar problems.
    Firstly, we can reinterpret our families of $(A,B)$-bimodules as bundles of finitely generated projective left $A$-modules, or equivalently, as locally trivial $A$-module bundles as discussed in Section~\ref{sec:SerreSwan}.
    However, they are in general not locally trivial as right $B$-module bundles.
    Secondly, we only consider families of bimodules, not bundles of algebras.
    Our construction does immediately give a symmetric monoidal bicategory of algebra bundles with constant fiber, which however does not satisfy descent in contrast to \cite{kristel2022insidious}.
\end{remark}

In the current setup, we can reformulate Segal's model for $K$-homology~\cite[Proposition~1.1,~1.3]{MR515311} as follows.

\begin{theorem}
\label{th:segal}
Let $X$ be a finite CW complex.
    The group completion of the $\E_\infty$-space 
    \begin{equation}
        \label{eq:segal}
        (\Hom_{\BimBanInf}(C(X),\R),\oplus)    
    \end{equation}
     of finite-dimensional $C(X)$-modules is equivalent to $ko \otimes \Sigma_+^\infty X$.
\end{theorem}
The proof rests on the realization that \eqref{eq:segal} is a configuration space.
More precisely, finite-dimensional representations of $C(X)$ are completely reducible into one-dimensional representations of $C(X)$ given by evaluation at a finite number of points $x_1,\dots,x_n \in X$; the one-dimensional representations are exactly the points of $X$ (e.g.\ Lemma~\ref{lem:gelfandhom} with $Y = \pt$).
Such a representation can be reinterpreted as a configuration space of finite-dimensional vector spaces, i.e.\ a skyscraper sheaf over $X$ which is supported at finitely many points.
When points collide, the vector spaces add.
\begin{lemma}
    Let $A$ be a unital Banach algebra.
    Then the mapping spaces in $\BimBanInf$ can be identified with the spaces of nonunital algebra homomorphisms
    \begin{equation}
        \label{eq:mappingspaceasunion}
        \Hom_{\BimBanInf}(A,\R) \simeq \colim_n \Hom^{\mathrm{nu}}(A,M_n(\R))  \,,
    \end{equation}
    where the colimit is along the nonunital block inclusions $M_n(\R)\hookrightarrow M_{n+1}(\R)$.
\end{lemma}
\begin{proof}
    Firstly, any nonunital homomorphism $\phi$ maps $1\in A$ to an idempotent and factors unitally through the corner subalgebra $\phi(1)M_n(\R) \phi(1)$.
    Let $\Hom^{\mathrm{nu}}_k(A,M_n(\R))$ be the subspace where $\rank(\phi(1))=\tr(\phi(1))=k$.
    We obtain $\Hom^{\mathrm{nu}}(A,M_n(\R))\cong \bigsqcup_k \Hom^{\mathrm{nu}}_k(A,M_n(\R))$.
    The unital factorisation leads to a homeomorphism
    \begin{equation*}
        \mathrm{Fr}_k(\R^n)\times_{\GL_k}\Hom(A,M_k(\R)) \xrightarrow{\cong} \Hom^{\mathrm{nu}}_k(A,M_n(\R)) \,,
    \end{equation*}
    where $\mathrm{Fr}_k(\R^n)$ is the space of $k$-frames in $\R^n$ with principal $\GL_k$-action on the frames, where $\GL_k$ acts on $M_k(\R)$ by conjugation.
    Finally, since $\colim_n \mathrm{Fr}_k(\R^n)\simeq E\GL_k$ we obtain
    \begin{align*}
        \colim_n \Hom^{\mathrm{nu}}(A,M_n(\R)) &\cong \bigsqcup_k   \colim_n \mathrm{Fr}_k(\R^n)\times_{\GL_k} \Hom(A,M_k(\R)) \\
        &\simeq \bigsqcup_k \Hom(A,M_k(\R))\sslash \GL_k \\
        &\simeq \Hom_{\BimBanInf}(A,\R) \,.
    \end{align*}
\end{proof}
\begin{remark}
    \label{rem:segal}
    Let $(Y,y)$ be a compact connected pointed Hausdorff space and let $C_0(Y,y)$ denote the \emph{nonunital} Banach algebra of functions vanishing at $y$.
    Segal proves that $\widetilde{F}(Y):=\colim_n \Hom(C_0(Y,y),M_n(\R))$ is group-complete and coincides with the \emph{reduced} $ko$-homology $ko\otimes \Sigma^\infty Y$~\cite[Prop.~1.1]{MR515311}. 
    
    Let $X$ be compact Hausdorff and let $X_+=X\sqcup \pt$.
    Note that $F(X):=\widetilde{F}(X_+)$ coincides with \eqref{eq:mappingspaceasunion}.
    However, $F(X)$ is not group-complete and because $X_+$ is not connected we cannot directly apply Segal's theorem.
    Thus, the group completion in Theorem~\ref{th:segal} is a sort of basepoint issue.
\end{remark}
\begin{proof}[Proof of Theorem~\ref{th:segal}]
    We use the notation of Remark~\ref{rem:segal}.
    Let $Y$ denote a finite pointed CW complex.
    Note that $\widetilde{F}(S^0)=F(\pt)^\gp=ko$ and there is an assembly natural transformation
    \begin{equation*}
        \alpha_Y:ko\otimes \Sigma^\infty Y \to \widetilde{F}(Y)^{\gp} \,.
    \end{equation*}
    To show that this is an equivalence for finite CW complexes $X$, we will rely on results from \cite{zbMATH01827571,zbMATH05225707} on configuration spaces.
    
    By \cite[Ex.~7(c)]{zbMATH05225707}, $\widetilde{F}(Y)$ is identified with the configuration space model used in loc. cit.
    and by \cite[Prop.~4.5]{zbMATH01827571} the natural map $\widetilde{F}(Y)\to \Omega \widetilde{F}(\Sigma Y)$ is a group completion.
    Moreover, $Y\mapsto \pi_*\Omega \widetilde{F}(\Sigma Y)$ defines a reduced homology theory on finite pointed CW complexes.
    We conclude that $\alpha_Y$ is an equivalence because $\pi_*\alpha$ is a natural transformation of homology theories which agree on $Y=S^0$.
    The theorem follows for $Y=X_+$.
\end{proof}

\begin{remark}
    Consider the group completion $k^{\mathrm{fin}}(B,A)$ of the $\E_\infty$-monoid $(\Hom_{\BimBanInf}(A,B)^{\cong},\oplus)$ of $(B,A)$-bimodules which are finitely generated as $B$-modules. 
    We have observed that for $A = \R$, the spectrum $k^{\mathrm{fin}}(B,A)$ is the connective $K$-theory spectrum of $B$.
    It follows by Segal's Theorem~\ref{th:segal} that for $A = C(X)$ and $B = \R$ the spectrum $k^{\mathrm{fin}}(B,A)$ agrees with connective $ko$-homology of $X$.
    Therefore it is tempting to conjecture that $k^{\mathrm{fin}}(B,A)$ is a spectral refinement of ``connective Kasparov theory''.
However, this is not the case in general since $(B,A)$-bimodules that are finitely generated and projective are ``too small'' to detect delicate infinite-dimensional aspects necessary to capture $KK$-theory.
For example, the Cuntz algebra $\mathcal{O}_n$ does not admit any finite-dimensional representations and so $\Hom_{\BimBanInf}(A,B) = \{0\}$.
However, for $n >2$, the $K$-homology of $\mathcal{O}_n$ is nontrivial, see e.g. \cite{MR3792521}.
\end{remark}

\newpage
\section{Graded \texorpdfstring{$K$}{K}-theory by group completion}
\label{sec:graded kthy}
This section will extend the setup of Section~\ref{sec:ungradedkthy} to the setting of graded Banach algebras.
We discuss our conventions on graded algebras and signs (\ref{sec:graded alg}) and we record important structure theorems.
In particular, we will quickly adapt the definitions from the ungraded case to define the $(\infty,1)$-category $\sBanAlgInf$ and the $(\infty,1)$-Morita category $\sBimBanInf$ of graded Banach algebras.

The decisive difference is that graded $K$-theory is not just the $K$-theory of the category of graded modules $k^{C_2}_A=k(\ModEn^{C_2}_A)$.
We develop a spectral refinement $k^{\ABS}$ of the Atiyah--Bott--Shapiro definition in Section~\ref{sec:connectivesuperkthy}.
As discussed there, making this into a lax symmetric monoidal functor requires an infinite chain of coherences.
The technical implementation of this is Section~\ref{sec:smith}, which exhibits a Smith ideal (Definition \ref{def:smith}) in the $(\infty,1)$-category $\InnInf$ of inner product spaces.
We deem this to be an important standalone result.
Using the symmetric monoidal functor $\Cl\colon \InnInf \to \sBanAlgInf \to \sBimBanInf$ constructed in Proposition~\ref{prop:Clfunctor} we can transport the Smith ideal to graded algebras and prove monoidality of $k^{\ABS}$.
We state a connective version of Bott periodicity in Theorem~\ref{thm:gradedbottperiodicity}.
This is used to define periodic graded $K$-theory in Section~\ref{sec:periodickthy}.
The section contains a statement of our \ThmB.
\ThmC\ is stated in Section~\ref{sec:picard}.
The longer proofs are gathered in Section~\ref{sec:proofs}.

\subsection{Graded algebra}
\label{sec:graded alg}

\paragraph{Graded vector spaces and graded algebras}

We fix a field $\mathbb{F}\in \{\R,\C\}$.
A \emph{graded vector space} is a vector space with a grading into even and odd degrees $V=V_{0}\oplus V_1$.
We write $|v|=i$ if $v\in V_i$. 
We denote by $\mathbb{F}^{p|q}$ the vector space $\mathbb{F}^{p+q}$ with $p$ even and $q$ odd coordinates.

Morphisms of graded vector spaces are required to preserve the grading.
The \emph{grading involution} of a graded vector space $V$ is defined as the morphism $v\mapsto (-1)^{|v|}v$.
The category of graded vector spaces and grading-preserving maps $\sVectOne$ carries a graded tensor product $\grotimes$
\begin{equation}
    \label{eq:gradedtensorproduct}
    (V\grotimes W)_0 = V_0\grotimes W_0 \oplus V_1\grotimes W_1\,, \quad  (V\grotimes W)_1=V_1\grotimes W_0 \oplus V_0\grotimes W_1 \,,
\end{equation}
which can be enhanced into a symmetric monoidal structure.
Importantly, the braiding of the category is given by the Koszul sign rule:
\begin{equation}
\label{eq:koszul}
    \begin{aligned}
        \beta: V\grotimes W &\longrightarrow W\grotimes V \\
    v\otimes w & \longmapsto (-1)^{|v||w|}w\otimes v \,.
    \end{aligned} \,.
\end{equation}
A \emph{graded subspace} $W \subseteq V$ is a subspace such that $W_i := W \cap V_i$ defines a grading.
\begin{convention}
\label{conv:graded}
    We say \emph{graded} for $\Z/2$-graded throughout; no other grading groups occur.
    The Koszul `super' sign~\eqref{eq:koszul} enters only through the braiding: it appears in opposite algebras, in the multiplication on tensor products of algebras and in bimodules, but not in the tensor product of the underlying graded vector spaces.
    Several standard category names retain their customary super ``s'' (e.g.\ $\sVectOne$).
    We reserve $C_2$-decorations for group actions and equivariance.
\end{convention}

A \emph{graded algebra} is an algebra object in $(\sVectOne,\grotimes)$.
Explicitly, a graded algebra is an ordinary algebra $A$ together with a grading $A=A_0\oplus A_1$ such that $|ab|=|a|+|b|$ modulo $2$ and $1 \in A_0$.
The main example of graded algebras are the Clifford algebras, see Definition~\ref{def:clifford} for our conventions.
Morphisms in $\Alg(\sVectOne)$ are grading-preserving algebra homomorphisms.
Using the Koszul sign, the multiplication on the tensor product $A\grotimes B$ of two graded algebras is
\begin{equation}
\label{gradedtensoralgebra}
    (a_1\grotimes b_1)(a_2\grotimes b_2)=(-1)^{|a_2||b_1|} a_1a_2\grotimes b_1b_2 \,.
\end{equation}

Let $\Mor(\sVectOne)$ be the symmetric monoidal $(2,2)$-category  of graded algebras, graded bimodules, and grading-preserving bimodule maps.
Explicitly, an $(A,B)$-graded bimodule is a bimodule $M$ which is a graded vector space $M = M_0 \oplus M_1$ such that the action respects the grading in the sense that $A_i M_j \subseteq M_{i+j}$ and $M_j B_i \subseteq M_{j+i}$ for $i,j \in \Z/2$.

\paragraph{Graded Banach algebras}

A \emph{graded Banach space} is a Banach space $X$ together with a \emph{grading}, which is a direct sum decomposition $X = X_0 \oplus X_1$ into closed subspaces.\footnote{In order to ensure algebras in graded Banach spaces agree with the usual notion of graded Banach algebras, we do not require this to be a product or coproduct in $\Ban_\leq$; the norm on $X$ need not be formed as an $\ell^1$ or $\ell^\infty$-norm of certain norms on $X_0$ and $X_1$. However, for any decomposition $X = X_0 \oplus X_1$ into closed subspaces the norm on $X$ is equivalent to the $\ell^1$-sum of $X_0$ and $X_1$, so topologically the distinction is irrelevant.}
Analogously to the ungraded case, one defines the symmetric monoidal $1$-categories $\sBan_{\leq}$, resp. $\sBan$, of \emph{graded Banach spaces} and grading-preserving contractive, resp. bounded, linear maps.
The symmetric monoidal structure is the completed projective graded tensor product $\grotimes$ as in Equation~\eqref{eq:gradedtensorproduct} and the contractive Koszul braiding~\eqref{eq:koszul}.
A \emph{graded Banach algebra} is an algebra object in $\sBan_{\leq}$, i.e.\ a graded algebra with a submultiplicative Banach norm.
Every algebra object in $\sBan$ is isomorphic to a graded Banach algebra by rescaling the norm.

\begin{notation}
    Let $A$ be a graded (Banach) algebra.
    We write $|A|$ for the underlying ungraded (Banach) algebra.\footnote{
    The functor $|-|:\sVectOne \to \VectOne$ that forgets the grading is monoidal but \emph{not} braided monoidal.
    Therefore it defines a (non-monoidal) functor $\Alg(\sBan) \to \Alg(\Ban)$.}
\end{notation}
If $f, g \colon X \to Y$ are grading-preserving bounded linear maps between graded Banach spaces, then $\overline{\im(f-g)} \hookrightarrow Y$ is a grading-preserving inclusion and hence $Y/\overline{\im(f-g)}$ is still the coequalizer in $\sBan_{\leq}$ and $\sBan$.
    Since this coequalizer is preserved by the projective graded tensor product, we can apply Lemma~\ref{lem:morexists} to obtain
    \begin{enumerate}
        \item the symmetric monoidal $1$-category of graded Banach algebras and grading-preserving bounded homomorphisms $(\Alg(\sBan),\grotimes)$.
        \item the symmetric monoidal $(2,2)$-category $\Mor(\sBan)$ of graded Banach algebras, graded bimodules and grading-preserving bimodule maps.
    \end{enumerate}

\begin{definition}
\label{def:parity}
    If $M$ is a graded (Banach) $(A,B)$-bimodule its \emph{parity shift} is the $(A,B)$-bimodule $\Pi M := \Pi \mathbb{F} \otimes M$.
\end{definition}
The parity shift $\Pi M$ is equal to $M$ as a vector space, but with $M_0$ and $M_1$ reversed.
The right $B$-action is the same, but the left $A$-action is modified by the grading automorphism $a\mapsto (-1)^{|a|}a$.
\begin{remark}
    Tensoring from the left and the right is equivalent by applying the braiding, which results in a comparison sign.
\end{remark}

We let $A^{p|q}=A^p\oplus \Pi A^q \cong \mathbb{F}^{p|q}\grotimes A$.

\begin{definition}
    A graded $A$-module is \emph{finitely generated projective} if it is isomorphic (via a grading-preserving $A$-linear map) to a direct summand of $A^{p|q}$.
\end{definition}
\begin{notation}
    The category of finitely generated projective graded $A$-modules is denoted $\ModOne^{C_2}_A$.
    The ungraded variant is denoted by $\ModOne_{|A|}$.
    We hope this results in less confusion than using the potentially ambiguous notation $\ModOne_A$.
\end{notation}

\begin{convention}
\label{conv:graded-linear}
The category of finitely generated projective graded modules $\ModOne^{C_2}_A$ is enriched over $\sBan$.
Explicitly, the enrichment consists of those bounded linear maps $T \colon M \to N$ satisfying 
\begin{equation}
\label{eq:graded-linear}
    T(am) = (-1)^{|T||a|} a T(m) \,,
\end{equation}
where $|T|$ is the degree of $T$ (not necessarily grading preserving).
    Whether the sign in Equation~\eqref{eq:graded-linear} is required for odd maps varies throughout the literature. 
Our sign was forced from categorical principles, after fixing the conventions of left modules, left tensoring, and acting with operators from the left, see Lemma \ref{lem:selfenrichment}.
\end{convention}

\begin{convention}
\label{conv:End-op}
As in the ungraded case (Convention~\ref{conv:leftright}), $A$-linear endomorphisms of $A^{p|q}$ are given by multiplications with matrices from the right, e.g.\ $\End_A(A) = A^{\op}$.
Since endomorphisms act from the left, the $(p|q)$-dimensional matrix algebra $M_{p|q}(A) = M_{p|q}(\R) \grotimes A$ is the opposite graded Banach algebra of endomorphisms of the $A$-module $A^{p|q}$. 
This distinction is important because e.g.\ $\Cl_{n} \ncong (\Cl_n)^{\op}$ in general.
\end{convention}

\begin{definition}
    A \emph{homotopy} of graded Banach $(A,B)$-bimodules is a graded Banach $(C(\Delta^1;A), B)$-bimodule which is finitely generated projective as an $C(\Delta^1;A)$-module.
\end{definition}
\begin{definition}
    Analogously to the ungraded case, we denote 
    \begin{itemize}
        \item the topologically enriched category of graded Banach algebras and grading-preserving homomorphisms by $\sBanAlgEn$. Here, the set of bounded algebra homomorphisms $\Hom(A,B)$ is equipped with the $k$-ification of the compact-open topology.
    \item the associated $(\infty,1)$-category by $\sBanAlgInf$.
    \item the $(2,1)$-category of graded Banach algebras $\sBimBanOne$, where a $1$-morphism $A\rightsquigarrow B$ is a graded $(B,A)$-bimodule which is finitely generated and projective as a $B$-module or, equivalently, those $1$-morphisms which are right adjoints in $\Mor(\sBan)$.
    The $2$-morphisms are grading-preserving bimodule isomorphisms.
    \item the symmetric monoidal $(\infty,1)$-category obtained by localizing the underlying $(2,1)$-category of $\sBimBanOne$ with respect to those graded bimodules that are homotopy equivalences by $\sBimBanInf$
    \item the Morita $(2,1)$-functor taking finitely generated projective modules by $\ModOne^{C_2} \colon \sBimBanOne \to \Cat_1^{\Sigma}$ and its simplicially enriched version by $\ModEn_A^{C_2}$.
    \end{itemize}
\end{definition}
Applying Appendix~\ref{app:algebras}  as in Section~\ref{sec:Banach}, there are functors $\sBanAlgOne \to \sBimBanOne$ and $(\sBanAlgOne)^{\op}_{\mathrm{fgp}} \to \sBimBanOne$ given by inducing bimodules by homomorphisms on either side, which induce functors of $(\infty,1)$-categories after localizing homotopy equivalences.
Regarding a Banach space as trivially graded induces a symmetric monoidal functor $\Ban \to \sBan$ and further functors $\Alg(\Ban) \to \sBanAlgOne, \BanAlgInf \to \sBanAlgInf \BimBanOne \to \sBimBanOne$ and $\BimBanInf \to \sBimBanInf$.
If $A$ is an ungraded Banach algebra, we will often implicitly consider it as a graded Banach algebra by applying one of these functors, i.e.\ we take $A_0 = A$ and $A_1 = 0$.

\paragraph{Graded \texorpdfstring{$C^*$}{C*}-algebras}

\begin{definition}
    A \emph{graded $C^*$-algebra} is a $C^*$-algebra $A$ which is also a graded algebra $A=A_0\oplus A_1$ such that $|a^*|=|a|$.
\end{definition}

\begin{warning}
\label{warning:opposite}
Given any symmetric monoidal category $(\mathcal{C}, \otimes,\beta)$ and an algebra $A \in \mathcal{C}$, there is a notion of \emph{opposite algebra $A^{\op}$} of $A$.
This algebra has the same underlying object but its multiplication map $A^{\op} \otimes A^{\op} \to A^{\op}$ is changed by the braiding $\beta$.
    With this definition, the opposite of an ungraded algebra is the same algebra with multiplication $a^{\op} b^{\op} = (ba)^{\op}$ as usual. One could define a $C^*$-structure on an algebra to be a $\C$-antilinear algebra homomorphism $* \colon A \to A^{\op}$ satisfying some conditions.
    
    The opposite $A^{\op}$ of a graded algebra $A$ on the other hand has multiplication $a^{\op}b^{\op}=(-1)^{|a||b|}(ba)^{\op}$, as dictated by the Koszul sign rule.
    The $*$ of a graded $C^*$-algebra $A$ therefore \emph{does not} induce a ($\C$-antilinear) algebra isomorphism $*:A\to A^{\op}$ in the category of graded algebras.
In particular, a graded $C^*$-algebra is \emph{not} a $*$-algebra object in $(\sVectOne, \grotimes)$.
    If $A$ is a $C^*$-algebra, then $A^{\op}$ is a $C^*$-algebra with $(a^{\op})^* := (-1)^{|a|} (a^*)^{\op}$.
\end{warning}

However, it is customary to use the Koszul sign on the tensor product, forcing some ad hoc modifications: (see~\cite[Sec.~14.]{MR1656031})

\begin{definition}
    The \emph{tensor product $A \grotimes B$} of graded $C^*$-algebras $A,B$ is the tensor product of underlying graded algebras equipped with $(a \grotimes b)^* = (-1)^{|a||b|} a^* \grotimes b^*$, completed under the maximal tensor product norm.
\end{definition}

The map $(A \grotimes B)^{\op} \to A^{\op} \grotimes B^{\op}$ given by $(a \otimes b)^{\op} \mapsto (-1)^{|a||b|} a^{\op} \otimes b^{\op}$ is a natural $*$-isomorphism.

\paragraph{Graded Wedderburn--Artin}
We need graded versions of Schur's lemma and the Wedderburn--Artin theorem, also see~\cite[Chapter 12]{MR2165457}, \cite[Theorem C.6]{MR3119923}, and~\cite{MR1701598}.
The following definitions are only applicable to finite-dimensional graded algebras.
A \emph{graded submodule} $N \subseteq M$ is a graded subspace which is also a submodule.
A graded module is \emph{simple} if it has no graded submodules.
A graded module is \emph{semisimple} if any graded submodule is a direct summand.
A graded algebra $A$ is \emph{simple} if it is simple as an $A \grotimes A^{\op}$-module or equivalently, if it has no proper closed two-sided graded ideal. 
A graded algebra $A$ is \emph{semisimple} if it is semisimple as an $A$-module.

\begin{example}
    The graded algebra $\Cxl_1$ is simple even though $|\Cxl_1|$ is not simple.
\end{example}

\begin{definition}
    A graded algebra $A = A_0 \oplus A_1$ is a \emph{graded division algebra} if every homogeneous nonzero element $a \in A_i, i = 0,1$ is invertible.
\end{definition}

\begin{lemma}[Graded Schur]
    If $M$ is a simple graded $A$-module, then $\End_A M$ is a graded division algebra.
\end{lemma}
\begin{proof}
    Let $f \colon M \to M$ be a nonzero homogeneous endomorphism.
    Then $\ker f$ and $\im f$ are graded submodules of $M$.
    Since $f \neq 0$ and $M$ is simple, we obtain $\ker f = 0$ and $\im f = M$, and so $f$ is invertible.
\end{proof}

\begin{proposition}[Graded Wedderburn--Artin]
    \label{prop:wedderburn}
    If $A$ is a finite-dimensional simple graded algebra, then $A \cong M_{p|q}(D)$ for some $p,q \geq 0$ and $D$ a graded division algebra.
    Moreover, if $A$ is finite-dimensional and semisimple, then it is a finite direct sum of simple graded algebras.
\end{proposition}
\begin{proof}
The former result is shown in~\cite[Theorem 2.10.10]{MR2046303}.
    If $A$ is finite-dimensional and semisimple, there is an isomorphism of left $A$-modules $A\cong \bigoplus_{i=1}^n M_i^{p_i|q_i}$ for $M_i$ a list of simple $A$-modules.
    Then by Schur's lemma $A\cong \End_A(A)^{\op} \cong \bigoplus_{i=1}^n \End_A(M_i^{p_i|q_i})^{\op} \cong \bigoplus_{i=1}^n M_{p_i|q_i}(D_i)^{\op}$ where $D_i = \End_A (M_i)^{\op}$, using~\eqref{eq:endop}.
\end{proof}

\begin{theorem}[{\cite{MR167498}}]
    There are ten isomorphism classes of finite-dimensional graded division algebras:
    \begin{equation*}
    A = \C, \, \Cxl_1, \quad \R,\, \Cl_1,\, \Cl_2, \,\Cl_3, \,\mathbb{H}, \,\Cl_{-3},\,\Cl_{-2},\,\Cl_{-1}
    \end{equation*}
\end{theorem}

We conclude that any finite-dimensional semisimple graded algebra is a finite direct sum of finitely many matrix algebras over the above ten graded division algebras.

\begin{warning}
We have that $M_{1|0}(\Cxl_1) \cong M_{0|1}(\Cxl_1)$ as graded algebras, and so the choice of $p$ and $q$ in the graded Wedderburn--Artin theorem is not unique in general.
However, the unordered list of graded division algebras $D$ occurring in the decomposition of a semisimple graded algebra is unique.
\end{warning}

\begin{remark}
\label{rem:semsim}
    The Morita $(2,1)$-category of finite-dimensional semisimple graded algebras, finite-dimensional bimodules and bimodule isomorphisms admits an inclusion functor to $\sBimBanOne$.
    In fact, all such algebras admit a graded $C^*$-algebra structure, and the induced norm is unique up to isomorphism.
    For finite-dimensional semisimple algebras, finite-dimensional bimodules agree with modules that are finitely generated projective on the left, and bimodule actions and bimodule isomorphisms are automatically continuous.
\end{remark}

\subsection{Smith ideal structure on Clifford algebras}
\label{sec:smith}
This section constructs a Smith ideal (Definition~\ref{def:smith}) on the arrow $\Cl_{-1}\to^{\op} \R$ in $\sBanAlgInf^{\op}$, the crucial input for lax monoidality of graded $K$-theory.

To achieve this we construct a symmetric monoidal functor 
\begin{equation*}
    \Cl\colon \InnInf \longrightarrow \sBanAlgInf 
\end{equation*}
from the $(\infty,1)$-category of inner product spaces and isometric embeddings to the $(\infty,1)$-category of graded Banach algebras and even homomorphisms. 
The functor $\Cl$ assigns to an inner product space $(V,\langle-,-\rangle)$ the negative Clifford algebra on $V$, i.e.\ the unit vectors will satisfy $v^2=-1$.
The Smith ideal will be the image of the Smith ideal $\R^1 \to^{\op} \R^0$ in $\InnInf^{\op}$ under the symmetric monoidal functor $\Cl$.

The existence of this Smith ideal has been sketched in~\cite{TheMscThesis}.
Here, we provide a rigorous proof.
\subsubsection{Inner product spaces}
\begin{definition}
    The \emph{category of inner product spaces} $\InnEn$ is the topologically enriched category whose objects are pairs $(V,\langle\cdot  ,\cdot \rangle )$, where $V$ is a finite-dimensional real vector space equipped with an inner product $\langle\cdot  ,\cdot \rangle$ and an orthonormal basis.
    The morphisms are inner product preserving linear maps (not necessarily preserving the basis),
    which carry a Euclidean topology.
    $\InnEn$ carries a symmetric monoidal structure induced by direct sum $\oplus$.
    We denote by $\InnInf=N^{\mathrm{hc}}(\InnEn)$ the associated symmetric monoidal $(\infty,1)$-category.
\end{definition}
\begin{remark}
    The symmetric monoidal categories of inner product spaces with and without chosen basis are equivalent by simply forgetting the basis.
    We keep the choice of basis for conveniently formulating the constructions below.
\end{remark}
\begin{definition}
     An isometric embedding $f\colon V\to W$ in $\InnEn$ is \emph{coordinate-positive} if it maps each basis vector to a linear combination of basis vectors in which each coefficient is nonnegative.
     We let $\InnEn_{\geq 0}\subseteq \InnEn$ be the wide symmetric monoidal subcategory on the coordinate-positive maps.
     Let $\mathsf{J}$ be the full topological suboperad of $\Op(\InnEn_{\geq 0}^{\op})$ on the objects $\R^0,\R^1$ (and sequences of these), see Appendix \ref{app:categories} for our notation for operads.
\end{definition}
\begin{example}
    The space of orthogonal linear embeddings of $\R^1$ into $\R^2$ is $\Hom_{\InnEn}(\R^1,\R^2)\cong S^1$.
    On the other hand, the space of coordinate-positive orthogonal linear embeddings is restricted to the coordinate-positive quadrant $\InnEn_{\geq 0}(\R^1,\R^2)\cong S^1\cap [0,\infty)\times [0,\infty)\cong D^1$.
\end{example}
\begin{example}
    We can explicitly describe the underlying operad $\mathsf{J}\to \Fin_*$ as follows:
    \begin{itemize}
        \item An object over $\langle n\rangle$ is a map $f\colon \langle n \rangle^\circ\to \{0,1\} $, i.e.\ an ordered sequence of $\R^0$'s and $\R^1$'s.
        \item A morphism $(\langle n \rangle,f )\to (\langle m\rangle, g)$ over $\gamma\colon \langle n\rangle \to \langle m \rangle $ is a collection of coordinate-positive orthogonal linear embeddings 
        \begin{equation*}
           i_j\colon \R^{g(j)} \longrightarrow \bigoplus_{i\in \gamma^{-1}(j)} \R^{f(i)} \,.
        \end{equation*}
        Such a morphism exists if and only if for $j=1,\dots,m$ there exists $i\in \gamma^{-1}(j)$ with $g(j)\leq f(i)$.
        The topological space of morphisms is 
        \begin{equation*}
            \bigsqcup_{\gamma}\prod_{j=1}^m \InnEn_{\geq 0}\left(\R^{g(j)},\bigoplus_{i\in \gamma^{-1}(j)} \R^{f(i)}\right) \,.
        \end{equation*}
        This topological space is either empty or contractible.
    \end{itemize}
\end{example}

We will recall briefly here the theory of Smith ideals in commutative rings, see~\cite[Section~2]{hebestreit2025notehigherringtheory} for a more detailed discussion. 

\begin{definition}
\label{def:smith}
    The symmetric monoidal $1$-category $[1]_{\mathrm{min}}$ is the walking arrow category $0\to 1$ together with the symmetric monoidal structure given by $i\otimes j=\min(i,j) \in \{0,1\}$. 
    A (commutative) \emph{Smith ideal in a symmetric monoidal $(\infty,1)$-category $\calC$} is a lax symmetric monoidal functor $[1]_{\mathrm{min}}\to \calC$.
\end{definition}
The underlying functor will assign an arrow $I\to R$ in $\calC$.
The lax monoidal structure yields maps $R\otimes R\to R$, $I\otimes R\to I$, etc.\ as well as higher coherences.
This equips $R$ with a commutative algebra structure, $I$ with a nonunital commutative algebra structure and with a compatible $R$-bimodule structure.

A quotient of a ring by an ideal is a ring and, conversely, the fiber of a map of rings is an ideal.
The higher algebraic version of this is:
\begin{lemma}[{\cite[Prop.~2.3]{hebestreit2025notehigherringtheory}}]
    Let $\calC$ be pointed with finite colimits.
    Let $I\to R$ be a Smith ideal.
    Then $\cofib(I\to R)$ attains the structure of a commutative algebra in $\calC$ such that
    \begin{equation*}
        R\to \cofib(I\to R)
    \end{equation*}
    is a map in $\CAlg(\calC)$.
\end{lemma}
Upon identifying lax symmetric monoidal functors $\calC\to \calD$ with commutative algebras with respect to Day convolution we obtain: 
\begin{corollary}
    \label{cor:cofiberlaxmon}
    Let $\calC, \calD$ be symmetric monoidal $(\infty,1)$-categories.
    Let $I\to R$ be a Smith ideal in $\Fun(\calC,\calD)$ with the Day convolution symmetric monoidal structure and suppose $\calD$ is pointed with finite colimits.
    Then the functor $\cofib(I\to R)\colon \calC \to \calD$ is lax symmetric monoidal and $R\to \cofib(I\to R)$ is a lax symmetric monoidal natural transformation.
\end{corollary}

We now construct a Smith ideal in $\InnInf^{\op}$.

\begin{lemma}
\label{lem:operadequivJ}
    The functor
    \begin{align}
        \label{eq:operadmap}
        \begin{aligned}
            \mathsf{J} &\longrightarrow \Op([1]_{\mathrm{min}}) \,, \\
        (\langle n\rangle,f) &\longmapsto (\langle n \rangle , 1-f) \,,\\
        [\gamma,(i_j)] & \longmapsto \gamma 
        \end{aligned}
    \end{align}
    defines a weak equivalence of topological operads.
\end{lemma}
\begin{proof}
    For the ease of comparison we spell out $\Op([1]_{\mathrm{min}})$ explicitly:
    \begin{itemize}
        \item An object over $\langle n\rangle$ is a map $f\colon \langle n\rangle^\circ \to \{0,1\}$, i.e.\ an ordered sequence of objects in $[1]_{\mathrm{min}}$.
        \item A morphism $(\langle n \rangle,f )\to (\langle m\rangle, g)$ over $\gamma\colon \langle n\rangle \to \langle m \rangle $ is a collection of maps in $[1]_{\mathrm{min}}$
        \begin{equation*}
            \min_{i\in \gamma^{-1}(j)} f(i) \to g(j) \,.
        \end{equation*}
        There is a unique such morphism if and only if there exists $i\in \gamma^{-1}(j)$ with $f(i)\leq g(j)$.
    \end{itemize}
    We observe that the map~\eqref{eq:operadmap} is essentially surjective and the spaces of multimorphisms on both sides are either empty or contractible at the same time.
\end{proof}
\begin{corollary}
    \label{cor:smithidealInn}
    The arrow $\R^1 \to^{\op} \R^0$ in the symmetric monoidal $(\infty,1)$-category $\InnInf^{\op}$ admits a refinement to a Smith ideal.
    It is the lax symmetric monoidal functor exhibited by the zig-zag of topological operads
\begin{equation*}
    \Op([1]_{\mathrm{min}}) \xleftarrow{\simeq} \mathsf{J} \to \Op(\InnEn^{\op}) \,.
\end{equation*}
\end{corollary}
\begin{proof}
    Applying $\Sing$ turns this zig-zag into a zig-zag of fibrant simplicial operads, where the back leg is again a weak equivalence using Lemma~\ref{lem:operadequivJ}. 
    \cite[Thm.~B]{arakawa2025relativeoperadsmodelinftyoperads} proves that the localization of fibrant simplicial operads at the weak equivalences is a model for the $(\infty,1)$-category of operads $\mathcal{O}p_\infty$.
    The operadic nerve~\cite[Def.~2.1.1.23]{HA} (which coincides with the homotopy coherent nerve on the underlying category)
    is an explicit equivalence to Lurie's model of $\infty$-operads.   
    Finally, we use that for symmetric monoidal $(\infty,1)$-categories a lax symmetric monoidal functor $\calC\to \calD$ is the same as a functor of operads $\Op(\calC)\to \Op(\calD)$.
\end{proof}
\subsubsection{Clifford algebras as a functor}
In this section, we will construct a symmetric monoidal functor from the category of inner product spaces with $\oplus$ to the category of graded Banach algebras that assigns the Clifford algebra.
Much of the content of this section is well-known, see~\cite[Chapter 1.1]{lawsonmichelsohn} and~\cite[Lemma 5.9]{MR3548460}.
\begin{definition}
\label{def:clifford}
Let $\mathbb{F}$ be either $\R$ or $\C$.
    Let $V$ be a finite-dimensional vector space over $\mathbb{F}$ with a nondegenerate symmetric bilinear form $B$.
    The \emph{Clifford algebra} $\Cl(V,B)$ is defined as the free (noncommutative) algebra on the vector space $V$ modulo the relations $vw + wv = 2 B(v,w)$ for all $v,w\in V$.
    We write $\Cl_n=\Cl_{+n}$ for $\Cl(\R^n,\langle,\rangle)$ and $\Cl_{-n}$ for $\Cl(\R^n,-\langle,\rangle)$, where $\langle,\rangle$ is the standard Euclidean inner product.
    We write $\Cxl_n = \Cl(\C^n, (,))$, where $(,)$ is the standard complex nondegenerate symmetric bilinear form.
\end{definition}
The free algebra on $V$ has a natural $\Z_{\geq0 }$-grading in which generators have degree one.
Since the defining relation of the Clifford algebra is of even degree, $\Cl(V,B)$ is a graded algebra where the vectors $V \subseteq \Cl(V,B)$ have odd degree.
By definition Clifford algebras satisfy the following universal property: an algebra homomorphism $\Cl(V,B) \to A$ is equivalent to a linear map $f \colon V \to A$ such that $f(v)f(w) + f(w) f(v) = 2B(v,w)$ for all $v,w\in V$.
If $A$ is a graded algebra, then the homomorphism is even if and only if $f$ lands in $A_1$.
\begin{example}
    Note that $\Cl_{-1} = \R[e]/(e^2 + 1)$, where $e \in \R^1$ is a norm one vector.
    This algebra is isomorphic to $\C$ as an ungraded algebra.
    However, this is clearly not a graded isomorphism as we consider $\C$ to be purely even.
    We have that $(\Cl_{+1})^{\op} \cong \Cl_{-1} \ncong \Cl_{+1}$ as graded $C^*$-algebras, also see Warning~\ref{warning:opposite}.
\end{example}

\begin{example}
\label{ex:baby8per}
    There are isomorphisms of graded algebras $\mathbb{H} \grotimes \Cl_{\pm 1} \cong \Cl_{\mp 3}$ for both signs given by 
    \begin{align*}
        i &\mapsto e_1 e_2
        \\
        j &\mapsto e_2 e_3 \qquad  f \mapsto e_1 e_2 e_3
        \\
        k &\mapsto e_1 e_3 \,,
    \end{align*}
    where $e_1,e_2,e_3 \in \Cl_{\mp 3}$ is the orthonormal basis of $\R^3$ and $f \in \Cl_{\pm 1}$ of $\R$.
\end{example}

As a vector space $\Cl(V,B)\cong \bigwedge^*V$ and so $\dim_\R(\Cl(V,B))=2^{\dim V}$.

Given a real vector space $V$ with nondegenerate symmetric bilinear form, let $V = V_+ \oplus V_-$ be its orthogonal decomposition into the positive definite and negative definite part of $B$.
The map $\id_{V_+} \oplus -\id_{V_-}$ extends uniquely to a $\C$-antilinear map on the complexification $V \otimes_\R \C$, which comes with a complex bilinear form $(.,.)$ which is $B$ complexified.
By the universal property, it extends uniquely to a $\C$-antilinear antihomomorphism $\Cl(V \otimes_\R \C) \to \Cl(V \otimes_\R \C)$.
This map is explicitly given by $(v_1 \dots v_k)^* = (-1)^{\# \{i: v_i \in V_-\}}\,\overline{v}_k \dots \overline{v}_1$ for homogeneous $v_i \in V_{\pm}$.
We see that the antihomomorphism squares to 1 and is grading-preserving.
It defines a complex graded $C^*$-algebra structure on $\Cxl_n$ and hence a real $C^*$-algebra structure on $\Cl(V)$ which corresponds to the real structure $v \mapsto \overline{v}$ on $\Cxl_n$ induced by the canonical real structure on $V \otimes_\R \C$.
This in particular makes $\Cl(V)$ into Banach algebras with norm
\begin{equation}
\label{Cliffordnorm}
\|v_1 \dots v_k\|^2 = \|(v_1 \dots v_k)^* v_1 \dots v_k \| = (v_1,\overline{v}_1)  \dots (v_k,\overline{v}_k).
\end{equation}
As special cases, we see that $\Cl_{+n}$ is a $C^*$-algebra using $v^* = v$ for $v \in \R^n$ and $v^* = -v$ in $\Cl_{-n}$.

\begin{example}
\label{ex:(1,1)per}
    Let $V$ be a finite-dimensional vector space and equip $V \oplus V^*$ with the symmetric hyperbolic form.
    The Fock space construction gives a canonical isomorphism $\mathcal{F} \colon \Cl(V \oplus V^*) \cong \End(\bigwedge V)$ of graded algebras.
    Here $\End(\bigwedge V)$ is graded using the graded vector space structure on $\bigwedge V$ given by even/odd number of wedges.
    Concretely it maps $v \in V \subseteq \Cl(V \oplus V^*)$ to the creation operator $a_v^\dagger = v \wedge (-)$ and $v^* \in V^*$ to the annihilation operator $i_{v^*}$.

    If $V$ is a complex Hilbert space, then $\End(\bigwedge V)$ is a $C^*$-algebra using the induced Hilbert space structure on $\bigwedge V$.
    The Hilbert space structure also induces a real structure on $V \oplus V^*$ and hence a $C^*$-algebra structure on $\Cl(V \oplus V^*)$.
    The even algebra homomorphism $\mathcal{F}$ is a $*$-homomorphism.

    Picking a basis $\{e_1, \dots, e_n\}$ of $V$ and using the dual basis on $V^*$, we see that the hyperbolic form has signature $(n,n)$.
    We therefore get an isomorphism $\Cl_{n,n} \cong \End(\R^{2^{n-1}|2^{n-1}})$ of real graded algebras.
\end{example}

Recall the $8$-periodicity of real Clifford algebras:
\begin{lemma}
    We have that $\Cl_{p + 8,q} \cong \Cl_{p,q + 8}$ are graded Morita equivalent to $\Cl_{p,q}$.
\end{lemma}
\begin{proof}
    Since $\Cl_{p + 8,q} \cong \Cl_{p,q} \grotimes \Cl_{8}$ it suffices to show that $\Cl_8$ and $\Cl_{-8}$ are Morita trivial.
    It follows from Example~\ref{ex:baby8per} and Example~\ref{ex:(1,1)per} that 
    \begin{equation*}
    \Cl_{4} \cong \Cl_1 \grotimes \Cl_3 \cong \Cl_1 \grotimes \Cl_{-1} \otimes \mathbb{H} \cong \Cl_{-4} \,,
    \end{equation*}
    and so the isomorphic graded algebras $\Cl_{-4}$ and $\Cl_{4}$ are Morita equivalent to $\mathbb{H}$.
    Therefore $\Cl_{\pm 8} = \Cl_{\pm 4} \grotimes \Cl_{\pm 4}$ is Morita equivalent to $\mathbb{H} \otimes \mathbb{H} \cong M_4(\R)$.
\end{proof}

We denote the inclusion homomorphisms 
\begin{equation}
\label{eq:i}
i_{\pm} \colon A \hookrightarrow A \grotimes \Cl_{\pm 1}.
\end{equation}
Actions by Clifford algebras with positive squares can be interpreted as grading operators:

\begin{proposition}
\label{prop:gradingvsclifford}
Let $A$ be a graded Banach algebra.
    There is a natural equivalence of (Banach) categories
    \begin{equation*}
    \ModEn^{C_2}_A\cong \ModEn_{|A\grotimes \Cl_{+1}|} \,,
    \end{equation*}
    which commutes with the functors to $\ModOne_{|A|}$ given by forgetting the grading and $i_+^*$.
\end{proposition}
\begin{proof}
There is a one-to-one correspondence between gradings of a Banach space $X$ and involutions $e \colon X \to X$.
Indeed, the grading corresponding to an involution is given by its eigenspace decomposition $X = X_+ \oplus X_-$, into closed subspaces $X_{\pm} = \ker (1 \pm e)$.
    Moreover, if $T \colon X \to Y$ is a bounded linear map between Banach spaces with involutions $e_X$ and $e_Y$, then $T$ is grading-preserving if and only if $Te_X = e_Y T$.

    Lifting a Banach $|A|$-module $M$ to a Banach $|A\grotimes \Cl_{+1}|$-module amounts to giving an involution $e \colon M \to M$ such that $ea = (-1)^{|a|} ae$ for all $a \in A$.
    This condition is equivalent to the grading $M = M_+ \oplus M_-$ corresponding to $e$ being an $A$-graded module.
    The notions of finitely generated projective correspond to each other in the constructed bijection.
\end{proof}
\begin{warning}
    The equivalence of Proposition~\ref{prop:gradingvsclifford} does not behave well with respect to graded tensor products.
\end{warning}
\begin{remark}
\label{rem:C2vsCliffnatural}
We record the functor $\ModOne_{|B \grotimes \Cl_{+1}|} \to \ModOne_{|A \grotimes \Cl_{+1}|}$ which corresponds under the isomorphism of Proposition~\ref{prop:gradingvsclifford} with tensoring with a graded $(A,B)$-bimodule $M$.
It is given by the ungraded tensor product $N \mapsto M \otimes_{|B|} N$ equipped with the diagonal $\Cl_1$-action $e(n \otimes m) = en \otimes em$.
Here we used the ungraded diagonal homomorphism
\begin{equation*}
|\Cl_1| \to |\Cl_1| \otimes |\Cl_1| \ncong |\Cl_1 \grotimes \Cl_1|
\end{equation*}
\end{remark}

Recall that $\sBanAlgEn$ is the symmetric monoidal topological category of Banach algebras, where the set $\Hom(A,B)$ of algebra homomorphisms is topologized with the $k$-ification of the compact-open topology.

\begin{proposition}
\label{prop:Clfunctor}
The Clifford algebra assignment defines a symmetric monoidal $\Top$-enriched functor
\begin{equation}
    \Cl\colon (\InnEn, \oplus) \longrightarrow (\sBanAlgEn, \grotimes) \,.
\end{equation}
\end{proposition}
\begin{proof}
The functor $\Cl$ on $\InnEn$ assigns to each inner product space $V$ the negative Clifford algebra $\Cl(V,-\langle ,\rangle)$.
For any isometric embedding $V\to W$ the universal property of $\Cl(V)$ induces an even homomorphism $\Cl(V) \to \Cl(W)$ explicitly given by $v_1 \dots v_k \mapsto T v_1 \dots T v_k$.
This assignment is functorial.

$\Hom_{\sBanAlgEn}(\Cl(V), \Cl(W))$ has the operator norm, since $\Cl(V)$ is finite-dimensional.
The map $\Hom_{\InnEn}(V,W) \to \Hom_{\sBanAlgEn}(\Cl(V), \Cl(W))$ is continuous for the operator norm from \eqref{Cliffordnorm} and from the fact that $T$ preserves $(.,.)$.

The obvious map $V \oplus W \to \Cl(V,B)\grotimes \Cl(W,C)$ satisfies the universal property and induces an isomorphism
    $\Cl(V\oplus W,B\oplus C)\cong \Cl(V,B)\grotimes \Cl(W,C)$.
    Under this isomorphism the flip map $V \oplus W \to W \oplus V$ corresponds to the Koszul braiding $a_1 \otimes a_2 \mapsto (-1)^{|a_1||a_2|} a_2 \otimes a_1$ for $a_1 \in \Cl(V,B)$ and $a_2 \in \Cl(W,C)$ by the anti-commutativity of elements in $V$ and $W$.
\end{proof}

\begin{example}
    The algebra $\Cl_{-2}$ is isomorphic to the quaternions as an ungraded algebra.
    Note that even though $\Cl_{-1} \grotimes \Cl_{-1} \cong \Cl_{-2}$ by Proposition~\ref{prop:Clfunctor}, we have $\Cl_{-1} \otimes \Cl_{-1} = \C \otimes \C \cong \C \oplus \C \ncong \mathbb{H}$.
    We conclude that the Koszul sign in the graded tensor product was of vital importance in Proposition~\ref{prop:Clfunctor}.
\end{example}

We also write $\Cl$ for the induced symmetric monoidal functor $\InnInf \to \sBanAlgInf$.
\begin{proposition}
    \label{prop:cliffordsmithideal}
    The arrow $\Cl_{-1} \to^{\op} \R$ in $\sBanAlgInf^{\op}$ refines to a Smith ideal.
\end{proposition}
\begin{proof}
    Use that $\Cl$ is symmetric monoidal and preserves Smith ideals. $\Cl_{-1} \to^{\op} \R$ is the image of the Smith ideal $\R^1 \to^{\op} \R^0$ from Corollary~\ref{cor:smithidealInn}.
\end{proof}

\subsection{Graded \texorpdfstring{$K$}{K}-theory}

\subsubsection{Connective graded \texorpdfstring{$K$}{K}-theory}
\label{sec:connectivesuperkthy}
This section will provide a definition of connective $K$-theory for graded algebras along the lines of Atiyah--Bott--Shapiro~\cite{ABS}.
The most obvious guess for the definition of the graded $K$-theory spectrum of a graded Banach algebra $A$ is to take the $K$-theory $k(\ModEn^{C_2}_A)=:k^{C_2}_A$ of the topologically enriched category of graded modules $\ModEn^{C_2}_A$.
This however does \emph{not} reproduce what has been called graded $K$-theory on homotopy groups; the graded $K$-theory also takes into account the $\sVectOne$-enrichment on $\ModEn^{C_2}_A$ through the parity shift.
For example, for an ungraded Banach algebra $A$, the spectrum $k^{C_2}_A \cong k_A \oplus k_A$ (Proposition~\ref{prop:kC_2}) does not reproduce the connective $K$-theory spectrum in the ungraded sense, whereas it does for our definition as we will see in Example~\ref{example:ungraded}.

As Atiyah, Bott and Shapiro observed for the case of $A = \Cl_{-n}$, the ``correct'' definition of the $K$-theory groups requires a further quotient of $k_A^{C_2}$ by those $A$-modules that extend to graded $A \grotimes \Cl_{-1}$-modules.
A subtle fact which is well-known among experts~\cite{MR2490588} is that this definition is no longer ``correct'' for arbitrary graded Banach algebras $A$, see Section~\ref{sec:cofib} for a further discussion.
We will lift this ``wrong'' definition to spectra and prove that it is only a $\pi_0$-issue, which disappears upon periodification.

\begin{definition}
\label{def:ABSKth}
    Recall the notation $i_-$ from \eqref{eq:i}.
    The \emph{ABS $K$-theory} of a graded Banach algebra $A$ is the connective spectrum
\begin{equation}
    k^{\ABS}_A := \cofib\left(k^{C_2}_{A\grotimes \Cl_{-1}} \xrightarrow{i_-^*} k^{C_2}_A\right) \,.
\end{equation}
\end{definition}
\begin{remark}
    In~\cite[Section~5]{ABS} (see also~\cite[Ch.~1, Sec.~9]{lawsonmichelsohn}), the groups
    \begin{equation*}
        K_0(\ModOne^{C_2}_{\Cl_{-n}})/i_-^*K_0(\ModOne^{C_2}_{\Cl_{-n-1}})
    \end{equation*}
    are introduced.
    They coincide with $\pi_0k^{\ABS}_{\Cl_{-n}}$.
\end{remark}
\begin{notation}
    The notations $\ModOne_A^{C_2}$ and $k_A^{C_2}$ are motivated by equivariant $K$-theory.
    Specifically, the category of graded algebras $A = A_0 \oplus A_1$ is equivalent to the category of algebras $A$ with a $C_2$-action $\phi \colon A \to A, \phi^2 = \id_A$.
    The functor in one direction is given by taking $\phi$ to be the grading homomorphism, and the functor in the other direction is given by taking the $\pm$-eigenvalue decomposition $A = A_0 \oplus A_1$ of $\phi$.
    The crossed product algebra $A \rtimes C_2$ is isomorphic to $A \grotimes \Cl_{+1}$~\cite[Proposition 14.5.4]{MR1656031}, and so $\ModOne_A^{C_2}$ can be identified with the category of $ A\rtimes C_2$-modules. 
    By the Green--Julg theorem~\cite[Section 11.7]{MR1656031} the spectrum $k_A^{C_2}$ can be identified with the $C_2$-equivariant $K$-theory spectrum of $A$.
\end{notation}
\begin{lemma}
    \label{lem:ABSpizero}
    The grouplike commutative monoid $k^{\ABS}_A$ can equivalently be realized as a cofiber in $\CMon(\Spc)$
    \begin{equation*}
        k^{\ABS}_A \simeq \cofib\left( N^{\mathrm{hc}}(\ModEn_{A\grotimes \Cl_{-1}}^{C_2,\cong}) \to N^{\mathrm{hc}}(\ModEn_A^{C_2,\cong}) \right) \,.
    \end{equation*}
    The abelian group $\pi_0 k^{\ABS}_A$ admits the following description as a quotient of commutative monoids.
    Its elements are equivalence classes $[M]$ of finitely generated projective $A$-modules equipped with the operation of direct sum and $M\oplus M'\sim 0$ if $M\oplus M'$ admits the structure of a $A\grotimes \Cl_{-1}$-module. 
    We have $[\Pi M]=-[M]$.
\end{lemma}

\begin{proof}
    Group completion commutes with cofibers.
    The first claim follows, once we show that the cofiber of commutative monoids on the right hand side is group complete.
    This is a $\pi_0$-condition, i.e.\ we have to show that $\pi_0$ of the cofiber is a group.
    $\pi_0\colon \CMon(\Spc) \to \CMon(\Set)$ maps cofibers to quotients.
    The quotient of commutative monoids $\pi_0\ModEn_A^{C_2} / \pi_0 \ModEn_{A\grotimes \Cl_{-1}}^{C_2}$ is exactly what is explicitly described in the Lemma.
    It suffices to show that the inverse of $[M]$ is $[\Pi M]$.
    This follows since $M\oplus \Pi M$ admits an $A\grotimes \Cl_{-1}$-module structure with odd generator $f$ acting by $\begin{pmatrix}
        & 1 \\
        -1
    \end{pmatrix}$ .
\end{proof}
\begin{example}
\label{example:ungraded}
    Let $A$ be an ungraded algebra considered as a purely even graded algebra.
    Then there are isomorphisms of Banach categories $\ModEn^{C_2}_A \cong (\ModEn_{|A|})^{\times 2}$ and an equivalence $\ModOne^{C_2}_{A\grotimes \Cl_{-1}}\cong \ModOne_{|A\grotimes \Cl_{1,1}|} \simeq \ModOne_{|A|}$, using that $|A\grotimes \Cl_{1,1}|\cong A\otimes M_2(\R) \simeq A$. The forgetful functor $i_-^*$ identifies with the diagonal functor, see Section~\ref{innergraded} for details.
    We obtain
    \begin{equation*}
        k^{\ABS}_A = \cofib(k_A\xrightarrow{\Delta} k_A^{\times 2}) \simeq k_A \,.
    \end{equation*}
    Explicitly: a graded $A$-module $M$ is equivalent to two ungraded $A$-modules $M_0$ and $M_1$, its even and odd parts.
    The module extends to a graded $A \grotimes \Cl_{-1}$-module if and only if $M_0 - M_1$ is zero in the Grothendieck group completion of the monoid $\pi_0 \ModOne_A$ of isomorphism classes of ungraded $A$-modules.
\end{example}

We now extend $k^{C_2}_A$ to a functor.
Informed by Section~\ref{sec:moritafunctoriality}, we can be quick in our definitions.
We let $\mathcal{K}^{C_2,\mathrm{alg}}\colon \sBimBanOne \to \Sp_{\geq 0}$ be the lax symmetric monoidal composition
\begin{equation*}
\sBimBanOne \xrightarrow{\ModOne^{C_2} = \Hom_{\sBimBanOne}(\mathbbm{1},-)} \Cat_1^{\Sigma} \xrightarrow{\mathcal{K}^{\mathrm{alg}}} \Sp_{\geq 0}.
\end{equation*}

\begin{definition}
    We let $k^{C_2}$ be the initial homotopy invariant functor with a natural transformation $\mathcal{K}^{C_2,\mathrm{alg}}\to k^{C_2}$, i.e.\ the left Kan extension of $\mathcal{K}^{C_2,\mathrm{alg}}$ along $\sBimBanOne\to \sBimBanInf$.
\end{definition}

Let $\ModEn^{C_2}_A(X)$ be the Banach category of (finitely generated projective) graded $A$-module bundles over a compact Hausdorff space $X$.

\begin{proposition}
    \label{prop:kC_2}
    The spectrum $k^{C_2}_A = k(\ModEn_A^{C_2})=(N^{\mathrm{hc}}(\ModEn^{C_2,\cong}_A),\oplus)^{\gp}$ enhances to a lax symmetric monoidal functor $k^{C_2}\colon \sBimBanInf \to \Sp_{\geq 0}$,
    can also be computed as 
    \begin{equation*}
        k^{C_2}_A = \colim_{\Delta^{\op}} \mathcal{K}^{C_2,\mathrm{alg}}_{C(\Delta^n;A)} \,,
    \end{equation*}
    and takes values in connective $ko$-modules.
    Moreover:
    \begin{enumerate}
        \item There is an equivalence $k^{C_2}_A \simeq k_{|A\grotimes \Cl_{+1}|}$, natural in graded Banach algebra homomorphisms.
        In particular, $k^{C_2}$ is homotopy invariant.
        \item Writing $k^{C_2}_A(X) := k\left(\ModEn^{C_2}_A(X)\right)$, there is a natural equivalence $k^{C_2}_A(X) \simeq k^{C_2}_{C(X;A)}$.
        \item The restriction $k^{C_2}\colon \sBanAlgOne \to \Sp_{\geq 0}$ is excisive and preserves filtered colimits along contractive homomorphisms.
        \item If $A$ is trivially graded, then $k^{C_2}_A \simeq k_{|A|}^{\times 2}$.
    \end{enumerate}
\end{proposition}
\begin{proof}
The localization functor $\sBimBanOne \to \sBimBanInf$ is symmetric monoidal as in Theorem~\ref{thm:localizationbimban}.
The functor $\ModOne^{C_2}$ is lax symmetric monoidal by Proposition~\ref{prop:daniel} and so $\mathcal{K}^{C_2,\mathrm{alg}}$ is as well.
The formula for $k^{C_2}_A$ follows as in Proposition~\ref{prop:compareK} and Theorem~\ref{thm:omnibusktheory}.
Agreement with the object level definition of $k^{C_2}_A$ is analogous to Theorem~\ref{th:Kthextension}. 
The $ko$-module structure comes from the $\VectEn_{\R}$-module structure on $\ModEn^{C_2}_A$.

The first point in the list is Proposition~\ref{prop:gradingvsclifford} applied to the definition of $k^{C_2}$; naturality along bimodules rather than homomorphisms is described by Remark~\ref{rem:C2vsCliffnatural}.
Since $C(X;|A\grotimes \Cl_{+1}|)=|C(X;A)\grotimes \Cl_{+1}|$, the second part follows from the first and the Serre--Swan theorem, Proposition~\ref{prop:serreswan}.
For excision, note that $-\grotimes \Cl_{+1}$ preserves Milnor squares, since it preserves pullbacks and surjections, so the claim follows from the first part and Corollary~\ref{cor:excision}.
Similarly, $A\mapsto |A\grotimes \Cl_{+1}|$ preserves filtered colimits and contractive homomorphisms, so preservation of filtered colimits follows from the first part and Proposition~\ref{prop:filteredcolimits}.
The last part is the isomorphism $\ModEn^{C_2}_A\cong (\ModEn_{|A|})^{\times 2}$ of Example~\ref{example:ungraded}.
\end{proof}

\begin{corollary}
The spectrum
    $k^{\ABS}_A$ enhances to a functor $k^{\ABS} \colon \sBimBanInf \to \Sp_{\geq 0}$ with
    \begin{equation*}
        k^{\ABS}=\cofib(k^{C_2}_{-\grotimes \Cl_{-1}} \to k^{C_2}) \,.
    \end{equation*}
\end{corollary}
Even though $k^{C_2}$ is lax symmetric monoidal, this does not obviously follow for $k^{\ABS}$.

\paragraph{Lax monoidality}
Lax monoidality of the functor $k^{\ABS}$ is furnished by the Smith ideal of Proposition~\ref{prop:cliffordsmithideal}.
To see why this is needed, consider the problem of supplying structure maps 
\begin{equation}
\label{eq:ABSlax}
    k^{\ABS}_A\otimes k^{\ABS}_B \to k^{\ABS}_{A\grotimes B}.
\end{equation}
Since tensor product in $\Sp_{\geq 0}$ preserves colimits, it maps the tensor of cofibers $k^{\ABS}_A$ to an iterated (total) cofiber.
This can be computed as the cofiber of a pushout.
The lax monoidality data \eqref{eq:ABSlax} thus amounts to providing a map
\begin{equation*}
    \begin{tikzcd}
    \cofib\left( k^{C_2}_A \otimes k^{C_2}_{B\grotimes \Cl_{-1}} \sqcup_{k^{C_2}_{A\grotimes \Cl_{-1}} \otimes k^{C_2}_{B\grotimes \Cl_{-1}}} k^{C_2}_{A\grotimes \Cl_{-1}} \otimes k^{C_2}_{B} \to k^{C_2}_A \otimes k^{C_2}_{B}\right) \ar[d]\\
    \cofib\left(k^{C_2}_{A\grotimes B\grotimes \Cl_{-1}} \to k^{C_2}_{A\grotimes B}\right) 
    \end{tikzcd} \,.
\end{equation*}
The map that immediately comes to mind is the tensor product $\grotimes$ on the level of cofiber diagrams, which is inherited from $k^{C_2}$, but by the universal property of the pushout one then needs to provide a homotopy witnessing the commutativity of
\begin{equation}
    \label{eq:monoidalitydiagram}
    \begin{tikzcd}
        k^{C_2}_{A\grotimes \Cl_{-1}} \otimes k^{C_2}_{B\grotimes \Cl_{-1}} \ar[d,"i_-^*\times \id"]\ar[r,"\id\times i_-^*"]& k^{C_2}_{A\grotimes \Cl_{-1}} \otimes k^{C_2}_{B} \ar[d,"\grotimes"]\\
    k^{C_2}_{A} \otimes k^{C_2}_{B\grotimes \Cl_{-1}} \ar[r,"\grotimes"] & k^{C_2}_{A\grotimes B\grotimes \Cl_{-1}}
    \end{tikzcd} \,.
\end{equation}
On the level of module categories, the two composites admit the following description.
Let $(M,f_1)$ and $(N,f_2)$ be an $A\grotimes \Cl_{-1}$-module and a $B\grotimes \Cl_{-1}$-module, respectively.
We choose to write the $\Cl_{-1}$-generator $f_1$ with $f_1^2=-1$ as a graded $A$-linear odd operator on the $A$-module $M$. 
We obtain two different $A\grotimes B\grotimes \Cl_{-1}$-module structures on $M\grotimes N$ with the two $\Cl_{-1}$-actions given by
$f_1\grotimes 1$ and $1\grotimes f_2$, respectively.
There is a homotopy witnessing the commutativity of \eqref{eq:monoidalitydiagram}:
\begin{equation}
\label{eq:essentialhomotopy}
    f_t= \cos(t) f_1\grotimes 1+ \sin(t)1\grotimes f_2\,, \quad t\in[0,\pi/2]\,.
\end{equation}
This defines a family of $\Cl_{-1}$-actions since $f_1\grotimes 1,1\grotimes f_2$ anticommute.
So by construction they define equivalent morphisms in $\BanAlgInf$ and hence in $\BimBanInf$, so that we obtain the desired monoidality data $k^{\ABS}_A\otimes k^{\ABS}_B \to k^{\ABS}_{A\grotimes B}$.

However, to show lax monoidality of the whole functor $k^{\ABS}$, we still need to show that the homotopy \eqref{eq:essentialhomotopy} can be assigned in a functorial and coherent way witnessing all higher associativities.

\begin{theorem}
    \label{thm:abslaxmonoidal}
    The Atiyah--Bott--Shapiro $K$-theory functor $k^{\ABS}\colon \sBimBanInf\to \Sp_{\geq 0}$ admits a lax symmetric monoidal refinement such that $k^{C_2}\to k^{\ABS}$ is a monoidal natural transformation.

    The restriction to ungraded algebras along $\BimBanInf \to \sBimBanInf$ can naturally be identified with the lax monoidal connective topological $K$-theory functor $k$ from Definition~\ref{def:ungradedBimBanKth}.
    \begin{equation*}
    \begin{tikzcd}
        \BimBanInf \ar[rr,"k"] \ar[rd] & & \Sp_{\geq 0} \\
        & \sBimBanInf \ar[ru,"k^{\ABS}"']&
    \end{tikzcd}
    \end{equation*}
\end{theorem}

\begin{proof}
    By Proposition~\ref{prop:cliffordsmithideal} we can endow $i_-^{\op}\colon \Cl_{-1}\to \R$ with the structure of a Smith ideal, i.e.\ it can be upgraded to a lax symmetric monoidal functor $[1]_{\mathrm{min}}\to \sBanAlgInf^{\mathrm{fgp},\op}$.
    By using wrong-way functoriality, we may view this as a lax symmetric monoidal functor to the $(\infty,1)$-category $\sBimBanInf$.
    Using Lemma~\ref{prop:kC_2}, we obtain a lax monoidal composite
    \begin{equation*}
        \begin{tikzcd}
            \left[1\right]_{\mathrm{min}} \times \sBimBanInf \ar[r]& \sBimBanInf \times \sBimBanInf \ar[r,"\grotimes"]& \sBimBanInf \ar[r,"k^{C_2}"] & \Sp \,,
        \end{tikzcd} \,,
    \end{equation*}
    with underlying functor $\sBimBanInf\to \Sp^{[1]}$ given by $A \mapsto (k^{C_2}_{A\grotimes \Cl_{-1}} \to k^{C_2}_A)$.
    We conclude using Corollary~\ref{cor:cofiberlaxmon}.

    There is a symmetric monoidal functor $\BimBanOne \to \sBimBanOne$ that after $\Hom(\mathbbm{1},-)$ gives the functor $\ModOne_A\to \ModOne^{C_2}_A$ regarding an $A$-module as a purely even graded module.
    These assemble into a symmetric monoidal natural transformation of functors $\BimBanOne\to \Cat_1^{\Sigma}$.
    This induces a natural transformation $\mathcal{K}^{\mathrm{alg}}\to \mathcal{K}^{C_2,\mathrm{alg}}$ between functors on $\BimBanOne$, 
    which descends to a monoidal natural transformation $k\to k^{C_2}$ of functors on $\BimBanInf$ and further to a map $k\to k^{C_2}\to k^{\ABS}$.
    This map is given by
    \begin{equation*}
        k_A \to \cofib \left( k^{C_2}_{A\grotimes \Cl_{-1}} \to k^{C_2}_A \right) 
    \end{equation*}
    on objects which was shown to be an equivalence in Example~\ref{example:ungraded}.
\end{proof}

The lax monoidality provides a $ko=k^{\ABS}_{\R}$-action on $k^{\ABS}_A$ and so---by $\R$-linearity of all maps involved---Theorem \ref{thm:abslaxmonoidal} enhances to a lax monoidal functor to $ko$-modules.

\subsubsection{Bott periodicity and the Clifford suspension isomorphism}
\label{sec:bott}
    We will now explain the connective version of Bott periodicity in the setting of graded Banach algebras.
    Below we will construct a specific element $\alpha\in \pi_1k^{\ABS}_{\Cl_{+1}}$ only depending on a choice of one of the two Morita equivalences $\Cl_{1,1} \Morita \R$.
    It will turn out that $\pi_1k^{\ABS}_{\Cl_{+1}}\cong\Z$, and $\alpha$ is a generator.
    Note that $\alpha$ is equivalent to a $ko=k^{\ABS}_{\R}$-module map
    \begin{equation*}
        \alpha\colon ko \longrightarrow \Omega k^{\ABS}_{\Cl_{+1}} \,,
    \end{equation*}
    where we choose to write $\Omega=\tau_{\geq 0}\Sigma^{-1}$ since we are working in connective spectra.
    Using the lax symmetric monoidal structure of $k^{\ABS}$, we obtain a map
    \begin{equation*}
        \overline{\alpha}\colon k^{\ABS}_A \simeq k^{\ABS}_A\otimes_{ko} ko \xrightarrow{\alpha} k^{\ABS}_A \otimes_{ko} \Omega k^{\ABS}_{\Cl_{+1}} \to \Omega k^{\ABS}_{A\grotimes \Cl_{+1}} \,,
    \end{equation*}
    that is natural in $A$.
    Moreover, $\overline{\alpha}$ is multiplicative, in the sense that there are commutative diagrams
    \begin{equation*}
        \begin{tikzcd}
            k^{\ABS}_A \otimes k^{\ABS}_B \ar[d,"\id\otimes \overline{\alpha}_B"]\ar[r]& k^{\ABS}_{A\grotimes B} \ar[d,"\overline{\alpha}_{A\grotimes B}"] \\
            k^{\ABS}_A \otimes \Omega k^{\ABS}_{B\grotimes \Cl_{+1}} \ar[r]& \Omega k^{\ABS}_{A\grotimes B \grotimes \Cl_{+1}}
        \end{tikzcd} \,.
    \end{equation*}
    Succinctly:
    \begin{proposition}
        $\overline{\alpha}\colon k^{\ABS}_{-}\to \Omega k^{\ABS}_{-\grotimes \Cl_{+1}}$ is a map of $k^{\ABS}$-modules in $\Fun(\sBimBanInf,\Sp_{\geq 0})$ with the Day convolution symmetric monoidal structure.
    \end{proposition}
    We summarize the key properties of $\overline{\alpha}$.
    \begin{theorem}
        \label{thm:gradedbottperiodicity}
        For every real graded Banach algebra $A$, the map $\overline{\alpha}_A\colon k^{\ABS}_A\to \Omega k_{A\grotimes \Cl_{+1}}^{\ABS}$ induces an isomorphism on all homotopy groups $\pi_i$ , $i\geq 1$.
    \end{theorem}
    The above theorem is a version of strong Bott periodicity. We will prove it in Section~\ref{sec:proofs} by reducing the statement to Karoubi's version of Bott periodicity.
    We will discuss the failure to be an isomorphism on $\pi_0$ in Section~\ref{sec:cofib}.

     We write $\overline{\alpha}_A^{(n)}$ for the composition 
    \begin{equation*}
        k^{\ABS}_A \xrightarrow{\overline{\alpha}} \Omega k^{\ABS}_{A\grotimes \Cl_{+1}} \xrightarrow{\Omega\overline{\alpha}} \Omega^2 k^{\ABS}_{A\grotimes \Cl_{+2}} \dots \xrightarrow{\Omega^{n-1} \overline{\alpha}} \Omega^n k^{\ABS}_{A\grotimes \Cl_{+n}} \,.
    \end{equation*}
    Pick a Morita equivalence $\Cl_{+4} \Morita \mathbb{H}$ which gives a map
        \begin{equation}
        \label{eq:ksp}
            \overline{\alpha}^{(4)}\colon ko \to \Omega^4k^{\ABS}_{\Cl_{+4}} \simeq \Omega^4k_{\mathbb{H}}
        \end{equation}
        relating to the $ko$-module $k_{\mathbb{H}}$ also known as $ksp$.
        The square of $\Cl_{+4} \Morita \mathbb{H}$ gives a Morita equivalence 
        \begin{equation}
        \label{eq:moritaCl8}
        \Cl_{+8} \cong \Cl_{+4} \grotimes \Cl_{+4} \Morita \mathbb{H} \otimes \mathbb{H} \cong M_4(\R) \Morita \R
        \end{equation}
        and hence a map
        \begin{equation}
        \label{eq:alphabott}
            \overline{\alpha}^{(8)}\colon ko \to \Omega^8k^{\ABS}_{\Cl_{+8}} \simeq \Omega^8 ko
        \end{equation}
        The maps \eqref{eq:ksp} and \eqref{eq:alphabott} are equivalences:
    
    \begin{corollary}
        \label{cor:ktheoryomegaspectrum}
        The connective $K$-theory of real Clifford algebras provides a delooping of $ko$:
        \begin{equation*}
            \overline{\alpha}^{(n)}\colon ko \xrightarrow{\simeq } \Omega^n k^{\ABS}_{\Cl_{+n}} \,.
        \end{equation*}
        The equivalence \eqref{eq:alphabott} is given by multiplication by the Bott class $\beta\in \pi_8ko$.
    \end{corollary}
    \begin{proof}
        $\overline{\alpha}_\R$ is an equivalence by Lemma~\ref{lem:alphaequivalence} and $\Omega^n\overline{\alpha}_A$ is always an equivalence for $n\geq 1$ by Theorem~\ref{thm:gradedbottperiodicity}.
        It must also map to a generator of $\pi_8ko$.
        Now $\beta$ is the preferred generator of $\beta \in \pi_8ko\cong \Z$ with the property that $4\beta = \alpha^2$ where $\alpha \in \pi_4ko$ is a generator~\cite[Section 2.3.D]{MR2723113}.
        Since we picked the choice of Morita equivalence \eqref{eq:moritaCl8} to be the square of a Morita equivalence $\Cl_{+4} \Morita \mathbb{H}$ and $\mathbb{H} \otimes_\R \mathbb{H} \cong M_4(\R) \Morita \R$ is an ungraded Morita equivalence, we took this canonical choice.\footnote{The choice of $V$ in the definition of $\alpha$~\eqref{eq:alphadef} does change the sign of $\alpha$ but not of $\beta$ because only its eighth root enters.}
    \end{proof}
    \begin{corollary}
        \label{cor:kuennethforclifford}
        For any graded Banach algebra, the monoidality constraint
        \begin{equation*}
            k^{\ABS}_{A}\otimes_{ko} k^{\ABS}_{\Cl_{-n}} \to k^{\ABS}_{A \grotimes \Cl_{-n}}
        \end{equation*}
        is an equivalence.
    \end{corollary}
    \begin{proof}
        By naturality of $\overline{\alpha}^{(n)}$ and the Morita equivalence $\Cl_{n,n}\Morita \R$ we have the commutative diagram
        \begin{equation*}
        \begin{tikzcd}[column sep = 40]
              k_{A}^{\ABS} \otimes_{ko} k_{\Cl_{-n}}^{\ABS} \ar[r,"\id\otimes\overline{\alpha}^{(n)}","\simeq"'] \ar[d]& k_{A}^{\ABS} \otimes_{ko} \Omega^n k_{\Cl_{n,n}}^{\ABS} \ar[d]\ar[r,"\simeq"] &k_{A}^{\ABS} \otimes_{ko} \Omega^n ko \ar[d,"\simeq"]\\
              k_{A\grotimes \Cl_{-n}}^{\ABS} \ar[r,"\overline{\alpha}^{(n)}","\simeq"'] & \Omega^n k^{\ABS}_{A\grotimes \Cl_{n,n}}  \ar[r,"\simeq"]& \Omega^n k^{\ABS}_{A} 
        \end{tikzcd} \,.
        \end{equation*}
        We see that the leftmost vertical map is an isomorphism.
    \end{proof}
    \subsubsection{Periodic graded \texorpdfstring{$K$}{K}-theory}
    \label{sec:periodickthy}
    \begin{definition}
        \label{def:periodick}
        The \emph{periodic graded $K$-theory of $A$} is the spectrum
        \begin{equation*}
            K^{\mathrm{gr}}_A :=  \colim_{n} \left(k^{\ABS}_A\xrightarrow{\overline{\alpha}} \Sigma^{-1}k^{\ABS}_{A\grotimes \Cl_{+1}} \xrightarrow{\Sigma^{-1} \overline{\alpha}}  \Sigma^{-2}k^{\ABS}_{A\grotimes \Cl_{+2}} \dots \right) \,.
        \end{equation*}
    \end{definition}
    It follows by Theorem~\ref{thm:gradedbottperiodicity} that $\Sigma^{-k} \overline{\alpha}$ is an isomorphism on $\pi_{i}$ for $i \geq -k+1$.
    In particular, the connective cover of the periodic graded $K$-theory is
    \begin{equation*}
        \tau_{\geq 0}K^{\mathrm{gr}}_A = \Omega k^{\ABS}_{A\grotimes \Cl_{+1}}  \,,
    \end{equation*}
    and it receives a map from $k^{\ABS}_A$ which is an equivalence on $\pi_i$ for $i\geq 1$.
    We study the map on $\pi_0$ in Section~\ref{sec:cofib}.

Applying Corollary~\ref{cor:ktheoryomegaspectrum}, we obtain $K^{\mathrm{gr}}_A = k_A^{\ABS}[\beta^{-1}]$; it follows by naturality in $A$ that $\alpha^8_A$ is given by acting by $\beta \in \pi_8 ko$.
It therefore also follows by commuting colimits that
\begin{equation*}
K^{\mathrm{gr}}_A = \cofib(k^{C_2}_{A \grotimes \Cl_{-1}} \to k_A^{C_2})[\beta^{-1}] = \cofib(k_{A \grotimes \Cl_{-1}}^{C_2}[\beta^{-1}] \to k_A^{C_2}[\beta^{-1}]) = \cofib(K_{A \grotimes \Cl_{-1}}^{C_2} \to K_A^{C_2}),
\end{equation*}
where $K_A^{C_2} := k_A^{C_2}[\beta^{-1}]$; recall from Proposition~\ref{prop:kC_2} that $k^{C_2}_A$ is a $ko$-module.
    \begin{remark}
        One can rephrase the definition of $K^{\mathrm{gr}}_A$ as the explicit $\Omega$-spectrum with connecting maps $\overline{\alpha}$
        \begin{equation*}
            (\Omega^{\infty+1} k^{\ABS}_{A\grotimes \Cl_{+1}},\dots,\Omega^{\infty+n+1}k^{\ABS}_{A\grotimes \Cl_{n+1}},\dots)\,.
        \end{equation*}
    \end{remark}

    \begin{theorem}
        \label{thm:periodicKtheorylax}
        Periodic graded $K$-theory admits a natural refinement to a lax symmetric monoidal functor such that $k^{\ABS}\to K^{\mathrm{gr}}$ is a monoidal natural transformation.
        It refines to a lax symmetric monoidal functor from the Morita $(\infty,1)$-category of graded Banach algebras and bimodules to $KO$-module spectra
        \begin{equation*}
            K^{\mathrm{gr}}\colon \sBimBanInf \longrightarrow \ModInf(KO) \,.
        \end{equation*}
        $K^{\mathrm{gr}}$ is homotopy invariant.
        The restrictions $k^{\ABS}, K^{\mathrm{gr}}\colon \sBanAlgInf \to \ModInf(KO)$ preserve filtered colimits along contractive homomorphisms and $K^{\mathrm{gr}}$ is excisive.
    \end{theorem}
    Recall that a functor $F\colon \sBanAlgOne\to \calC$ is excisive if it maps Milnor squares in the sense of Definition~\ref{def:excisive} to pullback squares.
    \begin{proof}
        Regarding $ko\in \CAlg(\Sp)$, we realize $\ModInf(KO)$ as the full subcategory of $\ModInf(ko)$ on the modules where $\beta$ acts invertibly.
        $M\mapsto M[\beta^{-1}]$ provides a left adjoint to the inclusion and is a smashing localization, i.e.\ $M[\beta^{-1}]=M\otimes_{ko}ko[\beta^{-1}]$. 
        Using the results in~\cite[Section~3]{Gepner2015}, this equips $KO=ko[\beta^{-1}]$ with a $\E_\infty$-multiplication such that $(-)[\beta^{-1}]:\ModInf(ko)\to \ModInf(KO)$ is symmetric monoidal.
        Now we need to exercise some care with connectivity.
        By definition, $K^{\mathrm{gr}}$ is the composition
        \begin{equation*}
            \begin{tikzcd}
                \Bim(\sBan)\ar[r,"k^{ABS}"]& \ModInf_{\Sp_{\geq 0}}(ko) \ar[r] & \ModInf(KO)
            \end{tikzcd} \,,
        \end{equation*}
        where the second map factors as
        \begin{equation}
            \label{eq:invertingbottconnective}
                        \ModInf_{\Sp_{\geq 0}}(ko) \hookrightarrow \ModInf(ko) \xrightarrow{-\otimes_{ko} KO} \ModInf(KO) \,.
        \end{equation}
        We claim that the composite \eqref{eq:invertingbottconnective} preserves finite limits.
        Let $M\to N\to L$ be a fiber sequence in $\ModInf_{\Sp_{\geq 0}}(ko)$.
        Then $M=\tau_{\geq 0}M'$, where $M'\to N\to L$ is a fiber sequence in $\ModInf(ko)$.
        It suffices to show that $M\otimes_{ko}KO\simeq M'\otimes_{ko}KO$.
        Since $-\otimes_{ko}KO$ is exact, it suffices to show that it annihilates coconnective $ko$-modules.
        This follows since $\pi_*(N\otimes_{ko}KO) \cong \colim_{k} \pi_{*+8k}N$.
        Then $K^{\mathrm{gr}}$ inherits excision from $\Omega k^{\ABS}_{-\grotimes \Cl_{+1}}$, which is excisive by Lemma~\ref{lem:loopsofabsexcisive}.
        Preservation of filtered colimits along contractive homomorphisms holds for $k^{C_2}$ by Proposition~\ref{prop:kC_2}, hence for $k^{\ABS}$ since cofibers commute with filtered colimits, and hence for $K^{\mathrm{gr}}$ since localizing at $\beta$ is itself a filtered colimit.
    \end{proof}
    \begin{proposition}
        There is a unique reduced and exact extension $A \mapsto \fib(K^{\mathrm{gr}}_{A_+} \to K^{\mathrm{gr}}_\R)$ of $K^{\mathrm{gr}} \colon \sBanAlgOne \to \Sp$ to nonunital graded Banach algebras.
        This extension also preserves filtered colimits along contractive homomorphisms.
    \end{proposition}
    \begin{proof}
        The proof is analogous to Lemma~\ref{lem:reducedextension}.
    \end{proof}
    \begin{proposition}
    \label{prop:clifshift}
    For every graded Banach algebra $A$, the generator $\alpha\in \pi_1k^{\ABS}_{\Cl_{+1}}$ induces isomorphisms of $KO$-modules
        \begin{equation}
            K^{\mathrm{gr}}_{A\grotimes \Cl_{p,q}} \simeq \Sigma^{p-q} K^{\mathrm{gr}}_A \,.
        \end{equation}
\end{proposition}
\begin{proof}
It follows essentially by Definition~\ref{def:periodick} that $K^{\mathrm{gr}}_{A \grotimes \Cl_{p,q}} \cong \Sigma^{p-p'} K^{\mathrm{gr}}_{A \grotimes \Cl_{p',q}}$ for all $p,p',q$.
    Therefore we can assume that $p \geq q$ without loss of generality.
    The result now follows since $K^{\mathrm{gr}}_{A \grotimes \Cl_{k,k}} \cong K^{\mathrm{gr}}_A$ by graded Morita invariance.
\end{proof}

    We collect the properties of graded $K$-theory established so far.
    \begin{theorem}[Theorem B]
        \label{thm:omnibusgradedktheory}
        The functors $k^{\ABS}$ and $K^{\mathrm{gr}}$ refine to lax symmetric monoidal functors
        \begin{equation*}
            k^{\ABS}\colon \sBimBanInf \to \ModInf(ko) \,, \qquad K^{\mathrm{gr}}\colon \sBimBanInf \to \ModInf(KO) \,,
        \end{equation*}
        such that the natural transformations $k^{C_2}\to k^{\ABS}\to K^{\mathrm{gr}}$ are monoidal.
        The restriction of $k^{\ABS}$ along $\BimBanInf \to \sBimBanInf$ is naturally identified with the connective topological $K$-theory functor $k$ of Theorem~\ref{thm:omnibusktheory}.
        Both functors are homotopy invariant and their restrictions to $\sBanAlgOne$ preserve filtered colimits along contractive homomorphisms; the restriction of $K^{\mathrm{gr}}$ to $\sBanAlgInf$ is moreover excisive.
        Finally, there is a natural equivalence of $KO$-modules
        \begin{equation*}
            K^{\mathrm{gr}}_{A\grotimes \Cl_{p,q}} \simeq \Sigma^{p-q} K^{\mathrm{gr}}_A \,.
        \end{equation*}
    \end{theorem}
    \begin{proof}
        Lax symmetric monoidality of $k^{\ABS}$, monoidality of $k^{C_2}\to k^{\ABS}$ and the identification of the ungraded restriction are Theorem~\ref{thm:abslaxmonoidal}.
        The corresponding statements for $K^{\mathrm{gr}}$, together with homotopy invariance, excision and preservation of filtered colimits, are Theorem~\ref{thm:periodicKtheorylax}. Homotopy invariance and preservation of filtered colimits for $k^{\ABS}$ also follow from that theorem.
        The last equivalence is Proposition~\ref{prop:clifshift}.
    \end{proof}

    \begin{corollary}
        \label{cor:strongmonoidalfd}
        Periodic graded $K$-theory defines a (strong) symmetric monoidal functor
        \begin{equation*}
            K^{\mathrm{gr}} \colon \Mor(\sVectOne^{\mathrm{fd}}_\R) \to \ModInf(KO)
        \end{equation*}
        from the $(2,1)$-category of finite-dimensional semisimple graded algebras.
    \end{corollary}
    \begin{proof}
        The existence of a lax monoidal functor follows by composing the symmetric monoidal functor $\Mor(\sVectOne^{\mathrm{fd}}_\R) \to \sBimBanOne$ of Remark~\ref{rem:semsim} with $\sBimBanOne \to \sBimBanInf$ and Theorem~\ref{thm:periodicKtheorylax}.
        Strong monoidality is a special case of Theorem~\ref{cor:periodicKunneth} together with $KO \cong K^{\mathrm{gr}}_{\R}$.
    \end{proof}
    \begin{proposition}
        If $X$ is a compact Hausdorff space and homotopy finitely dominated, there is an equivalence 
        \begin{equation*}
        \underline{\Hom}_{\Sp}(\Sigma^\infty_+ X, K^{\mathrm{gr}}_A) \cong K^{\mathrm{gr}}_{C(X;A)}
        \end{equation*}
    \end{proposition}
    \begin{proof}
        Follows as in Corollary~\ref{cor:homotopicalSwan} by excisiveness.
    \end{proof}
    All of the above results have analogues over complex Banach algebras which can now be phrased as corollaries after observing the complex analogue of Corollary~\ref{cor:ktheoryomegaspectrum} with 2-fold periodicity of complex Clifford algebras.
    We immediately get:
    \begin{theorem}
        Restricting the lax symmetric monoidal functors $k^{\ABS},K^{\mathrm{gr}}$ to complex graded Banach algebras, we obtain lax symmetric monoidal functors
        \begin{equation*}
            k^{\ABS}\colon \Bim(\sBan_{\C})_\infty \to \ModInf(ku) \,, \qquad K^{\mathrm{gr}}\colon \Bim(\sBan_\C)_\infty \to \ModInf(KU)\,.
        \end{equation*}
        Both functors are homotopy invariant and preserve filtered colimits.
        $K^{\mathrm{gr}}$ is excisive.
        Let $A$ be a complex Banach algebra. Then:
        \begin{enumerate}
            \item $k^{\ABS}_{A\grotimes_{\C} \Cxl_{+1}}\simeq k^{\ABS}_{A\grotimes \Cl_{+1}}\to \Omega k^{\ABS}_A$ is a $\pi_{\geq 1}$-equivalence.
            \item $K^{\mathrm{gr}}_{A\grotimes_{\C} \Cxl_{p,q}} \simeq \Sigma^{q-p}K^{\mathrm{gr}}_A$ and so $K_A^{\mathrm{gr}}$ is $2$-periodic.
            \item $K^{\mathrm{gr}}_A=k^{\ABS}_A[u^{-1}]$, where $u\in \pi_2ku$ is the complex Bott class.
        \end{enumerate}
    \end{theorem}
    \subsection{The Picard spectra of graded Banach algebras and \texorpdfstring{$KO$}{KO}-theory}
    \label{sec:picard}
    We denote by $\sBimBanInf^{\mathrm{fd}}\subseteq \sBimBanInf$ the full symmetric monoidal subcategory on the finite-dimensional semisimple graded Banach algebras.
    Recall that an object $c$ in a symmetric monoidal $(\infty,1)$-category $(\calC,\otimes)$ is \emph{$\otimes$-invertible} if it is dualizable and the evaluation and coevaluation maps $c^\vee \otimes c\to \mathbbm{1}$ and $\mathbbm{1}\to c\otimes c^\vee$ are equivalences.
    \begin{definition}
        \label{def:picard}
        The \emph{Picard $\infty$-groupoid} $\Pic(\calC)$ of a symmetric monoidal $(\infty,1)$-category $(\calC,\otimes)$ is the maximal subgroupoid of $\calC$ on the $\otimes$-invertible objects.
        \end{definition}

    $\Pic(\calC)$ is a grouplike $\E_\infty$-monoid and hence a connective spectrum.

    \begin{theorem}
        \label{thm:picardgroupoid}
        The restriction of $K^{\mathrm{gr}}$ to the Picard $\infty$-groupoid of $\grotimes$-invertible finite-dimensional semisimple graded algebras is a symmetric monoidal functor
            \begin{equation*}
                \Pic(\Bim(\sBan_{\R})^{\mathrm{fd}}) \longrightarrow \Pic(\ModInf(KO)) \,.
            \end{equation*}
        The map is an isomorphism on $\pi_0,\pi_1,\pi_2$ and $\Pic(\Bim(\sBan_{\R})^{\mathrm{fd}})$ is $2$-coconnective.
        The inclusion $\Pic(\Mor(\sVectOne))\to \Pic(\Bim(\sBan_{\R})^{\mathrm{fd}})$ is an equivalence.
        Hence, $K^{\mathrm{gr}}$ witnesses a splitting of Picard $\infty$-groupoids:
        \begin{equation}
            \label{eq:splitting}
            \Pic(\ModInf(KO)) \simeq \Pic(\Bim(\sBan_{\R})^{\mathrm{fd}}) \times \tau_{\geq 3}\Pic(\ModInf(KO))\,.
        \end{equation}
    \end{theorem}
        The $2$-type $\Pic(\Bim(\sBan_{\R})^{\mathrm{fd}})$ has the following description:
        The equivalence classes of objects are generated by the Clifford algebras:
        \begin{equation*}
            \pi_0\Pic(\Bim(\sBan_{\R})^{\mathrm{fd}})=\{\Cl_0,\dots,\Cl_7\}\cong \Z/8 \,.
        \end{equation*}
        The loop space $\Omega_{\mathbbm{1}}(\Pic(\Bim(\sBan_{\R})^{\mathrm{fd}}))$ is the $1$-type $\sLine_{\R}$ of real graded lines with
        \begin{align*}
            \pi_1\Pic(\Bim(\sBan_{\R})^{\mathrm{fd}})&=\pi_0 N^{\mathrm{hc}}(\sLine_{\R}^{\cong})=\{\R, \Pi \R\} \cong \Z/2 \\
            \pi_2\Pic(\Bim(\sBan_{\R})^{\mathrm{fd}})&=\pi_1 N^{\mathrm{hc}}(\sLine_{\R}^{\cong})=\pi_0\frac{\{\text{inv. linear isomorphisms } \R\to \R\}}{\text{homotopy}}\cong\Z/2 \,.
        \end{align*}

        \begin{remark}
            In ordinary Morita bicategories every algebra is dualizable with dual $A^{\op}$, where the evaluation is given by $A$ regarded as a morphism $\ev\colon A^{\op}\otimes A\rightsquigarrow k$, i.e.\ as a vector space with a right $A^{\op}\otimes A$-action.
            Usually, only the separable algebras are \emph{fully} dualizable, which means that the evaluation and coevaluation have adjoints.
            An algebra is \emph{separable} if it is projective as an $A\grotimes A^{\op}$-module.
            In our Morita $(\infty,1)$-category we have restricted the $1$-morphisms and so not all objects are dualizable.
            
            Still, for all semisimple finite-dimensional graded Banach algebras $A$ the opposite algebra $A^{\op}$ provides a dual in $\sBimBanOne$ and hence also in $\sBimBanInf$: finite-dimensional implies that $\ev_A$ is a well-defined morphism in $\sBimBanOne$.
            Semisimple implies separable (analogous to the ungraded case \cite[Corollary 11.12]{MR2894798}) and hence $\coev_A$ is a morphism in $\sBimBanOne$.
        \end{remark}
    \begin{proof}
        $K^{\mathrm{gr}}$ restricts to a symmetric monoidal functor on $\sBimBanInf^{\mathrm{fd}}$ by Theorem~\ref{cor:periodicKunneth}.
        We have argued that finite-dimensional semisimple graded algebras are dualizable.
        By Wedderburn--Artin~\ref{prop:wedderburn}, any such dualizable $A$ is Morita equivalent to a finite product of graded division algebras $A\Morita \prod D_i$ with dual $A^{\op}\Morita \prod_{i=1}^n D_i^{\op}$.
        In order to be invertible, the evaluation and coevaluation have to induce equivalences
        \begin{equation*}
            \prod_{i,j}  D_i\grotimes D_j^{\op} \Morita A^{\op} \grotimes A \xrightarrow{\ev} \R \,. 
        \end{equation*}
        This forces $n=1$, e.g.\ by comparing $K_0$ of both sides.
        Hence, $A$ must be equivalent to a graded division algebra.
        Using the classification of graded division algebras, we are left to rule out $A\Morita \C$ and $A\Morita \Cxl_1$.
        The dual $A^{\op}$ is also $\C$ and $\Cxl_1$, respectively.
        They cannot be invertible since $\C\grotimes_{\R} \C =\C\oplus \C \notMorita \R$ and $\Cxl_1\grotimes_\R \Cxl_1 = \Cxl_2 \otimes_\R \C \Morita \C \oplus \C \notMorita \R$.

        To understand $\Omega_{\mathbbm{1}}\Pic(\sBimBanInf)$, we look at the space of endomorphisms of the unit using the graded version of Equation~\eqref{eq:mappingspacesbimban1}:
        \begin{align*}
            \Hom_{\sBimBanInf}(\R,\R)
            &\simeq N^{\mathrm{hc}}(\sVectOne^{\cong}_{\R})
            \\
            &\simeq\bigsqcup_{p,q} B\GL_p(\R) \times B\GL_q(\R) \,.
        \end{align*}
        In the space $\Omega_{\mathbbm{1}}\Pic(\sBimBanInf)$ composition translates into tensor product of graded vector spaces.
        The only graded vector spaces invertible under $\grotimes$ are ${\R,\Pi \R}$, corresponding to the summands $p=1,q=0$ and $p=0,q=1$, i.e.\ 
        \begin{equation*}
            \Aut_{\sBimBanInf}(\R) \simeq \Z/2\times B\GL_{1}(\R) \simeq N^{\mathrm{hc}}(\sLine_{\R}^{\cong}) \,.
        \end{equation*}
        $K^{\mathrm{gr}}$ is symmetric monoidal when restricted to the bootstrap class (cf.\ Section~\ref{sec:bootstrap}) and the bootstrap class contains all the Clifford algebras.
        Hence, $K^{\mathrm{gr}}$ indeed restricts to a symmetric monoidal functor $\Pic(\sBimBanInf)\to \ModInf(KO)$ which must factor through $\Pic(\ModInf(KO))$.
        We are left to note that $K^{\mathrm{gr}}$ does induce an equivalence on $\pi_0$ which maps $\Cl_{n}$ to $\Sigma^n KO$.
        Upon taking based loops at the unit, we are left with the functor
        \begin{align*}
            N^{\mathrm{hc}}(\sLine_{\R}^{\cong}) \longrightarrow & \GL_1(KO) \\ 
            L \longmapsto & (L\grotimes -:KO \to KO) \,.
        \end{align*}
        On $\pi_1\Pic(\sBimBanInf)=\pi_0\sLine_{\R}$, the generating bimodule $\Pi \R$ maps to the involution $-1\colon KO\to KO$ since for every $V\in \sVectOne$, $\Pi \R \otimes V$ represents the negative in $k^{\ABS}_{\R}$ by Lemma~\ref{lem:ABSpizero}.
        Finally, on $\pi_2\Pic(\sBimBanInf)=\pi_1\sLine_{\R}$. 
        The isomorphism $-1\colon \R\to \R$ of graded lines also classifies a line bundle over $S^1$---the M\"{o}bius bundle---the generator of $\pi_1$.
        \begin{equation*}
            S^1 \to \sLine_{\R} \to \GL_1(KO) \subseteq \Omega^\infty KO \,.
        \end{equation*}
        The M\"{o}bius bundle represents a generator in $\widetilde{KO}^0(S^1)$.
    \end{proof}

    \begin{remark}
        In~\cite{beardsley2023brauerwallgroupstruncatedpicard}, the authors obtain a splitting $B\GL_1(KO)\simeq \tau_{\leq 2}B\GL_1(KO) \times \tau_{\geq 3}B\GL_1(KO)$ and state that they do not know whether there is an analogous splitting for $\Pic(\ModInf(KO))$.
        We obtained this splitting in \eqref{eq:splitting}.
        It is a folklore result, proven e.g.\ in~\cite[Thm~5.27]{beardsley2023brauerwallgroupstruncatedpicard}, that $\Pic(\Mor(\sVectOne)) \simeq \tau_{\leq 2}\Pic(\ModInf(KO))$ are abstractly equivalent as spectra, and our result proves that the map $K^{\mathrm{gr}}$ induces this equivalence.
        Freed~\cite[(1.40)]{FreedTwistedKTheoryOrientifolds} computes the homotopy groups of the Picard $2$-groupoid of real invertible graded algebras, but leaves its $k$-invariants undetermined~\cite[(1.121)]{FreedTwistedKTheoryOrientifolds}.
    \end{remark}
    \begin{remark}
        We expect that the full subcategory inclusion $\Pic(\Bim(\sBan_{\R})^{\mathrm{fd}})\to \Pic(\Bim(\sBan_\R))$ is an equivalence.
        However, there are many more (equivalent) invertible objects on the right.
        Examples for this are the nonsemisimple dual numbers $\R[\epsilon]/(\epsilon^2)$ which are homotopy equivalent to $\R$ and the infinite-dimensional Banach algebras $C(X)$, where $X$ is contractible.
    \end{remark}

    The same explicit calculation as in the real case shows:
    \begin{theorem}
        \label{thm:complexpicardgroupoid}
        The restriction of $K^{\mathrm{gr}}$ to $\grotimes_{\C}$-invertible complex graded algebras is a symmetric monoidal functor of Picard $\infty$-groupoids
            \begin{equation*}
                \Pic(\Bim(\sBan_{\C})^{\mathrm{fd}}) \longrightarrow \Pic(\ModInf(KU)) \,.
            \end{equation*}
        The map is an isomorphism on $\pi_0,\pi_1,\pi_2,\pi_3$ and $\Pic(\Bim(\sBan_{\C})^{\mathrm{fd}})$ is $3$-coconnective.
        Hence, $K^{\mathrm{gr}}$ witnesses a splitting of Picard $\infty$-groupoids:
        \begin{equation}
            \label{eq:complexsplitting}
            \Pic(\ModInf(KU)) \simeq \Pic(\Bim(\sBan_{\C})^{\mathrm{fd}}) \times \tau_{\geq 4}\Pic(\ModInf(KU))\,.
        \end{equation}
    \end{theorem}
        The $3$-type $\Pic(\Bim(\sBan_{\C})^{\mathrm{fd}})$ has the following description:
        The equivalence classes of objects are generated by the complex Clifford algebras:
        \begin{equation*}
            \pi_0\Pic(\Bim(\sBan_{\C})^{\mathrm{fd}})=\{\Cxl_0,\Cxl_1\}\cong \Z/2 \,.
        \end{equation*}
        The loop space $\Omega_{\mathbbm{1}}(\Pic(\Bim(\sBan_{\C})^{\mathrm{fd}}))$ is the $2$-type $\sLine_{\C}$ of complex graded lines with
        \begin{align*}
            \pi_1\Pic(\Bim(\sBan_{\C})^{\mathrm{fd}})&=\pi_0 N^{\mathrm{hc}}(\sLine_{\C}^{\cong})=\{\C, \Pi \C \} \cong \Z/2 \\
            \pi_2\Pic(\Bim(\sBan_{\C})^{\mathrm{fd}})&=\pi_1 N^{\mathrm{hc}}(\sLine_{\C}^{\cong})=\pi_0\frac{\{\text{inv. linear isomorphisms } \C\to \C\}}{\text{homotopy}} = 0 \\
            \pi_3\Pic(\Bim(\sBan_{\C})^{\mathrm{fd}}) &= \pi_1\frac{\{\text{inv. linear isomorphisms } \C\to \C\}}{\text{homotopy}} = \Z \,,
        \end{align*}
    since invertible linear isomorphisms $\C\to \C$ are given the Euclidean topology.
    Note that there is a difference between $\Pic(\Mor(\sVectOne_\C))$ and $\Pic(\Bim(\sBan_{\C})^{\mathrm{fd}})$ which is less prominent in the real case: the former is a $2$-type with $\pi_2=\C^\times$ whereas the latter is a $3$-type which accounts for the topology of $\C^\times$.
\subsection{Proofs}
\label{sec:proofs}

\paragraph{Proof strategy}

This section proves Theorem~\ref{thm:gradedbottperiodicity}: the map
$\overline{\alpha}_A \colon k^{\ABS}_A \to \Omega k^{\ABS}_{A \grotimes \Cl_{+1}}$
is an isomorphism on $\pi_i$ for $i \geq 1$.

Throughout it is more convenient to work with fibers than with cofibers.

\begin{definition}
\label{def:vDspectrum}
The \emph{(connective) van Daele $K$-theory spectrum} of a graded Banach algebra $A$ is
\begin{equation*}
k_A^{\mathrm{vD}} := \fib\left(i_-^* \colon k^{C_2}_{A \grotimes \Cl_{-1}} \to k^{C_2}_A\right) \,.
\end{equation*}
\end{definition}
The reader is referred to Section~\ref{sec:vD} for further details on van Daele $K$-theory.
Since $k^{\ABS}_A$ is the cofiber of the same map, $\fib \cong \Omega \cofib$ gives a natural equivalence $k^{\mathrm{vD}}_A \simeq \Omega k^{\ABS}_A$.
The proof of Theorem~\ref{thm:gradedbottperiodicity} then proceeds in three steps.

\begin{enumerate}
    \item \emph{Reduction to $\pi_0$.}
    The functor $\Omega k^{\ABS} = k^{\mathrm{vD}}$ is excisive and homotopy invariant, and a
    natural transformation between such functors is an equivalence as soon as it is an
    isomorphism on $\pi_0$
    (Lemmas~\ref{lem:eqofspectra}--\ref{lem:loopsofabsexcisive}).
    It therefore suffices to compute $\overline{\alpha}$ on $\pi_0$ after looping once.

    \item \emph{A strictly commutative model.}
    The class $\alpha \in \pi_1 k^{\ABS}_{\Cl_{+1}}$ is constructed from a graded irreducible
    $\Cl_{1,1}$-module, and the induced map is unravelled into a square of Banach categories
    commuting up to an explicit rotation homotopy (Lemma~\ref{lem:tviamodules}).
    Trading this homotopy for a family of modules over the interval rectifies the square
    into a strictly commutative one involving module bundles over $(D^1,S^0)$
    (Proposition~\ref{prop:bundlemodel} and Lemma~\ref{lem:tandalphasquare}).

    \item \emph{Karoubi triples.}
    Both sides of the rectified square are fibers of group completions of quasi-surjective
    maps (Lemma~\ref{lem:i-quasisur}), so Theorem~\ref{thm:fibergroupcompletion} describes
    their $\pi_0$ in terms of Karoubi triples.
    Under this description $\pi_0(\Omega \overline{\alpha})$ becomes the explicit
    formula~\eqref{eq:t0}, and the fact that it is an isomorphism is Karoubi's Bott
    periodicity theorem~\cite[Thm~2.2.2]{Karoubi68}.
\end{enumerate}

Apart from Karoubi's theorem, the only inputs are the group completion machinery of
Section~\ref{sec:fibersgroupcompletion} and excision for $k^{C_2}$.

\paragraph{Homotopy theory lemmas}

To apply the theory of Karoubi triples of Section~\ref{sec:fibersgroupcompletion}, it will be convenient to rewrite cofibers as fibers by using that $\fib f \cong \Omega \cofib f$ naturally in $f$.
The fiber within connective spectra is the connective cover of the fiber in all spectra, so the equation also holds in connective spectra by taking the connective cover of both sides.
\begin{lemma}
\label{lem:cofibdiag}
    Let $f:X\to Y$ be a map of spectra. The cofiber of $\Delta_f:X\to X\times_Y X$ is equivalent to $\fib f \cong \Omega \cofib f$.
\end{lemma}
\begin{proof}
The diagonal $\Delta_f \colon X \to X \times_Y X$ is a section of either projection map $p_1 \colon X \times_Y X \to X$ and so splits the pullback.
We can thus identify $X \times_Y X \cong X \oplus \fib f$ under which $\Delta_f$ becomes the inclusion in the first factor.
\end{proof}
\begin{lemma}
    \label{lem:eqofspectra}
    Let $t\colon E\to F$ be a map of connective spectra with the property that
    \begin{equation*}
        t_*: \pi_0\Hom(\Sigma_+^\infty(X),E)\to  \pi_0\Hom(\Sigma_+^\infty(X),F)
    \end{equation*}
    is an isomorphism for all compact Hausdorff spaces $X$ which are homotopy finitely dominated.
    Then $t$ is an equivalence.
\end{lemma}
\begin{proof}
    We just have to show that $t$ induces an isomorphism on all nonnegative homotopy groups $\pi_n$, which suggests to consider $X = S^n$.
   There is a disjoint basepoint issue, but we can write $\mathbb{S}^n=\cofib(\Sigma_+^\infty(\pt)\to \Sigma_+^\infty S^n)$. Since this cofiber sequence is split, there is a short exact sequence:
    \begin{equation*}
        \begin{tikzcd}
            0 \ar[r]& \pi_0\Hom(\mathbb{S}^n,E) \ar[d,"t_*"] \ar[r] & \pi_0\Hom(\Sigma_+^\infty(S^n),E) \ar[d,"t_*","\cong"']\ar[r]& \pi_0\Hom(\Sigma_+^\infty(\pt),E) \ar[d,"t_*","\cong"']\ar[r]& 0 \\
            0 \ar[r]& \pi_0\Hom(\mathbb{S}^n,F) \ar[r]& \pi_0\Hom(\Sigma_+^\infty(S^n),F) \ar[r]& \pi_0\Hom(\Sigma_+^\infty(\pt),F) \ar[r]& 0 
        \end{tikzcd} \,.
    \end{equation*}
    This implies that $t_*:\pi_nE\to \pi_n F$ is an isomorphism.
\end{proof}
\begin{lemma}
    \label{lem:eqoffunctors}
    Let $F,G\colon \sBanAlgInf \to \Sp_{\geq 0}$ be two functors that are excisive and homotopy invariant and let $F\Rightarrow G$ be a natural transformation that induces an isomorphism on $\pi_0$.
    Then $F\Rightarrow G$ is an equivalence.
\end{lemma}
\begin{proof}
    By Corollary~\ref{cor:homotopicalSwan}, we know that there is a natural equivalence
    \begin{equation*}
        F(C(X;A))\simeq \tau_{\geq 0} \underline{\Hom}_{\Sp}(\Sigma_+^\infty X,F(A))
    \end{equation*}
    for all homotopy finitely dominated compact Hausdorff spaces $X$.
    Thus, we know that there is a natural isomorphism $\pi_0\Hom(\Sigma_+^\infty X,F(A))\cong \pi_0\Hom(\Sigma_+^\infty X,G(A))$.
    We can conclude using Lemma~\ref{lem:eqofspectra}.
\end{proof}
\begin{lemma}
    \label{lem:loopsofabsexcisive}
    The functor $\Omega k^{\ABS}\colon \sBanAlgInf \to \Sp_{\geq 0}$ is excisive.
\end{lemma}
\begin{proof}
The functor $k^{C_2}$ is excisive by Proposition~\ref{prop:kC_2}.
    $\Omega k^{\ABS}_A=\fib(k^{C_2}_{A\grotimes \Cl_{-1}}\to k^{C_2}_A)$ is excisive as a fiber of excisive functors.
    This also uses that $-\grotimes \Cl_{-1}$ is excisive.
\end{proof}

\paragraph{A bundle model for \texorpdfstring{$\Omega k^{\ABS}_A$}{Omega k\textasciicircum ABS}}

Recall from Proposition~\ref{prop:kC_2} that $k^{C_2}_A(X)\simeq k^{C_2}_{C(X;A)}$.
We want to rewrite $\Omega k_A^{C_2}$ in terms of $\ModEn^{C_2}_A(S^1)$ for which we need to overcome a minor basepoint problem:

\begin{warning}
\label{rem:rewriteloops}
    If $B$ is an ungraded Banach algebra, it is not true that $\Omega k_B $ is the same as $k_{B}(S^1)$ because $\Omega$ takes based loops:
    \begin{equation*}
        k_{B}(S^1) \simeq \underline{\Hom}_{\Sp_{\geq 0}}(\Sigma_+^\infty S^1, k_B) \simeq \underline{\Hom}_{\Sp_{\geq 0}}(\mathbb{S}^1 \vee \mathbb{S}^0, k_B)\simeq \Omega k_B \oplus k_B.    
    \end{equation*}
\end{warning}
However, $\Omega k_B$ can be expressed as 
    \begin{equation*}
        \Omega k_B \simeq \fib 
        \left( k_{B}(D^1) \to k_{B}(S^0) \right)\,.
    \end{equation*}
    The same applies to $k^{C_2}_B$ for any graded algebra $B$, by Proposition~\ref{prop:kC_2}; applying this to $B = A \grotimes \Cl_{+1}$, we see that $\Omega k_A^{C_2} = \fib(k_A^{C_2}(D^1) \to k_A^{C_2}(S^0))$.
Since $D^1$ is contractible, the constant bundle functor $\ModEn^{C_2}_{A\grotimes \Cl_{-1}}\to \ModEn^{C_2}_{A\grotimes \Cl_{-1}}(D^1)$ is an equivalence after passing to the underlying symmetric monoidal $\infty$-groupoids $N^{\mathrm{hc}}\circ (-)^{\cong}$.
Similarly, $\ModEn^{C_2}_A(S^0)\cong (\ModEn_A^{C_2})^{\times 2}$.
Under this identification, $\ModEn_A^{C_2}(D^1) \to \ModEn_A^{C_2}(S^0)$ is the diagonal map.

\begin{proposition}
\label{prop:bundlemodel}
    There is a natural equivalence
    \begin{equation*}
    \Omega k^{\ABS}_A \simeq \cofib(k_{A \grotimes \Cl_{-1}}^{C_2}(D^1) \to k_{A \grotimes \Cl_{-1}}^{C_2}(S^0) \times_{k_{A}^{C_2}(S^0)} k_{A}^{C_2}(D^1)) \,.
    \end{equation*}
\end{proposition}
\begin{proof}
    Write $k^{C_2}_A=Y$, $k^{C_2}_{A\grotimes \Cl_{-1}}=Y$.
    The right hand side is of the form $\cofib(\Delta:X\to (X\times X) \times_{Y\times Y} Y)$.
    By Lemma~\ref{lem:cofibdiag} this is just $\fib(X\to Y)=\Omega\cofib(X\to Y)$.
\end{proof}

\begin{lemma}
\label{lem:fibKcommute}
      There is an equivalence
    \begin{equation*}
        k\left(\ModEn^{C_2}_{A\grotimes \Cl_{-1}}(S^0) \times_{\ModEn_A^{C_2}(S^0)} \ModEn_{A}^{C_2}(D^1)\right) \simeq k^{C_2}_{A\grotimes \Cl_{-1}}(S^0) \times_{k_A^{C_2}(S^0)} k_{A}^{C_2}(D^1) \,.
    \end{equation*}
    natural in $A$.
\end{lemma}
\begin{proof}
We apply Lemma~\ref{lem:groupcompletionpullback}. 
It is routine to verify that all maps are quasi-surjective.
\end{proof}

\paragraph{The generator \texorpdfstring{$\alpha\in \pi_1k^{\ABS}_{\Cl_{+1}}$.}{alpha.}}
            
Recall the definition of connective graded $K$-theory:
\begin{equation*}
    k^{\ABS}_A=\cofib\left(k^{C_2}_{A\grotimes \Cl_{-1}}\xrightarrow{i_-^*}k^{C_2}_A \right) \,.
\end{equation*}
By Lemma~\ref{lem:cofibdiag} we obtain 
\begin{equation}
\label{eq:diagrewrite}
    \Omega k^{\ABS}_A = \cofib\left(k^{C_2}_{A\grotimes \Cl_{-1}}\xrightarrow{\Delta_{i_-^*}} k^{C_2}_{A\grotimes \Cl_{-1}}\times_{k^{C_2}_A} k^{C_2}_{A\grotimes \Cl_{-1}} \right) \,.
\end{equation} 
To define $\alpha$ it suffices to define an element in $\pi_0\left(k^{C_2}_{\Cl_{+1} \grotimes \Cl_{-1}}\times_{k^{C_2}_{\Cl_{+1}}} k^{C_2}_{\Cl_{+1}\grotimes \Cl_{-1}}\right)$. 
We construct it as the image of an element
\begin{equation*}
    \label{eq:alphadef}
    (V,\overline{V})\in\ModEn^{C_2}_{\Cl_{1,1}} \times_{\ModEn^{C_2}_{\Cl_{+1}}} \ModEn^{C_2}_{\Cl_{1,1}} \,.
\end{equation*}
Here, we let $V$ be a choice of graded irreducible $\Cl_{1,1}$-module, which induces the Morita equivalence $V\grotimes-:\sVectOne\simeq \ModOne^{C_2}_{\Cl_{1,1}}$.
There is another such choice $\overline{V}$ which is $V$ pulled back along the algebra isomorphism
\begin{align*}
    \Cl_{1,1} & \longrightarrow \Cl_{1,1} \\
    e &\longmapsto e\\
    f& \longmapsto -f \,.
\end{align*}
The identity is an evident equivalence that witnesses $i_-(\overline{V})\cong i_-(V)$, so that the pair indeed defines an element in the pullback.

\begin{lemma}
    \label{lem:alphaequivalence}
    $\alpha$ induces an equivalence of $ko$-modules
    \begin{equation*}
        \alpha\colon ko\to \Omega k^{\ABS}_{\Cl_{+1}} \,.
    \end{equation*}
\end{lemma}
\begin{proof}
We will use the equivalence $ko\simeq k^{\ABS}_\R$ to reason about $\VectEn$-modules.
Applying Equation~\eqref{eq:diagrewrite}, we want to construct $\alpha$ as coming from taking cofibers of a commutative diagram
\begin{equation*}
\begin{tikzcd}
    0 \ar[d] \ar[r] & k^{C_2}_{\Cl_{1,1}}  \ar[d] \\
    k_{\R} \ar[r] & k^{C_2}_{\Cl_{1,1}} \times_{k^{C_2}_{\Cl_{+1}}} k^{C_2}_{\Cl_{1,1}}
\end{tikzcd}
\end{equation*}
obtained from group completing a diagram on the module category level.
    For this, we apply the graded Morita equivalence $\Cl_{1,1}\Morita \R$ and the ungraded Morita equivalence $|\Cl_{+2}|\Morita \R$ to obtain an equivalence of $\VectEn$-modules and (Banach) categories
    \begin{equation*}
        \ModEn^{C_2}_{\Cl_{+1}}\simeq \ModEn_{|\Cl_{+2}|} \simeq \VectEn \,.
    \end{equation*}
    The pair $(V,\overline{V})$ is identified with $(\R,\Pi \R)\in \sVectOne\times_{\VectOne} \sVectOne$.
    It therefore suffices to observe that the diagram
    \begin{equation*}
        \begin{tikzcd}[column sep = 35]
            0\ar[d] \ar[r]& \sVectEn \ar[d,"\Delta"] \\
            \VectEn \ar[r,"{(\id,\Pi)}"]& \sVectEn \times_{\VectEn} \sVectEn
        \end{tikzcd}
    \end{equation*}
    induces an equivalence on vertical cofibers after applying $k$. 
    This follows from the fact that right hand side group completes to $\cofib(ko^{\times 2}\to ko^{\times 3})$, using Lemma~\ref{lem:groupcompletionpullback}.
\end{proof}

\paragraph{Action of \texorpdfstring{$\alpha$}{alpha} by explicitly writing out \texorpdfstring{$k^{\ABS}$}{k\textasciicircum ABS}-module structure}

We write $t$ for the composite of $\overline{\alpha}$ with the Morita equivalence $\Cl_{1,1}\Morita \R$:
\begin{equation}
    \label{eq:tandalpha}
    t_A\colon k^{\ABS}_{A\grotimes \Cl_{-1}} \xrightarrow{\overline{\alpha}_{A\grotimes \Cl_{-1}}} \Omega k^{\ABS}_{A\grotimes \Cl_{1,1}} \simeq \Omega k^{\ABS}_{A} \,.
\end{equation}

\begin{lemma}
\label{lem:tviamodules}
  The map $t_A$ is induced on cofibers by the following square of Banach functors, which commutes up to a specified homotopy:
    \begin{equation}
        \label{sq:2}
        \begin{tikzcd}[row sep = 28, column sep = 40]
            \ModEn^{C_2}_{A\grotimes \Cl_{-2}} \ar[r,"i_-"]\ar[d,"i_-"]& \ModEn^{C_2}_{A\grotimes \Cl_{-1}} \ar[d,"\Delta"] \\
            \ModEn^{C_2}_{A\grotimes \Cl_{-1}} \ar[r,"{(\id,\overline{(-)})}"] \ar[ru,Rightarrow, "h",shorten >=1ex,shorten <=4ex]& \ModEn^{C_2}_{A\grotimes \Cl_{-1}}\times_{\ModEn^{C_2}_A} \ModEn^{C_2}_{A\grotimes \Cl_{-1}}
        \end{tikzcd} \,.
    \end{equation}
    Here, the bottom functor maps a graded $A\grotimes \Cl_{-1}$-module $N$ to $(N,\overline{N})$, where $\overline{N}$ has a sign-flipped $\Cl_{-1}$-action.
\end{lemma}
\begin{proof}
The proof consists of unraveling the tensoring by the explicit element $\alpha$.
Recall from Section~\ref{sec:bott} that the $k^{\ABS}$-module structure $k_{A}^{\ABS}\otimes k_B^{\ABS} \to k^{\ABS}_{A\grotimes B}$
is induced by the diagram
\begin{equation*}
    \begin{gathered}
        \begin{tikzcd}[column sep=large, row sep=large]
            \ModEn^{C_2}_{A\grotimes \Cl_{-1}}\times \ModEn^{C_2}_{B\grotimes \Cl_{-1}} \ar[r,"\id\times i_-^*"]\ar[d,"i_-^*\times \id"]& \ModEn^{C_2}_{A\grotimes \Cl_{-1}} \times \ModEn^{C_2}_{B} \ar[d,"{\grotimes}"]\\
            \ModEn^{C_2}_{A} \times \ModEn^{C_2}_{B\grotimes \Cl_{-1}} \ar[r,"{\grotimes}"] & \ModEn^{C_2}_{A\grotimes B\grotimes \Cl_{-1}}
        \end{tikzcd} \\
        \\
        \begin{tikzcd}[column sep=large, row sep=large]
            \left((M,f_1),(N,f_2) \right) \ar[rr,mapsto]\ar[d,mapsto]& &\left( (M,f_1),N\right) \ar[d,mapsto]\\
            \left(M,(N,f_2)\right) \ar[r,mapsto] & (M\grotimes N,f_2) & (M\grotimes N,f_1) 
        \end{tikzcd}
    \end{gathered}
\end{equation*}
that can be filled by the homotopy of $A\grotimes B \grotimes \Cl_{-1}$-modules
\begin{equation*}
    \theta\mapsto \left( M\grotimes N,f_1\cos(\theta)+f_2\sin(\theta)\right)\,, \quad \theta\in[0,\pi/2]\,.
\end{equation*}
Similarly we can explicitly unravel the action of $k^{\ABS}_A$ on $\Omega k^{\ABS}_B=\cofib (\Delta_{i_-})$ as induced from the diagrams\footnote{The pushout in the upper left corner does not need to exist for the diagram to make sense.}
\begin{equation*}
    \begin{tikzcd}[column sep=tiny]
         \mathcal{P}
         \ar[r]\ar[d]& \ModEn^{C_2}_{A\grotimes B\grotimes \Cl_{-1}} \ar[d,"\Delta_{i_-^*}"]\\
        \ModEn^{C_2}_{A}\times \left(\ModEn^{C_2}_{B\grotimes \Cl_{-1}}\times_{\ModEn^{C_2}_{B}} \ModEn^{C_2}_{B\grotimes \Cl_{-1}}\right) \ar[r,"\grotimes"]& \ModEn^{C_2}_{A\grotimes B\grotimes \Cl_{-1}}\times_{\ModEn^{C_2}_{A\grotimes B}} \ModEn^{C_2}_{A\grotimes B\grotimes \Cl_{-1}}
    \end{tikzcd} \,,
\end{equation*}
where the top left corner is the pushout
\begin{equation*}
    \mathcal{P}= \ModEn^{C_2}_{A} \times \ModEn^{C_2}_{B\grotimes \Cl_{-1}}
         \sqcup_{\ModEn^{C_2}_{A\grotimes \Cl_{-1}}\times \ModEn^{C_2}_{B\grotimes \Cl_{-1}}} \ModEn^{C_2}_{A\grotimes \Cl_{-1}} \times \left(\ModEn^{C_2}_{B\grotimes \Cl_{-1}}\times_{\ModEn^{C_2}_{B}} \ModEn^{C_2}_{B\grotimes \Cl_{-1}} \right) \,.
\end{equation*}
Taking $B=\Cl_{+1}$ and the chosen element $(V,\overline{V})$ we see that the map $\alpha$ is induced on cofibers by the diagram
\begin{equation*}
    \begin{tikzcd}[row sep=28, column sep = 50]
        \ModEn^{C_2}_{A\grotimes \Cl_{-1}} \ar[r,"-\grotimes i_-^*(V)"]\ar[d]& \ModEn^{C_2}_{A\grotimes \Cl_{1,1}} \ar[d,"\Delta"] \\
        \ModEn^{C_2}_A \ar[r,"{(-\grotimes V,-\grotimes \overline{V})}"]\ar[ru,Rightarrow, "h",shorten >=2ex,shorten <=6ex]\ar[d]& \ModEn^{C_2}_{A\grotimes \Cl_{1,1}} \times_{\ModEn^{C_2}_{A\grotimes \Cl_{+1}}} \ModEn^{C_2}_{A\grotimes \Cl_{1,1}} \ar[d] \\
        k^{\ABS}_A \ar[r,"\alpha"]& \Omega k^{\ABS}_{A\grotimes \Cl_{+1}}
    \end{tikzcd} \,.
\end{equation*}
The diagram commutes up to the homotopy $h$ which witnesses that for $(M,f_1) \in \ModEn^{C_2}_{A\grotimes \Cl_{-1}}$ there is a path of $A\grotimes \Cl_{1,1}$-module structures on $M\grotimes V$, in which the $\Cl_{-1}$-action is:
\begin{equation*}
    f(\theta)=\cos(\theta) f_1+\sin(\theta)f\,, \quad \theta\in [0,\pi] \,,
\end{equation*}
where $f$ is the $\Cl_{-1}$-action on $V$.
We use the graded Morita equivalences $A\grotimes {\Cl_{1,2}}\Morita A\grotimes \Cl_{-1}$ and $A\grotimes \Cl_{1,1}\Morita A$ to arrive at the desired diagram.
\end{proof}

\paragraph{Re-expressing \texorpdfstring{$\alpha$}{alpha} geometrically}
We want to replace diagrams that commute up to a specified homotopy with diagrams that strictly commute.

\begin{lemma}
    \label{lem:tandalphasquare}
    $t_A$ is induced upon applying $k=(-)^{\gp}\circ N^{\mathrm{hc}}\circ (-)^{\cong}$ and taking cofibers vertically from the commutative square of Banach functors
    \begin{equation}
        \label{eq:bundlediagram}
        \begin{tikzcd}[column sep = 40]
            \ModEn^{C_2}_{A\grotimes \Cl_{-2}} \ar[r,"{(M,f_1,f_2)\mapsto (M,f(\theta))}"]\ar[d,"i_-^*"]& \ModEn^{C_2}_{A\grotimes \Cl_{-1}}(D^1) \ar[d,"{((-)|_{S^0},i_-^*)}"] \\
            \ModEn^{C_2}_{A\grotimes \Cl_{-1}} \ar[r,"t"]& \ModEn^{C_2}_{A\grotimes \Cl_{-1}}(S^0) \times_{\ModEn_A^{C_2}(S^0)} \ModEn_{A}^{C_2}(D^1)
        \end{tikzcd} \,,
    \end{equation}
    where $f(\theta)=f_1\cos(\theta) +f_2 \sin(\theta)$ is the $\Cl_{-1}$-action varying along $D^1=[0,\pi]$ and 
        \begin{equation*}
            t\colon (M,f) \longmapsto \left( (M,\pm f), \underline{M}\to D^1\right) \,.
        \end{equation*}
\end{lemma}
\begin{proof}
    Abstractly, the bottom horizontal functor in Diagram~\eqref{sq:2} is the pullback along the automorphism of $\Cl_{-1}$ given by $f\mapsto -f$, which is the Clifford functor applied to the orientation reversal $-1 \colon \R\to \R$.
    There is a homotopy making the diagram commute which is itself induced by a homotopy between algebra homomorphisms $\Cl_{-1}\to \Cl_{-2}$.
    Namely, the standard embedding $f\mapsto f_1$ is homotopic to $f\mapsto -f_1$:
    This corresponds to a homotopy between the two orthogonal embeddings $\R \to \R^2$
    given by $x\mapsto (x,0)$ and $x\mapsto (-x,0)$: explicitly this is
    \begin{equation*}
        h(t,x)=(\cos(t)x,\sin(t)x) \,.
    \end{equation*}
    To relate squares~\eqref{eq:bundlediagram} and \eqref{sq:2}, one notes that there is a commutative cube (with prescribed homotopies) from \eqref{sq:2} to \eqref{eq:bundlediagram} which induces equivalences after applying $N^{\mathrm{hc}}\circ (-)^{\cong}$.
\end{proof}

\paragraph{Spelling out Karoubi triples}

We need to understand the map $\Omega t_A$ on $\pi_0$ using the description of Lemma~\ref{lem:tandalphasquare}.
We want to use the Karoubi $K$-theory machinery, cf.\ Section~\ref{sec:fibersgroupcompletion} and so we want to rewrite $\Omega t_A$ as a map between fibers of group completions.
For this, we use $\fib f \cong \Omega \cofib f$ to obtain
\begin{equation*}
                \Omega t_A\colon \fib\left(i_-^{\gp}\colon k^{C_2}_{A\grotimes \Cl_{-2}} \to k^{C_2}_{A\grotimes \Cl_{-1}}\right) \to \fib\left(k^{C_2}_{A\grotimes \Cl_{-1}}(D^1) \to k^{C_2}_{A\grotimes \Cl_{-1}}(S^0)\times_{k^{C_2}_A(S^0)} k^{C_2}_A(D^1)\right) \,.
    \end{equation*}
            To apply Theorem~\ref{thm:fibergroupcompletion}, we need 
            \begin{lemma}
            \label{lem:i-quasisur}
                The map $i_-^* \colon \ModOne^{C_2}_{A \grotimes \Cl_{-1}} \to \ModOne^{C_2}_A$ is quasi-surjective.
            \end{lemma}
            \begin{proof}
                We can apply Remark~\ref{rem:qsur} to the inclusion of the graded $A \grotimes \Cl_{-1}$-modules $(A \grotimes \Cl_{-1})^n$ into all graded $A$-modules, and the map from graded $A \grotimes \Cl_{-1}$-modules to graded $A$-modules.
Indeed, the $A \grotimes \Cl_{-1}$-modules $(A \grotimes \Cl_{-1})^n$ restrict to the graded $A$-modules $A^{n|n}$ which are cofinal in graded $A$-modules.
            \end{proof}

Explicitly spelling out Karoubi triples using Lemma~\ref{lem:triples} gives:
    \begin{itemize}
        \item $K^{\mathrm{vD}}_{A\grotimes \Cl_{-1}}:=\pi_0\fib(i_-^{\gp})$.
        A generic element is a tuple $[M,f,f_1,f_2]$, where $(M,f)$ is a $A\grotimes \Cl_{-1}$-module and $f_1,f_2$ are two extensions to an $A\grotimes \Cl_{-2}$-module.
        An elementary triple is one such that $f_1$ is homotopic to $f_2$ within $A\grotimes \Cl_{-2}$-module structures on $(M,f)$.
        We have $[M,f,f_1,f_2]=0$ if and only if there is an elementary triple $\tau$ such that $f_1\oplus \tau \simeq f_2\oplus \tau$, i.e.\ $[M,f,f_1,f_2]\oplus \tau$ is elementary.
        \item $K^{\mathrm{vD}}_A(D^1,S^0):=\pi_0\fib\left(k^{C_2}_{A\grotimes \Cl_{-1}}(D^1) \to k^{C_2}_{A\grotimes \Cl_{-1}}(S^0)\times_{k^{C_2}_A(S^0)} k^{C_2}_A(D^1)\right)$.
        By Lemma~\ref{lem:fibKcommute}, a generic element of this fiber is a triple $[M,f_1(\theta),f_2(\theta)]$, where $M$ is an $A$-module and $f_1(\theta),f_2(\theta),\theta\in D^1=[0,\pi]$ are two $A\grotimes \Cl_{-1}$-structures on the trivial bundle $\underline{M}\to D^1$ satisfying $f_1|_{S^0}=f_2|_{S^0}$.
        We have $[M,f_1(\theta),f_2(\theta)]=0$ if and only if $f_1(\theta)\oplus \tau(\theta) \simeq f_2(\theta)\oplus \tau(\theta)$, where the restriction to $S^0$ remains fixed along the homotopy.
        \end{itemize}

        The map $\pi_0(\Omega t_A)$ is therefore explicitly given by
        \begin{equation}
        \label{eq:t0}
        \begin{aligned}
             t:K^{\mathrm{vD}}_0(A\grotimes \Cl_{-1}) & \longrightarrow K_A^{\mathrm{vD}}(D^1,S^0)\\
             [M,f,f_1,f_2]&\longmapsto [M,\cos(\theta)f+\sin(\theta)f_1,\cos(\theta)f+\sin(\theta)f_2]
        \end{aligned} \,.
        \end{equation}

\paragraph{Karoubi's theorem}

\begin{theorem}[{\cite[Thm~2.2.2]{Karoubi68}}]
\label{th:karoubishift}
    The map $t$ of \eqref{eq:t0} is an isomorphism.
\end{theorem}
\begin{proof}
    Up to some minor modifications, this is proven in~\cite[Thm~2.2.2]{Karoubi68}, see the formula right before Lemma~2.2.3 in~\cite{Karoubi68} for his version of formula~\eqref{eq:t0}.
    We will now indicate the minor modifications and how to prove surjectivity of $t$.
    The injectivity is proven along the same lines and also contained in~\cite{Karoubi68}.
    \begin{enumerate}
        \item Karoubi uses $\Cl_{+1}$ instead of $\Cl_{-1}$.
        \item Every element of $K^{\mathrm{vD}}_A(D^1,S^0)$ can be written in the form 
                \begin{equation*}
                    [M,\cos(\theta)f+\sin(\theta)f_1,f_2(\theta)]\,,
                \end{equation*}
                where $(M,f,f_1)$ is a $A\grotimes \Cl_{-2}$-module and $f_2(0)=-f_2(\pi)=f$.
             This is proven in~\cite[Lemma~2.2.3]{Karoubi68}.
        \item  Every element of $K^{\mathrm{vD}}_A(D^1,S^0)$ can be written in the form 
                \begin{equation}
                \label{eq:triplealpha}
                     [M,\cos(\theta)f+\sin(\theta)f_1,\alpha(\theta) f \alpha(\theta)^{-1}]\,,
                \end{equation}
                where $(M,f,f_1)$ is a $A\grotimes \Cl_{-2}$-module and where $\alpha(\theta)\colon M\to M, \theta\in [0,\pi]$ is a continuous family of even $A$-module isomorphisms satisfying
                \begin{enumerate}
                \item \label{item1} $\alpha(0)=\id$.
                \item \label{item2} $\alpha(\pi)f=-f\alpha(\pi)$.
            \end{enumerate}
            This is proven in~\cite[Lemma~2.2.4]{Karoubi68}.
        \item Consider an $A\grotimes \Cl_{-2}$-module $(M,f,f_1)$ together with $\alpha(\theta),\theta\in[0,\pi]$ a family of $A$-linear automorphisms satisfying \ref{item1} and \ref{item2}.
            Writing $\theta/2=\phi$, we may uniquely extend the function $\alpha(\phi)$ defined on a quarter circle to a function on $\zeta = e^{i\phi}\in U(1) \cong [0,2\pi]/\sim$ satisfying
            \begin{enumerate}
                \item $\alpha(-\zeta)=-\alpha(\zeta)$
                \item $\alpha(\overline{\zeta})=f^{-1}\alpha(\zeta)f$.
            \end{enumerate}
        \item Using $(ff_1)^2=-1$, we can use the fact that the even endomorphisms of $M$ form a Banach algebra to write $\alpha$ as a Fourier series in $\zeta_M=\cos(\phi)+\sin(\phi)ff_1$. 
            \begin{equation*}
                \alpha(\zeta) = \sum_{n\in \Z} \alpha_n \zeta_M^n \,. 
            \end{equation*}
            All even terms $\alpha_{2n}$ vanish.
            By abuse of notation we will further write $\zeta \equiv \zeta_M$.
        \item $\alpha$ is homotopic to a Laurent series, i.e.\ a finite sum 
            \begin{equation*}
                \alpha(\zeta) = \sum_{n=-N}^N \alpha_n \zeta^n \,. 
            \end{equation*}
            This does not change the class of the triple in Equation~\eqref{eq:triplealpha}.
            This is proven in~\cite[Lemma~2.2.5]{Karoubi68}.
        \item It follows by a trigonometric computation that 
        \begin{equation}
        \label{eq:trig}
            \zeta f \zeta^{-1}=\cos(\theta)f+\sin(\theta)f_1.
        \end{equation}
        We can therefore add the elementary triple 
             $[M,\cos(\theta)f+\sin(\theta)f_1,\zeta f \zeta^{-1}]$ and preserve the class of the triple~\eqref{eq:triplealpha}.
             This replaces $\alpha$ by $\alpha \oplus \zeta$ acting on $M \oplus M$.
             There is an elementary homotopy of block matrices $\alpha \oplus \zeta \simeq \alpha \zeta^2 \oplus \zeta^{-1}$.
             We can thus inductively lower the pole order of $\alpha$ to obtain a pole of order $1$:
             \begin{equation*}
                \alpha(\zeta) = \sum_{n=-1}^N \alpha_n \zeta^n \,. 
            \end{equation*}
            This is proven in~\cite[Lemma~2.2.6]{Karoubi68}.
        \item The next step is to reduce to \emph{affine} $\alpha=a_{-1}\zeta^{-1}+a_1\zeta$ as follows.
            In~\cite[Lemma~2.2.8]{Karoubi68} there is a minor error in the formulas, but the following is the fix:
            Let $\alpha=\alpha_{-1}\zeta^{-1}+\dots+\alpha_{2n-1}\zeta^{2n-1}$.
            Elementary row and column operations provide a homotopy
            \begin{equation*}
                \alpha\oplus (\zeta^{-1})^{\oplus n}=\begin{pmatrix}
                    \alpha & & & \\
                    & \zeta^{-1}& & \\
                    & & \ddots & \\
                    & & & \zeta^{-1}
                \end{pmatrix} 
                =
                \begin{pmatrix}
                    \alpha \zeta & & & \\
                    & 1& & \\
                    & & \ddots & \\
                    & & & 1
                \end{pmatrix} \zeta^{-1}
                \simeq
                \begin{pmatrix}
                    \alpha_{-1} & \alpha_1 & \dots &\alpha_{2n-1} \\
                    -\zeta^2 & 1 & 0 & 0 \\
                    \vdots & \ddots & \ddots &\vdots\\
                    0 & \dots & -\zeta^2 & 1
                \end{pmatrix} \zeta^{-1}
                \,.
            \end{equation*}
            The right hand side is of the form $A \zeta+B\zeta^{-1}$.
            Now by~\cite[Lemma 2.2.7]{Karoubi68}, the additive inverse of $(\zeta_M^{-1})^{\oplus n}$ is the class represented by $(\zeta^{-1}_{M^\star})^{\oplus n}$.\footnote{Let $(M,f,f_1)$ be an $A\grotimes \Cl_{-2}$-module. Then $M^{\star}:=(M,-f,f_1)$.
            One has $[E^\star,\zeta_{E^*},\zeta_{E^\star}^{-1}]=-[E,\zeta_E,\zeta_E^{-1}]$.
            The notation is $[E,\zeta_E,\alpha]\equiv [E,\zeta_E f \overline{\zeta}_E,\alpha f \alpha^{-1}]$.
            Note that $\zeta_Ef\overline{\zeta_E}=\cos(\theta)f+\sin(\theta)f_1$, so this re-expresses Equation~\eqref{eq:triplealpha}.}
            That is, $\alpha$ is in the image of the affine cocycles.
            \item We can now conclude that $t$ is surjective as follows.
                Let $\alpha$ be affine. By \ref{item1}, we may write $\alpha = \cos(\phi)+g\sin(\phi)$.
                Since $\alpha$ is invertible, the spectrum of $g$ does not intersect the real axis.
                Using functional calculus, the spectrum may be deformed to $\{\pm i\}$:
                By~\cite[Lemma~2.2.9]{Karoubi68}, there is a homotopy $g_t$ of even $A$-linear operators such that $\alpha_t=\cos \phi + g_t\sin \phi $ satisfies \ref{item1} and \ref{item2} with $g_0=g$ and $g_1^2=-1$.
                \ref{item2} implies that $g_t$ anticommutes with $f$.
                Setting $f_2:=g_1f$ yields another odd graded $A$-linear operator that anticommutes with $f$ and squares to $-1$, and so a second graded $A \grotimes \Cl_{-2}$-module structure lifting the given $A \grotimes \Cl_{-1}$-module.
                Furthermore $\alpha_1f\alpha_1^{-1}=\cos(2\phi)f+\sin(2\phi)f_2$ by the same computation as \eqref{eq:trig}.
                Comparing with Equation~\eqref{eq:t0}, we see that the triple 
                \begin{equation*}
                    [M,\cos(\theta)f+\sin(\theta)f_1,\alpha(\theta) f \alpha(\theta)^{-1}] = [M,\cos(\theta)f+\sin(\theta)f_1,\cos(\theta)f+\sin(\theta)f_2]
                \end{equation*}
                lies in the image of $t$.
                This shows that $t$ is surjective.\qedhere
    \end{enumerate}
\end{proof}

\begin{proof}[Proof of Theorem~\ref{thm:gradedbottperiodicity}]
Using Equation~\eqref{eq:tandalpha}, $\overline{\alpha}_{A \grotimes \Cl_{-1}}$ is an equivalence if we can show $\Omega t_A$ is an equivalence.
By setting $A = B \grotimes \Cl_{+1}$ and using the Morita equivalence $\Cl_{1,1} \Morita \R$ this would then prove that $\overline{\alpha}_{B}$ is an equivalence for all $B$.
By Lemma~\ref{lem:eqoffunctors}, it is sufficient to prove that $\Omega t_A\colon \Omega k^{\ABS}_{A\grotimes \Cl_{-1}}\to \Omega^2 k_A^{\ABS}$ induces an isomorphism on $\pi_0$.

    By Lemma~\ref{lem:tandalphasquare} we need to show that 
    \begin{equation*}
                \Omega t_A\colon \fib\left(i_-^{\gp}\colon k^{C_2}_{A\grotimes \Cl_{-2}} \to k^{C_2}_{A\grotimes \Cl_{-1}}\right) \to \fib\left(k^{C_2}_{A\grotimes \Cl_{-1}}(D^1) \to k^{C_2}_{A\grotimes \Cl_{-1}}(S^0)\times_{k^{C_2}_A(S^0)} k^{C_2}_A(D^1)\right)
    \end{equation*}
            induces an isomorphism on $\pi_0$.
            Using the explicit expression for $\pi_0 \Omega t_A$ derived in Equation~\eqref{eq:t0}, we are done by Karoubi's result \ref{th:karoubishift}.
    \end{proof}

\section{Applications}
\label{sec:applications}
\subsection{The bootstrap class and K\"{u}nneth formula}
\label{sec:bootstrap}

In this and the following sections, graded Banach algebras are not assumed to be unital, see Section~\ref{sec:nonunital}.
Let $A,B$ be graded Banach algebras.
It is natural to ask when the natural map
\begin{equation}
    \label{eq:kuennethmap}
    K^{\mathrm{gr}}_A \otimes_{KO} K^{\mathrm{gr}}_B \to K^{\mathrm{gr}}_{A\grotimes B}
\end{equation}
is an equivalence of spectra.
If \eqref{eq:kuennethmap} is an equivalence, we say that \emph{the K\"{u}nneth formula} holds for $A,B$.
\begin{definition}
    Let $A$ be a graded Banach algebra.
    $A$ is said to lie in the \emph{graded bootstrap class} if the K\"{u}nneth formula~\eqref{eq:kuennethmap} holds for every graded Banach algebra $B$.
\end{definition}
There is a version of the previous definition for ungraded Banach algebras, in which case we will say that $A$ lies in the \emph{ungraded bootstrap class}.
The corresponding notion for complex nuclear $C^*$-algebras has been thoroughly investigated, see~\cite[Def.~22.3.4 and Sec.~23.1]{MR1656031} for further reference and~\cite{uuye2012notekunneththeoremnonnuclear} in the nonnuclear case.
For a modern account, see~\cite[3.2]{bunke2023survey}.
\begin{remark}
    For any commutative ring spectrum $R$ and $R$-modules $M,N$ there is a strongly convergent spectral sequence~\cite[Ch.~IV, Thm.~4.1]{EKMM},\cite[Prop.~7.2.1.19]{HA}
    \begin{equation}
        E^2_{p,q}=\Tor_{p,q}^{\pi_* R}(\pi_*M,\pi_*N) \Rightarrow \pi_{p+q}(M\otimes_R N) \,,
    \end{equation}
    which can be used to compute the homotopy groups of $K^{\mathrm{gr}}_A\otimes_{KO} K^{\mathrm{gr}}_B$.

    The above spectral sequence simplifies drastically when working with the $K$-theory of complex (graded) Banach algebras or general modules over $KU$, since $\pi_*KU=\Z[u,u^{-1}]$ is a principal ideal domain.
    Hence, $\Tor_{p,q}=0$ for $p\geq 2$ and the spectral sequence collapses on the $E^2$-page to yield the short exact sequence
    \begin{equation*}
        0 \to \pi_*M\otimes_{\pi_* KU} \pi_*N \to \pi_*(M\otimes_{KU}N) \to \Tor^{\pi_*KU}_{1,*}(M,N) \to 0 \,.
    \end{equation*}
    For $M=K_A$ and $N=K_B$ this is classically known as the K\"{u}nneth theorem for ungraded complex $K$-theory, cf.\ \cite[Thm.~23.1.3]{MR1656031}.
\end{remark}

\begin{lemma}
    \label{lem:ungradedbootstrap}
    Let $A$ be an ungraded algebra that lies in the ungraded bootstrap class.
    Then $A$, regarded as a graded algebra, also lies in the \emph{graded} bootstrap class.
\end{lemma}
\begin{proof}
    By assumption, we have $K_A\otimes_{KO} K_B \simeq K_{A\otimes B}$ for all ungraded Banach algebras $B$.
    Note that there is a natural equivalence $K_A\simeq K^{\mathrm{gr}}_A$.
    Let $B$ be a graded Banach algebra.
    There is a bifiber sequence of $KO$-modules
    \begin{equation*}
        K^{C_2}_{B\grotimes \Cl_{-1}} \to K^{C_2}_B \to K^{\mathrm{gr}}_B \,.
    \end{equation*}
    Firstly, the map $K_A\otimes_{KO} K^{C_2}_B \to K^{C_2}_{A\grotimes B}$ is an equivalence.
    This follows since we can naturally identify $K^{C_2}_B$ with $K_{|B\grotimes \Cl_{+1}|}$ as $KO$-modules and since $A$ lies in the ungraded bootstrap class by commutativity of the following diagram:
    \begin{equation*}
        \begin{tikzcd}
            K_A \otimes_{KO} K^{C_2}_B \ar[d,"\simeq"]\ar[rr]& &K^{C_2}_{A\grotimes B} \ar[d,"\simeq"] \\
            K_A \otimes_{KO} K_{|B\grotimes \Cl_{+1}|} \ar[r,"\simeq"] & K_{A\otimes|B\grotimes \Cl_{+1}|}\ar[r,"\simeq"]&K_{|A\grotimes B \grotimes \Cl_{+1}|}
        \end{tikzcd} \,.
    \end{equation*}
    Secondly, we obtain the following diagram, in which the horizontal sequences are bifiber sequences and the two vertical maps on the left are equivalences by the above:
    \begin{equation*}
        \begin{tikzcd}
            K_A\otimes_{KO} K^{C_2}_{B\grotimes \Cl_{-1}} \ar[r]\ar[d] & K_A\otimes_{KO} K^{C_2}_{B} \ar[r]\ar[d] & K_A\otimes_{KO} K^{\mathrm{gr}}_B \ar[d] \\
            K^{C_2}_{A\grotimes B \grotimes \Cl_{-1}} \ar[r]& \ar[r] K^{C_2}_{A\grotimes B} & K^{\mathrm{gr}}_{A\grotimes B}
        \end{tikzcd} \,.
    \end{equation*}
    This proves that $K^{\mathrm{gr}}_A\otimes_{KO} K^{\mathrm{gr}}_B \simeq K^{\mathrm{gr}}_{A\grotimes B}$ and hence $A$ lies in the graded bootstrap class.
\end{proof}

We do not expect the following proposition to be sharp.
\begin{proposition}
    The graded bootstrap class is closed under Morita equivalence, filtered colimits, finite direct sums and under $2$-out-of-$3$ for linearly split short exact sequences.
    It contains the ungraded bootstrap class.
    It contains all finite-dimensional Banach algebras.
\end{proposition}
\begin{proof}
    Morita-stability is clear.
    Stability under extensions, ideals and quotients for linearly split short exact sequences follows from the fact that $K^{\mathrm{gr}}$ is exact and that $-\grotimes B$ preserves these sequences.
    Theorem~\ref{thm:periodicKtheorylax} proves that $K^{\mathrm{gr}}$ commutes with filtered colimits. 
    Using that $-\grotimes B$ commutes with filtered colimits proves stability of the K\"{u}nneth formula under filtered colimits.
\end{proof}
\begin{remark}
    The bootstrap class heavily depends on the choice of tensor product via Equation~\eqref{eq:kuennethmap}.
    The $C^*$-tensor products $\otimes_{\mathrm{min}},\otimes_{\mathrm{max}}$ and their graded variants have preferable properties and larger bootstrap classes~\cite[Sec.~23.1]{MR1656031}.
\end{remark}
\subsection{Wood sequence and finite-dimensional K\"unneth theorem}
\label{sec:Wood}
We prove a version of the Wood sequence for graded Banach algebras $A$.
For a homotopical proof in the case $A =\R$ see~\cite{MR3515195}, and for an equivariant version see~\cite[Section 9.2]{MR3570153}.
A version for the $K$-theory of $C(X)$ is in~\cite[III, Thm.~5.18]{karoubi_k-theory_1978}.
A version for ungraded real $C^*$-algebras is~\cite[Theorem 1.18.]{MR1935138}, and a version for Kasparov theory is proven in~\cite[Section 3.2]{guerin2019exact}.

\begin{theorem}[Wood sequence]
    \label{thm:wood}
    Let $A$ be a graded Banach algebra.
    Then there is a natural bifiber sequence of $KO$-modules
    \begin{equation*}
        \Sigma K^{\mathrm{gr}}_A \xrightarrow{\eta} K^{\mathrm{gr}}_A \xrightarrow{c} K^{\mathrm{gr}}_{A\otimes_\R \C} \,,
    \end{equation*}
    where $c$ is induced from complexification and $\eta$ is multiplication by $\eta\in \pi_1KO \cong \pi_1\mathbb{S}$.
    Moreover, using complex Bott periodicity and rotating, we obtain
    \begin{equation*}
        \Sigma^{-2}K^{\mathrm{gr}}_A \xrightarrow{c} \Sigma^{-2}K^{\mathrm{gr}}_{A\otimes_\R \C}\simeq  K^{\mathrm{gr}}_{A\otimes_\R \C} \xrightarrow{r} K^{\mathrm{gr}}_A \,,
    \end{equation*}
    where the map $r$ is induced by forgetting the complex structure.
\end{theorem}

We will prove the above theorem by showing that $\C$ lies in the bootstrap class.
This we will deduce from the ungraded case, as the ungraded Wood sequence is a quick corollary of our setup:
\begin{lemma}[Ungraded Wood sequence]
    Let $A$ be an ungraded Banach algebra.
    Then there is a natural bifiber sequence of $KO$-modules
    \begin{equation*}
        \Sigma K_A \to K_A \to K_{A\otimes_\R \C} \,.
    \end{equation*}
\end{lemma}
\begin{proof}
    Consider the bifiber sequence of $KO$-modules
    \begin{equation*}
        K^{C_2}_{B\grotimes \Cl_{-1}} \to K^{C_2}_B \to K^{\mathrm{gr}}_B \,,
    \end{equation*}
    and plug in $B=A\grotimes \Cl_{-1}$.
    Then, this is naturally equivalent to the bifiber sequence
    \begin{equation*}
        K_{|A\grotimes \Cl_{-1}|} \to K_{|A|} \to \Sigma^{-1}K^{\mathrm{gr}}_{A} \,.
    \end{equation*}
    Upon rewriting $|A\grotimes \Cl_{-1}|\cong A\otimes \C$, we obtain
    \begin{equation}
        \label{eq:woodreal}
        K_{A\otimes \C} \xrightarrow{r} K_A \to \Sigma^{-1}K_A \,.
    \end{equation}
    It is easy to see by tracing through the identification that $r$ is induced from forgetting the complex structure.
    This is the rotated version of the Wood sequence.
    Consider the bifiber sequence of $KO$-modules coming from Karoubi $K$-theory (Section \ref{sec:karoubikt-h}):
    \begin{equation*}
        K^{\mathrm{gr}}_{B^{\op}} \to K_B^{C_2} \to K_{|B|} \,.
    \end{equation*}
    Apply this to $B=A\grotimes \Cl_{-1}$ to obtain
    \begin{equation}
        \label{eq:woodcomplex}
        \Sigma K_A\simeq \Sigma K_{A^{\op}} \to K_A \xrightarrow{c} K_{A\otimes \C}.
    \end{equation}
    Here we used that for ungraded algebras $K_A\simeq K_{A^{\op}}$ using the equivalence $\ModEn_A^{\cong} \to \ModEn_{A^{\op}}^{\cong}$ given by taking the adjoint $M \mapsto \Hom_A(M,A)$.
    It is again easy to see that the map $c$ is induced from complexification.

    Finally, we argue that the other map in \eqref{eq:woodcomplex} is induced from $\eta$.
    The case of \eqref{eq:woodreal} is analogous.
    For the special case $A=\R$, there are only two maps of $KO$-modules $\Sigma KO \to KO$ and hence it must coincide with $\eta$.
    For arbitrary $A$ we use naturality.
    Complexification on module categories fits into a commutative diagram:
    \begin{equation*}
        \begin{tikzcd}
            \ModEn_A \times \VectEn_{\R} \ar[r,"\id\times c"]\ar[d,"\otimes"] & \ModEn_A\times \VectEn_{\C} \ar[d,"\otimes"] \\
            \ModEn_A \ar[r,"c"]& \ModEn_{A\otimes \C} 
        \end{tikzcd} \,.
    \end{equation*}
    Upon applying $K$, this induces a map of bifiber sequences where the vertical maps are just multiplication:
    \begin{equation}
    \label{eq:WoodKU}
        \begin{tikzcd}
            K_A \otimes_{KO} \Sigma KO \ar[r,"\id\otimes \eta"]\ar[d,"\simeq"]& K_A\otimes_{KO} KO \ar[d,"\simeq"] \ar[r,"\id\otimes c"] & K_A\otimes_{KO} KU \ar[d] \\
            \Sigma K_A \ar[r]&K_A \ar[r,"c"]& K_{A\otimes \C}
        \end{tikzcd} \,.
    \end{equation}
    Since the cofiber of the zero map $\Sigma K_{\R} \to K_{\R}$ is not $KU$, this identifies $\Sigma K_A \to K_A$ with multiplication by $\eta$.
\end{proof}
\begin{corollary}
    \label{cor:complexnumbersbootstrap}
    The complex numbers $\C$ lie in the ungraded bootstrap class.
\end{corollary}
\begin{proof}
    We can use diagram~\eqref{eq:WoodKU}.
    Since the middle and left vertical maps are clearly equivalences, the right one is an equivalence, too.
\end{proof}
\begin{corollary}
    \label{cor:complexcliffordbootstrap}
    The complex Clifford algebras lie in the graded bootstrap class.
\end{corollary}
\begin{proof}
    We have that $\C$ is in the graded bootstrap class by combining Corollary~\ref{cor:complexnumbersbootstrap} with Lemma~\ref{lem:ungradedbootstrap}.
    The bootstrap class is closed under tensor product, so $\C\mathrm{l}_n=\C\grotimes \Cl_{+n}$ lies in the bootstrap class.
\end{proof}

\begin{corollary}
        \label{cor:periodicKunneth}
        For every finite-dimensional graded Banach algebra $B$ and any graded Banach algebra $A$, we have the K\"{u}nneth formula
        \begin{equation}
            K^{\mathrm{gr}}_{A\grotimes B} \simeq K^{\mathrm{gr}}_A \otimes_{KO} K^{\mathrm{gr}}_B \,.
        \end{equation}
\end{corollary}
\begin{proof}
    To show that all finite-dimensional semisimple algebras are bootstrap, we use Wedderburn--Artin~\ref{prop:wedderburn}, and stability under direct sums and Morita equivalence, to reduce to proving the K\"{u}nneth formula~\eqref{eq:kuennethmap} for the ten graded division algebras.
    We thus need a K\"{u}nneth theorem for the real and complex Clifford algebras.
    The real case follows from Corollary~\ref{cor:kuennethforclifford}. 
    The complex case follows from Corollary~\ref{cor:complexcliffordbootstrap}.
    Since $K^{\mathrm{gr}}_A\simeq K^{\mathrm{gr}}_{A/\mathrm{rad}(A)}$ by a proof similar to Proposition~\ref{prop:semisimplekthy}, we obtain that $A$ is bootstrap if and only if the semisimplification $A/\mathrm{rad}(A)$ is bootstrap.
\end{proof}

\subsection{Bott clock}
\label{sec:Bottclock}
Atiyah, Bott and Shapiro~\cite{ABS} computed the groups $\pi_n KO$ in terms of Clifford module quotients by explicitly constructing an isomorphism to $\widetilde{KO}^0(S^n)$.
This isomorphism refines to the level of spectra as $\Omega^n KO \cong K_{\Cl_{-n}}^{\mathrm{gr}}$ by Proposition~\ref{prop:clifshift}, further clarifying the relationship between Clifford algebras and $KO$-theory that has generated substantial discussion~\cite{MOhenriquesKOn, MOmathewclifford}.
In this section we in particular recover the loop spaces of $KO$ that Bott identified via Morse theory~\cite[IV.24]{MR163331} from representation theory of Clifford algebras.

There is a cofiber space of infinite loop spaces from the ABS Definition~\ref{def:ABSKth} of graded $K$-theory and Corollary~\ref{cor:ktheoryomegaspectrum}.
We can improve upon one of the key insights of~\cite{ABS}, namely upon applying $\pi_0$ to \eqref{eq:ABScofibseq}, we obtain $\pi_nko$ as the quotient $A_n$ of Grothendieck groups of Clifford modules. 
\begin{equation}
    \label{eq:ABScofibseq}
    k(\ModEn^{C_2}_{\Cl_{-n-1}})\to k(\ModEn^{C_2}_{\Cl_{-n}})\to  \Omega^n ko \,.
\end{equation}
Another useful sequence is the Karoubi fiber sequence of infinite loop spaces (hence also of spaces)
\begin{equation}
    \label{eq:Karoubifibseq}
    \Omega^nko\to k(\ModEn_{|\Cl_{n+1}|}) \to k(\ModEn_{|\Cl_{n}|}) \,,
\end{equation}
which allows to efficiently calculate all loop spaces $\Omega^nko=\Omega^n (\Z\times BO)$ as in Figure~\ref{fig:Cliffordtable} from the representation theory of $|\Cl_n|$.
We provide one example:
\begin{example}
    Using the identification
        \begin{equation*}
        \begin{tikzcd}
            \ModEn_{\Cl_{+4}}\ar[r]\ar[d,"\cong"]& \ModEn_{\Cl_{+3}}\ar[d,"\cong"] \\
            \VectEn_{\mathbb{H}} \ar[r]& \VectEn_\C  
        \end{tikzcd} \,,
        \end{equation*}
        we can easily continue this to a fiber sequence after group completion:
        \begin{equation*}
            U/Sp \to \Z\times BSp  \to \Z\times BU \,.
        \end{equation*}
        We can read off $k^{\Kar}_{\Cl_{+3}}\simeq U/Sp \simeq \Omega^3 ko$.
\end{example}
\begin{figure}[h!]
    \centering
    {
    \setlength{\tabcolsep}{4pt}
    \begin{tabular}{c|c|c|cc|cc|c}
     $n$ & ungraded $|\Cl_{n}|$ & $\Omega^{n} (\Z\times BO)$ & \makecell{\small \# ungraded\\irreps} & dimension & \makecell{\small \# graded\\irreps} & sdimension & $A_n$  \\
     \hline
     $0$ & $\R$ & $\Z\times BO$ & 1 & 1 & 2 & $\{1|0,0|1\}$ & $ \Z$ \\
     $-1$ & $\C$ & $U/O$ & 1 & 2 & 1 & $1|1$ & $ \Z/2$ \\
     $-2$ & $\mathbb{H}$ & $Sp/U$ & 1 & 4 & 1 & $2|2$ & $ \Z/2$ \\
     $-3$ & $\mathbb{H}\oplus \mathbb{H}$ & $Sp$ & 2 & 4 & 1 & $4|4$ & $ 0$ \\
     $-4$ & $M_2(\mathbb{H})$ & $\Z\times Sp/(Sp\times Sp)$ & 1 & 8 & 2 & $\{4|4,4|4\}$ & $ \Z$ \\
     $-5$ & $M_4(\C)$ & $U/Sp$ & 1 & 8 & 1 & $8|8$ & $ 0$ \\
     $-6$ & $M_8(\R)$ & $O/U$ & 1 & 8 & 1 & $8|8$ & $ 0$\\
     $-7$ & $M_8(\R)^{\oplus 2}$ & $O$ & 2 & 8 & 1 & $8|8$ & $ 0$ \\
     $-8$ & $M_{16}(\R)$ & $\Z\times BO$ & 1 & 16 & 2 & $8|8$ & $\Z$
    \end{tabular}}
    \caption{A table of Clifford algebras, their underlying ungraded algebras $|\Cl_n|$ and their graded and ungraded representations. There is a shift by one between graded and ungraded representations.
    One can read off the iterated loop spaces of the orthogonal group from the Clifford algebras. Making this precise is the subject of Section~\ref{sec:Bottclock}.}
    \label{fig:Cliffordtable}
\end{figure}
\begin{example}
    \label{ex:ringstructureko}
    One can also recover the ring structure on $\pi_*ko$ from Clifford modules, as first shown in~\cite[Thm.~6.9]{ABS}.
    Figure~\ref{fig:Cliffordtable} indicates the graded abelian group $\pi_*ko\cong A_*$, with ring structure
    \begin{equation*}
        \pi_*ko= \Z[\eta,\alpha,\beta]/(2\eta, \eta^3,\alpha\eta,\alpha^2=4\beta) \qquad |\eta|=1, |\alpha|=4, |\beta|=8 \,.
    \end{equation*}
    We indicate the reasoning:
    The generator $\eta$ of $A_1\cong \pi_1ko$ comes from an irreducible $\Cl_{-1}$-module $M_\eta$.
    Its square $M_{\eta}\grotimes M_{\eta}$ is an irreducible $\Cl_{-2}$-module, since its dimension is $2|2$.
    Let $\alpha\in \pi_4ko$ be a generator corresponding to an irreducible $\Cl_{-4}$-module $M_\alpha$ of dimension $4|4$.
    Its square has dimension $32|32$ and hence must be four copies of the irreducible $\Cl_{-8}$-module $M_\beta$ corresponding to the generator $\beta\in \pi_8 ko$.
\end{example} 
\subsection{ABS long exact sequence}
\label{sec:cofib}
Let $A$ be a graded Banach algebra.
Theorem~\ref{th:mastercomparison}, which is a slight reformulation of the definition of $k^{\ABS}_A$, produces a map of cofiber sequences, where the first two vertical maps are connective covers:
\begin{equation*}
\begin{tikzcd}
    k_{|A|} \ar[d] \ar[r] & k_{|A\grotimes \Cl_{+1}|} \ar[r] \ar[d] & k^{\ABS}_A \ar[d]
    \\
    K_{|A|} \ar[r] & K_{|A\grotimes \Cl_{+1}|} \ar[r] & K^{\gr}_A
\end{tikzcd} \,.
\end{equation*}
This induces long exact sequences
\begin{equation}
\label{eq:comparequotient}
\begin{tikzcd}
    \dots \ar[r]  & k_0(|A|) \ar[d,"f_0"] \ar[r,"d_1'"] & k_0(|A \grotimes \Cl_{+1}|) \ar[r,"d_2'"] \ar[d, "f_1"] & k_0^{\ABS}(A) \ar[r] \ar[d,"f_2"] & 0 \ar[d] \ar[r] &  \dots
    \\ 
    \dots  \ar[r] & \ar[r,"d_1"] K_0(|A|) & K_0(|A \grotimes \Cl_{+1}|) \ar[r, "d_2"] & K_0^{\ABS}(A) \ar[r, "d_3"]  & K_{-1}(|A|) \ar[r, "d_4"] & \dots
\end{tikzcd} \,,
\end{equation}
where we denoted $k_i(A) := \pi_i k_A$ here for notational consistency.
In the diagram $f_0$ and $f_1$ are isomorphisms.

\begin{lemma}
There is a long exact sequence
\begin{equation*}
0 \to k_0^{\ABS}(A) \xrightarrow{f_2} K_0^{\gr}(A) \xrightarrow{d_3} K_{-1}(|A|) \xrightarrow{d_4} K_{-1}(|A \grotimes \Cl_{+1}|)  \to \dots
\end{equation*}
\end{lemma}
\begin{proof}
Exactness everywhere except the first spot is clear.
    The top exact sequence says that the cokernel of $d_1'$ is identified with $k_0^{\ABS}$.
    The isomorphisms $f_1$ and $f_0$ identify this cokernel with the cokernel of $d_1$, which is identified with the kernel of $d_2$ by exactness.
    Hence $f_2$ gives an isomorphism $k_0^{\ABS} \cong \ker d_2$.
\end{proof}

We can conclude that $f_2$ is an inclusion
\begin{equation*}
\frac{\pi_0(k_{A}^{C_2})}{\pi_0(k_{A \grotimes \Cl_{-1}}^{C_2})} \hookrightarrow \pi_0 K^{\gr}_A
\end{equation*}
from the quotient of isomorphism classes of graded $A$-modules by graded $A \grotimes \Cl_{-1}$-modules into the graded $K$-theory of $A$, compare~\cite{MR2490588}.
In many cases (though not always) the inclusion $f_2$ is an isomorphism:

\begin{corollary}
\label{cor:quotientcondition}
    The following are equivalent:
\begin{enumerate}
        \item $f_2$ is an isomorphism
        \item $d_3 = 0$
        \item $d_4$ is injective
    \end{enumerate}
\end{corollary}

In particular, if the map
\begin{equation*}
K_{|A|} \to K_A^{C_2}\simeq K_{|A\grotimes \Cl_{+1}|}
\end{equation*}
of $KO$-modules admits a retraction, then the connecting map $d_3$ is zero and so $f_2$ is an isomorphism.

\begin{example}
    Let $A$ be an ungraded $C^*$-algebra. 
    Then $K_A \to K_A^{C_2} \cong K_A \oplus K_A$ is the diagonal map (see Example~\ref{example:ungraded}), and so the fiber sequence is split.
\end{example}

\begin{example}
    If $A$ is inner graded, then 
    \begin{equation*}
    |A \grotimes \Cl_{+1}| \cong |A| \otimes |\Cl_{+1}| \cong |A| \oplus |A|.
    \end{equation*}
    Under this identification, the homomorphism $|A| \to |A| \oplus |A|$ is the diagonal.
    Therefore as in the last example $K_{|A|} \to K_A^{C_2} \cong K_{|A|} \oplus K_{|A|}$ has a retraction and so
    \begin{equation*}
    \frac{\pi_0(k_{A}^{C_2})}{\pi_0(k_{A \grotimes \Cl_{-1}}^{C_2})} \cong \pi_0 K^{\gr}_A
    \end{equation*}
\end{example}

\begin{example}
If $A$ is a finite-dimensional graded $C^*$-algebra, then $f_2$ is an isomorphism.
Namely, $A$ is a finite direct sum of matrix algebras over graded division algebras by Proposition~\ref{prop:wedderburn}, so it suffices to consider the case where $A$ is a real or complex Clifford algebra.
This follows from the concrete descriptions of the comparison maps in Section~\ref{sec:Bottclock}.
\end{example}

The following example shows that the map $\overline{\alpha}_A\colon k^{\ABS}_A\to \Omega k^{\ABS}_{A\grotimes \Cl_{1}}$ is generally not an equivalence, compare~\cite[Example 3.61]{luukmasters}.
It also follows that $k^{\ABS}$ is not excisive.
\begin{example}
\label{ex:circle}
    Working over the complex numbers, we consider $A=C(S^1)\grotimes {\Cxl_{-1}}$.
    We then have 
    \begin{equation*}
        k^{C_2}_A\simeq k_{|C(S^1)\grotimes \Cxl_{1,1}|} \simeq ku(S^1) \,,
    \end{equation*}
    where $ku(S^1)=\underline{\Hom}_{\Sp_{\geq 0}}(\Sigma_+^\infty S^1,ku)$.
    Similarly,
    \begin{equation*}
        k^{C_2}_{A\grotimes \Cxl_{-1}}\simeq k^{C_2}(S^1)\simeq ku(S^1)^{\times 2} \,.
    \end{equation*}
    The map $i_-: k^{C_2}_{A\grotimes \Cxl_{-1}} \to  k^{C_2}_{A}$ can be identified with the sum $\oplus\colon ku(S^1)\times ku(S^1)\to ku(S^1)$.
    To see this, let $V$ be the unique graded irreducible representation of $\Cxl_{-2}$ inducing a Morita equivalence to $\C$.
    As a vector space $V=\C\oplus \Pi \C\cong \Cxl_{-1}$ with the canonical $\Cxl_{-2}\cong \Cxl_{-1}\grotimes \Cxl_{-1}^{\op}$-action. 
    \begin{equation*}
        \begin{tikzcd}
            \VectEn(S^1)^{\times 2} \ar[r,"\simeq"'] & \sVectEn(S^1) \ar[r,"-\grotimes V","\simeq"'] & \ModEn_{\Cxl_{-2}}^{C_2}(S^1) \ar[dd,bend left=15,"i_-"] \\
            (M,N) \ar[r,mapsto]& M\oplus \Pi N & \\
            &\VectEn(S^1) & \ModEn_{\Cxl_{-1}}^{C_2}(S^1) \ar[l,"\text{even part}","\simeq"']
        \end{tikzcd} \,.
    \end{equation*}
    It is easy to identify the even part $(M\oplus \Pi N)\grotimes V = M\oplus N$.
    In particular, we get that 
    \begin{equation*}
        k^{\ABS}_{A}\simeq \cofib(ku(S^1)^{\times 2}\to ku(S^1)) \simeq \Sigma ku(S^1) \,.
    \end{equation*}
    Furthermore, using the Morita equivalence $\Cxl_{1,1}\Morita \C$ and the comparison result for ungraded algebras we get:
    \begin{align*}
        k^{\ABS}_{A\grotimes \Cxl_{1}}\simeq k^{\ABS}_{C(S^1)}=k_{C(S^1)} = ku(S^1) \,.
    \end{align*}
    We obtain a $\pi_{\geq 1}$-equivalence
    \begin{equation*}
        \overline{\alpha}_A\colon \Sigma ku(S^1) \to \Omega ku(S^1) \,.
    \end{equation*}
    On $\pi_0$ this cannot be an isomorphism since $\pi_0\Sigma ku(S^1)=0$ and $\pi_0 \Omega ku(S^1)= \pi_0\Omega^2 ku= \Z$.
    This also proves that the map 
    \begin{equation*}
        \Sigma ku(S^1)\simeq k^{\ABS}_{C(S^1)\grotimes \Cxl_{-1}}\to \underline{\Hom}_{\Sp_{\geq 0}}(\Sigma_+^\infty S^1,k^{\ABS}_{\Cxl_{-1}}) \simeq \underline{\Hom}_{\Sp_{\geq 0}}(\Sigma_+^\infty S^1,\Omega ku)\simeq \Omega ku(S^1)
    \end{equation*}
    cannot be an equivalence and that $k^{\ABS}$ does not satisfy excision.
\end{example}

\begin{warning}
    If the equivalent Conditions of Corollary~\ref{cor:quotientcondition} hold for $A$, then they need not hold for $A \grotimes \Cl_{+1}$.
    For example for $A = C(S^1;\C)$ with trivial grading, we saw in Example~\ref{ex:circle} that $k_0^{\ABS}(C(S^1)) = 0$ even though $K_0^{\ABS}(C(S^1;\C)) = \Z$.
    In fact, in this case $d_3$ is an isomorphism with $K_{-1}(A) \cong KU^1(S^1) \cong \Z$.
\end{warning}

\subsection{Isomorphisms with other definitions of graded \texorpdfstring{$K$}{K}-theory}
\label{sec:isomorphismsdefs}

Graded $K$-theory of a graded ($C^*$-)algebra has been defined in several apparently different ways: by Karoubi~\cite{Karoubi68} using two gradings on a single ungraded module, by van Daele~\cite{MR947500, MR961241} using odd self-adjoint unitaries, and by Kasparov as $KK_*(\R,-)$~\cite{MR582160}.
The purpose of this section is to identify each of these with our Definitions~\ref{def:ABSKth} and~\ref{def:periodick}, up to explicit shifts and opposites.

Section~\ref{sec:vD} realizes van Daele $K$-theory as the $K$-theory of the Banach functor $\ModEn^{C_2}_{A \grotimes \Cl_{-1}} \to \ModEn^{C_2}_{A}$, which gives $K^{\mathrm{vD}}_0(A) \cong \pi_1 K^{\ABS}_A$ (Corollary~\ref{cor:vDvsABS}) and, combining with known comparisons of van Daele and Kasparov theory, the isomorphism $KK_i(\R,A) \cong K^{\ABS}_i(A)$ of Corollary~\ref{cor:KK}.
Section~\ref{sec:karoubikt-h} treats Karoubi $K$-theory, which turns out to be a shift of the graded $K$-theory of the opposite algebra (Proposition~\ref{prop:karvD}).
Two elementary tools are needed for these comparisons and are developed first: the \emph{Morita move} of Section~\ref{innergraded}, which trades graded modules for ungraded ones while reversing variance, and the \emph{flopposite} of Section~\ref{sec:flopping}, which converts between the sign conventions $\Cl_{+1}$ and $\Cl_{-1}$ at the cost of an opposite algebra.

Finally, Section~\ref{sec:atlas} assembles all comparison results into Table~\ref{table:mastercomparison}.
The variants of Definitions~\ref{def:ABSKth} and~\ref{def:periodick} obtained by independently interchanging fiber and cofiber, $\Cl_{-1}$ and $\Cl_{+1}$, contravariant and covariant functoriality, and graded and ungraded modules are all identified with shifts of $K^{\mathrm{gr}}_A$ or $K^{\mathrm{gr}}_{A^{\op}}$.
Several of these variants have, as far as we are aware, not been considered before.

\subsubsection{Kasparov \texorpdfstring{$KK$}{KK}-theory and van Daele \texorpdfstring{$K$}{K}-theory}
\label{sec:vD}

Throughout this section $A$ is a unital graded Banach algebra.
In this section we observe that van Daele's $K$-theory~\cite{MR947500, MR961241} is $K$-theory of the Banach functor $\ModEn_{A \grotimes \Cl_{-1}}^{C_2} \to \ModEn_A^{C_2}$, thus leading to a relationship with $k_{A}^{\ABS}$, see~\cite[Corollary 5.15]{MR3516985} and~\cite[Appendix D]{bourne2026analyticindextheoryspectral} for related results.
It was already known how to relate Kasparov's $KK$-theory~\cite{MR582160} to van Daele $K$-theory, and so we obtain a relationship between our approach and Kasparov theory in Corollary~\ref{cor:KK}.

\begin{definition}
\label{def:vDspectrumperiodic}
The \emph{periodic van Daele $K$-theory spectrum} is 
\begin{equation*}
K^{\mathrm{vD}}_A := \fib(K^{C_2}_{A \grotimes \Cl_{-1}} \to K^{C_2}_A).
\end{equation*}
\end{definition}

As already observed in Section~\ref{sec:proofs}, using $\Omega\cofib = \fib$ and $ko$-linearity in the periodic case immediately implies:

\begin{proposition}
\label{prop:ABSvsvD}
For any graded Banach algebra $A$, there are equivalences of spectra
    $k^{\mathrm{vD}}_A \cong \Omega k_A^{\ABS}$ and $K^{\mathrm{vD}}_A \cong \Omega K_A^{\mathrm{gr}}$.    
\end{proposition}

In order to show $\pi_0 K^{\mathrm{vD}}_A$ is isomorphic to van Daele's $K$-theory, we first introduce van Daele's constructions.
For $B$ a graded Banach algebra, we write
\begin{equation*}
\mathcal{F}_B^{\mathrm{vD}}(n)  := \{ x \in M_{n}(B)_{\mathrm{odd}}: x^2 = 1\}.
\end{equation*}
We are especially interested in the case $B = A \grotimes \Cl_{1,1}$.
From sums of graded $A \grotimes \Cl_{1,1}$-modules we get the maps $\mathcal{F}^{\mathrm{vD}}_{A \grotimes \Cl_{1,1}}(n) \times \mathcal{F}^{\mathrm{vD}}_{A \grotimes \Cl_{1,1}}(m) \to \mathcal{F}^{\mathrm{vD}}_{A \grotimes \Cl_{1,1}}(n+m)$ given by block diagonals $(x,y) \mapsto x \oplus y$.
We consider the transition maps $\mathcal{F}^{\mathrm{vD}}_{A \grotimes \Cl_{1,1}}(n) \to \mathcal{F}^{\mathrm{vD}}_{A \grotimes \Cl_{1,1}}(n+1)$ given by $x \mapsto x \oplus e$, where $e \in \mathcal{F}^{\mathrm{vD}}_{A \grotimes \Cl_{1,1}}(1)$ denotes the canonical generator corresponding to the $\Cl_{+1}$-factor.

\begin{definition}
\label{def:vDKgroup}
    The \emph{van Daele $K$-theory group of $A$} is the monoid
\begin{equation*}
DK(A) := \colim_{n} \pi_0 \mathcal{F}_{A \grotimes \Cl_{1,1}}^{\mathrm{vD}}(n)
\end{equation*}
under $\oplus$.
\end{definition}

The following remarks are for readers interested in a more detailed comparison with the literature, see~\cite[Section 5.1]{MR3665214} for a more detailed discussion.

\begin{remark}
The space $\mathcal{F}_{A \grotimes \Cl_{1,1}}^{\mathrm{vD}}(n)$ can be rewritten as matrices $x \in M_{n|n}(A)_{\mathrm{odd}}$ such that $x^2 = 1$ using $\Cl_{1,1} \cong M_{1|1}(\R)$.
     The transition maps are then formed by direct sum with the off-diagonal matrix 
    \begin{equation*}
    \begin{pmatrix}
        0 & 1 \\
        1 & 0
    \end{pmatrix}.
    \end{equation*}
\end{remark}

\begin{remark}
The original definition~\cite[Definition 3.1]{MR947500} of van Daele instead uses $\mathcal{F}_{A \grotimes \Cl_{2,2}}(n)$ in the colimit.
The reason is that he proves in~\cite[Proposition 2.11]{MR947500} that if $e \sim -e$ in $A \grotimes \Cl_{1,1}$, then $DK(A)$ is a group, and this is the case if $A = B \grotimes \Cl_{1,1}$.
However, this turns out to be unnecessary as it will follow from our results that $DK(A)$ is \emph{always a group}.
This can also be seen directly by applying~\cite[Proposition 5.2]{MR3516985} to the even unitary $1 \otimes ef$.
    Using $\mathcal{F}_{A \grotimes \Cl_{2,2}}(n)$ will give the same $K$-theory by $\Cl_{1,1}$-periodicity.
\end{remark}

\begin{remark}
Under the assumption on $A$ that $\mathcal{F}^{\mathrm{vD}}_A(1) \neq \emptyset$ one can define the van Daele $K$-theory group to be the colimit over $\mathcal{F}^{\mathrm{vD}}_A(n)$ using any fixed $e \in \mathcal{F}^{\mathrm{vD}}_A(1)$~\cite[Definition 2.8]{MR947500}.
    Our definition using $\mathcal{F}^{\mathrm{vD}}_{A \grotimes \Cl_{1,1}}$ with the preferred $e \in \mathcal{F}^{\mathrm{vD}}_{A \grotimes \Cl_{1,1}}(1)$ is equivalent in the special case that $\mathcal{F}^{\mathrm{vD}}_A(1) \neq \emptyset$, see~\cite[Proposition 3.3]{MR947500}.
\end{remark}

To show that $DK(A) \cong \pi_0 K^{\mathrm{vD}}_A$ we apply the following lemma, which rewrites Karoubi $K$-theory for systems with a special cofinal system.

\begin{lemma}    
\label{lem:seqcofin}
Let $\phi \colon M \to N$ be a quasi-surjective $\E_\infty$-map between telescopic monoids.
Suppose $e \in M$ is chosen so that $e^{\oplus n}$ is cofinal in $M$.
Then the fiber $K^{\mathrm{Kar}}(\phi)$ of $\phi^{\gp}:M^{\gp} \to N^{\gp}$ is equivalent as an $\E_\infty$-monoid to the colimit $\colim_n \fib_{\phi(e^{\oplus n})} (\phi)$ with connecting maps given by $(-)\oplus e$.
\end{lemma}
\begin{proof}
    \begin{equation*}
    \begin{tikzcd}
        \fib_0 \phi \ar[r] \ar[d, dashed] & M \ar[r,"\phi"] \ar[d,"\oplus e"] & N \ar[d,"\oplus \phi(e)"]
        \\
        \fib_{\phi(e)} \phi \ar[r] \ar[d,dashed] & M \ar[r,"\phi"] \ar[d,"\oplus e"] & N \ar[d,"\oplus \phi(e)"]
    \\
    \fib_{\phi(e^{\oplus 2})} \phi \ar[r] \ar[d,dashed] & M \ar[r,"\phi"] \ar[d, "\oplus e"] & N \ar[d,"\oplus \phi(e)"]
    \\
    \vdots & \vdots & \vdots
    \end{tikzcd} \,,
    \end{equation*}
    where the dashed arrows follow from the universal property of the fiber.
    Since filtered colimits commute with finite limits in $\mathcal{S}$, we obtain a fibration
    \begin{equation*}
    \begin{tikzcd}
    \colim_n \fib_{\phi(e^{\oplus n})} \ar[r] & M[\pi_0M^{-1}] \ar[r] \ar[d, "\simeq"] & N[\pi_0N^{-1}] \ar[d, "\simeq"]
    \\
    & M^{\gp} \ar[r] & N^{\gp}
    \end{tikzcd}.
    \end{equation*}
    Here we used cofinality of $e^{\oplus n}$ and Lemma~\ref{lem:quasi-surjlocalization} to identify the colimits with $M_\infty$ and $N[\pi_0N^{-1}]$, which we could identify with $M^{\gp}$ and $N^{\gp}$ since $M$ and $N$ are telescopic.
    The relation between the Karoubi $K$-theory space and fibers of group completions was shown in Theorem~\ref{thm:fibergroupcompletion}.
\end{proof}

\begin{corollary}
The monoid in sets $\colim_n \pi_0( \fib_{\phi(e^{\oplus n})}(\phi))$ is isomorphic to $K^{\Kar}(\phi)$.
In particular, it is a group.
\end{corollary}
\begin{proof}
    Since $\pi_0$ commutes with filtered colimits, we obtain that 
    \begin{equation*}
    K^{\Kar}(\phi) \cong \pi_0 \fib(M^{\gp} \to N^{\gp}) \cong \colim_n \pi_0 \fib_{\phi(e^{\oplus n})}(\phi) \,,
    \end{equation*}
    which was shown to be a group in Lemma~\ref{lem:triples}.
\end{proof}

We will apply Lemma~\ref{lem:seqcofin} to $\phi \colon N^{\mathrm{hc}}(\ModEn^{C_2,\cong}_{A \grotimes \Cl_{-1}}) \to N^{\mathrm{hc}}(\ModEn^{C_2,\cong}_{A})$ and $e = A \grotimes \Cl_{-1}$.
Observe that $\phi(e) = A^{1|1}$.
We proceed to identify the colimit of the fiber in the case at hand, and verify the hypotheses:

\begin{lemma}
\label{lem:fiberphi}
    The fiber of $\phi$
    over $A^{n|n} \in \ModOne^{C_2}_{A}$ is $\mathcal{F}_{A \grotimes \Cl_{1,1}}^{\mathrm{vD}}(n)$.
\end{lemma}
\begin{proof}
An extension of a graded $A$-module $M$ to a graded $A \grotimes \Cl_{-1}$-module is equivalent to a choice of an odd $T \in \End_A M$ such that $T^2 = -1$.
When $M = A^{n|n}$, this corresponds to a $x \in M_{n|n}(A)$ such that $x^2 = 1$ by the discussion at Convention~\ref{conv:End-op} and $\Cl_{-1}^{\op} \cong \Cl_1$.
Using $\Cl_{1,1} \cong M_{1|1}(\R)$ so that $M_n(A \grotimes \Cl_{1,1}) \cong M_{n|n}(A)$, we can therefore identify points of $\ModOne^{C_2}_{A \grotimes \Cl_{-1}}$ lying over $A^{n|n} \in \ModOne^{C_2}_{A}$ with $\mathcal{F}_{A \grotimes \Cl_{1,1}}^{\mathrm{vD}}(n)$.
\end{proof}

\begin{remark}
    The fact that van Daele's $K$-theory group is defined using odd elements that square to $1$ but the van Daele $K$-theory spectrum is defined using $\Cl_{-1}$ is somewhat unintuitive.
    Since we are interested in real $K$-theory these subtle sign choices are necessary to spot; $KO_i \ncong KO_{-i}$ for $i = 1,2$.
\end{remark}

\begin{lemma}
\label{lem:cofinCliffmods}
The submonoid of $N^{\mathrm{hc}}(\ModEn_{A \grotimes \Cl_{-1}}^{C_2,\cong})$ consisting of $(A \grotimes \Cl_{-1})^n$ for $n \geq 0$ is cofinal.
\end{lemma}
\begin{proof}
    First observe that for any graded Banach algebra $B$ the collection $B^{n|n}$ is cofinal in finitely generated projective graded $B$-modules.
    If moreover $\mathcal{F}^{\mathrm{vD}}_B(1) \neq \emptyset$, then $B^{1|0} \cong B^{1|0}$ and so $B^n$ is already cofinal.
    We now apply these results to $B = A \grotimes \Cl_{-1}$ to conclude.
\end{proof}

\begin{theorem}
\label{th:vDmonoid}
     $\pi_0 K^{\mathrm{vD}}_A$ is isomorphic to the monoid $DK(A)$.
\end{theorem}
\begin{proof}
It follows from Lemma~\ref{lem:i-quasisur} that $\phi$ is quasi-surjective.
We apply Lemma~\ref{lem:seqcofin} to
\begin{equation*}
    \phi \colon N^{\mathrm{hc}}(\ModEn^{C_2,\cong}_{A \grotimes \Cl_{-1}}) \to N^{\mathrm{hc}}(\ModEn^{C_2,\cong}_{A})
\end{equation*}
and $e = A \grotimes \Cl_{-1}$.
We have shown in Lemma~\ref{lem:cofinCliffmods} that $e^{\oplus n}$ is a cofinal system.
In Lemma~\ref{lem:fiberphi} we identified the fiber of $\phi$ over $\phi(e^{\oplus n})$ with $\mathcal{F}^{\mathrm{vD}}_{A \grotimes \Cl_{1,1}}(n)$, and that identification intertwines the respective connecting maps.
We have thus identified $V$ with $M_A^{\mathrm{vD}}$.

\end{proof}

Combining Proposition~\ref{prop:ABSvsvD} and Theorem~\ref{th:vDmonoid} we obtain

\begin{corollary}
\label{cor:vDvsABS}
There is an isomorphism of groups
    $K^{\mathrm{vD}}_0(A) \cong \pi_1 K^{\ABS}_A$.
\end{corollary}

\begin{remark}
    If we define higher van Daele $K$-groups as $K^{\mathrm{vD}}_i(A) := K^{\mathrm{vD}}(A \grotimes \Cl_{-i})$, then $K^{\mathrm{vD}}_i(A) \cong \pi_i K_A^{\mathrm{vD}}$ by our spectral suspension isomorphism \ref{prop:clifshift}.
    We thus recover the suspension theorem on the level of groups proven by van Daele in~\cite{MR961241}.
\end{remark}

\begin{corollary}
\label{cor:KK}
For all (real) graded $C^*$-algebras $A$ there are isomorphisms of groups
    $KK_i(\R,A) \cong K^{\ABS}_i(A)$ for all $i \in \Z$.
\end{corollary}
\begin{proof}
Recall that $KK_{q-p}(\R,A) = KK_0(\R,A \grotimes \Cl_{p,q})$~\cite{MR1267059}. 
It follows that
\begin{equation*}
    K^{\ABS}_{q-p}(A) \cong K^{\ABS}_{0}(A \grotimes \Cl_{p,q})
\end{equation*}
from Corollary~\ref{prop:clifshift} on $\pi_0$.
It therefore suffices to check $i = 1$.
    It was shown by Roe in~\cite{MR2082096} that $KK_1(\R,A) \cong K^{\mathrm{vD}}_{0}(A)$, see \cite[Theorem 5.11]{MR3516985},~\cite{MR4140811} and~\cite[Theorem 2]{MR4121611} for details.
    We are done by Corollary~\ref{cor:vDvsABS}.
\end{proof}

\subsubsection{The Morita move}
\label{innergraded}

Let $A$ be a graded Banach algebra.
We will use the Morita triviality $\Cl_{1,1} \Morita \R$ to relate with definitions of graded $K$-theory that involve ungraded modules, such as Karoubi $K$-theory in Section~\ref{sec:karoubikt-h}.

The main observation (Lemma~\ref{lem:covVScontr}) will be that there is a commutative diagram of Banach categories 
\begin{equation}
\label{eq:psiA}
\begin{tikzcd}
    \ModEn_{|A \grotimes \Cl_{+1}|}   & \ModEn_{|A|} \ar[d, "\Psi_A"] \ar[l,"(\R \to \Cl_{+1})_*", swap]
    \\
    & \ModEn_{|A \grotimes \Cl_{+1} \grotimes \Cl_{-1}|}\ar[ul,"(\R \to \Cl_{-1})^*"]
\end{tikzcd} \,,
\end{equation}
which we will call the \emph{Morita move}.
Note that the Morita move additionally exchanges the covariant functoriality of $\R \to \Cl_{\pm 1}$ for the contravariant functoriality of $\R \to \Cl_{\mp 1}$.
$\Psi_A$ is an equivalence of categories.
Intuitively, $\Psi_A$ is defined using the fact that $\Cl_{+1} \grotimes \Cl_{-1} \cong \Cl_{1,1} \cong M_{1|1}(\R)$ is Morita trivial.
However, we have to carefully distinguish graded and ungraded Morita equivalence, for which it is convenient to use the theory of inner gradings.

\begin{definition}
    A graded algebra $A$ is called \emph{inner graded} (also called \emph{evenly graded}) if there is an element $g\in A$ with $g=g^{-1}$ and $gag^{-1}=(-1)^{|a|}a$ is the grading operator for all homogeneous elements in $A$.
    If $A$ is a $C^*$-algebra, we require $g$ to be self-adjoint.
\end{definition}

\begin{lemma}
\label{lem:inneriso} 
    If $B$ is inner graded, then there is an isomorphism of ungraded algebras $|A\grotimes B|\cong |A|\otimes |B|$ sending $a\grotimes b \mapsto a \otimes g^{|a|} b$, natural in $A$ for graded homomorphisms ($*$-homomorphisms in case $A$ is $C^*$).
\end{lemma}
\begin{proof}
    The proof is easy, see~\cite[Proposition 14.5.1]{MR1656031}, \cite[Example 1.5]{MR2058474} and~\cite[Lemma 3.4]{MR3665214}.
\end{proof}

\begin{example}
    \label{ex:Cl11inner}
    The Clifford algebra $\Cl_{1,1}$ is inner graded.
    More explicitly, if $e,f$ are anticommuting odd generators with $e^2=1,f^2=-1$, then $g=ef$ provides an inner grading.
    Recall that $\Cl_{1,1}\cong \End(\R^{1|1})$ explicitly by 
    \begin{equation*}
    e \mapsto
    \begin{pmatrix}
    0 & 1 \\
    1 & 0
    \end{pmatrix} \quad 
    f \mapsto
    \begin{pmatrix}
        0 & -1 \\
        1 & 0
    \end{pmatrix} \,,
    \end{equation*}
    so that
    \begin{equation*}
    ef \mapsto 
    \begin{pmatrix}
        1 & 0 \\
        0 & -1
    \end{pmatrix}
    \end{equation*}
    is the standard grading operator on $\R^{1|1}$.
    The isomorphism 
    \begin{equation*}
        |A\grotimes\Cl_{1,1}|\cong |A|\otimes M_2(\R)\cong M_2(|A|)
    \end{equation*}
    is given by
    \begin{equation*}
    a \otimes x \mapsto 
    \begin{cases}
    a \otimes x & |a| = 0
    \\
    a \otimes ef x & |a| = 1,
    \end{cases} \,,
    \end{equation*}
    which is natural in graded ($*$-)homomorphisms, also see~\cite[Corollary 14.5.2 and Corollary 14.5.5]{MR1656031} and~\cite[Proposition 3.5]{MR3665214}.
\end{example}

\begin{example}
    The Clifford algebra $\Cl_{+2}$ is \emph{not} inner graded. Let $e_1,e_2$ be two anticommuting odd generators satisfying $e_i^2=1$. Then $g=e_1e_2$ does satisfy $gag^{-1}=(-1)^{|a|}a$, but $g= -g^{-1}$ and $g$ is skew-adjoint. 
    Note that for $A = \Cl_{-2}$ we have $|A\grotimes \Cl_{+2}|\ncong |A|\otimes |\Cl_{+2}|$; the left hand side equals $|\Cl_{2,2}| \cong M_4(\R)$ while the right hand side equals $|\Cl_{-2}| \otimes |\Cl_{2}| \cong M_2(\mathbb{H})$.
\end{example}

We obtain an equivalence of Banach categories given by the composition
    \begin{equation}
        \label{eq:gradedvsungradedfunctor}
         \Psi_A\colon \ModOne_{|A|} \xrightarrow{\simeq} \ModOne_{M_2(|A|)} \xrightarrow{\simeq} \ModOne_{|A\grotimes \Cl_{1,1}|} \,,
    \end{equation}
where the first functor uses the ungraded Morita equivalence between $|A|$ and $M_2(|A|)$ implemented by $|A|^{\oplus 2}$ and the second functor uses the isomorphism of Example~\ref{ex:Cl11inner}.
The previously observed naturality tells us that we get natural isomorphisms of functors $\BanAlgInf \to \BanCat$ and $C^*Alg_{\mathrm{gr}} \to \BanCat$.

\begin{lemma}
\label{lem:covVScontr}
    The Diagram~\eqref{eq:psiA} can be filled by a natural isomorphism.
\end{lemma}
\begin{proof}
Let $M$ be an $|A|$-module.
We have to give a natural isomorphism $\xi_M$ between the $|A \grotimes \Cl_{+1}|$-modules $|A \grotimes \Cl_{+1}| \otimes_{|A|} M$ and $\Psi_A(M)$ after forgetting the $\Cl_{-1}$-action.
    Using the formulas in Example~\ref{ex:Cl11inner}, we can compute explicitly that $\Psi_A(M) = M \oplus M$
with the $|A \grotimes \Cl_{1,1}|$-action 
    \begin{equation*}
    a(m_1,m_2) = (am_1, (-1)^{|a|} am_2), \quad e(m_1,m_2) = (m_2,m_1), \quad f(m_1,m_2) = (-m_2,m_1).
    \end{equation*}
    Consider the isomorphism $\xi_M \colon M \oplus M \cong |A \grotimes \Cl_{+1}| \otimes_{|A|} M$ of vector spaces given by 
    \begin{equation*}
    (m_1,m_2) \mapsto 1 \otimes 1 \otimes m_1 + 1 \otimes e \otimes m_2.
    \end{equation*}
    It follows by the computation $(a \otimes 1)(1 \otimes e \otimes m) = 1 \otimes e \otimes (-1)^{|a|} am$ that $\xi_M$ is an isomorphism of $|A \grotimes \Cl_{+1}|$-modules for the restricted $|A \grotimes \Cl_{+1}|$-action on $M \oplus M$.
    The isomorphism $\xi_M$ is natural in $|A|$-module maps.
\end{proof}

The analogue of Diagram~\eqref{eq:psiA} of opposite variance also commutes:

\begin{lemma}
    The diagram
   \begin{equation}
\begin{tikzcd}
    \ModEn_{|A \grotimes \Cl_{+1}|} \ar[r,"(\R \to \Cl_{+1})^*"] \ar[dr,"(\R \to \Cl_{-1})_*", swap]  & \ModEn_{|A|} \ar[d] 
    \\
    & \ModEn_{|A \grotimes \Cl_{+1} \grotimes \Cl_{-1}|}
\end{tikzcd}
\end{equation}
can be filled by a natural isomorphism.
\end{lemma}
\begin{proof}
    Taking right adjoints of Diagram~\eqref{eq:psiA} we obtain a diagram of the correct shape since $\Psi_A$ is an equivalence.
    The upper horizontal arrow is correct since restriction of scalars is right adjoint to extension of scalars.
    However, the bottom left arrow will be given by coextension of scalars $M \mapsto \Hom(\Cl_{-1}, M)$ instead of the desired extension of scalars $M \mapsto M \otimes \Cl_{+1}$ along $A \hookrightarrow A \grotimes \Cl_{+1}$.
    By the general theory of adjoint pairs between categories of modules~\cite{MR190183}, it suffices to show that $A \hookrightarrow A \grotimes \Cl_{-1}$ is a Frobenius extension of ungraded algebras, see~\cite[Theorem 1.2]{MR1690111} for a textbook account.
    One can check by direct computation that the Frobenius trace $\lambda(a + bf) = a$ is an ungraded $(|A|,|A|)$-bimodule map, and $1 \otimes 1 - f \otimes f \in (|A \grotimes \Cl_{-1}| \otimes |A \grotimes \Cl_{-1}|)$ is a dual basis.
\end{proof}

\subsubsection{Flopposites}
\label{sec:flopping}

There are variations in the literature about the preferred signs of Clifford algebras $\Cl_{\pm 1}$, and we now discuss how to compare these signs in our setup.
In this section, we restrict our graded Banach algebras to be $C^*$.
Recall from Warning~\ref{warning:opposite} that the categorically correct definition of $A^{\op}$ for $A$ a graded $C^*$-algebra involves a Koszul sign on the multiplication.
We will also consider the ``wrong'' opposite:

\begin{definition}
The \emph{flopposite} $A^{\flop}$ of a graded algebra $A$ is its ungraded opposite considered as a graded algebra. In other words, it is equal to $A$ as a graded vector space, but equipped with the multiplication $a_1^{\flop}\cdot a_2^{\flop}=(a_2a_1)^{\flop}$.
\end{definition}

If $A$ is a graded $C^*$-algebra, its flopposite is a graded $C^*$-algebra with $(a^{\flop})^* := (a^*)^{\flop}$.
The $C^*$-structure gives a natural graded $*$-isomorphism $A \to A^{\flop}$ and hence a natural graded $*$-isomorphism between $A^{\op}$ and $A^{\fl} := (A^{\flop})^{\op}$.
Observe that the $C^*$-algebra $A^{\fl}$ has multiplication explicitly given by $a_1^{\fl} \cdot a_2^{\fl} = (-1)^{|a_1||a_2|} (a_1a_2)^{\fl}$.

The following lemma and remark achieve an algebra-level incarnation of the conjugation equivalence of~\cite[Lemma 3.25]{stehouwer2025free} under the equivalence between ungraded $|A \grotimes \Cl_{+1}|$-modules and graded $A$-modules with even maps.

\begin{lemma}
\label{lem:magic2}
    Given a graded $C^*$-algebra $A$, there is a natural isomorphism 
    \begin{equation*}
    \Phi_A \colon |A^{\op} \grotimes \Cl_{+1}| \cong |A \grotimes \Cl_{+1}|
    \end{equation*}
    between functors $\CStarAlg_{\mathrm{gr}} \to \CStarAlg$.
\end{lemma}

\begin{proof}
    It suffices to show the analogous result replacing $A^{\op}$ by $A^{\fl}$.
    Define the map
    \begin{equation*}
    \Phi_A \colon |A \grotimes \Cl_{+1}| \longrightarrow |A^{\fl} \grotimes \Cl_{+1}|
    \end{equation*}
        for homogeneous $a$ by $a\otimes c\mapsto a^{\fl}\otimes e^{|a|}c$.
        This uniquely defines a linear map, which we show is an algebra isomorphism by noting that
        \begin{align*}
            \Phi_A(a_1\otimes c_1)\Phi_A(a_2\otimes c_2)&=(a_1^{\fl} \otimes e^{|a_1|}c_1) (a_2^{\fl} \otimes e^{|a_2|}c_2) \\
            &= (-1)^{|a_2|(|c_1|+|a_1|)} a_1^{\fl}\cdot a_2^{\fl} \otimes e^{|a_1|+|a_2|}c_1c_2 \\
            &= (-1)^{|a_2||c_1|} (a_1a_2)^{\fl}\otimes e^{|a_1a_2|}c_1c_2 \\
            &= \Phi_A( (-1)^{|a_1||c_2|}a_1a_2\otimes c_1c_2) \\
            &= \Phi_A((a_1\otimes c_1)(a_2\otimes c_2)) \,.
        \end{align*}
        Clearly $\Phi_A$ is natural in $A$.
        $\Phi_A$ is a $*$-homomorphism since the $*$ on $A^{\fl}$ is $(a^{\fl})^* = (-1)^{|a|} (a^*)^{\fl}$.
\end{proof}

\begin{warning}
    The isomorphism of Lemma~\ref{lem:magic2} cannot be lifted to a graded isomorphism.
    For example, for $A = \Cl_{+1}$, we have 
    \begin{equation*}
    \Cl_{+1}^{\op} \grotimes \Cl_{+1} \cong \Cl_{1,1} \ncong \Cl_{+2} \cong \Cl_{+1} \grotimes \Cl_{+1}
    \end{equation*}
    as graded algebras,
    even though $|\Cl_{1,1}| \cong M_2(\R) \cong |\Cl_{+2}|$ as ungraded algebras.
\end{warning}

\begin{lemma}
\label{lem:flopiso}
    There is a natural isomorphism which makes the diagram
\begin{equation*}
\begin{tikzcd}[column sep = 40]
    {|A^{\op} \grotimes \Cl_{+1}|} \ar[d,"\Phi_A"] \ar[r,"\R \to \Cl_{-1}"] & {|(A^{\op} \grotimes \Cl_{+1}) \grotimes \Cl_{-1}|} \ar[d,dashed]
    \\
    {|A \grotimes \Cl_{+1}|} \ar[r,"\R \to \Cl_{+1}"] & {|A \grotimes \Cl_{+1} \grotimes \Cl_{+1} |}
\end{tikzcd}
\end{equation*}
of ungraded $C^*$-algebras commute.
\end{lemma}
\begin{proof}  
It is straightforward to check that the following composition for the dashed arrow works:
\begin{align*}
A^{\op} \grotimes \Cl_{+1} \grotimes \Cl_{-1} &\xrightarrow{\beta_{\Cl_{+1},\Cl_{-1}}} A^{\op} \grotimes \Cl_{-1} \grotimes \Cl_{+1} \cong (A \grotimes \Cl_{+1})^{\op} \grotimes \Cl_{+1}
\\
&\xrightarrow{\Phi_{A^{\op} \grotimes \Cl_{+1}}} A \grotimes \Cl_{+1} \grotimes \Cl_{+1} \xrightarrow{\beta_{\Cl_{+1}, \Cl_{+1}}} A \grotimes \Cl_{+1} \grotimes \Cl_{+1}.
\end{align*}
Here, $\beta$ denotes the braiding isomorphism $a \otimes b \mapsto (-1)^{|a||b|} b \otimes a$ of graded $C^*$-algebras.
\end{proof}

\subsubsection{Karoubi \texorpdfstring{$K$}{K}-theory}
\label{sec:karoubikt-h}

Karoubi in~\cite[II.2.1]{Karoubi68} defined a $K$-theory group $K_0^{\Kar}(A)$ associated to a graded Banach algebra $A$, see~\cite[Section 1]{MR2513335} for a recent introduction.
A virtue of Karoubi's definition is that it is always ``correct'', as opposed to the Atiyah--Bott--Shapiro definition and it directly generalizes to super Banach categories.

\begin{definition}
\label{def:Karoubispectral}
\,
\begin{itemize}
    \item The \emph{connective Karoubi $K$-theory spectrum of $A$} is $k^{\Kar}_A := \fib(k_{A}^{C_2} \to k_{|A|})$, where the map is the group completion of forgetting the grading $f:\ModOne_A^{C_2} \to \ModOne_{|A|}$.
    \item The \emph{periodic Karoubi $K$-theory spectrum of $A$} is $K^{\Kar}_A := \fib(K_{A}^{C_2} \to K_{|A|}) \cong k^{\Kar}_A[\beta^{-1}]$.
    \item The \emph{Karoubi $K$-theory group} $K_0^{\Kar}(A)$ is defined as $\pi_0K^{\Kar}_A$, cf.\ Definition~\ref{def:abstractKar}.
    \end{itemize}
\end{definition}

For defining the periodic version we used that $k_A^{C_2} \to k_{|A|}$ is $ko$-linear.
We unpack the definition of $K_0^{\Kar}(A)$ by applying the theory of Section~\ref{sec:fibersgroupcompletion}.
It follows by Theorem~\ref{thm:fibergroupcompletion} and Lemma~\ref{lem:triples} that
\begin{equation*}
K_0^{\Kar}(A) \cong \pi_0 k(f) \cong K^{\Kar}(f).
\end{equation*}
Elements of $K_0^{\Kar}(A)$ are therefore equivalence classes of triples $(M_1, M_2, \phi)$ where $M_1$ and $M_2$ are finitely generated projective graded $A$-modules and $\phi \colon |M_1| \to |M_2|$ is an ungraded isomorphism of $|A|$-modules.
The monoid structure is given by $(M_1, M_2, \phi) \oplus (M_1',M_2',\phi') = (M_1 \oplus M_1', M_2 \oplus M_2', \phi \oplus \phi')$.
Triples $(M_1, M_2, \phi)$ and $(M_1', M_2', \phi')$ are identified when they are homotopic\footnote{Karoubi only identifies triples when they are isomorphic, i.e.\ the diagram below commutes strictly. However, our weaker notion of homotopy is a consequence of his definition~\cite[Proposition II.2.15]{karoubi_k-theory_1978}.} in the sense that there are isomorphisms $M_1 \cong M_1'$ and $M_2 \cong M_2'$ together with a homotopy of $|A|$-bimodule isomorphisms filling the diagram
\begin{equation*}
\begin{tikzcd}
    {|M_1|} \ar[r,"{\phi}"] \ar[d] & {|M_2|} \ar[d]
    \\
    {|M_1'|}  \ar[r,"{\phi'}"] & {|M_2'|}
\end{tikzcd} \,.
\end{equation*}
Finally, a quotient is taken by the submonoid of triples (homotopic to those) of the form $(M,M,\id)$.
One can equivalently formulate the triples as having two gradings on the same ungraded $A$-module.

These spectra can be related to $k_A^{\mathrm{vD}}$ (and hence $k_A^{\ABS}$) using flopposites as follows.

\begin{proposition}
\label{prop:karvD}
    We have $k_A^{\Kar} = k^{\mathrm{vD}}_{A^{\op} \grotimes \Cl_{+1}}$ and so $K_A^{\Kar} = \Omega^{-1} K^{\mathrm{vD}}_{A^{\op}}$
\end{proposition}
\begin{proof}
    Consider the commutative diagrams:
    \begin{equation*}
    \begin{tikzcd}
        k_A^{C_2} \ar[d,"\simeq"] & k^{C_2}_{A \grotimes \Cl_{-1}} \ar[l] \ar[d,"\simeq"]
        \\
        k_{A^{\op}}^{C_2} \ar[d,"\simeq"] & k_{A^{\op} \grotimes \Cl_{+1}}^{C_2} \ar[dl] \ar[l]
        \\
        k_{|A^{\op} \grotimes \Cl_{+1}|} & 
    \end{tikzcd} \,,
    \end{equation*}
    where the square is obtained from taking the group completion of the core of Lemma~\ref{lem:flopiso}, and the bottom triangle comes from Proposition~\ref{prop:gradingvsclifford}.
    Taking fibers we obtain
    \begin{equation*}
    k^{\mathrm{vD}}_A \cong k^{\Kar}_{A^{\op} \grotimes \Cl_{+1}}.
    \end{equation*}
    Using $\Cl_{+1}^{\op} = \Cl_{-1}$ and the Morita equivalence $\Cl_{1,1} \Morita \R$ we get 
    \begin{align*}
        k^{\mathrm{vD}}_{A^{\op} \grotimes \Cl_{+1}} \cong k^{\Kar}_{(A^{\op} \grotimes \Cl_{+1})^{\op} \grotimes \Cl_{+1}} = k^{\Kar}_A.
    \end{align*}
    The second point follows from the Clifford suspension theorem of periodic $K$-theory.
\end{proof}

By Proposition~\ref{prop:ABSvsvD} and Clifford periodicity, we get:

\begin{corollary}
    $k_A^{\Kar} \cong \Omega k_{A^{\op} \grotimes \Cl_{+1}}^{\ABS}$ and $K^{\Kar}_A \cong K_{A^{\op}}^{\mathrm{gr}}$
\end{corollary}

\subsubsection{Comparison atlas of graded \texorpdfstring{$K$}{K}-theory}
\label{sec:atlas}

In this section we consider alternatives to defining graded $K$-theory that naturally arise in our setup.
More specifically, in our Definitions~\ref{def:ABSKth} and ~\ref{def:periodick}, we could consider:
\begin{itemize}
    \item replacing cofiber by fiber
    \item replacing $\Cl_{-1}$ by $\Cl_{+1}$
    \item replacing contravariant functoriality of $\R \to \Cl_{-1}$ by covariant functoriality.
    \item replacing graded modules by ungraded modules
\end{itemize}
as well as analogues of Karoubi $K$-theory comparing graded and ungraded modules.
In this section we show how all such versions can be related.
Some reproduce the comparison results we have proven so far in this section and several other versions have not been considered in the literature as far as we are aware.
We hence prove that graded $K$-theory fits into a multitude of fiber and cofiber sequences.

For notational convenience, we will mostly replace $C_2$-gradings by $\Cl_{+1}$-actions in this section by invoking Proposition~\ref{prop:gradingvsclifford}, keeping in mind the naturality spelled out in Remark~\ref{rem:C2vsCliffnatural}.
The reader who is interested in the graded versions of results in this section is referred to Remark~\ref{rem:ungradede}.

\begin{lemma} 
\label{lem:connectivemaster}
    We have the following equivalences of $ko$-modules that are natural in the graded $C^*$-algebra $A$.
    \begin{enumerate}
        \item $\fib(k_{|A|} \to k_{|A \grotimes \Cl_{+1}}|) \cong k_A^{\mathrm{vD}}$
        \item $\fib( k_{|A \grotimes \Cl_{-1}|} \to k_{|A|}) \cong k^{\mathrm{vD}}_{A \grotimes \Cl_{-1}}$
        \item $\fib(k_{|A \grotimes \Cl_{+1}|} \to k_{|A|}) \cong k^{\mathrm{vD}}_{A^{\op} \grotimes \Cl_{+1}}$
        \item  $\fib(k_{|A|} \to k_{|A \grotimes \Cl_{-1}|}) \cong k^{\mathrm{vD}}_{A^{\op} \grotimes \Cl_{+2}}$
    \end{enumerate}
\end{lemma}
\begin{proof}
    From the Morita move we get
a commutative diagram
    \begin{equation}
    \label{moritamove}
    \begin{tikzcd}
    k_{A}^{C_2}    & k_{|A|} \ar[l] \ar[d,"\simeq"]
    \\
    k_{A}^{C_2} \ar[u,equals] & k_{A \grotimes \Cl_{-1}}^{C_2} \ar[l] 
    \end{tikzcd}
    \end{equation}
    of $ko$-modules.
    Taking fibers, we obtain 1.
We get $2$ by using $\Cl_{1,1} \Morita \R$ on both domain and codomain:
\begin{equation*}
k^{\mathrm{vD}}_{A \grotimes \Cl_{-1}} = \fib(k_{|A \grotimes \Cl_{-1} \grotimes \Cl_{-1} \grotimes \Cl_{+1}|} \to k_{|A \grotimes \Cl_{-1} \grotimes \Cl_{+1}}|) \cong \fib( k_{|A \grotimes \Cl_{-1}|} \to k_{|A|}).
\end{equation*}
Note that the left hand side of point 3 is $k_A^{\Kar}$ and so point 3 was shown in Proposition~\ref{prop:karvD} .
The version of Diagram~\eqref{moritamove} of the other variance identifies $k_A^{\Kar}$ with $\fib(k_{|A \grotimes \Cl_{+1}|} \to k_{|A \grotimes \Cl_{+1} \grotimes \Cl_{-1}|})$.
We get
\begin{equation*}
k^{\Kar}_{A \grotimes \Cl_{-1}} =
\fib(k_{|A \grotimes \Cl_{-1} \grotimes \Cl_{+1}|} \to k_{|A \grotimes \Cl_{-1}|}) \cong
\fib(k_{|A|} \to k_{|A \grotimes \Cl_{-1}|}). 
\end{equation*}
So we in particular obtain
\begin{equation*}
\fib(k_{|A|} \to k_{|A \grotimes \Cl_{-1}|}) \cong k^{\Kar}_{A \grotimes \Cl_{-1}} \cong k^{\mathrm{vD}}_{A^{\op} \grotimes \Cl_{+2}} \,.
\end{equation*}
All isomorphisms applied are natural in $A$.
\end{proof}

\begin{theorem}
    \label{th:mastercomparison}
    Let $A$ be a graded $C^*$-algebra.
    There are equivalences of $KO$-modules as listed in Table~\ref{table:mastercomparison} and they are natural in $A$.
\end{theorem}
\begin{proof}
    Comparing versions with fibers and cofibers is immediate using $\fib = \Omega \cofib$, and so it suffices to prove one of the versions for each quadrant.
    Recall that $k^{\ABS}_A \cong \cofib(k_{|A|} \to k_{|A \grotimes \Cl_{+1}|})$ indeed gives the top left entry of Table~\ref{table:mastercomparison}.
    We can compare the left with the right column in the table diagonally by using Lemma~\ref{lem:connectivemaster}.
    We now apply the Clifford suspension isomorphism \ref{prop:clifshift}.
\end{proof}

\begin{table}[ht]
\mastercomparisonbody

\caption{Comparison of different conventions for defining the graded
$K$-theory of the graded Banach algebra $A$. The entry labeled
\emph{ABS} is Definition~\ref{def:periodick} of graded $K$-theory,
motivated by Atiyah--Bott--Shapiro~\cite{ABS}. The entry labeled
\emph{van Daele} is the spectral refinement
\ref{def:vDspectrumperiodic}. The entry labeled \emph{Karoubi} is the
spectral refinement \ref{def:Karoubispectral} of Karoubi $K$-theory.}
\label{table:mastercomparison}
\end{table}

\begin{remark}
     There are connective variants of Theorem~\ref{th:mastercomparison}.
     Beware that the map $k^{\mathrm{ABS}}_A\to \Omega^j k^{\mathrm{ABS}}_{A\grotimes\Cl_{+j}}$ is in general only an isomorphism on $\pi_{\geq j+1}$.
\end{remark}

\begin{remark}
    Let $p,q \geq 0$ be fixed integers.
    Consider the inclusions $\Cl_{p,q} \hookrightarrow \Cl_{p+1,q}$ and $\Cl_{p,q} \hookrightarrow \Cl_{p,q+1}$ given by tensoring $\R \hookrightarrow \Cl_{+1}$  and $\R \hookrightarrow \Cl_{-1}$ with the identity on $\Cl_{p,q}$.
    By taking fiber/cofiber of the map on $K$-theory induced covariantly/contravariantly by these homomorphisms after further tensoring with the identity on some graded $C^*$-algebra $B$, we obtain eight spectra that correspond to the entries of Table~\ref{table:mastercomparison} with $A = B \grotimes \Cl_{p,q}$. 
    Using Clifford periodicity $K^{\mathrm{gr}}_{A \grotimes \Cl_{p,q}} \cong \Sigma^{p-q} K_A^{\mathrm{gr}}$, we see that also these constructions all lead to a shift of $K^{\mathrm{gr}}_B$ or $K^{\mathrm{gr}}_{B^{\op}}$ that can easily be computed.
\end{remark}

\begin{remark}
\label{rem:ungradede}
    In this section we chose to work ungraded to simplify the exposition.
    One can translate Lemma~\ref{lem:connectivemaster} and Theorem~\ref{th:mastercomparison} to versions involving gradings by applying the dictionary~\ref{prop:gradingvsclifford} comparing $\Cl_{+1}$-actions and gradings.
    For example, the ``$\Cl_{+1}$-version'' of ABS differs by an opposite using point 3 of Lemma~\ref{lem:connectivemaster}:
    \begin{equation*}
    \cofib(k^{C_2}_{A \grotimes\Cl_{+1}} \to k^{C_2}_{A}) \cong k_{A^{\op}}^{\ABS}.
    \end{equation*}

\end{remark}

\begin{remark}
If $A$ is a graded Banach algebra that is not $C^*$, the isomorphisms in Table~\ref{table:mastercomparison} for those entries that don't involve $A^{\op}$ follow by the same proof.
However, the flopping techniques developed in Section~\ref{sec:flopping} that we used to get the top right and lower left entries need to be replaced:
Consider the contravariant equivalence $\ModOne_A^{C_2} \to (\ModOne_{A^{\op}}^{C_2})^{\op}$ given by $M \mapsto \Hom_A(A,M)$, where we converted the right $A$-action into a left $A^{\op}$-action.\footnote{Recall the signs in Convention \ref{conv:graded-linear}.}
After converting to a covariant equivalence on the core and group completing, we get an equivalence $k_A^{C_2} \cong k_{A^{\op}}^{C_2}$ of spectra.
It follows by a direct computation that there is an equivalence $\Hom(A, M \grotimes \Cl_1) \cong  \Hom(A, M) \grotimes \Cl_{-1}$ of $(A \grotimes \Cl_1)^{\op} \cong A^{\op} \grotimes \Cl_{-1}$-modules, where $M \grotimes \Cl_1$ denotes the extension of scalars along $A \to A \grotimes \Cl_1$.
This equivalence obviously respects the respective inclusions and so we obtain a commutative diagram
\begin{equation*}
\begin{tikzcd}
    N^{\mathrm{hc}}(\ModEn_A^{C_2,\cong}) \ar[d, "\simeq"] \ar[r] & N^{\mathrm{hc}}(\ModEn_{A \grotimes \Cl_{+1}}^{C_2,\cong}) \ar[d,"\simeq"]
    \\
    N^{\mathrm{hc}}(\ModEn_{A^{\op}}^{C_2,\cong}) \ar[r] & N^{\mathrm{hc}}(\ModEn_{A^{\op} \grotimes \Cl_{-1}}^{C_2,\cong})
\end{tikzcd} \,.
\end{equation*}
The group completion of this diagram is all we need to get a graded Banach algebra version of Lemma~\ref{lem:connectivemaster} and Theorem~\ref{th:mastercomparison}.
\end{remark}

\newpage

\appendix
\section{Appendix}

\subsection{Enriched categories}
\label{sec:enriched}

Let $\mathcal{V}$ be a monoidal $1$-category.
We use~\cite{MR2177301} as a general reference for enriched category theory.
A \emph{$\mathcal{V}$-enriched category} $\mathsf{C}$ consists of a collection of objects together with hom-objects $\mathsf{C}(x,y) \in \mathcal{V}$, composition morphisms $\mathsf{C}(y,z) \otimes \mathsf{C}(x,y) \to \mathsf{C}(x,z)$ and units $\mathbbm{1} \to \mathsf{C}(x,x)$ satisfying the evident associativity and unit axioms. A \emph{$\mathcal{V}$-functor} is given on hom-objects by morphisms of $\mathcal{V}$ compatible with composition and units.
We denote the $(2,1)$-category of $\mathcal{V}$-enriched categories, $\mathcal{V}$-functors and invertible $\mathcal{V}$-natural transformations by $\Cat_1(\mathcal{V})$.
If $\mathcal{V} \to \mathcal{W}$ is a lax monoidal functor, there is an induced change-of-enrichment functor $\Cat_1(\mathcal{V}) \to \Cat_1(\mathcal{W})$.
For the functor $\Hom_{\mathcal{V}}(\mathbbm{1},-)\colon \mathcal{V} \to \Set$, this gives the \emph{underlying category}. 
If $\mathcal{V}$ is symmetric monoidal, then $\Cat_1(\mathcal{V})$ carries the symmetric monoidal structure $\boxtimes$ of~\cite[Section 1.4]{MR2177301}, given on objects by
$(\mathsf{C}\boxtimes\mathsf{D})\big((c_1,d_1),(c_2,d_2)\big) = \mathsf{C}(c_1,c_2) \otimes \mathsf{D}(d_1,d_2)$.

\paragraph{Enriched groupoids}
Enriched groupoids only make sense in cartesian monoidal categories.
Let $\mathcal{V}$ be a $1$-category with finite products, with the cartesian symmetric monoidal structure.

\begin{definition}
\label{def:enrichedgroupoid}
A \emph{$\mathcal{V}$-enriched groupoid} is a $\mathcal{V}$-enriched category $\mathsf{C}$ equipped with the \emph{additional datum} of inverse maps $i_{c,d} \colon \mathsf{C}(c,d)\to \mathsf{C}(d,c)$ such that
\begin{equation*}
\begin{tikzcd}[column sep =55]
    \mathsf{C}(c,d) \ar[d,"\Delta"] \ar[r] & \ast \ar[r,"\id_d"] & \mathsf{C}(d,d)
    \\
    \mathsf{C}(c,d) \times \mathsf{C}(c,d) \ar[rr, "\id \times i_{c,d}"] &&  \mathsf{C}(c,d) \times \mathsf{C}(d,c) \ar[u, "\circ"]
\end{tikzcd}
\end{equation*}
and the analogous diagram on the other side commute.
\end{definition}

\begin{example}
    Let $\mathcal{V}=(\Top,\times)$ be the category of compactly generated weak Hausdorff spaces with the Kelley product.
    A groupoid enriched in $\Top$ is a category $\mathsf{C}$ enriched in $\Top$ such that every $f\colon c\to d$ admits a (necessarily unique) two-sided inverse and such that the inverse map $\mathsf{C}(c,d)\to \mathsf{C}(d,c)$ is continuous.
\end{example}

\begin{definition}
\label{rmk:weakgroupoid}
A \emph{weak enriched groupoid} is a $\mathcal{V}$-enriched category whose underlying category is a groupoid.
    Not every weak enriched groupoid is a $\mathcal{V}$-enriched groupoid.
    For example, a $\Top$-enriched category with one object is a topological monoid $G$; its underlying category is a groupoid if and only if $G$ is a group in sets, but the inversion map need not be continuous.
\end{definition}
We write $\Grpd_1(\mathcal{V})$ for the $(2,1)$-category of weak $\mathcal{V}$-enriched groupoids. 
For some $\mathcal{V}$ the existence of the inverse maps is merely a condition, so that $\Grpd_1(\mathcal{V})\subseteq \Cat_1(\mathcal{V})$ is a full subcategory.

The composite~\eqref{eq:topktheory2} requires passing from enriched categories to enriched groupoids compatibly with products.
This is achieved by the iso-core, which we now construct for a general enriching category; the simplicial case is the one used in Section~\ref{sec:enrichKth}.

\begin{definition}
\label{def:fullsubobject}
    Write $V_0 := \Hom_{\mathcal{V}}(\mathbbm{1},V)$ for the set of points of $V \in \mathcal{V}$, as in the definition of the underlying category.
    We say that $\mathcal{V}$ \emph{has full subobjects} if for every $V \in \mathcal{V}$ and every subset $S \subseteq V_0$ there is a subobject $V|_S \hookrightarrow V$ such that a map $W \to V$ factors through $V|_S$ (necessarily uniquely) if and only if the image of $W_0 \to V_0$ is contained in $S$.
\end{definition}

\begin{example}
\label{ex:fullsubobject}
    In $\sSet$, the object $V|_S$ is the simplicial subset of those simplices all of whose vertices lie in $S$; a simplicial map lands in it precisely when it does so on vertices.
    In $\Top$, the object $V|_S$ is the subset $S$ with the subspace topology.
\end{example}

\begin{construction}
\label{constr:isocore}
    Let $\mathcal{V}$ have full subobjects and let $\mathsf{C} \in \Cat_1(\mathcal{V})$ have underlying category $\mathsf{C}_0$.
    The \emph{iso-core} $\mathsf{C}^{\cong}$ is the $\mathcal{V}$-enriched category with the same objects as $\mathsf{C}$ and hom-objects
    \begin{equation*}
        \mathsf{C}^{\cong}(x,y) := \mathsf{C}(x,y)\big|_{\mathrm{Iso}(x,y)} \,,
    \end{equation*}
    where $\mathrm{Iso}(x,y) \subseteq \mathsf{C}_0(x,y)$ is the set of isomorphisms of $\mathsf{C}_0$.
    The units factor through these subobjects because identities are isomorphisms.
    For the composition, recall that $\mathcal{V}$ is cartesian, so that a point of $\mathsf{C}^{\cong}(x,y) \times \mathsf{C}^{\cong}(y,z)$ is a pair of isomorphisms; its image in $\mathsf{C}_0(x,z)$ is invertible and Definition~\ref{def:fullsubobject} applies.
    Associativity and unitality are inherited from $\mathsf{C}$.
\end{construction}

Since $(\mathsf{C}^{\cong})_0$ is the core of $\mathsf{C}_0$, the iso-core is a weak enriched groupoid in the sense of Definition~\ref{rmk:weakgroupoid}.
It is in general not an enriched groupoid: inversion need not be a morphism of $\mathcal{V}$.

\begin{lemma}
\label{lem:isocore}
    The iso-core defines a $2$-functor
    \begin{equation*}
        (-)^{\cong} \colon \Cat_1(\mathcal{V}) \longrightarrow \Grpd_1(\mathcal{V}) \,,
    \end{equation*}
    right adjoint to the full inclusion $\Grpd_1(\mathcal{V}) \subseteq \Cat_1(\mathcal{V})$, with counit the inclusion $\mathsf{C}^{\cong} \hookrightarrow \mathsf{C}$.
    In particular $\Grpd_1(\mathcal{V})$ is coreflective in $\Cat_1(\mathcal{V})$ and $(-)^{\cong}$ preserves products.
\end{lemma}
\begin{proof}
    Let $\mathsf{G}$ be a weak enriched groupoid and $F \colon \mathsf{G} \to \mathsf{C}$ a $\mathcal{V}$-functor.
    The image of $\mathsf{G}_0(x,y) \to \mathsf{C}_0(Fx,Fy)$ is contained in $\mathrm{Iso}(Fx,Fy)$.
    By Definition~\ref{def:fullsubobject} each $F(x,y)$ therefore factors uniquely through $\mathsf{C}^{\cong}(Fx,Fy)$, and these factorizations assemble into a unique $\mathcal{V}$-functor $\mathsf{G} \to \mathsf{C}^{\cong}$ lifting $F$.
    The components of an invertible $\mathcal{V}$-natural transformation factor by the same argument, so the induced map of hom-groupoids is an isomorphism.
    Applying this to $\mathsf{C}^{\cong} \hookrightarrow \mathsf{C} \xrightarrow{F} \mathsf{D}$ produces the functoriality $F^{\cong} \colon \mathsf{C}^{\cong} \to \mathsf{D}^{\cong}$, and the displayed bijection is the required adjunction.
    Right adjoints preserve limits.
\end{proof}

Since the monoidal structure on $\mathcal{V}$ is cartesian, $\boxtimes$ is the cartesian product of $\Cat_1(\mathcal{V})$, and likewise on the full subcategory $\Grpd_1(\mathcal{V})$.
Lemma~\ref{lem:isocore} thus exhibits $(-)^{\cong}$ as a symmetric monoidal functor, so it induces
\begin{equation*}
    (-)^{\cong} \colon \CMon(\Cat_1(\mathcal{V})) \longrightarrow \CMon(\Grpd_1(\mathcal{V})) \,.
\end{equation*}

\paragraph{Direct sums in enriched categories}

Enriched colimits make sense in case the enriching category has enough limits~\cite[Section 3.8]{MR2177301}.

\begin{definition}
\label{def:enrichedcopr}
Suppose $\mathcal{V}$ has finite products.
An \emph{enriched coproduct} of $x,y$ in a $\mathcal{V}$-category $\mathsf{C}$ is an object $x \oplus y$ together with isomorphisms
\begin{equation}
\label{eq:enrichedcoproduct}
\mathsf{C}(x \oplus y, z) \cong \mathsf{C}(x,z) \times \mathsf{C}(y,z)
\end{equation}
in $\mathcal{V}$, natural in $z$.
\end{definition}
Products and initial/terminal objects are defined analogously.
Change of enrichment for functors $\mathcal{V} \to \mathcal{W}$ that preserve products preserves enriched coproducts.

The underlying object of an enriched coproduct is a coproduct in the underlying category, since $\Hom_{\mathcal{V}}(\mathbbm{1},-)$ preserves limits.
The converse is a condition in general: the canonical morphism in~\eqref{eq:enrichedcoproduct} must be an isomorphism \emph{in $\mathcal{V}$}, not merely a bijection on underlying sets.
For Banach categories, however, no additional condition is needed:

\begin{lemma}
\label{lem:openmapping}
    Let $\mathsf{C}$ be a Banach category and let $x \oplus y$ be a direct sum in the underlying additive category.
    Then $x \oplus y$ is an enriched coproduct, i.e.\ the canonical map~\eqref{eq:enrichedcoproduct} is an isomorphism in $\Ban$.
\end{lemma}
\begin{proof}
    The canonical map $\mathsf{C}(x \oplus y, z) \to \mathsf{C}(x,z) \times_{\ell^\infty} \mathsf{C}(y,z)$, $f \mapsto (f\iota_x, f\iota_y)$, is a bounded linear bijection of Banach spaces: it is bounded since precomposition is, and its set-theoretic inverse is $(g,h) \mapsto g\pi_x + h\pi_y$.
    By the open mapping theorem it is a topological isomorphism.
\end{proof}

The isomorphism of Lemma~\ref{lem:openmapping} need not be isometric, so the statement holds in $\Cat_1(\Ban)$ but not immediately for $\Cat_1(\Ban_{\leq})$.
For $\mathcal{V} = \Top$ or $\sSet$ the analogous statement fails since a continuous bijection need not be a homeomorphism, and so enriched coproducts are genuinely stronger than underlying ones.

If $\mathcal{V}$ is additive, an object is an enriched coproduct if and only if it is an enriched product, if and only if it admits $\iota_j, \pi_j$ with $\pi_j \iota_k = \delta_{jk}$ and $\sum_j \iota_j\pi_j = \id$.
Since this characterization is by equations between morphisms and uses no limits, we obtain that enriched direct sums are absolute~\cite{MR749468}:

\begin{lemma}
\label{lem:enrichedbiproduct}
    Let $\mathcal{V}$ be additive.
    Every $\mathcal{V}$-functor between $\mathcal{V}$-categories preserves enriched direct sums.
\end{lemma}

\paragraph{Idempotent completion}

Let $\mathcal{V}$ be a symmetric monoidal $1$-category in which idempotents split. 
Our main example is $\mathcal{V} = \Ban$: a bounded idempotent $P$ on a Banach space $V$ splits as $V \cong \im(P) \oplus \ker(P)$, because $\im(P) = \ker(\id - P)$ is closed and hence itself a Banach space.
An \emph{idempotent} in a $\mathcal{V}$-category $\mathsf{C}$ is an idempotent endomorphism $e = e^2 \in \mathsf{C}(c,c)$ of the underlying category, and a \emph{splitting} of $e$ consists of morphisms $r\colon c \to d$ and $s \colon d \to c$ of the underlying category with $s\circ r = e$ and $r \circ s = \id_d$.
In contrast to direct sums (Lemma~\ref{lem:openmapping}), no enriched condition needs to be verified: if $e$ splits in the underlying category, then composition with $r$ and $s$ exhibits $\mathsf{C}(d,x)$ as a splitting in $\mathcal{V}$ of the idempotent $(-) \circ e$ on $\mathsf{C}(c,x)$, and similarly on the other side.
The underlying reason is that idempotents are always absolute colimits, while for direct sums this depends on $\mathcal{V}$.

We call $\mathsf{C}$ \emph{idempotent complete} if every idempotent splits.

\begin{definition}
\label{def:idemcompletion}
    The \emph{idempotent completion} (or \emph{Karoubi completion}) of a $\mathcal{V}$-category $\mathsf{C}$ is the $\mathcal{V}$-category $\Idem(\mathsf{C})$ whose objects are pairs $(c, e)$ of an object $c \in \mathsf{C}$ and an idempotent $e \in \mathsf{C}(c,c)$, and whose hom-objects
    \begin{equation*}
        \Idem(\mathsf{C})\big((c,e),(c',e')\big) := e' \circ \mathsf{C}(c,c') \circ e
    \end{equation*}
    are the splittings in $\mathcal{V}$ of the idempotents $f \mapsto e' \circ f \circ e$ on the hom-objects $\mathsf{C}(c,c')$.
    Composition is inherited from $\mathsf{C}$, and the identity of $(c,e)$ is $e$.
\end{definition}

\begin{warning}
\label{warn:idemnorm}
    A nonzero idempotent for $\Ban$ satisfies $\|e\| \geq 1$, and possibly $\|e\| > 1$, so the identities of $\Idem(\mathsf{C})$ need not be contractive when $\mathsf{C}$ is enriched in $\Ban_{\leq}$.
\end{warning}

\begin{lemma}
\label{lem:idemcompletion}

    Let $\mathsf{C}$ be a $\mathcal{V}$-category.
    \begin{enumerate}
        \item The $\mathcal{V}$-functor $\iota \colon \mathsf{C} \to \Idem(\mathsf{C})$, $c \mapsto (c, \id_c)$, is fully faithful, and $\Idem(\mathsf{C})$ is idempotent complete.
        \item $\Idem$ is a $(2,1)$-functor which is left adjoint to the inclusion of the idempotent-complete $\mathcal{V}$-categories. In particular, for every idempotent-complete $\mathcal{V}$-category $\mathsf{D}$, precomposition with $\iota$ is an equivalence
        \begin{equation*}
            \iota^* \colon \Fun_{\mathcal{V}}\big(\Idem(\mathsf{C}),\mathsf{D}\big) \xrightarrow{\ \simeq\ } \Fun_{\mathcal{V}}(\mathsf{C},\mathsf{D})
        \end{equation*}
        of categories of $\mathcal{V}$-functors and $\mathcal{V}$-natural transformations.
    \end{enumerate}
\end{lemma}
\begin{proof}
    For $\mathcal{V} = \Set$ one reference is~\cite[Propositions~1.9, 1.10, Theorem~1.12 and Corollary~1.13]{kammermeier2024higher}; the same formulas define the enriched structure verbatim, the only additional input being the splitting of idempotents in $\mathcal{V}$ described above.
    The construction goes back to~\cite[Theorem~I.6.10]{karoubi_k-theory_1978}, where it is carried out for additive and Banach categories under the name pseudo-abelian completion; see~\cite{MR0850527}~\cite[Section~5.5]{MR2177301} for idempotent completion as part of Cauchy completion in enriched category theory.
\end{proof}

\begin{remark}
\label{rmk:idemadditive}
    If $\mathsf{C}$ has finite enriched coproducts, then so does $\Idem(\mathsf{C})$: the coproduct of $(c,e)$ and $(c',e')$ is $(c \oplus c', e \oplus e')$.
    Moreover, if $\mathsf{C}$ and $\mathsf{D}$ are additive, then $\Idem(\mathsf{C} \boxtimes \mathsf{D})$ is additive even though $\mathsf{C} \boxtimes \mathsf{D}$ is not: the coproduct of $(c,d)$ and $(c',d')$ is the splitting of the idempotent $p \otimes q + p' \otimes q'$ on $(c \oplus c', d \oplus d')$, where $p, p'$ and $q, q'$ denote the complementary projections onto the summands.
\end{remark}

The relevance for this paper is the following description of the categories of finitely generated projective modules.

\begin{lemma}
\label{lem:idemfree}
    Let $A$ be a Banach algebra and let $\Free_A \subseteq \ModEn_A$ be the full Banach subcategory on the free modules $A^n$ for $n \in \N$.
    The inclusion extends to an equivalence of Banach categories
    \begin{equation*}
        \Idem(\Free_A) \xrightarrow{\ \simeq\ } \ModEn_A \,.
    \end{equation*}
    In particular, $\ModEn_A$ is idempotent complete.
    Moreover, we have 
    \begin{equation*}
        \Idem(\ModEn_A\boxtimes \ModEn_B) \xrightarrow{\simeq} \ModEn_{A\otimes_\pi B}
    \end{equation*}
\end{lemma}
\begin{proof}
    The extension sends $(A^n, e)$ to the direct summand $eA^n \subseteq A^n$, which is finitely generated projective and carries the subspace norm as its unique Banach topology; every finitely generated projective module is of this form by definition, so the functor is essentially surjective.
    It is fully faithful because module maps $eA^n \to e'A^m$ are exactly the maps $f \colon A^n \to A^m$ satisfying $e'fe = f$, and the two norms in question agree up to equivalence by the norm comparison in Section~\ref{sec:Banach}.
    The tensor product of modules defines a functor $\ModEn_A\boxtimes\ModEn_B \to \ModEn_{A\otimes_\pi B}$. 
    Its image contains all free modules and it is fully faithful.
    Hence, it is an equivalence after idempotent completion.
\end{proof}

\begin{remark}
\label{rmk:cauchyslogan}
    Since $\Free_A$ is in turn the free completion of the one-object Banach category with endomorphism algebra $A^{\op}$ (cf.~\eqref{eq:endop}) under finite direct sums, Lemma~\ref{lem:idemfree} identifies $\ModEn_A$ with the Cauchy completion of this one-object Banach category.
    For $\mathcal{V} = \mathrm{Ab}$ this recovers the classical description of the finitely generated projective modules over a ring as its Cauchy completion~\cite{MR0850527}.
\end{remark}

\paragraph{Symmetric monoidal enriched categories}

Let $\mathcal{V}$ be a symmetric monoidal $1$-category.
A \emph{symmetric monoidal $\mathcal{V}$-enriched category} is a commutative monoid object in the symmetric monoidal $(2,1)$-category $(\Cat_1(\mathcal{V}), \boxtimes)$, i.e.\ an object of $\CMon(\Cat_1(\mathcal{V}))$.
Unwinding the definition, this is the data of a $\mathcal{V}$-category $\mathsf{C}$ together with a $\mathcal{V}$-functor $\otimes\colon \mathsf{C}\boxtimes \mathsf{C}\to \mathsf{C}$, a unit, and the usual associators, unitors and symmetric braiding, satisfying the axioms up to coherent invertible $\mathcal{V}$-natural transformations.
We write $\Cat_1^{\otimes}(\mathcal{V}) := \CMon(\Cat_1(\mathcal{V}))$ and $\Grpd_1^{\otimes}(\mathcal{V}) := \CMon(\Grpd_1(\mathcal{V}))$.

The following construction produces the symmetric monoidal structure on an enriched category with enriched coproducts that underlies Definition~\ref{def:topktheory2}; there it is denoted $(-)^{\oplus}$.
The construction is a straightforward enriched generalization of the classical construction of the cocartesian symmetric monoidal structure.
\begin{construction}
\label{constr:oplus}
    Let $\mathcal{V}$ have finite products and let $\mathsf{C}$ be a $\mathcal{V}$-enriched category with enriched finite coproducts (Definition \ref{def:enrichedcopr}).
    By Definition~\ref{def:commutativemonoid} it suffices to construct $\mathsf{C}^\oplus \in \CMon(\Cat_1(\mathcal{V}))$ as a functor $\Fin_*\to \Cat_1(\mathcal{V})$ satisfying the Segal conditions.
    For $\langle n \rangle \in \Fin_*$, let $\mathsf{C}_n^\oplus$ be the $\mathcal{V}$-category whose objects $\underline{c}$ consist of a tuple $(c_1,\dots,c_n)\in \mathrm{Ob}(\mathsf{C})^n$ together with choices of coproducts $\bigoplus_{j\in S} c_{j}$ for all subsets $S \subseteq \langle n \rangle^\circ$ (as well as the coprojections), and whose hom-objects are
    \begin{equation*}
        \mathsf{C}_n^\oplus(\underline{c},\underline{d})=\prod_{i=1}^n \mathsf{C}(c_i,d_i) \,,
    \end{equation*}
    with componentwise composition.
    For $f\colon\langle n\rangle \to \langle m \rangle$ define $f_*\colon\mathsf{C}_n^\oplus \to \mathsf{C}_m^\oplus$ on objects by
    \begin{equation*}
    f_*(\underline{c})=\Big(\bigoplus_{j\in f^{-1}(1)} c_j, \dots ,\bigoplus_{j\in f^{-1}(m)}c_j \Big)\,,
    \end{equation*}
    with the choices of coproducts inherited from $\underline{c}$, and on hom-objects by the block direct sum
    \begin{equation*}
        \prod_{j\in f^{-1}(i)} \mathsf{C}(c_j,d_j) \to \mathsf{C}\Big(\bigoplus_{j\in f^{-1}(i)} c_j,\bigoplus_{j\in f^{-1}(i) } d_j  \Big) \,,
    \end{equation*}
    which is a morphism in $\mathcal{V}$ by~\eqref{eq:enrichedcoproduct}.
    The isomorphisms $f_*\circ g_* \cong (f\circ g)_*$ are furnished by the universal property of the coproduct, and coproducts are associative.
    The Segal maps $\mathsf{C}_n^{\oplus} \to (\mathsf{C}_1^{\oplus})^{\times n} = \mathsf{C}^{\times n}$ are equivalences (forgetting the choices of coproducts), so $\mathsf{C}^\oplus$ defines an object of $\CMon(\Cat_1(\mathcal{V}))$.
    This defines a functor
    \begin{equation}
        \label{eq:cocartfunctor}
        \Cat_1^\Sigma(\calV) \to \CMon(\Cat_1(\calV)) \,.
    \end{equation}
    It is possible to characterize the image of \eqref{eq:cocartfunctor}: It consists of those symmetric monoidal $\calV$-categories $(\mathsf{C},\otimes)$ such that $\mathbbm{1}\in \mathsf{C}$ is initial, $\otimes$ commutes with coproducts and such that the unique maps $c_i \to \mathbbm{1}\otimes \dots \otimes  c_i\otimes \dots \otimes \mathbbm{1} \to c_1 \otimes \dots \otimes c_n$ realize $c_1\otimes \dots \otimes c_n$ as the coproduct of $c_1,\dots ,c_n$.
    In this case we say that $(\mathsf{C},\otimes)$ carries a cocartesian monoidal structure.
    The forgetful functor 
    \begin{equation*}
        \CMon(\Cat_1(\calV))^{\mathrm{cocart}} \to \Cat_1^\Sigma(\calV)
    \end{equation*}
    is an equivalence, with explicit inverse \eqref{eq:cocartfunctor}.
    This connects it to \cite[Variant~2.4.3.12]{HA}.
\end{construction}

\paragraph{Pullbacks and the iso-comma category}

Given $\mathcal{V}$-functors $\mathsf{A} \xrightarrow{F} \mathsf{C} \xleftarrow{G} \mathsf{B}$, the \emph{iso-comma category}~\cite[Section 4]{MR998024} has objects triples $(a,b,\phi)$ with $a \in \mathsf{A}$, $b \in \mathsf{B}$ and $\phi \colon Fa \to Gb$ an isomorphism in the underlying category of $\mathsf{C}$, and hom-objects
\begin{equation*}
    (\mathsf{A}\times_{\mathsf{C}}\mathsf{B})\big((a,b,\phi),(a',b',\phi')\big)
    := \mathsf{A}(a,a') \times_{\mathsf{C}(Fa,Gb')} \mathsf{B}(b,b') \,,
\end{equation*}
the pullback in $\mathcal{V}$ along
\begin{align*}
\mathsf{A}(a,a') &\xrightarrow{F} \mathsf{C}(Fa,Fa') \xrightarrow{\phi'_*} \mathsf{C}(Fa,Gb')
\\
\mathsf{B}(b,b') &\xrightarrow{G} \mathsf{C}(Gb,Gb') \xrightarrow{\phi^*} \mathsf{C}(Fa,Gb') \,.
\end{align*}
This is an explicit model for the (\emph{weak}) pullback
\begin{equation*}
\begin{tikzcd}
    \mathsf{A} \times_{\mathsf{C}} \mathsf{B}\ar[r] \ar[d] & \mathsf{B} \ar[d,"G"]
    \\
    \mathsf{A} \ar[r,"F"] & \mathsf{C}
    \arrow["\lrcorner"{anchor=center, pos=0.125}, draw=none, from=1-1, to=2-2]
\end{tikzcd}
\end{equation*}
in the $(2,1)$-category $\Cat_1(\mathcal{V})$~\cite[Example 15]{MR2664622}.

\begin{remark}
    The underlying category of the iso-comma category is the iso-comma category of the underlying categories, since $\Hom_{\mathcal{V}}(\mathbbm{1},-)$ preserves limits.
    Conversely, the pullback of underlying categories only enhances to the pullback of $\mathcal{V}$-categories when the hom-sets are given the enrichment above.
\end{remark}

\begin{remark}
\label{rmk:isofibrationapproach}
    An alternative route to $(2,2)$-categorical pullbacks avoids the iso-comma construction: call a $1$-morphism $F\colon x \to y$ in a $(2,2)$-category $B$ an \emph{isofibration} if $B(z,x) \to B(z,y)$ is an isofibration of categories for every $z$~\cite{MR1223657}.
    A strict pullback in which one leg is an isofibration computes the $(2,1)$-categorical pullback, and the comparison functor from the iso-comma category is then an equivalence; for $B = \Cat_1$ this recovers the classical statement for the canonical model structure on categories~\cite{MR1173014}, and~\cite[Theorem 4.3]{MR2369168} treats general strict $2$-categories.
    We do not use this route: for the passage to $(\infty,1)$-categories we instead verify the fibrancy hypotheses of Lemma~\ref{lem:htpypullbackoftopcats} directly, which simultaneously handles the enrichment.
\end{remark}

\subsection{Lemmas about \texorpdfstring{$\infty$}{infinity}-categories}
\label{app:categories}

We freely use the theory of $(\infty,1)$-categories developed in~\cite{HTT} and~\cite{HA} and we made an effort for most arguments to be \emph{model-independent}, but if we need to be precise, we will use quasicategories.

We also use the language of $\infty$-operads in the ``category of operators'' model of~\cite{HA}.
Some facts we will repeatedly use are:
\begin{itemize}
    \item Let $\Fin_*$ denote the category of pointed finite sets. Let $\langle n\rangle$ denote the pointed set $\{*,1,\dots,n\}$ with basepoint $*$.
     An $\infty$-operad is a functor $\mathcal{O} \to \Fin_*$ satisfying the Segal condition $\mathcal{O}_{\langle n\rangle } \simeq \mathcal{O}_{\langle 1\rangle }^n$ and admitting cocartesian lifts of inert morphisms, where $\mathcal{O}_{\langle n\rangle }$ is the fiber of $\mathcal{O}$ over $\langle n\rangle$.
    \item There is a faithful functor $\Op$ from symmetric monoidal $(\infty,1)$-categories and symmetric monoidal functors to $\infty$-operads, whose image are precisely the $\infty$-operads which admit cocartesian lifts for \emph{all} morphisms in $\Fin_*$.
      \item A \emph{lax} symmetric monoidal functor $\calC\to \calD$ is the same as a map of operads $\Op(\calC)\to \Op(\calD)$.
    \item For $\calC,\calD$ symmetric monoidal $(\infty,1)$-categories, the category $\Fun(\calC,\calD)$ attains the Day convolution symmetric monoidal structure.
    There is an equivalence $\CAlg(\Fun(\calC,\calD)) \simeq \Fun^{\mathrm{lax}}(\calC,\calD)$ between commutative algebra objects with respect to Day convolution and lax symmetric monoidal functors.
\end{itemize}
The $n$-ary operations in the operad $\Op(\calC)$ associated to $(\mathcal{C},\otimes)$ are given by $\Hom(\bigotimes_i x_i, y)$ from some finite list $x_i \in \mathcal{C}$ to some $y \in \mathcal{C}$.

We will now use definitions and lemmas used throughout the paper.
\begin{definition}
    \label{def:commutativemonoid}
    Let $\calC$ be a symmetric monoidal $(\infty,1)$-category.
    A commutative monoid in $\calC$ is a functor $\Fin_* \to \calC$ satisfying the Segal condition $\prod_i \rho_i:F({\langle n\rangle }) \xrightarrow{\simeq} F({\langle 1\rangle })^n$.
    Here, $\rho_i:\langle n\rangle \to \langle 1\rangle$ is the inert map sending $i$ to $1$ and all other elements to the basepoint.
\end{definition}
\begin{lemma}
\label{lem:htpypullbackoftopcats}
Let $\Cat_1(\Top)$ denote the $1$-category of topologically enriched categories with the Dwyer--Kan/Bergner model structure, and let $\sSet_{\mathrm{Joyal}}$ denote the $1$-category of simplicial sets equipped with the Joyal model structure. Let
\begin{equation*}
\mathsf{C} \xrightarrow{f} \mathsf{E} \xleftarrow{g} \mathsf{D}
\end{equation*}
be a diagram in $\Cat_1(\Top)$ where $g$ is a fibration (i.e.\, $g$ is an isofibration and induces Serre fibrations on all topological mapping spaces). Let $N^{\mathrm{hc}}\colon \Cat_1(\Top) \to s\mathrm{Set}$ be the homotopy coherent nerve functor. Then the homotopy coherent nerve of the strict pullback is the homotopy pullback of quasicategories:
\begin{equation*}
N^{\mathrm{hc}}(\mathsf{C} \times_{\mathsf{E}} \mathsf{D}) \simeq N^{\mathrm{hc}}(\mathsf{C}) \times^h_{N^{\mathrm{hc}}(\mathsf{E})} N^{\mathrm{hc}}(\mathsf{D})
\end{equation*}
\end{lemma}
One can forgo the isofibration assumption by using the iso-comma category.
\begin{proof}
    The Bergner model structure on the category of simplicially enriched categories $\Cat_1(\sSet)$ is right proper~\cite[Prop.~3.5]{bergner07}.
    Fibrations are functors which are isofibrations and which induce fibrations on mapping simplicial sets.
    In right proper model categories a strict pullback along a fibration is a homotopy pullback.
    Since $\Sing\colon \Top\to \sSet$ is also a right Quillen functor, changing enrichment maps $g$ to a fibration in the Bergner model structure and $\mathsf{C} \times_{\mathsf{E}} \mathsf{D}$ is a homotopy pullback in $\Cat_1(\sSet)$.
    \cite[Thm.~2.2.5.1]{HTT} proves that $N^{\mathrm{hc}}\colon \Cat_1(\sSet) \to \sSet_{\mathrm{Joyal}}$ is the right Quillen functor in a Quillen equivalence.
    It will hence preserve homotopy pullbacks.
\end{proof}

Let $\calC,\calD$ be $(\infty,1)$-categories.
The \emph{coend} of a functor $F\colon \calC \times \calC^{\op} \to \calD$
is the colimit over the twisted arrow category $\mathrm{Tw}(\calC)$ of the composite functor $\mathrm{Tw}(\calC) \to \calC\times \calC^{\op} \to \calD$, written suggestively as
\begin{equation*}
    \int^{c\in \calC} F(c,c) =\colim_{\mathrm{Tw}(\calC)} F \,.
\end{equation*}
The \emph{end} of a functor $F\colon \calC^{\op}\times \calC \to \calD$ is the limit over the composite functor $\mathrm{Tw}(\calC)^{\op} \to \calC^{\op}\times \calC \to \calD$.
We refer to~\cite{laxcolimits17} for a more in-depth treatment.
\begin{lemma}
    \label{lem:pointwiseKancoend}
    Let $\gamma\colon \calC \to \calD$ and $F\colon \calC\to \calE$ be functors and let $\calE$ be complete.
    The right Kan extension $\gamma_*F$ of $F$ along $\gamma$ is computed as the end
    \begin{equation*}
        \gamma_*F(d) = \int_{c\in \calC} F(c)^{\Hom_{\calD}(d,\gamma(c))} \,,
    \end{equation*}
    where $(-)^K$ is the cotensoring of $\calE$ over $\Spc$.
    If dually $\calE$ is cocomplete,
    then the left Kan extension $\gamma_!F$ of $F$ along $\gamma$ is given by the formula
    \begin{equation*}
        \gamma_!F (d) = \int^{c\in \calC} F(c)\otimes \Hom_\calD(\gamma(c),d) \,,
    \end{equation*}
    where $\otimes$ denotes the tensoring of $\calE$ over $\Spc$.
    In particular, we have
    \begin{equation*}
        F(d) = \int^{c\in \calC} F(c)\otimes \Hom(c,d) \,.
    \end{equation*}
\end{lemma}
\begin{proof}
By a Yoneda argument it suffices to consider $\calE=\Spc$.
    From~\cite[Prop.~5.1]{laxcolimits17}, we have that mapping spaces of functors are computed as ends:
    \begin{equation*}
        \mathrm{Nat}(F,G)= \int_{\calC} \Hom_{\calC}(F(c),G(c)) \,.
    \end{equation*}
    We calculate using the Yoneda lemma:
    \begin{align*}
        \gamma_*F(d) &= \mathrm{Nat}(\Hom_\calD(d,-),\gamma_*F) \\
        &= \mathrm{Nat}(\Hom_{\calD}(d,\gamma(-)),F) \\
        &= \int_{\calC} \Hom_{\Spc}\left(\Hom_{\calD}(d,\gamma(c)),F(c)\right) \\
        &= \int_{\calC} F(c)^{\Hom_{\calD}(d,\gamma(c))} \,.
    \end{align*}
    The left Kan extension formula is obtained by passing to opposites.
\end{proof}
Let us recall the \emph{derived mapping space lemma}, which is proven in~\cite[Theorem~2.2]{arakawa2025derivedmappingspacesinftycategories}.
\begin{lemma}[Derived Mapping Space Lemma]
\label{lem:derivedmappingspace}
    Let $(\calC,W)$ be a relative category and $\gamma\colon \calC\to \calC[W^{-1}]$ its Dwyer--Kan localization.
    Let $Y\to Y_\bullet$ be a simplicial resolution, cf.\ Definition~\ref{def:resolution}. Then there is an equivalence
    \begin{equation*}
        \Hom_{\calC[W^{-1}]}(\gamma(X),\gamma(Y)) \xrightarrow{\simeq} \colim_{[n]\in \Delta^{\op}} \Hom(X,Y_n).
    \end{equation*}
\end{lemma}
\begin{remark}
    By Condition~\ref{cond:2} and the universal property of $\calC[W^{-1}]$, the functor $\colim_{\Delta^{\op}} \Hom_\calC(-,Y_\bullet)$ is in the image of
    \begin{equation*}
        \gamma^* :\Fun(\calC[W^{-1}]^{\op},\Spc) \to \Fun(\calC^{\op},\Spc) \,.
    \end{equation*}
    The left adjoint $\gamma_!$ satisfies $\gamma_!\Hom_\calC(-,X)=\Hom_{\calC[W^{-1}]}(-,\gamma(X))$.
    The image of $\gamma^*$ spans a full reflective subcategory and $\gamma^*\gamma_!$ is the reflector.
    Hence, the right vertical map in the following diagram is an equivalence:
    \begin{equation*}
        \begin{tikzcd}
            \Hom_{\calC}(-,Y) \ar[r]\ar[d,"\eta"] & \colim_{\Delta^{\op}}\Hom_{\calC}(-,Y_\bullet) \ar[d,"\eta","\simeq"'] \\
             \gamma^*\gamma_!\Hom_{\calC}(-,Y) \ar[r] &  \gamma^*\gamma_!\colim_{\Delta^{\op}}\Hom_{\calC}(-,Y_\bullet)
        \end{tikzcd} \,.
    \end{equation*}
    This supplies the natural comparison map in Lemma~\ref{lem:derivedmappingspace}.
\end{remark}
\begin{lemma}
    \label{lem:leftKanextension}
    Let $\gamma\colon \calC \to \calC[W^{-1}]$ be the localization of $(\calC,W)$, $\calD$ cocomplete, and $F\colon \calC \to \calD$ a functor.
    Let $A_\bullet\colon I\to \calC_{A/}$ be a resolution. 
    Then we have 
    \begin{equation*}
        \gamma^*\gamma_!F (A)\simeq \colim_I F(A_\bullet).
    \end{equation*}
\end{lemma}
\begin{proof}
    Using the Derived Mapping Space Lemma~\ref{lem:derivedmappingspace} and the pointwise formula for Kan extensions~\ref{lem:pointwiseKancoend}, we have 
    \begin{align*}
        \gamma_!F(\gamma(A))&\simeq \int^{B\in\calC} F(B)\otimes \Hom_{\calC[W^{-1}]}(\gamma(B),\gamma(A)) \\
        &\simeq \int^B F(B)\otimes \colim_I \Hom_{\calC}(B,A_\bullet) \\
        & \simeq \colim_I F(A_\bullet) \,.
    \end{align*}
\end{proof}

Any simplicially enriched category $\mathsf{C}$ determines a simplicial object in $\Cat_1$ where $[n]\mapsto j(\mathsf{C})_n$ is the category whose objects are those of $\mathsf{C}$ and whose set of morphisms $j(\mathsf{C})(c_0,c_1)$ is $\sSet(\Delta^n,\mathsf{C}(c_0,c_1))$.
\begin{lemma}
    \label{lem:homotopycoherentnerve}
    Let $\mathsf{C}$ be a fibrant simplicially enriched category.
    Then there is an equivalence of $(\infty,1)$-categories
    \begin{equation*}
        N^{\mathrm{hc}}(\mathsf{C}) \simeq \colim_{[n]\in\Delta^{\op}} j(\mathsf{C})_n
    \end{equation*}
\end{lemma}
\begin{proof}
    \cite[Prop.~1.3.4.14]{HA} proves that in the model structure on $\Cat_1(\sSet)$,
    $\mathsf{C}$ is the homotopy colimit of $j(\mathsf{C})_n \in \Cat_1\subseteq\Cat_1(\sSet)$.
    Using the fact that $N^{\mathrm{hc}}$ is the right adjoint in a Quillen equivalence, we obtain
    \begin{equation*}
        N^{\mathrm{hc}}(\mathsf{C})\simeq N^{\mathrm{hc}}(\colim_{\Delta^{\op}} j(\mathsf{C})) \simeq \colim_{\Delta^{\op}} N^{\mathrm{hc}}j(\mathsf{C})_n \simeq \colim_{\Delta^{\op}} j(\mathsf{C})_n \,,
    \end{equation*}
    where we identified the homotopy coherent nerve of a $1$-category with the inclusion $\Cat_1\subseteq \Cat_\infty$.
\end{proof}
\begin{lemma}\label{lem:geomrealizationmonoidal}
  Let $\mathcal{C}^{\otimes}$ be a symmetric monoidal $(\infty,1)$-category which admits
  geometric realizations, and suppose that $\otimes$ preserves them in each variable
  separately. Equip $\mathcal{C}^{\Delta^{\op}}$ with the pointwise symmetric monoidal
  structure. Then
  \begin{equation*}
    |-| \;=\; \colim_{\Delta^{\op}} \colon \mathcal{C}^{\Delta^{\op}} \longrightarrow \mathcal{C}
  \end{equation*}
  admits a canonical symmetric monoidal refinement.
\end{lemma}

\begin{proof}
  The constant-diagram functor $\const\colon\mathcal{C}\to\mathcal{C}^{\Delta^{\op}}$ is
  restriction along $\Delta^{\op}\to\ast$ and is therefore symmetric monoidal for the
  pointwise structure.
  By assumption it admits the left
  adjoint $|-|$. The left adjoint of a symmetric monoidal functor is canonically
  \emph{oplax} symmetric monoidal~\cite[Prop.~A]{zbMATH07785229}. It therefore suffices to show that the oplax structure maps are
  equivalences.

  The unit of $\mathcal{C}^{\Delta^{\op}}$ is $\const_{\mathbbm{1}}$ and the
  structure map is the canonical map $\colim_{\Delta^{\op}}\const_{\mathbbm{1}}\to\mathbbm{1}$, an
  equivalence because $\Delta^{\op}$ is weakly contractible.

  Let $X,Y\in\mathcal{C}^{\Delta^{\op}}$ and write
  $X\boxtimes Y\colon \Delta^{\op}\times\Delta^{\op}\to\mathcal{C}$ for the external
  tensor product. 
  Then $|X\otimes Y|\to |X|\otimes |Y|$ is an equivalence:
  \begin{align*}
      \colim_{\Delta^{\op}}X_n\otimes Y_n &\simeq \colim_{\Delta^{\op}} \delta^*(X\boxtimes Y)  \\
      &\simeq \colim_{\Delta^{\op}\times \Delta^{\op}} X\boxtimes Y \\
      &\simeq |X|\otimes |Y|.
  \end{align*}
  Here, the second equivalence uses that $\Delta^{\op}$ is sifted and the last equivalence uses that $\otimes$ commutes with colimits in each variable. 
\end{proof}

\subsection{Algebras and bimodules in a symmetric monoidal category}
\label{app:algebras}

We collect the formal properties of algebras and bimodules internal to a symmetric monoidal $1$-category $\mathrm{C}$ that are used throughout, specialized in the body of the paper to $\mathrm{C} = \Ban$, $\Ban_{\leq}$, $\sBan_{\leq}$, $\VectOne$ and $\sVectOne$.
We refer to~\cite{MR450361} for an elementary introduction, ~\cite{MR2534210} for the bicategorical statements and to~\cite{MR3590516} and references therein for a higher categorical treatment.

Throughout this subsection, $\mathrm{C}$ is a symmetric monoidal category with reflexive coequalizers which are preserved by $\otimes$ in each variable.
An \emph{algebra} in $\mathrm{C}$ is an object $A$ with an associative unital multiplication $A \otimes A \to A$, $\mathbbm{1} \to A$; algebras and their homomorphisms form a symmetric monoidal category $\Alg(\mathrm{C})$. 
The multiplication of $A \otimes B$ uses the braiding.
Given algebras $A,B$, an \emph{$(A,B)$-bimodule} is an object $M$ with commuting unital actions $A \otimes M \to M$ and $M \otimes B \to M$; bimodules and their maps form a category $\Mor(\mathrm{C})(B,A)$.
For an $(A,B)$-bimodule $M$ and a $(B,C)$-bimodule $N$, the \emph{relative tensor product} $M \otimes_B N$ is the coequalizer
\begin{equation}
\label{eq:generalreltensor}
\begin{tikzcd}
    M \otimes B \otimes N \ar[r,shift left,"\rho \otimes \id"]
    \ar[r,shift right,swap,"\id \otimes \lambda"]
    & M \otimes N \ar[r] & M \otimes_B N \,,
\end{tikzcd}
\end{equation}
which is reflexive with common section given by the unit $M \otimes N \to M \otimes B \otimes N$, and hence exists by assumption.
Because $\otimes$ preserves these coequalizers, $M \otimes_B N$ inherits the structure of an $(A,C)$-bimodule and the construction is associative up to coherent isomorphism.

\begin{lemma}[{\cite[Theorem 11.5]{MR2534210}}]
\label{lem:morexists}
Let $\mathrm{C}$ be a symmetric monoidal $1$-category with reflexive coequalizers which are preserved by $\otimes$ in each variable.
Then there is a symmetric monoidal bicategory $\Mor(\mathrm{C})$ with objects the algebras in $\mathrm{C}$, $1$-morphisms the bimodules, $2$-morphisms the bimodule maps, composition given by~\eqref{eq:generalreltensor}, identity $1$-morphism on $A$ given by $A$ as an $(A,A)$-bimodule, and monoidal structure induced by $\otimes$.
\end{lemma}

See~\cite[Example 11.6]{MR2534210} for the verification in the motivating case.
We write $M \colon A \rightsquigarrow B$ for a $1$-morphism of $\Mor(\mathrm{C})$, i.e.\ an object of the hom-category $\Mor(\mathrm{C})(A,B)$, a $(B,A)$-bimodule.

\paragraph{Self-enrichment}

Suppose now that $\mathrm{C}$ is in addition closed and has equalizers.
Then the hom-categories of $\Mor(\mathrm{C})$ are themselves enriched in $\mathrm{C}$: for $(A,B)$-bimodules $M, N$ the internal hom of bimodule maps is the equalizer
\begin{equation}
\label{eq:selfenrich}
    \underline{\Hom}_{A,B}(M,N) \longrightarrow \underline{\mathrm{C}}(M,N) \rightrightarrows \underline{\mathrm{C}}(A \otimes M \otimes B, N)
\end{equation}
of the two morphisms comparing the actions, where $\underline{\mathrm{C}}$ denotes the internal hom of $\mathrm{C}$.

\begin{definition}
    Let $(\mathrm{C}, \otimes,\beta)$ be a symmetric monoidal category and $A \in \mathrm{C}$ an algebra. The \emph{opposite algebra $A^{\op}$} of $A$ is the same object with multiplication $A^{\op} \otimes A^{\op} \to A^{\op}$ changed by the braiding $\beta$.
\end{definition}

The category $\Mor(\mathrm{C})(A,B)$ is equivalent to the category of $A^{\op} \otimes B$-modules.

\begin{lemma}[{\cite[Section 1]{lewismandell2007},     \cite[\S 4.2]{Vasilakopoulou2014}}
]
\label{lem:selfenrichment}
    Let $\mathrm{C}$ be a closed symmetric monoidal category with equalizers and let $A,B,C$ be algebras in $\mathrm{C}$.
    Then $\Mor(\mathrm{C})(B,A)$ is $\mathrm{C}$-enriched with hom-objects $\underline{\Hom}_{A,B}(M,N)$, and the relative tensor product
    \begin{equation*}
        -\otimes_B- \colon \Mor(\mathrm{C})(B,A) \times \Mor(\mathrm{C})(C,B) \longrightarrow \Mor(\mathrm{C})(C,A)
    \end{equation*}
    is a $\mathrm{C}$-enriched bifunctor.
    Moreover, $\Mor(\mathrm{C})(B,A)$ is tensored and cotensored over $\mathrm{C}$, by $M\otimes X$ and $\underline{\mathrm{C}}(X,M)$ respectively.
\end{lemma}
\begin{proof}
    An $(A,B)$-bimodule is the same thing as a left $A \otimes B^{\op}$-module, and under this identification $\underline{\Hom}_{A,B}(M,N)$ is the module function object of~\cite[Section 1]{lewismandell2007}, defined there also by the equalizer \eqref{eq:selfenrich}.
    That this construction enriches the module category over $\mathrm{C}$ is deduced there from the natural bijection between module maps $M \to N$ and morphisms $\mathbbm{1} \to \underline{\Hom}_{A,B}(M,N)$ in $\mathrm{C}$; the tensoring and cotensoring come from the parametrized adjunctions in loc.\ cit., and the enrichment of the relative tensor product is the statement there that $-\otimes_B-$ is a $\mathrm{C}$-enriched bifunctor.
\end{proof}

\begin{remark}
    A more general and clean version of Lemma \ref{lem:selfenrichment} is: $\Mor(\mathrm{C})$ is enriched in $\Cat_1(\mathrm{C})$. 
    We could not find this statement in the literature, and we won't need this generality in this paper.
\end{remark}

\begin{proposition}[{\cite{laxsymdaniel}}]
\label{prop:daniel}
For any symmetric monoidal $(2,2)$-category $B$, the assignment $b \mapsto \Hom_B(\mathbbm{1},b)$ defines a lax symmetric monoidal $2$-functor $B \to \Cat_1$ to the $(2,2)$-category of categories.
\end{proposition}

\paragraph{Bimodules induced by homomorphisms}

Every algebra homomorphism $f \colon A \to B$ induces two bimodules.
Writing $B_f$ for $B$ with its left $B$-action and with the right $A$-action $b \cdot a := b f(a)$, and ${}_f B$ for $B$ with the left $A$-action $a \cdot b := f(a) b$ and its right $B$-action, we obtain
\begin{equation*}
    B_f \colon A \rightsquigarrow B \,, \qquad
    {}_f B \colon B \rightsquigarrow A \,.
\end{equation*}
The first assignment is functorial and symmetric monoidal, giving $\Alg(\mathrm{C}) \to \Mor(\mathrm{C})$; the second is contravariant, giving $\Alg(\mathrm{C})^{\op} \to \Mor(\mathrm{C})$.
The evaluation $B_f \otimes_A {}_f B \to B$ exhibits $B_f$ as left adjoint to ${}_f B$. 
On module categories this is the usual adjunction between extension and restriction of scalars.

\begin{definition}
\label{def:bimgeneral}
    Let $\mathrm{C}$ be as in Lemma~\ref{lem:morexists}.
    We write $\Bim(\mathrm{C})$ for the symmetric monoidal $(2,1)$-category with
    \begin{itemize}
        \item the same objects as $\Mor(\mathrm{C})$, that is, the algebras in $\mathrm{C}$;
        \item those $1$-morphisms $M \colon A \rightsquigarrow B$ of $\Mor(\mathrm{C})$ which admit a right adjoint;
        \item the invertible $2$-morphisms between those.
    \end{itemize}
\end{definition}

This is well-defined: adjoints compose and the unit $\mathbbm{1}$ is self-adjoint, so $\Bim(\mathrm{C}) \subseteq \Mor(\mathrm{C})$ is closed under composition and under $\otimes$.

\begin{remark}
\label{rmk:onesidedrestriction}
Since $B_f$ admits the right adjoint $_f B$ for every $f$, while ${}_f B$ in general does not, the covariant functor $\Alg(\mathrm{C}) \to \Bim(\mathrm{C})$ survives in general.
However, the contravariant one only restricts to a wide subcategory of $\Alg(\mathrm{C})$, thus Definition \ref{def:bimgeneral} breaks the symmetry of these two constructions.
\end{remark}

\bibliographystyle{alpha}
\bibliography{biblio.bib}{}

\end{document}

%% file: preamble.tex
\usepackage[utf8]{inputenc}
\usepackage[T1]{fontenc}

\usepackage{amsmath}
\usepackage{amsfonts}
\usepackage{amssymb}
\usepackage{amsthm}
\usepackage{mathtools}
\usepackage{bbm}

\usepackage{enumitem}

\usepackage{tikz-cd}
\usepackage{tikz}
\usetikzlibrary{decorations.markings}
\usetikzlibrary{decorations.pathmorphing}
\usetikzlibrary{cd}

\usepackage{tabularx}
\usepackage{booktabs}
\usepackage{array}
\usepackage{makecell}
\usepackage{geometry}
\usepackage{hyperref}

\newcommand{\Tor}{\operatorname{Tor}}

\newcommand{\Spec}{\operatorname{Spec}}

\newcommand{\Aut}{\operatorname{Aut}}
\newcommand{\Hom}{\operatorname{Hom}}
\newcommand{\Mor}{\operatorname{Mor}}
\newcommand{\Morita}{\sim}
\newcommand{\notMorita}{\nsim}
\newcommand{\End}{\operatorname{End}}
\DeclareMathOperator{\GL}{GL}
\newcommand{\pt}{\operatorname{pt}}
\newcommand{\ModOne}{\mathrm{Mod}}
\newcommand{\ModInf}{\mathrm{Mod}}
\newcommand{\ModEn}{\mathsf{Mod}}
\newcommand{\LMod}{\operatorname{LMod}}

\newcommand{\id}{\operatorname{id}}

\newcommand{\Ext}{\operatorname{Ext}}
\newcommand{\sVectOne}{\mathrm{sVect}}
\newcommand{\sVectEn}{\mathsf{sVect}}

\newcommand{\VectOne}{\mathrm{Vect}}
\newcommand{\VectEn}{\mathsf{Vect}}

\newcommand{\ev}{\operatorname{ev}}
\newcommand{\coev}{\operatorname{coev}}

\newcommand{\sLine}{\operatorname{sLine}}

\newcommand{\tr}{\operatorname{tr}}
\newcommand{\Fun}{\operatorname{Fun}}

\newcommand{\Pic}{\operatorname{Pic}}

\newcommand{\ob}{\operatorname{ob}}

\newcommand{\op}{\operatorname{op}}
\newcommand{\flop}{\operatorname{flop}}
\newcommand{\fib}{\operatorname{fib}}
\newcommand{\cofib}{\operatorname{cofib}}
\DeclareMathOperator{\Cl}{Cl}
\DeclareMathOperator{\Cxl}{\mathbb{C}\mathrm{l}}
\newcommand{\Fin}{\mathrm{Fin}}
\newcommand{\ABS}{\mathrm{ABS}}
\newcommand{\Kar}{\mathrm{Kar}}

\newcommand{\Z}{\mathbb{Z}}
\newcommand{\N}{\mathbb{N}}

\newcommand{\C}{\mathbb{C}}
\newcommand{\R}{\mathbb{R}}
\newcommand{\E}{\mathbb{E}}

\newcommand{\calC}{\mathcal{C}}
\newcommand{\calD}{\mathcal{D}}
\newcommand{\calE}{\mathcal{E}}

\newcommand{\calV}{\mathcal{V}}

\newcommand{\grotimes}{\hat{\otimes}}

\newcommand{\im}{\operatorname{Im}}

\newcommand{\BanAlgInf}{\mathrm{Alg}(\Ban)_{\infty}}
\newcommand{\BanAlgEn}{\mathsf{Alg(Ban)}}

\newcommand{\sBanAlgInf}{\mathrm{Alg}(\sBan)_{\infty}}
\newcommand{\sBanAlgEn}{\mathsf{Alg(sBan)}}
\newcommand{\sBanAlgOne}{\mathrm{Alg}(\sBan)}
\newcommand{\BanCat}{\mathrm{BanCat}}
\newcommand{\BimBanInf}{\Bim(\Ban)_{\infty}}
\newcommand{\BimBanOne}{\Bim(\Ban)}
\newcommand{\sBimBanInf}{\Bim(\sBan)_{\infty}}
\newcommand{\sBimBanOne}{\Bim(\sBan)}
\newcommand{\Bim}{\operatorname{Bim}}
\newcommand{\const}{\mathrm{const}}
\newcommand{\gr}{\textrm{gr}}

\newcommand{\Op}{\operatorname{Op}}

\newcommand{\ThmB}{\hyperlink{thmB}{Theorem~B}}
\newcommand{\ThmC}{\hyperlink{thmC}{Theorem~C}}

\newcommand{\mastercomparisonbody}{%
\centering
\small
\setlength{\tabcolsep}{3pt}%
\renewcommand{\arraystretch}{1.35}%
\begin{tabularx}{\textwidth}{
  @{}
  >{\raggedright\arraybackslash}p{0.20\textwidth}
  >{\raggedleft\arraybackslash}l
  @{\;}c@{\;}
  >{\raggedright\arraybackslash}X
  >{\raggedleft\arraybackslash}l
  @{\;}c@{\;}
  >{\raggedright\arraybackslash}X
  @{}
}
\toprule
&
\multicolumn{3}{c}{$\Cl_{+1}$}
&
\multicolumn{3}{c}{$\Cl_{-1}$}
\\
\cmidrule(lr){2-4}
\cmidrule(l){5-7}

covariant cofiber
&
$K_A^{\mathrm{gr}}$
&
$\cong$
&
$\cofib\bigl(
K_{|A|}\to K_{|A\grotimes\Cl_{+1}|}
\bigr)$

\scriptsize (ABS)
&
$\Omega^{-2}K_{A^{\op}}^{\mathrm{gr}}$
&
$\cong$
&
$\cofib\bigl(
K_{|A|}\to K_{|A\grotimes\Cl_{-1}|}
\bigr)$
\\

covariant fiber
&
$\Omega K_A^{\mathrm{gr}}$
&
$\cong$
&
$\fib\bigl(
K_{|A|}\to K_{|A\grotimes\Cl_{+1}|}
\bigr)$

\scriptsize (van Daele)
&
$\Omega^{-1}K_{A^{\op}}^{\mathrm{gr}}$
&
$\cong$
&
$\fib\bigl(
K_{|A|}\to K_{|A\grotimes\Cl_{-1}|}
\bigr)$
\\

\midrule

contravariant cofiber
&
$\Omega^{-1}K_{A^{\op}}^{\mathrm{gr}}$
&
$\cong$
&
$\cofib\bigl(
K_{|A\grotimes\Cl_{+1}|}\to K_{|A|}
\bigr)$
&
$\Omega K_A^{\mathrm{gr}}$
&
$\cong$
&
$\cofib\bigl(
K_{|A\grotimes\Cl_{-1}|}\to K_{|A|}
\bigr)$
\\

contravariant fiber
&
$K_{A^{\op}}^{\mathrm{gr}}$
&
$\cong$
&
$\fib\bigl(
K_{|A\grotimes\Cl_{+1}|}\to K_{|A|}
\bigr)$

\scriptsize (Karoubi)
&
$\Omega^{2}K_A^{\mathrm{gr}}$
&
$\cong$
&
$\fib\bigl(
K_{|A\grotimes\Cl_{-1}|}\to K_{|A|}
\bigr)$
\\

\bottomrule
\end{tabularx}
}
\newcommand{\sslash}{/\!\!/}

\newcommand{\Free}{\operatorname{Free}}
\newcommand{\Idem}{\operatorname{Idem}}
\newcommand{\Ring}{\operatorname{Ring}}

\newcommand{\Ho}{\mathrm{Ho}}

\newcommand{\Top}{\operatorname{Top}}

\newcommand{\sBan}{\mathrm{sBan}}

\newcommand{\CStarAlg}{\operatorname{C}^*\!\!\operatorname{Alg}}
\newcommand{\Cat}{\operatorname{Cat}}
\newcommand{\Spc}{\operatorname{Spc}}

\newcommand{\Sp}{\operatorname{Sp}}
\DeclareMathOperator*{\colim}{colim}

\newcommand{\Grpd}{\operatorname{Grpd}}
\newcommand{\CMon}{\operatorname{CMon}}
\newcommand{\CGrp}{\operatorname{CGrp}}
\newcommand{\Alg}{\operatorname{Alg}}
\newcommand{\Ban}{\mathrm{Ban}}
\newcommand{\Set}{\operatorname{Set}}
\newcommand{\sSet}{\operatorname{sSet}}
\newcommand{\InnInf}{\mathrm{Inn}_{\infty}}
\newcommand{\InnEn}{\mathsf{Inn}}
\newcommand{\CHaus}{\mathrm{CHaus}}
\newcommand{\CAlg}{\mathrm{CAlg}}
\DeclareMathOperator{\Sing}{Sing}

\newcommand{\gp}{\mathrm{gp}}
\newcommand{\rank}{\mathrm{rank}}

\newcommand{\fl}{\operatorname{fl}}

\newtheorem{theorem}{Theorem}[section]
\newtheorem*{theorem*}{Theorem}
\newtheorem*{theoremA*}{Theorem~A}
\newtheorem*{theoremB*}{Theorem~B}
\newtheorem*{theoremC*}{Theorem~C}

\newtheorem{proposition}[theorem]{Proposition}
\newtheorem{lemma}[theorem]{Lemma}

\newtheorem{corollary}[theorem]{Corollary}

\theoremstyle{definition}

\newtheorem{definition}[theorem]{Definition}

\theoremstyle{remark}

\newtheorem{remark}[theorem]{Remark}
\newtheorem{convention}[theorem]{Convention}
\newtheorem{notation}[theorem]{Notation}
\newtheorem{example}[theorem]{Example}
\newtheorem{warning}[theorem]{Warning}

\definecolor{Blue} {rgb} {0.282352,0.239215,0.803921}
\definecolor{Green} {rgb} {0.133333,0.545098,0.133333}

\newcounter{jfc}

\makeatletter
\newcommand{\nameditem}[1]{%
    \item[#1]% Creates the description item
    \protected@edef\@currentlabel{#1}% Tells \ref to use this exact text
}
\makeatother

\newtheorem{construction}[theorem]{Construction}

\makeatletter
\renewcommand\paragraph{\@startsection{paragraph}{4}%
  \z@{.5\linespacing\@plus.7\linespacing}{-\fontdimen2\font}%
  {\normalfont\bfseries}}
\makeatother